\documentclass[11pt]{amsart}
\usepackage[dvipsnames]{xcolor}
\usepackage{amsfonts,amsthm,amsmath,amssymb,etoolbox,comment,graphicx,array,tabu,commath,latexsym,environ,tikz,tikz-cd,stackengine,diagbox,caption}
\usepackage{mathrsfs}
\usepackage{soul}
\usepackage{eso-pic}
\usepackage[toc,page]{appendix}
\usepackage{enumitem}
\usepackage{mathtools}

\usepackage[colorlinks, linkcolor = blue, anchorcolor = blue, citecolor = blue]{hyperref}

\newcommand{\fg}{\mathfrak{g}}
\newcommand{\cA}{\mathcal{A}}

\newcommand{\cC}{\mathcal{C}}
\newcommand{\cD}{\mathcal{D}}

\newcommand{\cF}{\mathcal{F}}

\newcommand{\cH}{\mathcal{H}}

\newcommand{\cK}{\mathcal{K}}

\newcommand{\cM}{\mathcal{M}}

\newcommand{\cO}{\mathcal{O}}

\newcommand{\cR}{\mathcal{R}}
\newcommand{\cS}{\mathcal{S}}
\newcommand{\cT}{\mathcal{T}}

\newcommand{\cV}{\mathcal{V}}

\newcommand{\cW}{\mathcal{W}}
\newcommand{\cX}{\mathcal{X}}

\newcommand{\cZ}{\mathcal{Z}}

\newcommand{\bB}{\mathbf{B}}
\newcommand{\bC}{\mathbf{C}}

\newcommand{\bF}{\mathbf{F}}

\newcommand{\bI}{\mathbf{I}}

\newcommand{\bL}{\mathbf{L}}
\newcommand{\bM}{\mathbf{M}}

\newcommand{\RP}{{\mathbb{RP}}}
\newcommand{\R}{\mathbb R}

\newcommand{\Z}{\mathbb Z}
\newcommand{\N}{\mathbb N}
\newcommand{\Zc}{\mathcal Z}

\newcommand{\Fc}{\mathcal F}
\newcommand{\Pc}{\mathcal P}

\newcommand{\Fb}{\mathbf F}
\newcommand{\Mb}{\mathbf M}
\newcommand{\id}{\mathrm {id}}

\newcommand{\area}{\operatorname{area}}

\newcommand{\dist}{\operatorname{dist}}
\newcommand{\dmn}{\mathrm{dmn}}

\newcommand{\diam}{\mathrm{diam}}

\renewcommand{\tilde}{\widetilde}

\newcommand{\rank}{\operatorname{rank}}
\newcommand{\link}{\operatorname{link}}
\newcommand{\spt}{\operatorname{spt}}
\newcommand{\vol}{\operatorname{vol}}

\newcommand{\x}{\times}

\newcommand{\ins}{\mathrm{in}}
\newcommand{\out}{\mathrm{out}}

\newcommand{\graph}{\mathrm{graph}}
\newcommand{\Ker}{\mathrm{ker}}
\newcommand{\Ric}{\operatorname{Ric}}

\newcommand{\Tan}{\operatorname{Tan}}

\newcommand{\injrad}{\mathrm{injrad}}

\newcommand{\GS}{\mathcal{S}}

\newcommand{\Is}{\mathfrak{Is}}
\newcommand{\VC}{\cV\cC}

\newcommand{\va}{\mathbf{a}}
\newcommand{\vb}{\mathbf{b}}
\newcommand{\bd}{\mathbf{d}}
\newcommand{\vu}{\mathbf{u}}
\newcommand{\vv}{\mathbf{v}}

\usepackage[style=alphabetic, backend=biber, sorting=nyt, url=false ]{biblatex}
\title[Five Minimal tori in 3-spheres of positive Ricci curvature]{Existence of 5 minimal tori in 3-spheres of positive Ricci curvature}
\author{Adrian Chun-Pong Chu}\author{Yangyang Li}

\address{Cornell University, Ithaca, NY 14853, USA}
\email{cc2938@cornell.edu}
\address{The University of Chicago, Department of Mathematics, Eckhart Hall,
5734 S University Ave,
Chicago, IL, 60637}
\email{yangyangli@uchicago.edu}

\date{\today}

\newtheorem{alphatheorem}{Theorem}[section]

\numberwithin{equation}{section}
\newtheorem{thm}{Theorem}[section]
\newtheorem{cor}[thm]{Corollary}
\newtheorem{prop}[thm]{Proposition}
\newtheorem{lem}[thm]{Lemma}

\newtheorem{claim}[thm]{Claim}

\theoremstyle{definition}
\newtheorem{defn}[thm]{Definition}

\newtheorem{exmp}[thm]{Example}
\newtheorem{rmk}[thm]{Remark}

\newtheorem*{question*}{Question}

\begin{document}

\begin{abstract}
    In 1989, B. White conjectured that every Riemannian 3-sphere contains at least 5 embedded minimal tori. We confirm this conjecture for  3-spheres of positive Ricci curvature. While our proof uses min-max theory, the underlying heuristics are largely inspired by mean curvature flow.
\end{abstract}
\maketitle
\setcounter{tocdepth}{1}
\tableofcontents

\section{Introduction}\label{sect:intro}

    While enumerative geometry is traditionally a branch of algebraic geometry, there are, in fact, enumerative problems in differential geometry as well. 
    
    In 1905, Poincar\'e conjectured the existence of $3$ simple closed geodesic in every smooth, Riemannian 2-sphere~\cite{Poi05}. The number $3$ is optimal, as is shown in the example of an ellipsoid by Morse~\cite{Mor96}. In 1917, G. D. Birkhoff~\cite{Bir17} initially confirmed the existence of an (immersed) closed geodesic using a min-max argument, which can be further refined to show that the geodesic is simple. A universally accepted solution to Poincar\'e's conjecture was later given by M. Grayson in 1989 through curve shortening flow~\cite{Gra89}, building upon the earlier work of Lyusternik–Schnirelmann~\cite{LS29}.
    
    In one dimension higher, S.-T. Yau~\cite{Yau82} conjectured that there exist at least $4$ embedded minimal $2$-spheres in every Riemannian $3$-sphere.  Simon-Smith~\cite{Smi83} showed that there is always at least one embedded minimal $2$-sphere using min-max theory. More recently, Wang-Zhou~\cite{WZ23} confirmed Yau's conjecture for the case of metrics with positive Ricci curvature and bumpy metrics  (see also the work of Haslhofer-Ketover~\cite{HK19}). Furthermore, B. White~\cite{Whi89, Whi91} conjectured that every Riemannian $3$-sphere contains at least $5$ embedded minimal tori. In fact, he proved that every $3$-sphere of positive Ricci curvature contains at least one embedded minimal torus using a degree theory he developed alongside Hamilton's Ricci flow.

    In this paper, we confirm B. White's conjecture for 3-spheres with positive Ricci curvature using min-max theory:

    \begin{alphatheorem}\label{thm:main}
        Every smooth Riemannian $3$-sphere of positive Ricci curvature contains at least $5$ embedded minimal tori.
    \end{alphatheorem}
    
    Let us briefly explain where the number $5$ comes from. 
    
    The well-known Willmore conjecture~\cite{Wil65, Wil71}, which was proved by Marques-Neves~\cite{MN14}, implies that in the unit 3-sphere $\mathbb S^3\subset \R^4$, the \emph{Clifford torus}
    \[
        T^2 = \left\{x \in \mathbb{R}^4 : x^2_1 + x^2_2 = x^2_3 + x^2_4 = \frac{1}{2}\right\}\,,
    \]
    up to rigid motions in $\mathbb S^3$, is the unique embedded minimal surface with the second smallest area, after the equator. Moreover, the resolution of Lawson's conjecture on minimal tori~\cite{Law70} by S. Brendle~\cite{Bre13} indicates that the Clifford torus is the only embedded minimal torus in the $\mathbb S^3$. Consequently, the space of embedded minimal tori in $\mathbb{S}^3$ is exactly the space $\cC$ of Clifford tori, which is homeomorphic to $\mathbb{RP}^2 \times \mathbb{RP}^2$.  
    
    Let us also denote $\mathbb{S}^3$ by $(S^3, \bar g)$, where $\bar g$ is the Riemannian metric induced from $\mathbb{R}^4$. If we perturb the metric $\bar g$, the area functional defined on $\cC$ changes. However, the {\it Lyusternik-Schnirelmann number} of $\RP^2\x\RP^2$ is $5$, meaning  every smooth real function on $\RP^2\x\RP^2$ has at least $5$ critical points. This is due to the theorem that the Lyusternik-Schnirelmann number of a space is always strictly greater than its {\it cup length} in $\Z_2$-coefficients (the maximum number of first cohomology classes whose cup product is non-zero), and the topological fact that the cup length of $\RP^2\x\RP^2$ is $4$. Based on these facts, B. White proved that every $3$-sphere equipped with a metric sufficiently close to the round metric contains at least $5$ embedded minimal tori. 
    
    It is worth noting that the topological fact that the cup length of $\RP^2\x\RP^2$ is $4$ will also play a crucial role in our proof of Theorem~\ref{thm:main}.

    For more works on the construction of geodesics or minimal surfaces with controlled topological types, we refer to the works ~\cite{Str84, GJ86, Zho16, HK23, Ko23a, Ko23b}. Recently, X. Li and Z. Wang also have an independent work on the existence of minimal tori \cite{LW24}.


\subsection{Heuristics for the proof}\label{subsect:heuristicsProof}

    Let us describe the heuristics behind the proof of Theorem~\ref{thm:main}. It is worth noting that while the heuristics rely on mean curvature flow, we do not actually use mean curvature flow in practice. Instead, we primarily use Simon-Smith min-max theory, incorporating some concepts from  Almgren-Pitts min-max theory.

    Let $(S^3,g_0)$ be a Riemannian $3$-sphere of positive Ricci curvature as given in Theorem~\ref{thm:main}. In their proof of the Willmore conjecture~\cite{MN14}, F. C. Marques and A. Neves defined a $5$-parameter family of \emph{flat cycles}, known as the {\em canonical family}, each associated to a closed embedded surface in the unit $3$-sphere. Building upon their construction, we define in $(S^3,g_0)$ a 9-parameter family $\Psi$ of surfaces of genus $1$ or $0$, possibly with singularities; see details in \S \ref{subsubsect:introPsi}.
    
    Now, we apply mean curvature flow to each member of $\Psi$. Recall that given a smooth, embedded surface $\Sigma\subset M$, {\it the mean curvature flow starting from $\Sigma$} is a family of surfaces $\{\Sigma_t\}_{t\in[0,T)}$ in $M$, with $\Sigma_0=\Sigma$, satisfying the equation
    \begin{equation*}
        \partial_t x=\mathbf H(x),
    \end{equation*}
    where $x$ is the position vector and $\mathbf H$ is the mean curvature vector. Moreover, mean curvature flow can be viewed as the gradient flow of the area functional.
    
    Following the heuristics from gradient flow in the finite-dimensional setting, we assume all such mean curvature flows exist uniquely for all time until the surfaces vanish. Although this assumption is likely false in reality due to the possibility of {\em fattening} (i.e. non-unique solutions of mean curvature flow~\cite{IW24}), we will ignore this complication in our heuristic argument. Thus, as $t\to +\infty$, we expect the family $\Psi$ to evolve into some family $\Xi$ that is ``optimal", in the sense that $\Xi$ contains minimal tori or minimal spheres, as well as  {\it eternal or ancient mean curvature flows connecting these minimal surfaces to one another or to the empty set}. These mean curvature flows act as gradient flow lines connecting critical points (see Figure~\ref{fig:pulltight}). Note that we assume all minimal tori or minimal spheres appearing in $\Xi$ have multiplicity one: This assumption is heuristically justified, as it is conjectured that under the condition of positive Ricci curvature, long time limits of mean curvature flow have multiplicity one~\cite[Conjecture 1.3]{CS24}.

    \begin{figure}[!ht]
        \centering
        \makebox[\textwidth][c]{\includegraphics[width=4in]{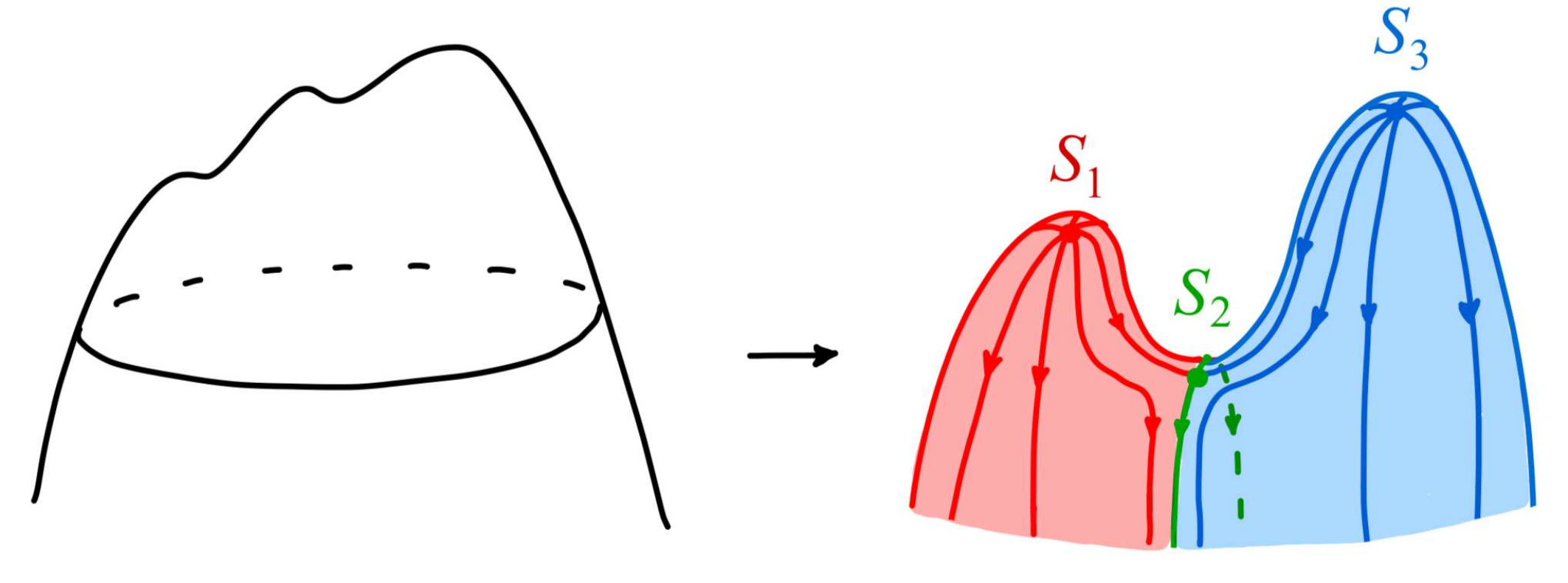}}
        \caption{The figure on the left shows the original family $\Psi$. The figure on the right shows $\Xi$, with the minimal surfaces $S_1,S_2,S_3$ detected, and ancient mean curvature flow lines originating from them. The red part is $T(S_1)$, the green part is $T(S_2)$, and the blue part is $T(S_3)$. }
        \label{fig:pulltight}
    \end{figure}

    For each minimal surface $S$ detected by $\Xi$, there are eternal or ancient mean curvature flows originating from $S$. Let $T(S)$ denote the collection of all time slices in such mean curvature flows. Note that $T(S)$ is a subset of $\Xi$, and can be viewed as  the ``unstable manifold" for $S$ (see Figure~\ref{fig:pulltight}). Moreover, $\Xi=\bigcup_S T(S)$ where $S$ ranges over all minimal surfaces (of genus $0$ or $1$) detected by $\Xi$. If $S$ is a minimal sphere, then each element of $T(S)$ should have genus $0$, possibly with singularities, since mean curvature flow does not increase genus~\cite{BK24}.

    To prove Theorem~\ref{thm:main}, suppose for the sake of contradiction that there are only $4$ minimal tori $T_1,T_2,T_3,T_4$ detected by $\Xi$. The case with fewer than $4$ minimal tori can be handled similarly. Then, letting $D_0:=\bigcup_S T(S)$ where $S$ ranges over the minimal {\em spheres} detected by $\Xi$, and $D_i:=T(T_i)$ for $i=1,\cdots,4$, we have
    \begin{equation}\label{eq:XiUnion}
        \Xi=D_0\cup D_1\cup D_2\cup D_3\cup D_4\,.
    \end{equation}

    Next, let us describe a crucial topological fact about $\Psi$. In fact, ensuring this topological fact holds is the main reason we constructed $\Psi$ using the Marques-Neves canonical family. Recall that the cycle space $\cZ_2(S^3;\Z_2)$ is weakly homotopy equivalent to $\RP^\infty$, and thus, the cohomology ring $H^*(\cZ_2(S^3;\Z_2);\Z_2)$ can be written as $\Z_2[\bar\lambda]$, where $\bar\lambda$ denotes the unique non-trivial element in $H^1(\cZ_2(S^3;\Z_2);\Z_2)$. Since each member of $\Psi$ can be viewed as a mod $2$ cycle, we can  view $\Psi$ as a subset of $\cZ_2(S^3;\Z_2)$. Let $\lambda$ be the element in $H^1(\Psi;\Z_2)$ induced by $\bar\lambda$.  The crucial topological fact is that, there exist $\alpha,\beta\in H^1(\Psi;\Z_2)$ such that
    \begin{equation}\label{eq:cupLength9}
        \lambda^5\cup\alpha^2\cup\beta^2\ne 0\,.
    \end{equation}

    By construction, $\Xi$ and $\Psi$ are homotopic to each other as sweepouts, so we can view $\lambda,\alpha$ and $\beta$ as elements of $H^1(\Xi;\Z_2)$ as well. Hence, by an elementary Lyusternik–Schnirelmann argument, \eqref{eq:XiUnion} and \eqref{eq:cupLength9} together immediately imply that one of the following must hold:
    \begin{enumerate}
        \item The element $\lambda^5|_{D_0}$ in $ H^5(D_0;\Z_2)$ induced by $\lambda^5$ is non-zero.
        \item The element $\alpha|_{D_1}$ in $ H^1(D_1;\Z_2)$ induced by $\alpha$ is non-zero.
        \item The element $\alpha|_{D_2}$ in $ H^1(D_2;\Z_2)$ induced by $\alpha$ is non-zero.
        \item The element $\beta|_{D_3}$ in $ H^1(D_3;\Z_2)$ induced by $\beta$ is non-zero.
        \item The element $\beta|_{D_4}$ in $ H^1(D_4;\Z_2)$ induced by $\beta$ is non-zero.
    \end{enumerate}
    And to derive a contradiction, it suffices to show that statements (1) to (5) are all impossible.

    To show that (1) is impossible, first note that by construction, $D_0$ consists entirely of genus $0$ surfaces, possibly with singularities. Now, let us consider $D_0$ as a family in the unit round $3$-sphere $\mathbb{S}^3$, instead of the original $3$-sphere $(S^3,g_0)$. Note  the following two facts:
    \begin{itemize}
        \item If we apply the Simon-Smith min-max theorem to $D_0$ in $\mathbb{S}^3$, then by Wang-Zhou's multiplicity one theorem~\cite{WZ23}, we detect the multiplicity-one equatorial $2$-sphere, which has area $4\pi$. 
        \item On the other hand, if (1) holds, meaning $\lambda^5|_{D_0}\ne 0$, then by definition, $D_0$ is a $5$-sweepout in the sense of the Almgren-Pitts min-max theory. Since the Simon-Smith width must be greater than or equal to the Almgren-Pitts width, the Simon-Smith width of the family $D_0$ (viewed in $\mathbb{S}^3$) is at least the $5$-width of $\mathbb{S}^3$, which is $2\pi^2$ by C. Nurser~\cite{Nur16}.
    \end{itemize}
    The above two facts are contradictory, so (1) is impossible.

    To show that (2) is impossible, note that by backtracking the mean curvature flow lines, $D_1$ can be deformation retracted to the single minimal torus $\{T^1\}$. Thus, $H^1(D_1;\Z_2)$ is zero, making (2) is impossible. Similarly, (3) through (5) are also impossible. This concludes the proof of Theorem~\ref{thm:main}.

\subsection{Details behind the  heuristics.}\label{subsect:mainIdea} 
    The above heuristic argument contains several details and inaccuracy that need to be addressed.

\subsubsection{The $9$-parameter family $\Psi$}\label{subsubsect:introPsi}
    Let us now explain the construction of the family $\Psi$. Let $\cC$ (resp. $\tilde \cC$) denote the set of unoriented (resp. oriented) Clifford tori in the unit 3-sphere $\mathbb S^3$. Note that $\cC\cong\RP^2\x\RP^2$, and $\tilde \cC$ is a double cover of this space. For each $\Sigma\in\tilde\cC$, based on Marques-Neves' canonical family, C. Nurser in~\cite{Nur16} constructed a 5-sweepout
    \[
        \Phi^\Sigma_5:\RP^5\to\cZ_2(\mathbb S^3;\Z_2)\,.
    \]    
    Here is his construction: The set of conformal diffeomorphisms of $\mathbb S^3$ can be parametrized by the open unit 4-ball $\mathbb B^4$, which we denote by $\{F_v\}_{v\in {\mathbb B^4}}$. For each image $F_v(\Sigma)$, which is oriented, we consider the level surfaces of the {\em signed} distance function to $F_v(\Sigma)$. This gives us a $5$-parameter family of surfaces, possibly with singularities and possibly empty, parametrized by ${\mathbb B^4}\x [-\pi,\pi]$, since the set of possible signed distances from $F_v(\Sigma)$ lies in the range $[-\pi,\pi]$.
    
    Marques-Neves~\cite[\S 5]{MN14} showed that by ingeniously reparametrizing this family {\em near} the boundary of its parameter space, one can continuously extend the family to the boundary, obtaining a $\overline{{\mathbb B^4}}\x [-\pi,\pi]$-family. Moreover, any two ``antipodal" points on the boundary of this parameter space can be identified, as they represent the same surface (but with opposite orientation), thereby yielding an $\RP^5$-family $\Phi^\Sigma_5$. 
    
    Repeating this for every oriented Clifford torus $\Sigma$, we obtain an $\RP^5\x\tilde\cC$-family. The subfamilies $\Psi^\Sigma_5$ and  $\Psi^{-\Sigma}_5$ have the same images for each $\Sigma\in\tilde \cC$, where $-\Sigma$ denotes $\Sigma$ with an opposite orientation. By identifying all such pairs of subfamilies $\Psi^\Sigma_5$ and  $\Psi^{-\Sigma}_5$, we obtain a family $\Psi$ with parameter space $Y$, which is an $\RP^5$-bundle over $\cC\cong\RP^2\x\RP^2$. This family consists only of genus $0$ and genus $1$ surfaces, possibly with singularities, and is suitable for running the Simon-Smith min-max process. 

\subsubsection{Overcoming the issue of fattening}\label{subsubsect:fatteningIssue}
    In the heuristic argument of \S \ref{subsect:heuristicsProof}, we proposed using mean curvature flow to obtain an ``optimal family" $\Xi$ from $\Psi$. In practice, the authors were unable to implement this strategy due to the phenomenon of fattening. Indeed, if any member of $\Psi$ fattens along the mean curvature flow, then the flow no longer generates a canonical deformation that is continuous across the entire family $\Psi$ of initial conditions. As a result, it is unclear how one might obtain the desired family $\Xi$ using this approach.

    Instead of using mean curvature flow, we employ a scheme of repeatedly applying the Simon-Smith min-max to $\Psi$ in order to obtain $\Xi$. Furthermore, we will define a process called {\it pinch-off process} that resembles mean curvature flow, so that each member $\Psi(x)$ will deform under this pinch-off process to some member $\Xi(x)$. Crucially, this process is genus-non-increasing in time. More precisely, a pinch-off process involves three types of behavior: (1) isotopy, (2) neck-pinch surgery, (3) shrinking some components in points: See Figure \ref{fig:pinchOff2}.  The ``genus-non-increasing" nature of pinch-off process (discussed in \S \ref{min-max_iiiOffAndMCF}) should be compared with the analogous results on mean curvature flow by B. White~\cite{Whi95}. At the end, although the family $\Xi$ we obtain may not be as well-behaved as to be considered ``optimal" in the sense described in \S \ref{subsect:heuristicsProof}, it is sufficient for our purposes.  
    \begin{figure}[h]
        \centering
        \makebox[\textwidth][c]{\includegraphics[width=2.5in]{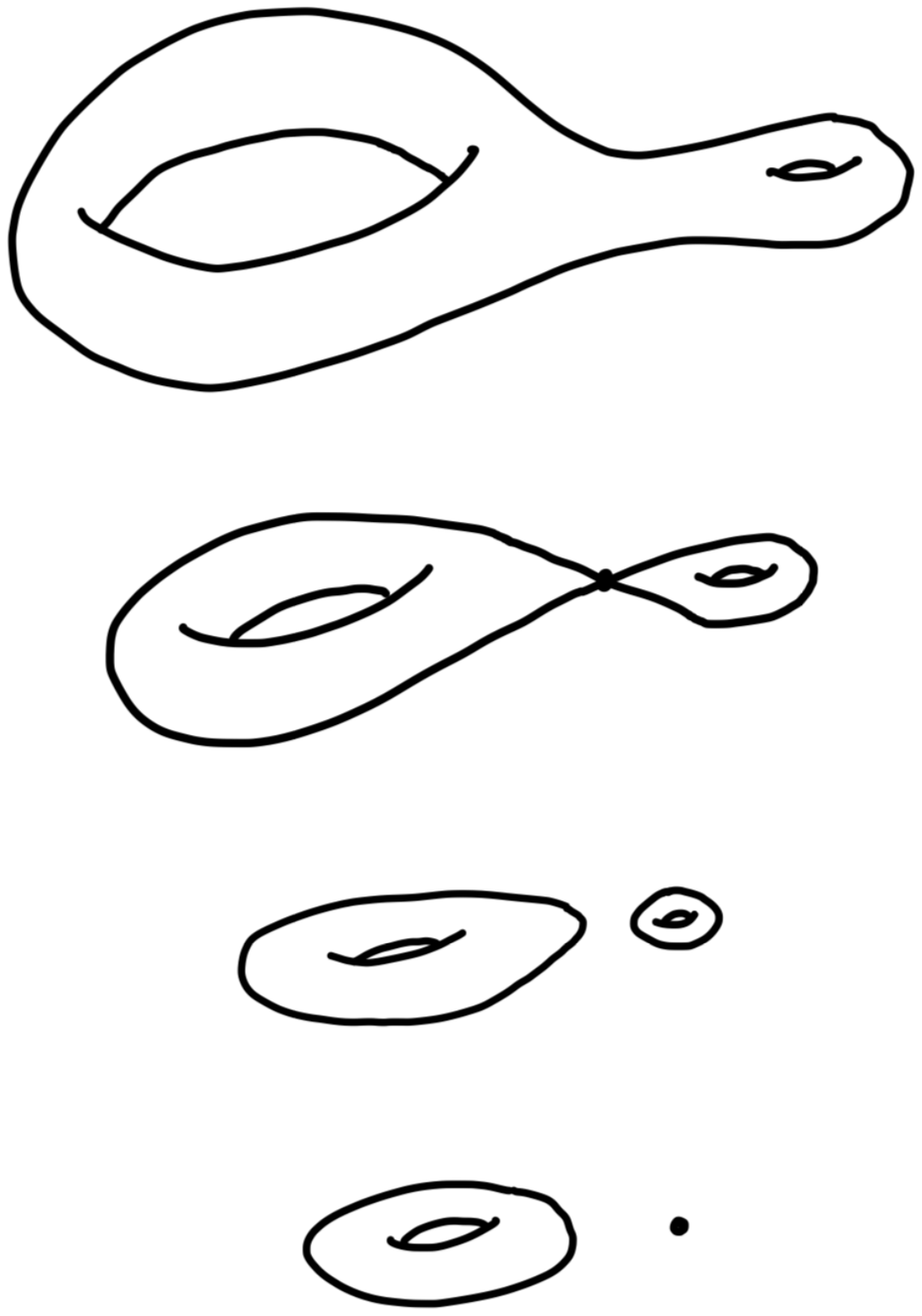}}
        \caption{This is an example of pinch-off process. It has one neck-pinch surgery, and one connected component shrunk to a point.}
        \label{fig:pinchOff2}
    \end{figure}

    One might ask why we need the pinch-off process. Indeed, solely by the usual Simon-Smith min-max theorem, one already could homotope the initial family $\Psi$ to some other family that is, a certain sense, optimal. The problem is the following. When this optimal family detects a minimal {\it sphere}, let us consider the ``cap" region of this optimal family that is near the minimal sphere. Every member of this cap would only be varifold-close to the minimal sphere, and thus could actually have genus $1$ but with a short non-trivial loop. As we will see in \S \ref{subsubsect:obtainXi}, this would be problematic (and deviates from the heuristic picture in \S \ref{subsect:heuristicsProof} too). Hence, we need to use the pinch-off process, to eliminate the short non-trivial loops, and ensure that the cap consists of only genus $0$ surfaces. 

    We remark that, this ``interpolation" process of modifying the whole high area cap such that every member becomes genus 0 is a non-trivial step. Namely, we need to perform the pinch-off process continuously for every member. We will briefly explain the strategy in \S \ref{sect:introInterpolate}.    

\subsubsection{Obtaining $\Xi$ through repeated applications of the min-max theorem}\label{subsubsect:obtainXi}

    First, we assume that the 3-sphere $(S^3,g_0)$ given has only finitely many minimal tori. Then, we perturb the metric $g_0$ such that it has only finitely many minimal spheres, all of which are non-degenerate and have linearly independent areas over $\Z$, with no linear combination over $\Z$ of their areas equal to the area of any minimal torus; see \S \ref{sect:perturbMetric}.
    
    In the first stage of min-max process, we apply the Simon-Smith min-max construction to $\Psi$, homotoping it obtain a new pulled-tight ``optimal" family $\tilde\Psi^1$ that detects finitely many min-max minimal surfaces of genus $0$ or $1$. For simplicity, we will assume the new parameter space $\dmn(\tilde \Psi^1)$ is also $Y$. 
    
    In the new perturbed metric, using D. Ketover's strengthened genus bound~\cite{Ket19} and Wang-Zhou's multiplicity one theorem~\cite{WZ23}, which builds on Sarnataro-Stryker's work~\cite{SS23}, it is not hard to see that there are two possible cases about the detected minimal surfaces: (1) each min-max minimal surface is a multiplicity one minimal torus, or (2) only one min-max minimal surface is detected, and it is a multiplicity-one sphere. Note that the Frankel property is needed, which is ensured by the assumption of positive Ricci curvature. For simplicity, in case (1), we assume that there is only one minimal torus.

    In both cases, we consider a subset $C^1\subset \dmn(\tilde \Psi^1)$ such that the restriction $\tilde\Psi^1|_{C^1}$ is close to the minimal surface $V^1$ detected. The family $\tilde\Psi^1|_{C^1}$ is like a ``cap" of large area, such that its complement $\Psi^1:=\tilde\Psi^1|_{\overline{Y\backslash C^1}}$ has the property that the maximal area of $\Psi^1$ is strictly less than the width of $\Psi$, $\area(V^1)$. Let $Y^1$ denote the new domain $\dmn(\Psi^1)=\overline{Y\backslash C^1}$.
    
    For case (2) only, we will prove  certain interpolation results that allow us, without loss of generality, to assume that $\tilde\Psi^1|_{C^1}$ consists entirely of genus $0$ surfaces, possibly with singularities: {\it This is  where the pinch-off process described in \S \ref{subsubsect:fatteningIssue} comes into place.} 

    Now, we proceed to the second stage of the min-max process, this time applied to $\Psi^1$, regardless of whether case (1) or (2) occurred previously. The width at this stage will be strictly smaller than the previous width, leading to some new minimal surface. We repeat this process iteratively. 
    
    Namely, at the $k$-th stage of the  min-max process, we start with a family $\Psi^{k-1}$ with parameter space $Y^{k-1}$. Applying the Simon-Smith min-max theorem to $\Psi^{k-1}$ would produce a  family $H^k$ with parameter space $[0,1]\x Y^{k-1}$ , resembling a homotopy, such that $H^k(0,\cdot)=\Psi^{k-1}$, and the new family $\tilde\Psi^k:=H^k(1,\cdot)$ (with parameter space $Y^{k-1}$) is an ``optimal" family detecting a multiplicity-one minimal surface $V^k$ of genus $0$ or $1$. 
    
    Let $C^k \subset Y^{k-1}$ be a ``cap" such that $\tilde\Psi^k|_{C^k}$ has large area and is near $V^k$. As before, if $V^k$ has genus $0$, we can assume, by applying pinch-off process, that $\tilde\Psi^k|_{C^k}$ consists entirely of genus $0$ surfaces, possibly with singularities. We then consider the new parameter space $Y^k:=\overline{Y^{k-1}\backslash C^k}$, and proceed with the $(k+1)$-th stage of   min-max process, by applying the Simon-Smith min-max theorem  to the  family $\Psi^k:=\tilde\Psi^k|_{\overline{Y^{k-1}\backslash C^k}}$. For simplicity, we assume that each $Y^k$ is a $9$-dimensional manifold with boundary.

    Since we have perturbed the metric $g_0$ so that all minimal spheres are non-degenerate, there must be only finitely many of them, by Choi-Schoen's compactness result~\cite{CS85}. Consequently, since the width of $\Psi^k$ decreases strictly with each iteration, the repetitive min-max process must terminate. Namely, there exists a $K>0$ such that at the $K$-th stage of min-max process, the width of $\Psi^{K-1}$ is zero. At this point, there exists a family $H^K$ (which acts like a homotopy) with parameter space $[0,1]\x Y^{K-1}$, where $H^K(0,\cdot)=\Psi^{K-1}$, and $\tilde \Psi^K:=H^K(1,\cdot)$ has a maximal area arbitrarily close to zero.

    Finally, let us construct the desired ``optimal" family $\Xi$. Consider the following list of $2K-1$ Simon-Smith families of genus $\leq 1$. For the sake of visualization, readers may find it helpful  to refer to Figure~\ref{fig:gluingSchemeIntro}. 

    \begin{description}
        \item[{\makebox[4em][r]{(1)}}] $\tilde \Psi^1|_{C^1}$
        \item[{\makebox[4em][r]{\color{red}(2)}}]   {\color{red} $H^2|_{[0,1]\x\partial Y^1}$}
        \item[{\makebox[4em][r]{(3)}}] $\tilde \Psi^2|_{C^2}$
        \item[{\makebox[4em][r]{\color{RoyalBlue}(4)}}] {\color{RoyalBlue}$H^3|_{[0,1]\x\partial Y^2}$}
        \item[{\makebox[4em][r]{...}}] 
        \item[{\makebox[4em][r]{\color{ForestGreen}{($2K-4$)}}}] {\color{ForestGreen}$H^{K-1}|_{[0,1]\x\partial Y^{K-2}}$}
        \item[{\makebox[4em][r]{($2K-3$)}}] $\tilde \Psi^{K-1}|_{C^{K-1}}$
        \item[{\makebox[4em][r]{\color{brown}($2K-2$)}}] {\color{brown}$H^K|_{[0,1]\x\partial Y^{K-1}}$}
        \item[{\makebox[4em][r]{($2K-1$)}}] $\Psi^K$
    \end{description}
    
    \begin{figure}[h]
        \centering
        \makebox[\textwidth][c]{\includegraphics[width=5.5in]{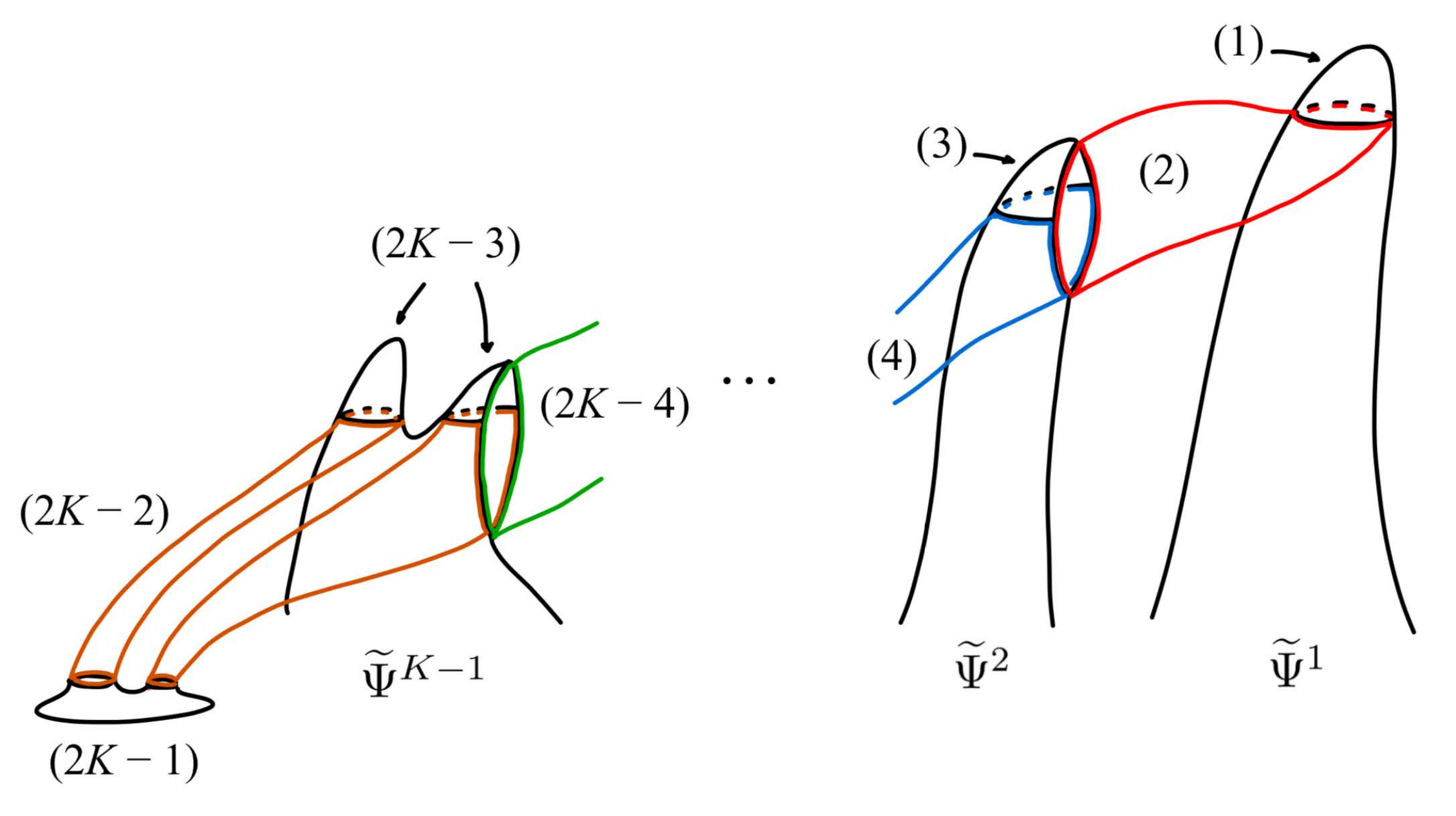}}
        \caption{Constructing $\Xi$.}
        \label{fig:gluingSchemeIntro}
    \end{figure}
    Immediately from the definition of these $2K-1$ families, there exists a natural way to glue their boundaries together, as illustrated in Figure~\ref{fig:gluingSchemeIntro}. Through this process, we obtain a new $9$-parameter Simon-Smith family $\Xi$ that can be deformed to the original family $\Psi$; see Proposition~\ref{prop:XiPsiHomotopic}.
    
\subsubsection{Decomposing $\Xi$ into ``unstable manifolds"} 
    
    In the heuristic argument in \S \ref{subsect:heuristicsProof}, we decomposed $\Xi$ into a union of ``unstable manifolds" $T(S)$ associated with each minimal surface $S$ detected in $\Xi$. There are two issues with this approach.

    First, the family $\Xi$ we obtained in \S \ref{subsubsect:obtainXi}, when viewed as a subset of the cycle space $\cZ_2(S^3;\Z_2)$,  may have complicated topology and thus be difficult to work with from an algebraic topology perspective. Thus, in practice we actually work with its parameter space $\dmn(\Xi)$, and decompose $\dmn(\Xi)$ instead.

    Second, our family $\Xi$ no longer consists of mean curvature flow lines. Instead of defining $T(S)$, we are going to define the {\em trace} $T(C)$ of each cap $C=C^1,C^2,\cdots,C^{K-1}$, and each trace $T(C)$ will be a subset of $\dmn(\Xi)$.

    Let us fix a cap $C=C^k$, obtained at the $k$-th stage of the min-max process. Consider the following sets, {\em which can all be viewed as subsets of $\dmn(\Xi)$} (see Figure \ref{fig:trace}):
    \begin{itemize}
        \item {\color{red}$B_k$}$\;:=C$,
        \item  ${\color{RoyalBlue}B_{k+1}}:=[0,1]\x\partial Y^k=\dmn(H^{k+1})$,
        \item  ${\color{ForestGreen}B_{k+2}}:=[0,1]\x\left((\{1\}\x\partial Y^k)\cap(\{0\}\x\partial Y^{k+1})\right)\subset \dmn(H^{k+2})$. Here, the subsets  $\{1\}\x\partial Y^k,\{0\}\x\partial Y^{k+1}\subset\dmn(\Xi)$ can both also be viewed as subsets of $\dmn(\tilde\Psi^{k+1})$
        \item ${ B_{k+3}}:=[0,1]\x\left((\{1\}\x\partial Y^{k+1})\cap(\{0\}\x\partial Y^{k+2})\right)\subset\dmn(H^{k+3})$.
        \item ...
    \end{itemize}
    Finally, we define the trace of $C$ by
    \[
        T(C):=B_k\cup B_{k+1}\cup\cdots\cup B_K\subset\dmn(\Xi)\,.
    \]
    A crucial property of $T(C)$ is that, by definition, $T(C)$ admits a strong deformation retraction back onto the cap $C$.
    \begin{figure}[h]
        \centering
        \makebox[\textwidth][c]{\includegraphics[width=2.7in]{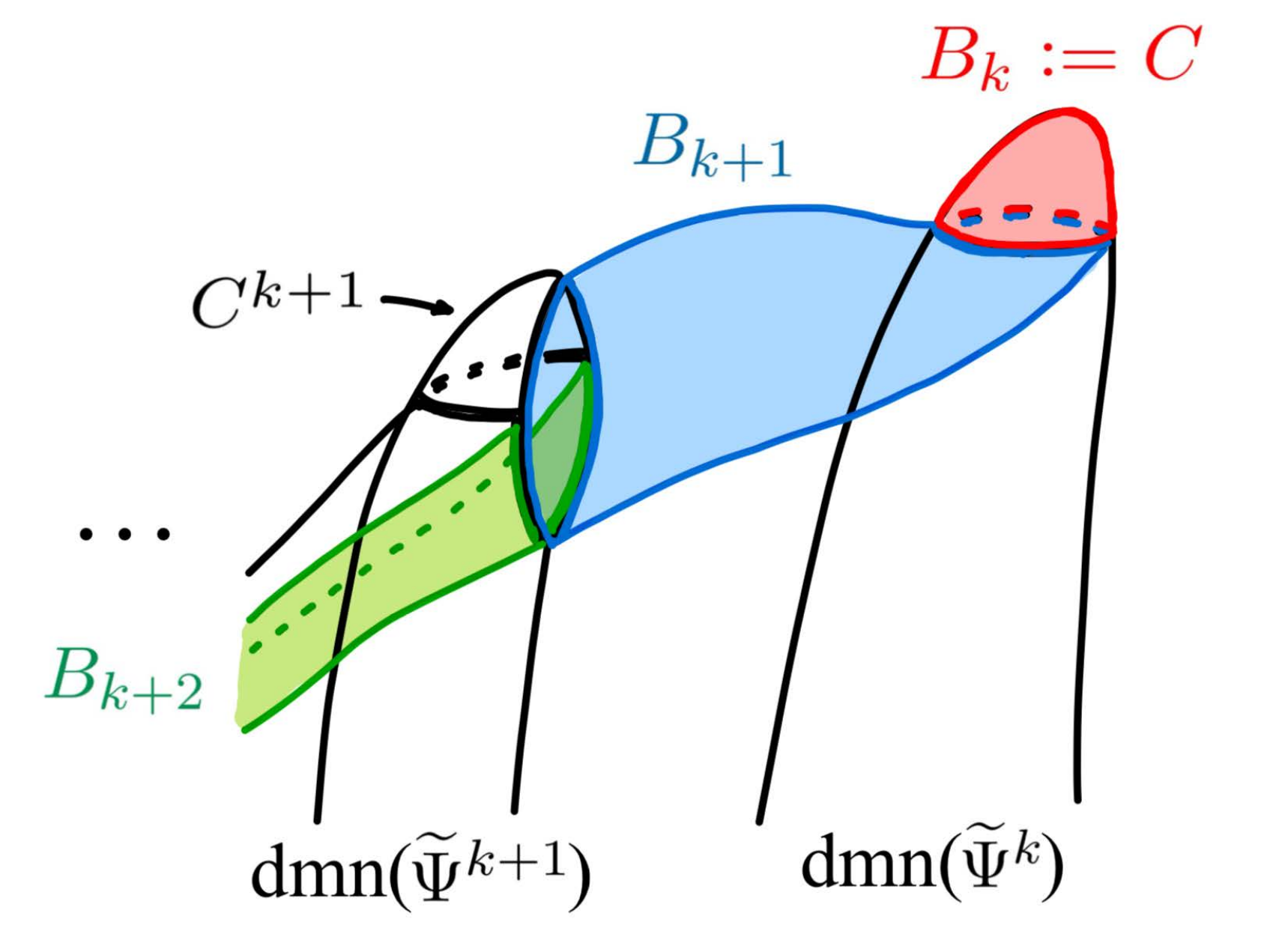}}
        \caption{The trace $T(C)$ in $\dmn(\Xi)$.}
        \label{fig:trace}
    \end{figure}
        
    In fact, if $C=C^k$ is a cap associated with a genus 0 minimal surface, then by construction $\Xi|_{T(C)}$ consists entirely of genus 0 surfaces, possibly with singularities, because $\tilde\Psi^k|_{C^k}$ does and the ``homotopies" given by $H^{k+1},H^{k+2},...$ do not increase genus (this should be compared with the situation in \S \ref{subsect:heuristicsProof}, where  mean curvature flow does not increase genus).

    Now, let $N$ be the number of minimal tori we detected throughout the first $K-1$ stages of the min-max process. To prove Theorem~\ref{thm:main}, we only need to show that $N\geq 5$.

    Let us see why $N=4$ is impossible; the cases for $N=0,...,3$ are similar. Given a cap $C=C^k$, we say that it has {\it genus $0$} if the minimal surface $V^k$ associated has genus $0$, and {\it genus $1$} if $V^k$ has genus $1$. For example, when $N=4$, we have four genus $1$ caps, and finitely many genus $0$ caps. We can express $\dmn(\Xi)$ as the union of the following five subsets: 
    \begin{itemize}
        \item Let $D_0$ be the union of $\dmn(\Psi^K)$
        and $\bigcup_C T(C)$, where $C$ ranges over all the genus $0$ caps. 
        \item For each genus $1$ cap $C$, we consider its trace $T(C)$. There are in total $4$ such traces, which we will denote by $D_1,\cdots,D_4$.
    \end{itemize}
    It follows directly from the definition of $\Xi$ that 
    \[
        \dmn(\Xi)=D_0\cup D_1\cup D_2\cup D_3\cup D_4\,.
    \]

\subsubsection{Topological arguments}
    With this setup, we can proceed with the remaining steps of the heuristic argument in \S \ref{subsect:heuristicsProof}. However, instead of considering the cohomology rings of $\Psi$ and $\Xi$, we should focus on those of $Y$ and $\dmn(\Xi)$. For example, we should instead define $\lambda$ as an element of $ H^1(Y;\Z_2)$, by $\lambda:=\Psi^*(\bar\lambda)$, where $\Psi$ is viewed as a map into $\cZ_2(S^3;\Z_2)$. The crucial topological fact is that, there exist $\alpha,\beta\in H^1(Y;\Z_2)$ such that
    \begin{equation}
        \lambda^5\cup\alpha^2\cup\beta^2\ne 0\,.
    \end{equation}
    Moreover, since $Y$ and $\dmn(\Xi)$ are actually homotopy equivalent, $\lambda,\alpha$ and $\beta$ can be viewed as elements of $H^1(\dmn(\Xi);\Z_2)$ as well.

    Then, as in \S \ref{subsect:heuristicsProof}, we apply a Lyusternik-Schnirelmann argument to derive a contradiction. The proof that $\lambda^5|_{D_0}=0$ proceeds similarly to the earlier argument. However, proving $\alpha|_{D_i}=0$ for  $i =1,...,4$ requires more essential modifications. For instance, unlike in \S \ref{subsect:heuristicsProof}, $D_1$ may not be contractible (let us focus on $i=1$, as the cases of $i=2,3,4$ are the same). Thus, we would instead prove $\alpha|_{D_1}=0$ by showing that the map 
    \[
        H_1(D_1;\Z_2)\to H_1(\dmn(\Xi);\Z_2)
    \] 
    induced by the inclusion $D_1\hookrightarrow\dmn(\Xi)$ is trivial. 

    Recall that $D_1$ admits a strong   deformation retraction onto the genus $1$ cap $C$ it corresponds to. Therefore, it suffices to show that for any loop $c\subset C$, $[c]=0$ in $H_1(\dmn(\Xi);\Z_2)$. To achieve this, we show that the cap $C$ can be assumed to be close in the {\em $\bF$-metric for currents} to the minimal torus corresponding to $C$. Consequently, the restriction $\Xi|_c$ actually gives a family of genus $1$ surfaces with singularities. Using the homotopy between $\dmn(\Xi)$ and $Y$, we can deform $\Xi|_c$ back to some subfamily $\Psi|_{c_0}$ of $\Psi$, for some loop $c_0\subset Y=\dmn(\Psi)$, while ensuring that throughout the homotopy, {\em all punctate surfaces are of genus $1$}. This is possible because the process of obtain the final family $\Xi$ from the original family $\Psi$ is, in some sense, genus non-increasing (similar to running mean curvature flow). In particular, it suffices to show that the loop $c_0\subset Y$ is homologically trivial.

    Let $Z_0\subset Y$ be the set of parameters corresponding to (unoriented) Clifford tori under $\Psi$. From the definition of $\Psi$ and the fact that all members of $\Psi|_{c_0}$ have genus one, it is not hard to show that one can deform $c_0$ into $Z_0$. Thus, it suffices to show that $c_0$ is homologically trivial in $Z_0$. Note that $Z_0\cong \RP^2\x\RP^2$, so $H_1(Z_0;\Z_2)$ consists four elements:   $(0,0), (0,1), (1,0), (1,1)$. 
    
    First, using the fact that $\Xi|_{c}$ is close to a single minimal torus,  $[c_0]$ cannot be $(0,1)$ or $(1,0)$, as both correspond to $1$-sweepouts under $\Psi$. As for $(1,1)$, one can check that it gives rise to, under $\Psi$, a loop of {\em oriented} Clifford tori, $\{\Sigma(\theta)\}_{\theta\in[0,2\pi]}$, such that the followings hold:
    \begin{itemize}
        \item We can continuously pick an interior region $\ins(\Sigma(\theta))$, bounded by $\Sigma(\theta)$, for each $\theta$.
        \item While the interior regions $\ins(\Sigma(0))$ and $\ins(\Sigma(2\pi))$ coincide, the isomorphism map from $\pi_1(\ins(\Sigma(0)))$ to $\pi_1(\ins(\Sigma(2\pi)))$ (both groups are $\Z$) induced by the motion $\{\Sigma(\theta)\}_{\theta\in[0,2\pi]}$ is not the identity but $-1$.
    \end{itemize}  We would show that this is impossible, due to the fact that $\Xi|_{c}$ is close to a single minimal torus, and that $\Psi|_{c_0}$ was obtained from $\Xi|_{c}$ through a homotopy which is a family consisting of {\em genus one} surfaces possibly with singularities.

    Hence, $[c_0]$ must be the trivial element $(0,0)$, as desired. This means $\alpha|_{D_1}=0$. Similarly, $\alpha|_{D_2}$, $\beta|_{D_3}$, $\beta|_{D_4}$ are all trivial. This leads to a contradiction, as shown in \S \ref{subsect:heuristicsProof}. This completes the proof of Theorem~\ref{thm:main}.

\subsubsection{Interpolation}\label{sect:introInterpolate}
Lastly, let us briefly explain the interpolation process mentioned at the end of \S \ref{subsubsect:fatteningIssue}.

Suppose we have a ``cap'' region $C$ consisting of genus $1$ surfaces, {\it each being $\bF$-close to a minimal sphere} ($\bF$ is the distance for currents). It is expected that each such surface contains a short non-trivial loop. If such short non-trivial loops could be selected in a continuous manner, 
one might perform a continuous family of surgeries on the cap to produce a new cap consisting of genus $0$ surfaces. However, to the best of the authors' knowledge, there is no canonical method for locating such a continuous family of short loops.

Instead, we replace the neck-pinch surgery by a two-step ``pinch-off process". First, for each surface $\Sigma$ in $C$, we identify a $3$-ball $B(p_\Sigma, r_\Sigma)\subset S^3$ that contains a short non-trivial loop of $\Sigma$. We then perform a sequence of neck-pinch surgeries on $\Sigma$  along each loop component of $\Sigma \cap \partial B(p_\Sigma, r_\Sigma)$, so that afterwards $\Sigma $ avoids the sphere $\partial B(p_\Sigma, r_\Sigma)$. In the second step, we shrink everything inside $B(p_\Sigma, r_\Sigma)$ to a point. 

The advantage of this alternative method is that it avoids the need to choose the $3$-balls continuously. Instead, we fix an open covering $\{O_i\}_i$ of $C$ and select a suitable $3$-balls $B_i$ for each $O_i$ such that $\partial B_i \cap \partial B_j = \emptyset$ if $O_i \cap O_j \neq \emptyset$. If such a selection exists, one can construct a continuous family of two-step pinch-off processes in $C$ using a standard cut-off argument. Fortunately, this is indeed feasible, and a similar construction was already carried out by Pitts~\cite{Pit81}; see Section~\ref{sect:min-max_iii} for details.

The actual construction is more involved than the outline given above, due to the following reasons:
\begin{enumerate}
    \item Proving the existence of a short  non-trivial loop on each surface requires some careful argument.
    \item It is crucial to control the area during  interpolation  in order to ensure at each repetition  of the  min-max procedure, the width drops.
    \item The ``surfaces'' in the cap $C$ may not be smooth, so we need multiple balls to handle each one appropriately.
\end{enumerate}

\subsection{Open problems}

    First of all, the full conjecture of B. White is still open. However, in light of the work of~\cite{WZ23}, we are more optimistic about the existence of at least five embedded minimal tori for a generic metric, though this number might not be sharp for generic metrics, as discussed below.

    Recall that the Morse number of a manifold is the minimum number of critical points a Morse function can have. Since $\mathbb{RP}^2 \times \mathbb{RP}^2$ has Morse number $9$, B. White~\cite{Whi91} conjectured the existence of at least $9$ embedded minimal tori in a $3$-sphere equipped with a generic metric.
    
    Inspired by the works of~\cite{MN16, Li23a, Li23b}, and as a further inquiry based on the current work, we conjecture that in every $3$-sphere with positive Ricci curvature, there exist at least 5 embedded minimal tori of Morse index at most $9$. 

    It is also natural to raise the question of counting minimal surfaces with higher genus. For instance, let $\cM_2$ denote the space of embedded, unoriented, genus $2$ Lawson surfaces in the unit $3$-sphere. This space is a closed $6$-dimensional manifold. In the spirit of~\cite{Whi91}, we conjecture that in every Riemannian $3$-sphere, the number of closed, embedded, genus $2$ minimal surfaces is at least one plus the cup length of $\cM_2$. This conjecture is, of course, based on the fact that the Lyusternik–Schnirelmann number of a manifold is bounded below by one plus the cup length. We also conjecture that in every Riemannian $3$-sphere equipped with a generic metric, the number of closed, embedded, genus $2$ minimal surfaces is at least the Morse number of $\cM_2$. 
    
    Finally, we pose the following questions as potential steps towards resolving the fattening issue described in \S \ref{subsubsect:fatteningIssue} from a purely mean curvature flow perspective. Given a smooth, embedded, closed surface $\Sigma$ in a 3-manifold, let $\mathfrak{M}$ denote the union of all time slices of all mean curvature flows with initial data $\Sigma$, in whichever appropriate weak sense. Is the set $\mathfrak{M}$ connected? More ambitiously, is $\mathfrak{M}$ contractible? 
    
    If $\mathfrak{M}$ is contractible, it may be possible to run mean curvature flow to the whole family $\Psi$ and surpass the time of fattening; see also the recent work of J. Bernstein, L. Chen, and L. Wang regarding the set of expanders flowing out of a cone~\cite{BW22,BCW24}.

\subsection{Organization} 
    This paper is divided in two main parts. In part I, we first present the preliminary results in \S \ref{sect:prelim}, which primarily concern about various min-max theories, and then prove the main theorem in \S \ref{sect:mainProof}. In part II, we prove various theorems and propositions used in part I.
\subsection*{Acknowledgment} 
    We are deeply grateful to Zhihan Wang for his generous help, which has been instrumental in the completion of this project. We also wish to thank Andr\'e Neves for his constant support and guidance. The second author would also like to thank Jonanthan Zung for answering questions on geometric topology, especially those related to Example~\ref{ex:tori_interpolation}.

    This material is based upon work supported by the National Science Foundation under Grant No. DMS-1928930, while the first author was in residence at the Simons Laufer Mathematical Sciences Institute (formerly MSRI) in Berkeley, California, during the Fall semester of 2024. The second author was partially supported by the AMS-Simons travel grant.
    
\part{Main arguments}

\section{Preliminaries}\label{sect:prelim}
    Throughout this paper, all the ambient Riemannian manifolds $(M, g)$ are assumed to be smooth and closed. We will use $M$ to denote the manifold when the metric $g$ is clear from the context.
    
\subsection{Notations} 
    \begin{itemize}
        \item $\bI_n(M;\Z_2)$: the set of integral $n$-dimensional currents in $M$ with $\Z_2$-coefficients.
        \item $\Zc_{n}(M;\Z_2)\subset \bI_n(M;\Z_2)$: the subset consisting of elements $T$ such that $T=\partial Q$ for some $Q\in\bI_{n+1}(M;\Z_2)$. Such $T$ are also referred to as {\it (modulo $2$) flat $n$-cycles}.
        \item $\Zc_{n}( M;\nu;\Z_2)$ with $\nu=\cF,\bF,\bM$: the set $\Zc_{n}( M;\Z_2)$ equipped with the three topologies given corresponding respectively to the  {\em flat norm $\Fc$},  the {\em $\Fb$-metric}, and the  {\em mass norm $\Mb$} (see~\cite{Pit81} and the survey paper~\cite{MN20}). For the flat norm, there are two definitions that, by the isoperimetric inequality, would induce the same topology:
        \[
            \cF(T):=\inf\{\bM(P)+\bM(Q):T=P+\partial Q\}\,,
        \]
        and 
        \[
            \cF(T):=\inf\{\bM(Q):T=\partial Q\}\,.
        \]
        In this paper, we will use the second definition.
        \item $\mathcal{V}_n(M)$: the closure, in the varifold weak topology, of the space of $n$-dimensional rectifiable varifolds in $M$. There is a $\bF$ metric on $\mathcal{V}_n(M)$, which also induces the varifold weak topology when restricted to any subset of $\mathcal{V}_n(M)$ whose elements have volume bounded above by a constant.
        \item $\norm{V}$: the Radon measure induced on $M$ by  $V\in \mathcal{V}_n(M)$.
        \item $|T|$: the varifold in $\cV_n(M)$ induced by a current $T\in \Zc_{n}(M;\Z_2)$, or a countable $n$-rectifiable set $T$. In the same spirit, given a map $\Phi$ into $\Zc_n(M;\Z_2)$, the associated map into $ \mathcal{V}_n(M)$ is denoted by $|\Phi|$.
        \item $\spt(\cdot)$: the support of a current or a measure.
        \item $[W]$: the $\Z_2$-current induced by $W$, if $W$ is a countably $2$-rectifiable set with $\cH^2(W) < \infty$. In the same spirit, given a map $f$ whose images are countably $2$-rectifiable sets, the associated map into the space $\cZ_n(M; \Z_2)$ is denoted by $[f]$.
        \item $[\mathcal{W}] := \{[\Sigma_i]\} \subset \cZ_n(M;\Z_2)$ for a set $\mathcal{W}$ of varifolds $\{V_i\} \subset \cV_n(M)$, each associated with a countably $n$-rectifiable set $\Sigma_i$.
        \item $\bB^\nu_\varepsilon(\cdot)$:  the open $\varepsilon$-neighborhood of an element or a subset of the space $\Zc_n(M;\nu;\Z_2)$.
        \item $\bB^{\Fb}_\varepsilon(\cdot)$: the open $\varepsilon$-neighborhood  of an element or  a subset of $\cV_n(M)$ under the $\Fb$-metric.
        \item $\cC(M)$: the space of Caccioppoli sets in $M$, equipped with the metric induced by the Lebesgue measure of the symmetric difference.
        \item $\partial^* \Omega$: the reduced boundary of $\Omega \in \cC(M)$.
        \item $\nu_\Omega$: the inward pointing normal of $\Omega \in \cC(M)$.
        \item $\Gamma^\infty(M)$: the set of smooth Riemannian metrics on $M$.
        \item $B_r(p)$: the open $r$-neighborhood of a point $p$.
        \item $\fg(S)$: the genus of a surface $S$. 
        \item $\ins(S)$ ($\out(S)$): the inside (outside), open region of an oriented hypersurface $S$, provided that $S$ separates $M$.
    \end{itemize}
    
    For an $m$-dimensional cube $I^m = \mathbb{R}^m \cap \{x : 0 \leq x_i \leq 1, i = 1,2, \cdots, m\}$, we can give it {\em cubical complex structures} as follows.
    \begin{itemize}
        \item  $I(1,j)$: the cubical complex on $I := [0,1]$ whose $1$-cells and $0$-cells are respectively 
        \[
            [0,1/3^j],[1/3^j,2/3^j],\dots,[1-1/3^j,1]\quad \textrm{ and }\quad [0],[1/3^j],[2/3^j],\dots,[1]\,.
        \]
        \item $I(m,j)$: the cubical complex structure
        \[
            I(m,j)=I(1,j)\otimes\dots\otimes I(1,j)\quad(m\textrm{ times})
        \]
        on $I^m$. $\alpha = \alpha_1 \otimes \cdots \otimes \alpha_m$ is a {\em $q$-cell} of $I(m, j)$ if and only if each $\alpha_i$ is a cell in $I(1, j)$ and there are exactly $q$ $1$-cells. A cell $\beta$ is a {\em face} of a cell $\alpha$ if and only if $\beta \subset \alpha$ as sets.
    \end{itemize}

    We call $X \subset I(m, j)$ a {\em cubical subcomplex of $I(m, j)$} if every face of a cell in $X$ is also a cell in $X$. For convenience, we also call $X$ a {\em cubical complex} without referring to the ambient cube. We denote by $|X|$ the {\em underlying space} of $X$. For the sake of simplicity, we will also consider a complex and its underlying space as identical unless there is ambiguity.

    Given a cubical subcomplex $X$ of some $I(m, j)$, for $j' > j$, one can {\em refine} $X$ to a cubical subcomplex
    \[
        X(j') := \{\sigma \in I(m, j'): \sigma \cap |X| \neq \emptyset\}
    \]
    of $I(m, j')$. For the sake of convenience, we will denote the refined cubical subcomplex by $X$ unless there is ambiguity.
    
\subsection{Simon-Smith min-max theory}
    In his Ph.D. thesis~\cite{Smi83}, written under the supervision of Simon, F. Smith combined the Almgren-Pitts theory with results from Almgren-Simon~\cite{AS79} and Meeks-Simon-Yau~\cite{MSY82}, proving the existence of a minimal 2-sphere in any Riemannian 3-sphere. This theory is now known as \emph{Simon-Smith min-max theory}. Later, Pitts and Rubinstein~\cite{PR86} announced, without publishing proofs, an extension of the work to more general min-max constructions in 3-dimensional manifolds, dealing with multiparameter families, surfaces of higher genus, and bounds on the Morse index. This program was eventually completed by De Lellis-Pellandini~\cite{DP10} and Ketover~\cite{Ket19}. Most recently, Wang-Zhou~\cite{WZ23} proved a multiplicity one theorem for the Simon-Smith min-max theory.

    In most of the aforementioned literature, the focus is on the space of smooth surfaces, occasionally augmented with certain degenerating sets of measure zero. This setup allows the use of smooth topology, which helps simplify arguments related to regularity and genus bounds. However, in this paper, such a space is not sufficiently broad, as we need to interpolate between tori and spheres, requiring allowance for singularities both on and off the surface; see \S \ref{sect:min-max_iii}. We will refer to this type of singular surface as a \emph{punctate surface} to distinguish it from existing notions in the literature.
    
    In this subsection, let $(M, g)$ be a closed Riemannian 3-manifold.

    \begin{defn}[Punctate surfaces]\label{def:punctate_surf}
        A closed subset $\Sigma \subset M$ is called a {\em punctate surface} provided that:
        \begin{enumerate}
            \item The 2-dimensional Hausdorff measure $\cH^2(\Sigma) \in (0, \infty)$.
            \item There exists a finite set $P \subset \Sigma$ such that $\Sigma \setminus P$ is an orientable, smooth, embedded surface. In this case, we call $P$ a \emph{punctate set} of $\Sigma$.
            \item Let $\Sigma_{\mathrm{iso}}$ be the set of isolated points of $\Sigma$. Then the complement of $\Sigma\backslash \Sigma_{\mathrm{iso}}$ is a disjoint union of two open regions of $M$, each having $\Sigma\backslash \Sigma_{\mathrm{iso}}$ as its topological boundary.
        \end{enumerate}
        We denote the space of punctate surfaces in $M$ by $\GS(M)$. 

        For each $\fg_0 \in \N$ (the set of non-negative integers), $\Sigma \in \GS(M)$ is called {\em a genus $\fg_0$ punctate surface} provided that there exists a punctate set $P$ and a sequence $r_i \to 0$ such that
        \[
            \fg(\Sigma) := \lim_{r \to 0} \fg(\Sigma \setminus B_{r_i}(P)) = \fg_0\,.
        \]
        
        Additionally, as in the aforementioned literature, we augment the space of punctate surfaces with degenerating sets of measure zero, and denote 
        \[
            \GS^*(M) := \GS(M)\cup\{\Sigma\subset M:\Sigma\textrm{ is a closed subset}, \cH^2(\Sigma)=0\}\,.
        \]
    \end{defn}

    \begin{figure}[!ht]
        \centering
        \includegraphics[width=0.7\linewidth]{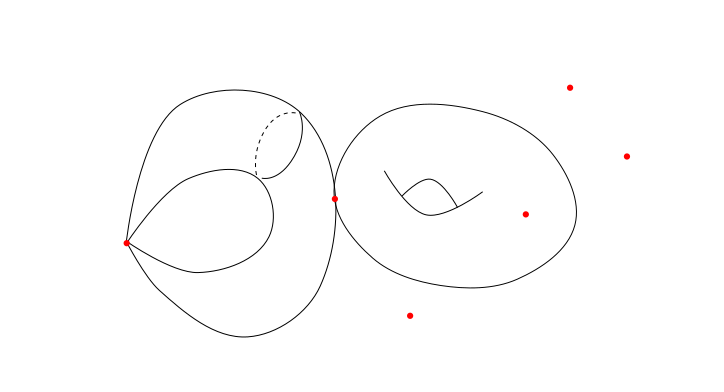}
        \caption{A punctate surface of genus $1$ with a punctate set (red dots).}
        \label{fig:a punctate surface}
    \end{figure}

    \begin{rmk}
        A punctate surface is simply the union of a surface with finitely many singularities and a finite set of points outside the surface (see Figure~\ref{fig:a punctate surface}). Therefore, the smallest possible punctate set of $\Sigma$ is the union of the singular points of $\Sigma \setminus \Sigma_\text{iso}$ and $\Sigma_\text{iso}$. As for why we concern about isolated points of a punctate surface, and why we allow the punctate set to contain some regular points of a punctate surface, see Remark \ref{rmk:moreAboutPunctateSet}.
    \end{rmk}
    
    \begin{rmk}
        For any $\Sigma \in \GS(M)$ with a punctate set $P$, by Sard's theorem, there exists a set of full measure $E \subset (0, \infty)$ such that for every $r \in E$, $\Sigma \setminus B_r(P)$ is a smooth surface with  boundary. Note that the boundary is a finite union of disjoint loops.
    
        A smooth surface with boundary can be obtained from a closed surface by removing finitely many disjoint disks. The genus of this surface is the genus of the original closed surface. Also note that when a surface is disconnected, its genus is defined as the sum of the genus of each of its connected components. 

        Moreover, for $r_1 > r_2 > 0$ such that both $\Sigma \setminus B_{r_i}(P)$ are smooth surfaces with boundary, by~\cite[Chapter~1~\S 2.1~Lemma~1.5]{CM12}, we have the inequality on their genus,
        \[
            \mathfrak{g}(\Sigma\setminus B_{r_1}(P)) \leq \mathfrak{g}(\Sigma\setminus B_{r_2}(P))\,.
        \]
        Hence, the limit
        \[
            \lim_{E \ni r \to 0} \mathfrak{g}(\Sigma \setminus B_{r}(P)) \in \N \cup \{\infty\}
        \]
        always exists. 
        
        In particular, the definition of $\mathfrak{g}(\Sigma)$ is independent of the choice of the sequence $r_i$. In addition, one can verify that the definition is independent of the choice of the punctate set $P$.
    \end{rmk}

    \begin{defn}[Simon-Smith family]\label{def:Simon_Smith_family}
        Let $X$ be a cubical subcomplex of some $I(m, j)$.
        A map $\Phi: X \to \GS^*(M)$ is called a {\em Simon-Smith family}, provided that:
        \begin{enumerate}[label=\normalfont(\arabic*)]
            \item \label{item:Hausdorff_cts} 
            The composition map $x\mapsto\cH^2 \circ \Phi(x)$ is continuous. 
            \item \label{item:closedFamily}
            For any $x_0 \in X$ and any open set $U \supset \Phi(x_0)$, there exists a neighborhood $O \subset X$ of $x_0$ such that for any $x \in O$, $\Phi(x) \subset U$.
            \item \label{item:SingPointsUpperBound} For each $x_0 \in X$ with $\Phi(x_0) \in \GS(M)$, we can choose a punctate set $P_\Phi(x_0) \subset \Phi(x_0)$ such that 
            \[
                N_P(\Phi) := \sup_{x:\Phi(x) \in \GS(M)} |P_\Phi(x)|<\infty\,.
            \]
            \item \label{item:SimonSmithFamilyLocalSmooth}
            For the punctate set $P_\Phi(x_0)$ chosen in~\ref{item:SingPointsUpperBound}, for any $x_0\in X$ with $\Phi(x_0) \in \GS(M)$ and on any open set $U \subset \subset M \backslash P_\Phi(x_0)$, $\Phi(x)\to \Phi(x_0)$ smoothly whenever $x\to x_0$.
        \end{enumerate}
        In this case, we call $X$ the {\it parameter space} of $\Phi$.
        
        In addition, for a non-negative integer $\mathfrak{g}_0$, we call $\Phi$ a {\it  Simon-Smith family of genus $\leq \mathfrak{g}_0$}, if $\Phi$ also satisfies:
        \begin{enumerate}[label=\normalfont(\arabic*)]
            \setcounter{enumi}{4}
            \item\label{item:SimonSmithFamilyGenus} For each $x \in X$ with $\Phi(x) \in \GS(M)$, $\Phi(x)$ has genus at most $\fg_0$.
        \end{enumerate}
    \end{defn} 

    \begin{rmk}
        Condition (2) implies that the family $\Phi$ is in some sense closed. More precisely, the subset 
        $$\{(x,p)\in(X,M):p\in\Phi(x)\}$$
        is closed in $X\x M$. This closedness property will be really  convenient in later parts (in \S \ref{min-max_iiiOffAndMCF} and the proof of Proposition \ref{prop:identifyHomoGroups}). 
    
        In the literature, there exist various notions of {\em generalized family of surfaces} in the Simon-Smith min-max setting. Our definition of Simon-Smith family can be viewed as a generalization of those in the following sense:
        Our continuity conditions (2), (3) and (4) are essentially the same as those in~\cite[Definition~2.2]{DR18}, but we also allow general measure-zero sets in the family. Our genus bound condition (5) generalizes the definitions in~\cite{DP10,CFS20,Fra21} without assuming a dense set of smooth surfaces in the family.
    \end{rmk}

    \begin{exmp}\label{ex:tori_interpolation} 
        At the end of \S \ref{subsubsect:fatteningIssue} we mentioned a ``pinch-off process," for the purpose of interpolating from a family of tori to a family of genus $0$ surfaces; see Theorem \ref{thm:mapping_cylinder}. This is, in fact, why we need to consider surfaces with singularities, more precisely, punctate surfaces. Let us illustrate the necessity of this with the following example.
    
        Let $\mathbb{S}^2 \subset \mathbb{R}^3$ be the unit $2$-sphere centered at the origin. We consider an $\mathbb{S}^2$-family of tori $\{\Sigma_p\}_{p \in \mathbb{S}^2}$, constructed as follows: Fix a small constant $\varepsilon > 0$ and for each $p \in \mathbb{S}^2$, let $\Sigma_p$ be obtained by replacing the small neighborhoods of $p$ and $-p$ with a thin cylinder connecting $p$ and $-p$ (see Figure~\ref{fig:tori}). For sufficiently small $\varepsilon > 0$, each $\Sigma_p$ is very close to $\mathbb{S}^2$ as varifolds or as reduced boundaries of Caccioppoli sets. However, we claim that there does not exist a $[0, 1] \times \mathbb{S}^2$-family $H$ of smooth surfaces with genus 0 or 1 in $\R^3$ such that:
        \begin{enumerate}[label=\normalfont(\arabic*)]
            \item For each $p \in \mathbb{S}^2$, $H(0, p) = \Sigma_p$ and $H(1, p) = \mathbb{S}^2$.
            \item $H|_{[0, 1)}$ is a continuous family of tori in the smooth topology.
            \item For any sequence $(t_i, p_i) \to (1, p)$, the tori $H(t_i, p_i)$ converge smoothly to $\mathbb{S}^2$ except at finitely many points in $\mathbb{S}^2$.
        \end{enumerate} 

    \begin{figure}[!ht]
        \centering
        \includegraphics[width=0.7\linewidth]{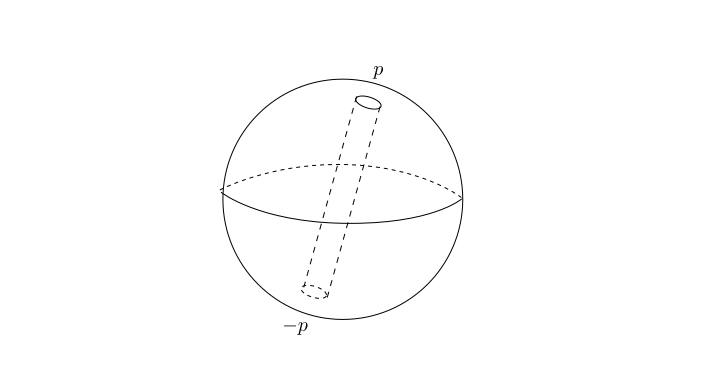}
        \caption{A torus close to a round sphere.}
        \label{fig:tori}
    \end{figure}
        
        Indeed, suppose by contradiction such map $H$ exists.  Then for each $t \in [0, 1)$, consider a solid torus bundle $E_t$ over $\mathbb{S}^2$, where the fiber at each $p \in \mathbb{S}^2$ is the solid torus bounded by $H(t, p)$. Note that $E_0$ is homotopy equivalent to the unit tangent bundle to $\mathbb{S}^2$, and thus, has a nontrivial Euler class. However, for $t$ sufficiently close to $1$, every fiber of $E_t$ contains the origin, implying that $E_t$ has a section, and thus the Euler class of $E_t$ is $0$. This leads to a contradiction.

        On the other hand, if surfaces with singularities are allowed in the interpolation, then the above task of interpolating the family $\{\Sigma_p\}_{p\in\mathbb S^2}$ to $\mathbb S^2$ can clearly be done, by pinching all thin necks at the origin.

    \end{exmp}

    \begin{rmk}\label{rmk:moreAboutPunctateSet}
        In the definition of punctate surface, we also care about isolated points of a punctate surface, and allow the punctate set to contain some regular points of a punctate surface for the following reasons:
        \begin{enumerate}
            \item First, fix two points $p,q\in M$ with distance at least $2$, and let $\Sigma_n$ be the union of two disjoint spheres, one centered at $p$ with radius $1$, and the other centered at $q$ with radius $1/n$. In our Simon-Smith min-max argument, we need to allow such convergence phenomenon within a family (see Theorem \ref{thm:mapping_cylinder}). Thus, in view of the closedness condition, Definition \ref{def:Simon_Smith_family} \ref{item:closedFamily}, the limit of $\Sigma_n$ is a sphere centered at $p$, {\it union the isolated point $q$}. Hence, the limit is a punctate surface, where its punctate set contains $\{q\}$.
            \item In the construction of the family  $\Psi$ in \S \ref{sect:family_Psi},  we allow a family to contain a sequence of smooth tori converging to a {\it smooth} sphere, smoothly {\it except at a regular point} of the sphere. Therefore, we need to consider punctate sets that contain regular points of the surfaces.
        \end{enumerate}
    \end{rmk}

    \begin{prop}[Simon-Smith families are sweepouts]\label{prop:SS_AP}
        \begin{enumerate}[label = \normalfont(\arabic*)]
            \item[]
            \item \label{item:prop:SS_AP_1} Every element $\Sigma$ in $\GS^*(M)$ is associated with a unique $2$-cycle $[\Sigma]$ in $\cZ_2(M; \Z_2)$. In particular, if $\cH^2(\Sigma)=0$, then $[\Sigma] = 0$.
            \item \label{item:prop:SS_AP_2} Moreover, for a Simon-Smith family $\Phi$, the induced map 
            \[
                [\Phi]: X \to \cZ_2(M;\bF;\Z_2)
            \]
            is continuous with respect to the $\bF$-metric.
        \end{enumerate}
    \end{prop}
    \begin{rmk}
        (2) implies that our Simon-Smith family induces a sweepout in the Almgren-Pitts min-max theory. 
    \end{rmk}
    \begin{proof}[Proof of Proposition~\ref{prop:SS_AP}]
        By Definition~\ref{def:punctate_surf}, for $\Sigma \in \GS(M)$, we can choose a punctate set $P \subset \Sigma$ such that $\Sigma \backslash P$ is a smooth surface with finite $\cH^2$ measure, so we can define
        \[
            [\Sigma] := [\Sigma \setminus P] \in \cZ_2(M; \Z_2)\,.
        \]
        Note that the definition is independent of the choice of punctate sets $P$. For $\cH^2(\Sigma)=0$, we can simply define
        \[
            [\Sigma] := 0 \in \cZ_2(M; \Z_2)\,.
        \]

        For (2), for any $x \in X$ and any sequence $\{y_i\} \subset X$ with $\lim_i y_i = x$, it suffices to prove that $[\Phi(y_i)] \to [\Phi(x)]$ with respect to the $\bF$-metric.
        There are two cases: (1) $\cH^2(\Phi(x))=0$, or (2) $\cH^2(\Phi(x))>0$.
        \medskip
    
        \paragraph*{\bf Case 1: $\cH^2(\Phi(x))=0$}
        In this case, by Definition~\ref{def:Simon_Smith_family}~\ref{item:Hausdorff_cts},
            \[
                \bM([\Phi(y_i)]) = \cH^2(\Phi(y_i)) \to \cH^2(\Phi(x)) = 0\,,
            \]
            and thus, $[\Phi(y_i)] \to 0 = [\Phi(x)]$ with respect to the $\bF$-metric.
        \medskip
        \paragraph*{\bf Case 2: $\cH^2(\Phi(x))>0$}
        In this case, we can assume that every $\Phi(y_i) \in \GS(M)$ and let $P := P_\Phi(x)$ defined in Definition~\ref{def:Simon_Smith_family}~\ref{item:SingPointsUpperBound}. For any $\varepsilon > 0$, by the facts that $\cH^2(\Phi(x)) < \infty$, that $\Phi(x)$ is closed, and by Sard's theorem, there exists $r > 0$ such that
        \begin{equation}\label{eqn:cH_x_inside}
            \cH^2(\Phi(x) \cap \overline{B}_{2r}(P)) < \frac{\varepsilon}{4}\,,
        \end{equation}
        \begin{equation}\label{eqn:small_balls}
            \cH^3(\overline{B}_{2r}(P)) < \varepsilon / 4\,,
        \end{equation}
        and $\partial B_{2r}(P)$ intersects transversally with $\Phi(x)$ and each $\Phi(y_i)$, where each intersection is a finite union of loops.
            
        By Definition~\ref{def:Simon_Smith_family}~\ref{item:SimonSmithFamilyLocalSmooth}, $\Phi(y_i) \to \Phi(x)$ smoothly outside $\overline{B}_{r}(P)$. In particular, we have:
        \begin{itemize}
            \item 
            \begin{equation}\label{eqn:cF_i}
                \limsup_i \cF([\Phi(y_i)], [\Phi(x)]) \leq \cH^3(\overline{B}_{2r}(P)) \leq \frac{\varepsilon}{4}\,.
            \end{equation}
            \item Outside $\overline{B}_{2r}(P)$, we have
            \begin{equation}\label{eqn:bF_i_outside}
                \lim_i \bF(|\Phi(y_i) \setminus \overline{B}_{2r}(P)|, |\Phi(x) \setminus \overline{B}_{2r}(P)|) = 0\,.
            \end{equation}
            In particular, $\lim_i \cH^2(\Phi(y_i) \setminus \overline{B}_{2r}(P)) = \cH^2(\Phi(x) \setminus \overline{B}_{2r}(P))$.
            \item Inside $\overline{B}_{2r}(P)$, 
            \begin{equation}\label{eqn:cH_i_inside}
            \begin{aligned}
                & \quad \limsup_i \cH^2(\Phi(y_i) \cap \overline{B}_{2r}(P))\\
                &= \limsup_i \left(\cH^2(\Phi(y_i)) - \cH^2(\Phi(y_i) \setminus \overline{B}_{2r}(P))\right)\\
                &= \lim_i \cH^2(\Phi(y_i)) - \lim_i \cH^2(\Phi(y_i) \setminus \overline{B}_{2r}(P))\\
                &= \cH^2(\Phi(x)) - \cH^2(\Phi(x) \setminus \overline{B}_{2r}(P))\\
                &= \cH^2(\Phi(x) \cap \overline{B}_{2r}(P)) < \frac{\varepsilon}{4}\,.
            \end{aligned}
            \end{equation}
        \end{itemize}
        Note that by the triangle inequality, we have
        \[\begin{aligned}
            &\quad \bF(|\Phi(y_i)|, |\Phi(x)|)\\
            &= \bF(|\Phi(y_i)|, |\Phi(x)|) + \cF(|\Phi(y_i)|, |\Phi(x)|)\\
            &\leq \bF(|\Phi(y_i)|, |\Phi(y_i) \setminus \overline{B}_{2r}(P)|) + \bF(|\Phi(x) \setminus \overline{B}_{2r}(P)|, |\Phi(x)|)\\
            & \quad + \bF(|\Phi(y_i) \setminus \overline{B}_{2r}(P)|,|\Phi(x) \setminus \overline{B}_{2r}(P)|) \\
            &= \bF(|\Phi(y_i) \cap \overline{B}_{2r}(P)|, 0) + \bF(0, |\Phi(x) \cap \overline{B}_{2r}(P)|)\\
            &\quad + \bF(|\Phi(y_i) \setminus \overline{B}_{2r}(P)|,|\Phi(x) \setminus \overline{B}_{2r}(P)|) \\
            &= \cH^2(\Phi(y_i) \cap \overline{B}_{2r}(P)) + \cH^2(\Phi(x) \cap \overline{B}_{2r}(P))\\
            &\quad + \bF(|\Phi(y_i) \setminus \overline{B}_{2r}(P)|, |\Phi(x) \setminus \overline{B}_{2r}(P)|)\,.
        \end{aligned}\]
        Applying \eqref{eqn:cH_x_inside}, \eqref{eqn:cF_i}, \eqref{eqn:bF_i_outside} and \eqref{eqn:cH_i_inside}, for any $\varepsilon > 0$, we obtain
        \begin{align*}
            \limsup_i \bF([\Phi(y_i)], [\Phi(x)]) &\leq \limsup_i \cF([\Phi(y_i)], [\Phi(x)]) \\
            &\qquad + \limsup_i \bF(|\Phi(y_i)|, |\Phi(x)|)\\
            &\leq \varepsilon\,,
        \end{align*}
        and thus, $\lim_i \bF([\Phi(y_i)], [\Phi(x)]) = 0$.
    \end{proof}

    \begin{defn}[Homotopy class]\label{def:homotopyClass}
        Two Simon-Smith families $\Phi_0$ and $\Phi_1$, which are parametrized by the same parameter space $X$, are said to be {\em homotopic} to each other if there exists a continuous map
        \[
           \varphi: [0, 1] \times X \to \operatorname{Diff}^\infty(M)
        \] such that:
        \begin{enumerate}
            \item $\varphi(0, x) = \operatorname{Id}$ for all $x \in X$.
            \item $\varphi(1, x)(\Phi_0(x)) = \Phi_1(x)$ for all $x \in X$.
        \end{enumerate}
        Note that $\textrm{Diff}^\infty(M)$ denotes the diffeomorphism group of $M$ equipped with the $C^\infty$ topology.
        
        The set of all families homotopic to a Simon-Smith family $\Phi$ is called the {\em homotopy class associated to $\Phi$}, and is denoted by $\Lambda(\Phi)$.
    \end{defn}
    \begin{rmk}\label{rmk:homotopy}
        The family
        \[
            H: [0, 1] \times X \to \GS^*(M), \quad (t, x) \mapsto \varphi(t, x)(\Phi(x))
        \] is also a Simon-Smith family. In particular, for every $t \in [0, 1]$,
        \[
            \Phi_t: X \to \GS^*(M), \quad x \mapsto \varphi(t, x)(\Phi(x))
        \] is a Simon-Smith family in $\Lambda(\Phi_0)$.

        Furthermore, for any choice of punctate sets
        \[
            P_\Phi: \{x \in X : \Phi(x) \in \GS(M)\} \to M
        \]
        for $\Phi$ satisfying the conditions \ref{item:SingPointsUpperBound} and \ref{item:SimonSmithFamilyLocalSmooth} in Definition~\ref{def:Simon_Smith_family}, any isotopy $\varphi$ would  induce a choice of $P_{\Phi'}(x) := \varphi(1, x) \circ P_\Phi(x)$ for every $\Phi':=\varphi(1,\cdot)\circ\Phi \in \Lambda(\Phi)$. Consequently, we may assume that $N_{P} < \infty$ is a constant within the homotopy class $\Lambda(\Phi)$.
        
        In addition, if $\Phi$ is of genus $\leq \mathfrak{g}_0$, then so is every $\Phi' \in \Lambda(\Phi)$.
    \end{rmk}

    \begin{defn}[Simon-Smith width]
        Given a homotopy class $\Lambda$, its {\em width} is defined by
        \[
            \mathbf{L}(\Lambda):=\inf_{\Phi\in \Lambda}\sup_{x\in X}\cH^2(\Phi(x))\,.
        \]
    \end{defn}
    
    \begin{defn}[Minimizing sequence and min-max sequence]
        A sequence $\{\Phi_i\}$ in $\Lambda$ is said to be {\em minimizing} if 
        \[
            \lim_{i\to\infty} \sup_{x\in X}\cH^2(\Phi_i(x))=\mathbf{L}(\Lambda)\,.
        \]
        If $\{\Phi_i\}$ is a minimizing sequence in $\Lambda$ and $\{x_i\} \subset X$ satisfies
        \[
            \lim_{i\to\infty} \cH^2(\Phi_i(x_i))=\mathbf{L}(\Lambda)\,,
        \]
        then $\{\Phi_i(x_i)\}$ is called a {\em min-max sequence}. 
        For a minimizing sequence $\{\Phi_i\}$, we define its \textit{critical set} $\bC(\{\Phi_i\})$ to be the set of all subsequential varifold-limit of its min-max sequences:
        \[
            \bC(\{\Phi_i\}):=\{V=\lim_j|\Phi_{i_j}(x_j)|: x_j\in X, \| V\|(M)=\bL(\Lambda)\}\,.
        \]
        And $\{\Phi_i\}$ is called \textit{pulled-tight} if every varifold in $\bC(\{\Phi_i\})$ is stationary.
    \end{defn}

    \begin{defn}\label{def:minimal_surf_varifolds}
        In a closed Riemannian $3$-manifold $(M, g)$, for $L > 0$, let $\cW_L(M, g)$ be the set of all varifolds $W \in \cV_2(M)$, with $\|W\|(M) = L$, of the form 
        \[
            W=m_1|\Gamma_1|+\cdots+m_l|\Gamma_l|\,,
        \]
        where $\{m_j\}$ is a set of positive integers and $\{\Gamma_j\}$ is a disjoint set of smooth, connected, embedded minimal surface in $(M, g)$. 
        
        For a nonnegative integer $\mathfrak{g}_0 \geq 0$, let $\cW_{L, \leq \mathfrak{g}_0}(M, g) \subset \cW_L(M, g)$ consist of all the varifolds satisfying the genus bound 
        \begin{equation}\label{eq:genus_bound}
        \sum_{j\in I_O}m_j\fg(\Gamma_j)+ \frac{1}{2}\sum_{j\in I_N}m_j(\fg(\Gamma_j)-1)\leq \fg_0\,,
        \end{equation}
        where $\Gamma_j$ is orientable if $j \in I_O$ and non-orientable if $j \in I_N$. 
        
        For simplicity, when the ambient manifold $(M, g)$ is clear from context, we may use the notations $\cW_L$ and $\cW_{L, \leq \mathfrak{g}_0}$ without further specification.
    \end{defn}
    \begin{rmk}
        Note that the genus of a non-orientable surface $\Sigma$ is the number of cross-caps needed to be attached to a two-sphere in order to obtain a surface homeomorphic to $\Sigma$.
    \end{rmk}
    
    \begin{thm}[Simon-Smith min-max theorem]\label{thm:minMax}
        Given an orientable closed Riemannian $3$-manifold $(M, g)$, a cubical subcomplex $X$ of some $I(m, k)$, a Simon-Smith family $\Phi: X \to \GS^*(M)$ of genus $\leq \fg_0$, and a positive real number $r > 0$, if $L := \mathbf{L}(\Lambda(\Phi)) > 0$, 
        then there exists a minimizing sequence $\{\Phi_{i}\}$ in $\Lambda(\Phi)$ such that:
	\begin{enumerate}[label=\normalfont(\arabic*)]
            \item\label{item:minMaxPulltight} The sequence $\{\Phi_i\}$ is pulled-tight and moreover,
            \[
                \bC(\{\Phi_{i}\})\cap \cW_{L, \leq \mathfrak{g}_0} \neq \emptyset\,.
            \]
            \item\label{item:minMaxMultiplicityone} There exists a $W \in \cW_{L, \leq \mathfrak{g}_0}$ such that every connected component $\Gamma_j$ in $W$ satisfies:
		\begin{enumerate}[label=\normalfont(\alph*)]
			\item If $\Gamma_j$ is unstable and two-sided, then $m_j = 1$.
			\item If $\Gamma_j$ is one-sided, then its connected double cover is stable.
		\end{enumerate}
            \item \label{item:minMaxW} Furthermore, there exists $\eta > 0$ such that for all sufficiently large $i$,
		\[
			\cH^2(\Phi_{i}(x)) \geq L - \eta \implies |\Phi_i(x)| \in \bB^{\Fb}_{r}(\cW_{L, \leq \mathfrak{g}_0})\,.
		\]
	\end{enumerate}
    \end{thm}
    
    This min-max theorem essentially builds upon the foundational results in~\cite{Smi83, DP10, Ket19, MN21, WZ23}. We postpone the detailed proof of this theorem to \S \ref{sect:min-max_i}.
    
\subsection{Relative Simon-Smith min-max theory}
    In this subsection, we fix {\it a pair of parameter spaces}, a cubical subcomplex $X$ and its subcomplex $Z \subset X$. We adopt the concept of relative $(X, Z)$-homotopy class introduced in~\cite{Zho20} to the Simon-Smith min-max theory as follows.

    \begin{defn}[Relative homotopy class]
        Two Simon-Smith families $\Phi_0$ and $\Phi_1$ parametrized by the same parameter space $X$ with $\Phi_0\vert_Z = \Phi_1\vert_Z$ are said to be {\em homotopic relative to $\Phi_0\vert_Z$} to each other if there exists a continuous map
        \[
           \varphi: [0, 1] \times X \to \operatorname{Diff}^\infty(M)
        \] 
        such that: 
        \begin{enumerate}
            \item $\varphi(0, x) = \operatorname{Id}$ for all $x \in X$.
            \item $\varphi(t, z) = \operatorname{Id}$ for all $t \in [0, 1]$ and $z \in Z$.
            \item $\varphi(1, x)(\Phi_0(x)) = \Phi_1(x)$ for all $x \in X$.
        \end{enumerate} 
        
        The set of all families homotopic relative to $\Phi\vert_Z$ to a Simon-Smith family $\Phi$ is called the {\em relative $(X, Z)$-homotopy class of $\Phi$}, and is denoted by $\Lambda_Z(\Phi)$.
    \end{defn}

    \begin{defn}[Relative Simon-Smith min-max width]
        Given a $(X, Z)$-relative homotopy class $\Lambda_Z$, its {\em width} is defined by
        \[
            \mathbf{L}(\Lambda_Z):=\inf_{\Phi\in \Lambda_Z}\sup_{x\in X}\cH^2(\Phi(x))\,.
        \]
    \end{defn}
    \begin{defn}[Minimizing sequence and min-max sequence for Relative Simon-Smith min-max]
        A sequence $\{\Phi_i\}$ in $\Lambda_Z$ is said to be {\em minimizing} if 
        \[
            \lim_{i\to\infty} \sup_{x\in X}\cH^2(\Phi_i(x))=\mathbf{L}(\Lambda_Z)\,.
        \]
        If $\{\Phi_i\}$ is a minimizing sequence and $\{x_i\} \subset X$ satisfies 
        \[
            \lim_{i\to\infty} \cH^2(\Phi_i(x_i))=\mathbf{L}(\Lambda_Z)\,,
        \]
        then $\{\Phi_i(x_i)\}$ is called a {\em min-max sequence}. 
        For a minimizing sequence $\{\Phi_i\}$, we define its \textit{critical set} $\bC(\{\Phi_i\})$ to be the set of all subsequential varifold-limit of its min-max sequences:
        \[
            \bC(\{\Phi_i\}):=\{V=\lim_j|\Phi_{i_j}(x_j)|: x_j\in X, \| V\|(M)=\bL(\Lambda_Z)\}\,.
        \]
        And $\{\Phi_i\}$ is called \textit{pulled-tight} if every varifold in $\bC(\{\Phi_i\})$ is stationary.
    \end{defn}

    \begin{lem}
        For a Simon-Smith family $\Phi: X \to \GS^*(M)$ and a subcomplex $Z \subset X$, we have 
        \[
            \bL(\Lambda(\Phi)) \leq \bL(\Lambda_Z(\Phi))\,.
        \]
    \end{lem}
    \begin{proof}
        This follows immediately from the fact that $\Lambda_Z(\Phi) \subset \Lambda(\Phi)$.
    \end{proof}

    \begin{thm}[Relative Simon-Smith min-max theorem]\label{thm:relative_minMax}
        Given an orientable closed Riemannian $3$-manifold $(M, g)$, a pair of cubical subcomplexes $Z \subset X$ of some $I(m, k)$, a Simon-Smith family $\Phi: X \to \GS^*(M)$ of genus $\leq \fg_0$, and a positive real number $r > 0$, if 
        \[
            L = \mathbf{L}(\Lambda_Z(\Phi)) > \sup_{z \in Z} \cH^2(\Phi(z))\,,
        \]
        then there exists a minimizing sequence $\{\Phi_i\}$ in $\Lambda_Z(\Phi)$ such that conclusions (1), (2) and (3) in Theorem~\ref{thm:minMax} still hold.
    \end{thm}

    This relative min-max theorem can be proved in much the same way as~\ref{thm:minMax}, carefully following the steps of that proof. A proof is also provided in \S \ref{sect:min-max_i}.
    
\subsection{Almgren-Pitts min-max theory}
    In 1962, after F.J. Almgren completed his thesis~\cite{Alm62}, H. Federer suggested that he develop a Morse theory for minimal surfaces. In~\cite{Alm65}, Almgren initiated a program to develop a geometric version of the calculus of variations in the large, aimed at proving the existence of critical points for the area functional. This program is based on what is nowadays called the \emph{Almgren isomorphism theorem}~\cite{Alm62, Alm65} (see also~\cite[\S 2.5]{LMN18}). In particular, when equipped with the flat topology, the modulo $2$ cycle space $\Zc_{n}(M;\Z_2)$ is weakly homotopy equivalent to $\R\mathbb P^\infty$. The rich topological structure of $\Zc_{n}(M;\Z_2)$ makes Morse theory possible. 
    
    In the following, we denote the cohomology ring of $\Zc_{n}(M;\Z_2)$ by $\Z_2[\bar\lambda]$, where $\bar \lambda \in H^1(\Zc_{n}( M;\Z_2), \Z_2)$ is the generator. Let $\Pc_p$ be the set of all $\Fb$-continuous maps $\Phi:X\to \Zc_2(M; \bF; \Z_2)$, where $X$ is a finite simplicial complex, such that $\Phi^*(\bar\lambda^p)\ne 0$. Elements of $\Pc_p$ are called {\it $p$-sweepouts.} Note that every finite cubical complex is homeomorphic to a finite simplicial complex and vice versa (see~\cite[\S 4]{BP02}). So in the above notion of $p$-sweepouts, we may as well require $X$ to be a finite cubical complex.

    \begin{defn}[Almgren-Pitts $p$-width]
        Denoting by $\dmn(\Phi)$ the domain of $\Phi$, the {\it $p$-width} of $(M,g)$ is defined by 
        \[
            \omega_p(M,g):=\inf_{\Phi\in\Pc_p}\sup_{x\in \dmn(\Phi)}\Mb(\Phi(x))\,.
        \]
    \end{defn}
        
    \begin{rmk}\label{def_equiv_width}
        There is an equivalent definition of $p$-widths from~\cite[Remark 5.7]{MN21}: First, an $\Fc$-continuous map $\Phi:X\to\Zc_n(M;\Z_2)$ is said to have {\it no concentration of mass} if 
        \[
            \lim_{r\to 0}\sup_{x\in X,q\in M}\norm{\Phi(x)}(B_r(q))=0\,.
        \]
        Then, when defining the $p$-width, instead of using the collection $\Pc_p$, we use the collection of all $\Fc$-continuous maps $\Phi$ with no concentration of mass such that $\Phi^*(\bar\lambda^p)\ne 0$. 

        The interpolation results therein imply that the two definitions yield the same $p$-widths.
    \end{rmk}

    As in Simon-Smith min-max theory, we can define the notion of minimizing sequences, critical sets, and min-max sequences for the $p$-widths.

    In this paper, the only result we require from the Almgren-Pitts min-max theory is that the $5$-width of the unit $3$-sphere $(S^3,\bar g)$ is greater than $4\pi$. In fact, C. Nurser~\cite{Nur16} proved that
    \[
        \omega_5(S^3,\bar g)=2\pi^2\,.
    \]

\section{Proof of Theorem~\ref{thm:main} }\label{sect:mainProof} 

    In this section, we will prove our main theorem, Theorem~\ref{thm:main}: Given a  $3$-sphere $(S^3,g_0)$ of positive Ricci curvature, it contains at least $5$ embedded minimal tori. The proof consists of five key steps: 
    \begin{enumerate}
        \item Define a 9-parameter Simon-Smith family $\Psi$ of genus $\leq 1$,  building upon the $5$-parameter canonical family discovered by Marques-Neves in their proof of the Willmore conjecture~\cite{MN14}.
        \item Repeatedly apply the Simon-Smith min-max theory to $\Psi$. Namely, each time we detect a minimal surface, which must be either a multiplicity one minimal sphere or torus by Wang-Zhou's multiplicity one theorem~\cite{WZ23}, we remove the region in the family near the minimal surface (which is like a cap), resulting in a new family of maximal area strictly less than the area of the minimal surface. We repeat the min-max process on this new family.
        \item  Assemble pieces we obtained in step 2 to get a $9$-parameter family $\Xi$, which is, in a certain sense, homotopic to $\Psi$.
        \item  Let $N$ be the number of embedded minimal tori obtained via min-max in Step 2. Then we show that the parameter space $\dmn(\Xi)$ of $\Xi$ can be decomposed into a union of subsets $D_0,D_1,\cdots,D_N$, such that the restriction $\Xi|_{D_0}$ is not a 5-sweepout, while for each $i=1,\cdots,N$, the image of $H_1(D_i;\Z_2)$ mapped into $H_1(\dmn(\Xi);\Z_2)$ is trivial.
        \item  Show that $N\geq 5$ by analyzing the cohomology ring of $\dmn(\Xi)$ and applying a Lyusternik-Schnirelmann argument.
    \end{enumerate}
    
\subsection{A 9-parameter family}\label{subsect:9param_family}
    Let us first define a $9$-parameter family $\Psi$ of surfaces in the unit $3$-sphere. It is worth noting that F. C. Marques previously considered a similar family to prove that the $8$-width of the unit $3$-sphere is $2\pi^2$~\cite{Mar23}.

    For simplicity, throughout \S\ref{subsect:9param_family}, we denote the unit $3$-sphere by $\mathbb{S}^3$, and the unit open $4$-ball in $\R^4$ by $\mathbb{B}^4$. We fix an unoriented Clifford torus $\Sigma_0$ in $\mathbb{S}^3$.
    
    Consider for each $v\in \mathbb{B}^4$ the conformal diffeomorphism $F_v:\mathbb{S}^3 \to \mathbb{S}^3$ given by
    \begin{equation}\label{eq_conf_map}
        F_v(x)=\frac{1-|v|^2}{|x-v|^2}(x-v)-v\,,
    \end{equation}
    which pushes everything away from $v/|v|$ and towards $-v/|v|$ whenever $v\ne 0$. 
    
    Next, {\it we choose an orientation for $\Sigma_0$.} This gives an inward normal direction for $\Sigma_0$, and consequently, for each $F_v(\Sigma_0)$. With this orientation, we can define the signed distance $d_v:\mathbb{S}^3 \to \R$ to $F_v(\Sigma_0)$, which is positive outside $F_v(\Sigma_0)$ and negative inside. 
    
    Then we define the following continuous family of $2$-cycles: For each pair $(v,t)\in \mathbb{B}^4\x [-\pi,\pi]$, let
    \begin{equation} \label{5family}
        \Sigma(v,t) := \partial\{x\in \mathbb{S}^3:d_v(x)<t\}\in \cZ_2(\mathbb S^3;\Z_2)\,.
    \end{equation}
    Marques-Neves~\cite[\S 5]{MN14} showed that by ingeniously reparametrizing this family of cycles {\it near} the boundary of the parameter space $\mathbb{B}^4 \x [-\pi,\pi]$, one can actually continuously (in the flat topology) extend this collection to the boundary, and consequently obtains a continuous family of $2$-cycles in $\cZ_2(\mathbb S^3;\Z)$ parametrized by $\overline{\mathbb{B}^4} \x [-\pi,\pi]$. Furthermore, any two ``antipodal" points $(v,t)$ and $(-v,-t)$ on the boundary would correspond to two $2$-cycles that differ exactly by a sign. Then, by identifying the boundary of parameter space via $(v,t)\sim (-v,-t)$ and forgetting the orientations of all cycles, C. Nurser in~\cite[\S 3.4.2]{Nur16} obtained an $\cF$-continuous map 
    \[
        \Phi_5:\RP^5\to\Zc_2( \mathbb S^3;\Z_2)\,,
    \] 
    which he showed to be a 5-sweepout by proving $\Phi_5^*(\bar\lambda^5)\ne 0$ and $\Phi_5$ has no concentration of mass ($\bar\lambda$ is the non-trivial element of $H^1(\cZ_2(\mathbb S^3;\Z_2);\Z_2)$). 
    The details of this procedure was carried out carefully in 
   ~\cite[\S 3.4.2]{Nur16}, but for the sake of completeness, we will also provide the details in \S\ref{sect:family_Psi}. 

    Now, for each {\it oriented} Clifford torus $\Sigma$, one can repeat the above procedure, with $\Sigma$ in place of $\Sigma_0$, to obtain a map
    \[
        \Phi^\Sigma_5:\RP^5\to\Zc_2(\mathbb S^3;\Z_2)\,.
    \]
    Thus, if we let $\tilde \cC$ be the space of oriented Clifford tori, we immediately obtain a map 
    \begin{equation}\label{eq:PsiTilde}
        \RP^5\x\tilde \cC\to\Zc_2(\mathbb S^3;\Z_2)\,.
    \end{equation}
    However, the topology of the parameter space of this family would not satisfy our need. We would want to identify Clifford tori of opposite orientations. Thus, we wish to identify the subfamily $\Phi^\Sigma_5$ and  $\Phi^{-\Sigma}_5$ of the family (\ref{eq:PsiTilde}), where $-\Sigma$ is  $\Sigma$ with an opposite orientation. We claim that this can be done. More precisely, let us consider instead the space $\cC$ of {\it unoriented} Clifford tori, which is topologically an $\RP^2\x\RP^2$~\cite[\S 3.4.3]{Nur16}. We will prove in \S \ref{sect:family_Psi} the following theorem.

    \begin{thm}\label{thm:PsiDef}
        There exists an $\RP^5$-bundle $Y$ over $\RP^2\x\RP^2$, and a Simon-Smith family of genus $\leq 1$
        \[
            \Psi:Y\to \GS^*(S^3)\,,
        \]
        satisfying the following property: If $\Sigma$ is an oriented Clifford torus and $b\in\RP^2\x\RP^2$ corresponds to $\{\Sigma,-\Sigma\}$, then there exists a homeomorphism $f:Y_{b}\to\RP^5$ such that $\Phi^\Sigma_5\circ f=[\Psi|_{Y_{b}}]$.
    \end{thm}

    Note that in the above, $Y_b$ denotes  the fiber at $b$,  $\Psi|_{Y_b}$  the restriction of $\Psi$ on $Y_b$, and  $[\Psi|_{Y_b}]$  the induced map into $\cZ_2(\mathbb S^3;\Z_2)$. Also note that since we are in the Simon-Smith setting, $\Psi$ maps into the set $\GS^*(S^3)$ of punctate surfaces instead of the cycle space $\cZ_2(\mathbb S^3;\Z_2)$.

    \begin{rmk}
        In~\cite[\S 3.4.3]{Nur16} C. Nurser introduced a $9$-parameter family which is parametrized by $\RP^5\x\RP^2\x\RP^2$. However, that family is actually not well-defined. The issue is that to define an $\RP^5$-family as described above, one needs to specify an orientation for the Clifford torus. There is not a continuous way to assign an orientation to each unoriented Clifford torus in $\RP^2\x\RP^2$.
    \end{rmk}

\subsection{Repetitive min-max} 

    Let $(S^3,g_0)$ be the $3$-sphere given in our main theorem, Theorem~\ref{thm:main}. To prove the theorem, we can assume $(S^3,g_0)$ has only finitely many embedded minimal tori; otherwise, we would already have at least five embedded minimal tori and the proof would be complete.

    In \S\ref{subsect:9param_family}, we obtained a $9$-parameter Simon-Smith family $\Psi$ of genus $\leq 1$ on the unit $3$-sphere. But this family can also be viewed as a Simon-Smith family in $S^3$ with any metric.
    
    We will soon run the min-max process repeatedly starting from $\Psi$. To ensure that the process terminates in finitely many steps and that all the minimal surfaces detected have multiplicity one, we need to perturb $g_0$ to a new metric $g'$ with the following properties:
    \begin{prop}\label{prop:perturbMetric}
        For a Riemannian metric $g$ of positive Ricci curvature on $S^3$, if there are finitely many embedded minimal tori in $(S^3, g)$, then there exists another metric $g'$ of positive Ricci curvature such that: 
        \begin{enumerate}[label=\normalfont(\arabic*)]
            \item\label{item:sameNumber}  $(S^3,g)$ and $(S^3,g')$ have the same number of embedded minimal tori.
            \item $(S^3,g')$ has finitely many embedded minimal spheres, each of which is non-degenerate.
            \item \label{item:area_sphere_neq_torus} For any embedded minimal sphere $S$ in $(S^3, g')$ and any positive integers $m$, the value $m \cdot \area_{g'}(S)$ can never equal the $g'$-area of an embedded minimal torus in $(S^3, g')$.
            \item \label{item:area_sphere_linear_indep} For a pair of distinct embedded minimal spheres $S_1, S_2$ in $(S^3, g')$ and a pair of positive integers $m_1, m_2$, we have
            \[
                m_1\cdot \area_{g'}(S_1) \neq m_2 \cdot \area_{g'}(S_2)\,.
            \]
        \end{enumerate}
    \end{prop}
    
    The proof of the proposition above will be postponed to  \S \ref{sect:perturbMetric}. In fact, from the proof, the metric $g'$ can be made arbitrarily close to $g$ in $C^\infty$, and all the $g'$-minimal tori coincide with the $g$-minimal tori. However, we will not need these facts. 

    Now, let us apply the above proposition to $(S^3,g_0)$ to obtain a modified metric $g'$ satisfying the above properties. To prove Theorem~\ref{thm:main}, it suffices to show that $(S^3,g')$ has at least five embedded minimal tori. Therefore, in what follows, we will focus on $(S^3, g')$.
    We will need the following useful lemma.

    \begin{lem}\label{lem:W_one_elmt}
        In $(S^3, g')$, for any $L > 0$, suppose that $\mathcal{W}_{L, \leq 1}$ (defined using (\ref{eq:genus_bound})) has a varifold associated with a multiplicity-one minimal surface $\Sigma$.
        \begin{enumerate}[label=\normalfont(\arabic*)]
            \item If $\Sigma$ is a minimal torus, then every varifold in $\mathcal{W}_{L, \leq 1}$ is associated with a multiplicity-one minimal torus. 
            \item If $\Sigma$ is a multiplicity-one minimal sphere, then 
            \[
                \mathcal{W}_{L, \leq 1} = \{|\Sigma|\}\,.
            \]
        \end{enumerate}
        In particular, $\mathcal{W}_{L, \leq 1}$ consists of finitely many varifolds, each associated with a multiplicity-one minimal surface.
    \end{lem}
    \begin{proof}
        Since $(S^3, g')$ has positive Ricci curvature, the Frankel property holds, i.e., every pair of connected embedded minimal surfaces intersects with each other. Consequently, every varifold in $\mathcal{W}_{L, \leq 1}$ is associated with a surface, which is either a minimal torus with multiplicities, or a minimal sphere with multiplicities. 

        For (1), by \ref{item:area_sphere_neq_torus} of Proposition~\ref{prop:perturbMetric}, every varifold in $\mathcal{W}_{L, \leq 1}$ is associated with a multiplicity-one minimal torus. 

        For (2), by \ref{item:area_sphere_neq_torus} of Proposition~\ref{prop:perturbMetric} again, every varifold in $\mathcal{W}_{L, \leq 1}$ is associated with a minimal sphere with multiplicities. By  \ref{item:area_sphere_linear_indep}
        of Proposition~\ref{prop:perturbMetric}, we have $\mathcal{W}_{L, \leq 1} = \{[\Sigma]\}$.
    \end{proof}

    We now apply the min-max process repeatedly, starting with our canonical family $\Psi$, treated as a family in $(S^3,g')$. We choose some $d_0>0$ to be less than the distance  $\bF(\Sigma,\Sigma')$   for any two minimal surfaces $\Sigma,\Sigma'$  of genus 0 or 1 in $(S^3, g')$  ($\bF$ denotes the distance for currents).  Note, this is possible as we assumed  $(S^3, g')$ has only  finitely many minimal tori and spheres.  (Though, in \S \ref{sect:genus1caps}, we will need to choose an even smaller $d_0$, depending on $g'$ only; see the paragraph following the proof of Proposition~\ref{prop:tauSigma}.)
    
\subsubsection{First stage.}\label{subsubsect:firstStage} 
    In $(S^3, g')$, applying Theorem~\ref{thm:minMax} to $\Psi$ with $r=d_0$, we obtain a varifold $V^1$ in $\cW^1:=\cW_{\bL(\Lambda(\Psi)),\leq 1}$, where $V^1$ is induced by a multiplicity-one minimal surface of genus $0$ or $1$. By Lemma~\ref{lem:W_one_elmt}, we have the following dichotomy:
    \begin{itemize}
        \item[Case 1] $\cW^1$ consists finitely many finitely many varifolds, each associated to a multiplicity-one minimal torus.
        \item[Case 2] $\cW^1 = \{V^1\}$ where $\spt \|V^1\|$ has genus $0$.
    \end{itemize}

    Let us investigate case 1 first. 
    \medskip

    \paragraph*{\bf Case 1.}
    From Theorem~\ref{thm:minMax}~\ref{item:minMaxW} it follows immediately that there exists a $\delta>0$ and some $\Psi'\in\Lambda(\Psi)$ such that 
    \[
        \bM(\Psi'(x))\geq \bL(\Lambda(\Psi))- \delta\;\;\Rightarrow\;\;|\Psi'(x)|\in\bB^\bF_{d_0}(\cW^1)\,.
    \]
    However, this is not enough for us to prove Theorem~\ref{thm:trivialInFirsthomo}. A stronger result is required.
    
    \begin{thm}\label{thm:currentsCloseInBoldF}
        Let $(M, g)$ be an orientable closed Riemannian $3$-manifold, $\Lambda$ be a homotopy class of a Simon-Smith families of genus $\leq \fg_0$ such that $L = \mathbf{L}(\Lambda) > 0$. Suppose that $\cW_{L, \leq\mathfrak{g}_0}$ consists of finitely many varifolds, each associated with a connected multiplicity-one minimal surface.
        Then for any $r > 0$, there exists $\eta>0$ and $\Phi \in \Lambda$ such that 
        \[
          \bM(\Phi(x))\geq L -\eta \implies [\Phi(x)] \in \bB^\bF_r([\cW_{L, \leq \mathfrak{g}_0}]) \,.
        \]
    \end{thm}

    The notation $[\cW_{L, \leq \mathfrak{g}_0}]$ used here is as introduced in \S \ref{sect:prelim}, and $\bB^\bF_r([\cW_{L, \leq \mathfrak{g}_0}])$ denotes the $r$-neighborhood of $[\cW_{L, \leq \mathfrak{g}_0}]$ in $\cZ_2(M; \bF; \Z_2)$. We will prove Theorem~\ref{thm:currentsCloseInBoldF} in \S \ref{sect:min-max_ii}. 

    Applying the above theorem with $(M, g) = (S^3, g')$, $\Lambda = \Lambda(\Psi)$, $\mathfrak{g}_0 = 1$ and $r=d_0$, we obtain $\eta > 0$ and $\Phi \in \Lambda(\Psi)$ satisfying the aforementioned property. For later reference, we denote $\delta_1 := \eta$ and $\tilde\Psi^1 := \Phi$.
    
    By refining the cubical complex structure of the parameter space $Y$, we can obtain a $9$-dimensional subcomplex $C^1$ of $Y$, such that 
    \begin{equation}\label{eq:Phi_iMassBound}
        \bM(\tilde \Psi^1(x))\geq \bL(\Lambda(\Psi))-\delta_1/2 \implies x\in C^1,
    \end{equation}
    and  
    \begin{equation}\label{eq:C1liesIn} 
        x\in C^1 \implies [\tilde
        \Psi^1(x)]\in\bB^\bF_{d_0}([\cW^1]).
    \end{equation} 
    Roughly speaking, $C^1$ is the ``cap'' where $\tilde \Psi^1$ has large area and  is close to some embedded minimal tori.
    
    Since $(S^3,g')$ has only finitely many embedded minimal tori, by the definition of $d_0$, and by further refining $Y$, we can assume that $C^1$ can be decomposed into a {\it disjoint} union of $9$-dimensional  subcomplexes $C^1_1,\cdots,C^1_{n_1}$ such that: For each $j=1,\cdots,n_1$ there is a distinct embedded minimal torus  $T^1_j\in \GS^*(S^3)$, with  $|T^1_j|\in \cW^1$, such that the image 
    \[
        [\tilde \Psi^1](C^{1}_j)\subset \bB^\bF_{d_0}([T^1_j])\,.
    \]
    
    Finally, we define 
    \[
        \Psi^1 := \tilde \Psi^1|_{\overline{Y\backslash C^1}}\,.
    \]
    By \eqref{eq:Phi_iMassBound}, $\Psi^1$ satisfies that
    \begin{equation}\label{eq:Psi1MassBound}
        \sup\cH^2\circ\Psi^1<\bL(\Lambda(\Psi)),
    \end{equation}
    Later when we apply the min-max process for the second time, we will start with $\Psi^1$, and the area upper bound guarantees new min-max minimal surfaces will be detected. For notational convenience later, let us denote $\tilde Y^1:= \dmn(\tilde\Psi^1)=Y$  and $Y^1=\overline{Y\backslash C^1}=\dmn(\Psi^1)$.
    \medskip
    
    \paragraph*{\bf Case 2.}
    In this case, the goal is once again to construct a map $\Psi^1$ that satisfies the bound (\ref{eq:Psi1MassBound}). Additionally, we want $\Psi^1$, when restricted to the boundary of the removed cap, to form a Simon-Smith family of genus $0$.
    To achieve this, we will need to apply an important interpolation theorem. But first, let us introduce some concepts.

    \begin{figure}[h]
        \centering
        \makebox[\textwidth][c]{\includegraphics[width=3in]{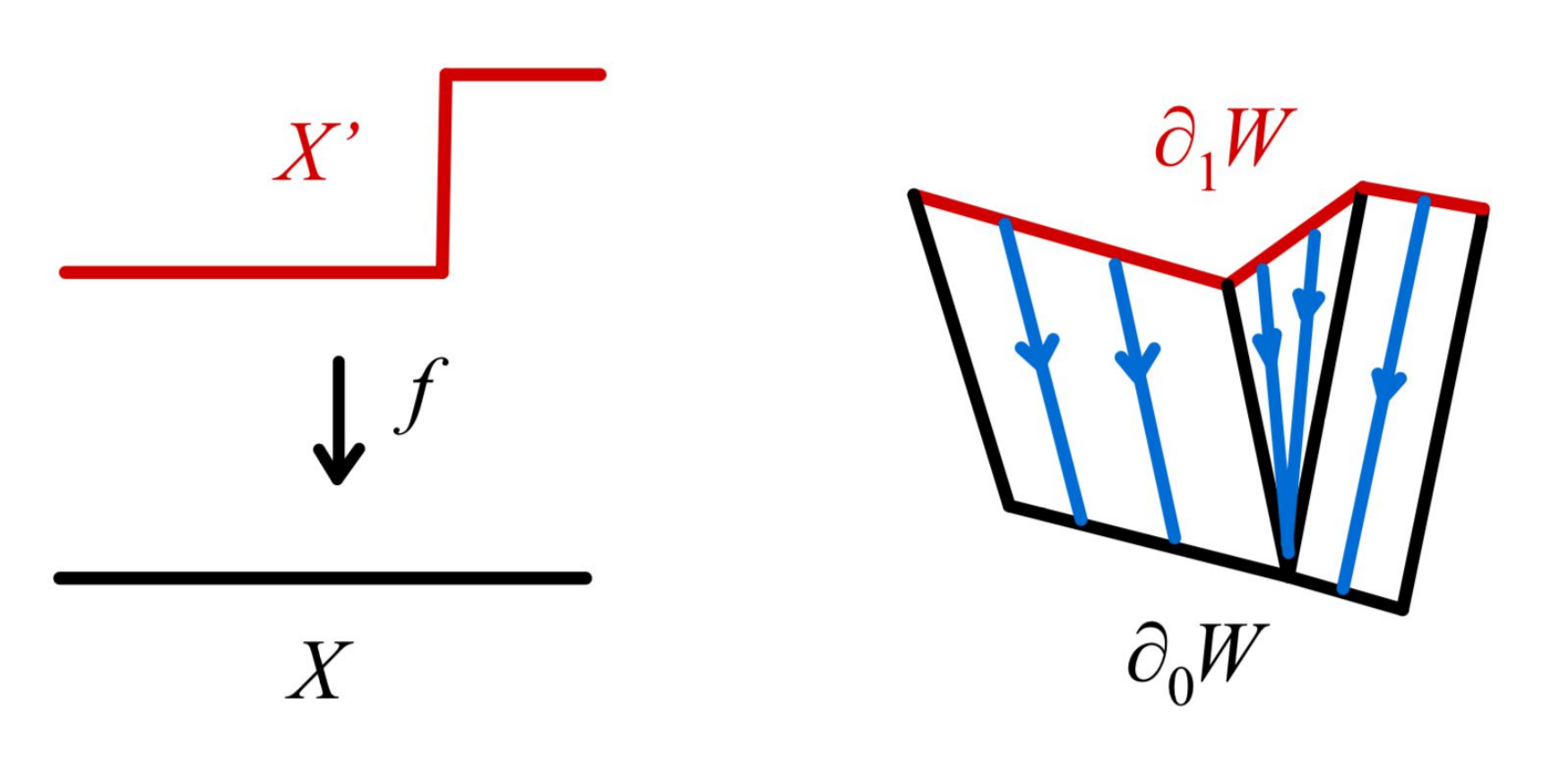}}
        \caption{The figure on the left shows a map $f:X'\to X$. The figure on the right shows the mapping cylinder $W:=M_f$, where the blue lines describe the deformation retraction of $W$ onto $\partial_0 W$ induced by the map $F_W$.}
        \label{fig:w}
    \end{figure}
    
    Given two topological spaces $X$ and $X'$, and a continuous map $f:X'\to X$, the {\it mapping cylinder} of $f$ is defined by 
    \[
        M_f := (([0, 1] \times X') \sqcup X) / \sim\,,
    \]
    where $(0, x) \sim f(x)$ for all $x \in X'$. 
    
    In the following, we will primarily be interested in the case where $X'$ and $X$ are both finite cubical complexes, and $f:X'\to X$ is a surjective homotopy equivalence that is a {\it cubical  map}, i.e., every cell of $X'$ is mapped to a cell of $X$. In this case, given the mapping cylinder $W:=M_f$, which can be viewed as a simplicial complex, we define the following subsets of $W$:
    \[ 
        \partial_0 W:=(\{0\}\x X')/\sim\,,\quad \partial_1 W:=(\{1\}\x X')/\sim\,.
    \]
    Note that $\partial_0W\cong X$ and $\partial_1W\cong X'$. Heuristically, $W$ is like a cobordism between $\partial_0 W$ and $\partial_1 W$.    Also,  we will sometimes by abuse of notation just denote $\partial_0W$ by $X$, and $\partial_1W$ by $X'$. Now, consider the map 
    \[
        [0,1]\x [0,1] \x X'\to [0,1] \x X'
    \]
    defined by sending each $(t,s,x')$ to $(ts,x')$. This induces a map 
    \begin{equation}\label{eq:Fw}
        F_W:[0,1]\x W\to W
    \end{equation}
    (see Figure~\ref{fig:w}). The key property of $F_W$ is that the map $(t,w)\mapsto F_W(1 - t, w)$,  $(t, w) \in [0, 1] \x W$, is a {\it strong deformation retraction} of $W$ onto $\partial_0W $: This means
    for all $t\in [0,1],w\in W,x\in \partial_0 W$ we have
    \[
        F_W(1, w) = w\,,\quad F_W(0, w)\in \partial_0 W\,,\quad F_W(t, x)= x\,.
    \]
    By~\cite[\S 4]{BP02}, $W$ can also be viewed as a cubical complex.

    Now, we can state the interpolation theorem.

    \begin{thm}\label{thm:mapping_cylinder}
        In a closed $3$-dimensional Riemannian manifold $(M, g)$, let $\Phi: X \to \GS^*(M)$ be a Simon-Smith family of genus $\leq 1$, where $X$ is a cubical subcomplex in $I^m$. Suppose that $L = \bL(\Lambda(\Phi)) > 0$, and the set $\cW_{L, \leq 1}$ consists of exactly one varifold associated with a multiplicity-one nondegenerate minimal sphere. Then after refining $X$, there exists a cubical subcomplex $X'$ in $I^{m+1}$, a homotopy equivalence
        $f: X' \to X\,,$
        which is a surjective cubical map, and a Simon-Smith family $\Phi': X' \to \GS^*(M)$ of genus $\leq 1$ with the following properties.
        \begin{enumerate}[label=\normalfont(\arabic*)]
            \item There exists an $\eta>0$ such that for $x \in X'$,  
                \[
                    \cH^2(\Phi'(x))\geq L - \eta\quad\Longrightarrow\quad \fg(\Phi'(x))=0\,.
                \]
            \item\label{item:mapping_cylinder} Let $W$ be the mapping cylinder $M_f = (([0, 1] \times X') \sqcup X) / \sim$ of $f$, where $(0, x) \sim f(x)$ for all $x \in X'$. There exists a Simon-Smith family $H: W \to \cS^*(M)$ of genus $\leq 1$ such that:
            \begin{enumerate}[label=\normalfont(\alph*)]
                \item $H|_{\partial_0 W}=\Phi$ and $H|_{\partial_1 W} = \Phi'$.
                \item\label{item_item:pinch-off} For all $x \in X'$, $t \mapsto H(t, x)$ is a {\em pinch-off process} (see Definition~\ref{defn:pinch_off}). 
            \end{enumerate}
        \end{enumerate}
    \end{thm}
    
    Heuristically, if a torus is sufficiently to a sphere, one can find a loop within a small ball. Under mean curvature flow, it is expected that the flow would undergo a {\em neck pinch} along the loop, yielding a genus $0$ punctate surface. However, as explained in \S \ref{subsubsect:fatteningIssue}, we cannot employ mean curvature flow to reduce the genus. Instead, we {\em pinch off} the small ball containing the loop. This process involves a finite number of neck-pinch surgeries, so that the surface becomes disjoint from the small ball, and then shrinking the components inside the ball into a point.

    Note that the following definition is purely topological.
    
    \begin{defn}\label{defn:pinch_off}
        In a closed Riemannian manifold $(M, g)$, a Simon-Smith family $\Phi: [a, b] \to \GS^*(M)$, where the parameter space $[a, b]$ is viewed as a time interval, is called a {\em pinch-off process} if the following holds. 
        
        There exist finitely many spacetime points $(t_1,p_1),\cdots,(t_n,p_n)$, where $t_i\in[a,b]$ and $p_i\in\Phi(t_i)$, such that: 
        \begin{enumerate}[label=\normalfont(\arabic*)]
            \item For each $(t,p)$, where $t\in[a,b]$ and $p\in\Phi(t)$, different from any $(t_i,p_i)$, there exists an open time interval $J\subset [a,b]$ around $t$, a one-parameter group of diffeomorphisms $\{\varphi_{t'}\}_{t' \in J} \subset \operatorname{Diff}(M)$ and a open ball $U\subset S^3$ around $p$ such that for all $t' \in J$,
            \[
                \varphi_{t'}(\Phi(t)) \cap U = \Phi(t') \cap U\,.
            \]
            \item For each $i=1,\cdots,n$, there exists an open time interval $J\subset [a,b]$ around $t_i$ and a ball $U\subset S^3$ around $p_i$ such that the family 
            \[
                \{\Phi(t')\cap U\}_{t'\in J}
            \] is of one of the following types:
            \begin{enumerate}[label=\normalfont(\alph*)]
                \item A {\it surgery process}: Pinching a (topological) cylinder at the point $p_i$, at time $t_i$, to obtain a (topological) double cone, and then either splitting it into two smooth discs or leaving it  intact.
                \item A {\it shrinking process}: For each $t'\in J$, $\Phi(t')$ does not intersect $\partial U$, and $\{ \Phi(t')\cap U\}_{t' \in J\cap[a,t_i)}$ is induced by a one-parameter group of diffeomorphisms in $U$, and $\Phi(t')\cap U = p_i$ for all $t' \in J \cap [t_i,b]$.
           \end{enumerate}
        \end{enumerate}
    \end{defn}

    \begin{rmk} 
        Let us compare pinch-off process and mean curvature flow. Mean curvature flow could introduce more complicated singularities than pinch-off process. Namely, under mean curvature flow, besides the above two types of singularities, there could also be singularities modeled by non-compact self-shrinkers of genus greater than $0$. 
        
        Furthermore, as we will prove in  \S \ref{min-max_iiiOffAndMCF}, a pinch-off process simplifies the topology of the initial condition in certain sense. This can be compared to B. White's result that mean curvature flow simplifies the topology of the surfaces~\cite{Whi95}.
    \end{rmk}

    We will prove the interpolation theorem, Theorem~\ref{thm:mapping_cylinder}, in \S\ref{sect:min-max_iii}. Here is a brief outline: First, we show that if a torus is close to a sphere, in the varifold sense, it must contain a non-trivial loop in a {\it small} ball. Heuristically, we can then pinch this loop to obtain a punctate surface with genus $0$. The main challenge lies in the fact that the choice of loops may not be continuous in a Simon-Smith family, causing the pinching process to fail in producing a desired deformation map $H$. To address this, we will extend Pitts's combinatorial arguments~\cite[Theorem~4.10]{Pit81}, and construct a map $H$ which consists of a union of pinch-off processes.

    Note that $H$ itself is not a homotopy between $\Phi$ and $\tilde\Phi$ in the sense of Simon-Smith min-max theory due to surgeries and shrinking, but the induced maps $[\Phi]$ and $[\Phi']$  are in the same homology class in $\cZ_2(M;\bF;\Z_2)$, via $[H]$. 

    Applying  Theorem~\ref{thm:mapping_cylinder} to $\Psi$, we obtain:
    \begin{itemize}
        \item a cubical complex $\tilde Y^1$, 
        \item a surjective, cubical, homotopy equivalence $f^1:\tilde Y^1\to Y$,
        \item a mapping cylinder $W^1:=M_{f^1}$ (which can be viewed as a simplicial complex) with
        \[
            Y=\partial_0 W^1, \quad\tilde Y^1= \partial_1 W^1\,,
        \] and the associated map $F^1:=F_{W^1}:[0,1]\x W^1\to W^1$ (see \eqref{eq:Fw}),
        \item a Simon-Smith family  of genus $\leq 1$
        \[
            \tilde\Psi^1:\tilde Y^1\to\cS^*(S^3)\,,
        \]
        \item a Simon-Smith family of genus $\leq 1$ \[
            H^1: W^1 \to \GS^*(S^3)\,,
        \]
        with $H^1|_{\partial_0W^1}=\Psi$, $H^1|_{\partial_1W^1}=\tilde\Psi^1$,
        \item a constant $\delta_1>0$ (in place of $\eta$),
    \end{itemize} 
    satisfying the properties listed in the theorem. 
    
    Then by refining $\tilde Y^1$, we can obtain a cubical subcomplex $C^1$ of $\tilde Y^1$ such that 
    \[
        \cH^2(\tilde\Psi^1(x))\geq \bL(\Lambda(\Psi))-\delta_1/2 \implies x\in C^1\,,
    \]
    and $\tilde\Psi^1|_{C^1}$ is a Simon-Smith family of genus 0. Finally, we let 
    \[
        \Psi^1:=\tilde\Psi^1|_{\overline{\tilde Y^1\backslash C^1}}\,.
    \]
    As in Case 1, a key property of $\Psi^1$ is that 
    \[
        \sup\cH^2\circ\Psi^1<\bL(\Lambda(\Psi)).
    \]
    Let us denote $Y^1 := \dmn(\Psi^1)$.

    In summary, in both cases, we obtain a Simon-Smith family $\Psi^1$ of genus $\leq 1$ with parameter space $\overline{\tilde Y^1\backslash C^1}$, and with  $\sup\cH^2\circ\Psi^1$  strictly less than the width $\bL(\Lambda(\Psi))$. Thus, if we apply the Simon-Smith min-max construction again to $\Psi^1$, all min-max minimal surfaces will have smaller area than those obtained previously. We will continue repeating the min-max process.

\subsubsection{$k$-th stage}
    In this subsubsection, we explain in detail what is obtained in the $k$-th stage of min-max for $k\geq 1$. Before the $k$-th stage, we have a Simon-Smith family of genus $\leq 1$ 
    \[
        \Psi^{k-1}:Y^{k-1}\to \GS^*(S^3)\,,
    \] 
    where $Y^{k-1}$ is some cubical complex. If $k=1$, we take $\Psi^0:=\Psi$ and $Y^0:=Y$. 
    
    Applying the min-max theorem, Theorem~\ref{thm:minMax}, to $\Psi^{k-1}$ with $r=d_0$, we obtain a varifold $V^k$ in $\cW^k:=\cW_{\bL(\Lambda(\Psi^{k-1})),\leq 1}$, where $V^k$ is induced by a multiplicity one, smooth, embedded, minimal surface of genus $0$ or $1$. Again, there are two cases: 
    \begin{itemize}
        \item[Case 1.] each element of $\cW^{k}$ is associated with a multiplicity one torus;
        \item[Case 2.] $V^{k}$ is associated with a smooth, embedded, minimal sphere of multiplicity one and  $\cW^{k}=\{V^{k}\}$.
    \end{itemize}

    \begin{rmk}
        In the following, the superscript $^k$ indicates that the object concerned is constructed from running the min-max process for the $k$-th stage.
    \end{rmk}

    \medskip
    
    \paragraph*{\bf Case 1.}Arguing as before, we obtain the following objects:
    \begin{itemize}
        \item a Simon-Smith family of genus $\leq 1$ 
        \[
            \tilde\Psi^k: Y^{k - 1} \to \GS^*(S^3)
        \] 
        in $\Lambda(\Psi^{k-1})$, 
        \item a ``cap" $C^k$, which is a subcomplex of $Y^{k-1}$ and can be decomposed into a disjoint union of cubical subcomplexes $C^k_1, \cdots ,C^k_{n_k}$,
        \item the restriction 
            \[
                \Psi^k:=\tilde\Psi^{k}|_{\overline{Y^{k-1}\backslash C^k}}\,,
            \]
        \item a new domain $Y^k:=\dmn(\Psi^k)=\overline{Y^{k-1}\backslash C^k}$, 
    \end{itemize}
    satisfying the following properties:
    \begin{itemize}
        \item For each $j=1, \cdots, n_k$, there is a distinct embedded minimal torus $T^k_j\in \GS^*(S^3)$, with $|T^k_j|\in\cW^k$, such that the image 
        \[
            [\tilde\Psi^{k}](C^{k-1}_j)\subset \bB^\bF_{d_0}([T^k_j])\,.
        \]
        \item $\sup_{Y^k} \cH^2\circ\Psi^k<\bL(\Lambda(\Psi^{k-1}))$.
    \end{itemize}
    Since $\tilde\Psi^k$ and $\Psi^{k-1}$ are homotopic in the sense of Simon-Smith min-max theory, there exists a homotopy 
    \[
        H^k:[0,1]\x Y^{k-1} \to\GS^*(S^3)\,,
    \]
    such that 
    \[
        H(0,\cdot)=\Psi^{k-1}, \quad H^k(1,\cdot)=\tilde \Psi^k\,,
    \]
    and for each $(t,x)$, $H(t,x)$ is obtained from $\Psi^{k-1}(x)$ via an isotopy of $S^3$, according to Definition~\ref{def:homotopyClass}. 

    For consistency in both cases, we introduce the following notation, which mirrors the setup in Case 2. Define
    \[
        W^k:=[0,1]\x Y^{k-1},\quad\partial_0 W^k=\{0\}\x Y^{k-1},\quad\partial_1 W^k=\{1\}\x Y^{k-1}\,,
    \] 
    and a map 
    \[
        F^k:[0,1]\x W^k\to W^k\,, \quad (s,(t,x)) \mapsto (st,x)\,.
    \]
    Note that the map $(s,w)\mapsto F^k(1-s,w)$, with $(s,w)\in [0,1]\x W^k$, is a strong deformation retraction of $W^k$ onto $\{0\}\x Y^{k-1}$. For convenience, we will identify $Y^{k-1}$ with the subset $\{0\}\x Y^{k-1}$ in $W^k$, and denote $\{1\}\x Y^{k-1}$ by $\tilde Y^k$.
    \begin{rmk}\label{rmk:case1cylinder}
        For example, in Case 1, if we have a subset $C\subset Y^{k-1}$, then $(F^k(0,\cdot))^{-1}(C)$  is the subset $[0,1]\x C$ of $W^k$.
    \end{rmk}

    \medskip

    \paragraph*{\bf Case 2.}
    In this case, by applying Theorem~\ref{thm:mapping_cylinder} to $\Psi^{k-1}$ with $r=d_0$,  we obtain:
    \begin{itemize}
        \item a cubical complex $\tilde Y^k$,
        \item a surjective, cubical, homotopy equivalence $f^k:\tilde Y^k\to Y^{k-1}$,
        \item a mapping cylinder $W^k:=M_{f^k}$ (viewed as a simplicial complex) with
        \[
            Y^{k-1}=\partial_0 W^k\,, \quad\tilde Y^k= \partial_1 W^k\,,
        \]
        and the associated map $F^k:=F_{W^k}:[0,1]\x W^k\to W^k$ (see \eqref{eq:Fw}),
        \item a Simon-Smith family of genus $\leq 1$
        \[
            \tilde\Psi^k:\tilde Y^k\to\cS^*(S^3)\,,
        \]
        \item a Simon-Smith family of genus $\leq 1$ 
        \[
            H^k: W^k \to \GS^*(S^3),
        \] 
        with $H^k|_{\partial_0W^k}=\Psi^{k-1}$, $H^k|_{\partial_1W^k}=\tilde\Psi^k$,
        \item a ``cap" $C^k$, which is a cubical subcomplex of $\tilde Y^k$, 
        \item the restriction 
            \[
                \Psi^k:=\tilde\Psi^{k}|_{\overline{Y^{k-1}\backslash C^k}}\,,
            \]
        \item a new domain $Y^k:=\dmn(\Psi^k)=\overline{Y^{k-1}\backslash C^k}$,
    \end{itemize}
    satisfying the following:
    \begin{itemize}
        \item The map $(t,w)\mapsto F^k(1-t,w)$, where $(t,w)\in [0,1]\x W^k$, is a strong deformation retraction of $W^k$ onto $Y^{k-1}$.
        \item For each $x\in \tilde Y^k$, the family $t\mapsto H^k(F^k(t,x))$, $t\in [0,1]$, is a pinch-off process.
        \item $\tilde\Psi^k|_{C^k}$ is a Simon-Smith family of genus 0.
        \item $\sup\cH^2\circ\Psi^k<\bL(\Lambda(\Psi^{k-1}))$.
    \end{itemize}

    As before, in both cases, the mass bound 
    \begin{equation}\label{eq:massBoundPsiK}
        \sup\cH^2\circ\Psi^k<\bL(\Lambda(\Psi^{k-1}))
    \end{equation}
    holds. Thus, the width $\bL(\Lambda(\Psi^k))$ is strictly decreasing with respect to $k$. Combined with the assumption that $(S^3,g')$ admits only finitely many embedded minimal spheres and tori, our repetitive min-max process must terminate in finitely many steps. Specifically, for some $K \in \mathbb{N}^+$, when we apply the min-max process for the $K$-th time to the family $\Psi^{K-1}$, the width $\bL(\Lambda(\Psi^{K-1}))$ reaches zero. 

\subsubsection{Last stage}
    We have  that the width $\bL(\Lambda(\Psi^{K-1}))$ is zero. Thus, there exists some Simon-Smith family $\Psi^K\in\Lambda(\Psi^{K-1})$ and a homotopy in the sense of Simon-Smith min-max,
    \[
        H^{K}:[0,1]\x Y^{K-1}\to \GS^*(S^3),
    \]
    such that:
    \begin{itemize}
        \item $H^K(0,\cdot)=\Psi^{K-1}$, $H^K(1,\cdot)=\Psi^K$.
        \item For each $(t,x)$,  $H^K(t,x)$ is obtained from $\Psi^{K-1}(x)$ via some diffeomorphism of $S^3$ according to Definition~\ref{def:homotopyClass}.
        \item  
        $\sup_{Y^{K - 1}} \cH^2_{\bar g}\circ \Psi^K<1$, where $\bar g$ denotes the standard round metric on $S^3$.
    \end{itemize}
    Crucially, note that in the last bullet point, the Hausdorff measure is taken with respect to {\em the unit 3-sphere $(S^3,\bar g)$}.

    We denote 
    \[
        W^K:=[0,1]\x Y^{K-1}\,,\quad \partial_0 W^K=\{0\}\x Y^{K-1}\,,\quad \partial_1 W^K=\{1\}\x Y^{K-1}\,,
    \] and define a map 
    \[
        F^K:[0,1]\x W^K\to W^K\,, \quad (s, (t, x)) \mapsto (st,x)\,.
    \]
    Note that the map $(s,w)\mapsto F^K(1-s,w)$, with $(s,w)\in [0,1]\x W^K$, is a strong deformation retraction of $W^K$ onto $\{0\}\x Y^{K-1}$. As before, we identify $Y^{K-1}$ with the subset $\{0\}\x Y^{K-1}$ in $W^K$.

\subsection{A new family $\Xi$}\label{subsect:new_family_Xi}
    We have already run the min-max process $K$ times, generating multiple families  from the original family $\Psi$, each with a different parameter space. In this subsection, we will use these families to reconstruct a new Simon-Smith family $\Xi$, whose parameter space is homotopy equivalent to the original parameter space $Y$ (which is an $\RP^5$-bundle over $\RP^2\x\RP^2$). This new family remains a Simon-Smith family of genus $\leq 1$.

    Consider the following list of $2K-1$ Simon-Smith families of genus $\leq 1$. For visualization purposes, readers may find it helpful to refer to Figure ~\ref{fig:gluingScheme}. When interpreting the notations below, we should keep in mind that we always identify $\partial_0W^k$ with $Y^{k-1}=\dmn(\Psi^{k-1})$ (and also thus their respective subsets), and we identify  $\partial_1W^k$ with $  \dmn(\tilde \Psi^{k})$.
    \begin{description} 
        \item[{\makebox[4em][r]{(1)}}] $\tilde \Psi^1|_{C^1}$.
        \item[{\makebox[4em][r]{\color{red}(2)}}] {\color{red} $H^2|_{(F^2(1,\cdot))^{-1}(A_1)}$}, where $A_1:= C^1\cap Y^1\;( \subset \partial_0 W^2\subset W^2)$. Note that $A_1$ is simply $\partial C^1$ if $\tilde Y^1$ is a topological manifold and $C^1$ is a $\dim(\tilde Y^1)$-chain in it. Also,  $H^2|_{(F^2(1,\cdot))^{-1}(A_1)}$ is the same  as {$H^2|_{[0,1]\x A_1}$} if Case 1 occurred at the second stage of the min-max process; see Remark~\ref{rmk:case1cylinder}.
        \item[{\makebox[4em][r]{(3)}}] $\tilde \Psi^2|_{C^2}$.
        \item[{\makebox[4em][r]{\color{RoyalBlue}(4)}}] {\color{RoyalBlue}$H^3|_{(F^3(1,\cdot))^{-1}(A_2)}$}, where $A_2\subset \partial_0 W^3$ is  defined as the union of the following two subsets of $\partial_0 W^3\;(\subset W^3)$:
        \begin{itemize}
            \item $(f^2)^{-1}(A_1)\cap Y^2 \;(\subset Y^2\cong\partial_0 W^3)$.
            \item $C^2\cap Y^2 \;(\subset\partial_0 W_3)$.
        \end{itemize}
        In other words, 
            \[
                A_2:=\left((f^2 )^{-1}(A_1)\cap Y^2\right)\cup \left(C^2 \cap Y^2\right)\,.
            \] 
        (In Figure  ~\ref{fig:gluingScheme}, 
        we see a  blue bridge, and a boundary part of it (on the right). This boundary part consists of two curves: The ``top curve" corresponds to the set described in the second bullet point, while the ``side curve" (which  coincides with part of the left boundary of the red bridge) corresponds to the set in the first bullet point.)
        Note that $H^3|_{(F^3(1,\cdot) )^{-1}(A_2)}$ coincides with { $H^3|_{[0,1]\x A_2}$} if Case 1 occurred at the third stage of min-max.
        \item[{\makebox[4em][r]{(5)}}] $\tilde \Psi^3|_{C^3}$.
        \item[{\makebox[4em][r]{(6)}}]$H^4|_{(F^4(1,\cdot))^{-1}(A_3)}$, where $A_{3}\subset \partial_0 W^4$ is defined by 
            \[
                A_3:=\left((f^3 )^{-1}(A_2)\cap  Y^3\right)\cup \left( C^3\cap Y^3\right)\,.
            \]
        \item[{\makebox[4em][r]{...}}]
        \item[{\makebox[4em][r]{\color{ForestGreen}{($2K-4$)}}}] {\color{ForestGreen}...}
        \item[{\makebox[4em][r]{($2K-3$)}}] $\tilde \Psi^{K-1}|_{C^{K-1}}$.
        \item[{\makebox[4em][r]{\color{brown}($2K-2$)}}] {\color{brown}$H^K|_{(F^K (1,\cdot))^{-1}(A_{K-1})}$}, where $A_{K-1}\subset \partial_0 W^K$ is defined by 
            \[
                A_{K-1}:=\left((f^{K-1} )^{-1}(A_{K-2})\cap   Y^{K-1}\right)\cup \left( C^{K-1}\cap   Y^{K-1}\right)\,.
            \]
        \item[{\makebox[4em][r]{($2K-1$)}}] $\Psi^K$.
    \end{description}

    \begin{figure}[h]
        \centering
        \makebox[\textwidth][c]{\includegraphics[width=5.5in]{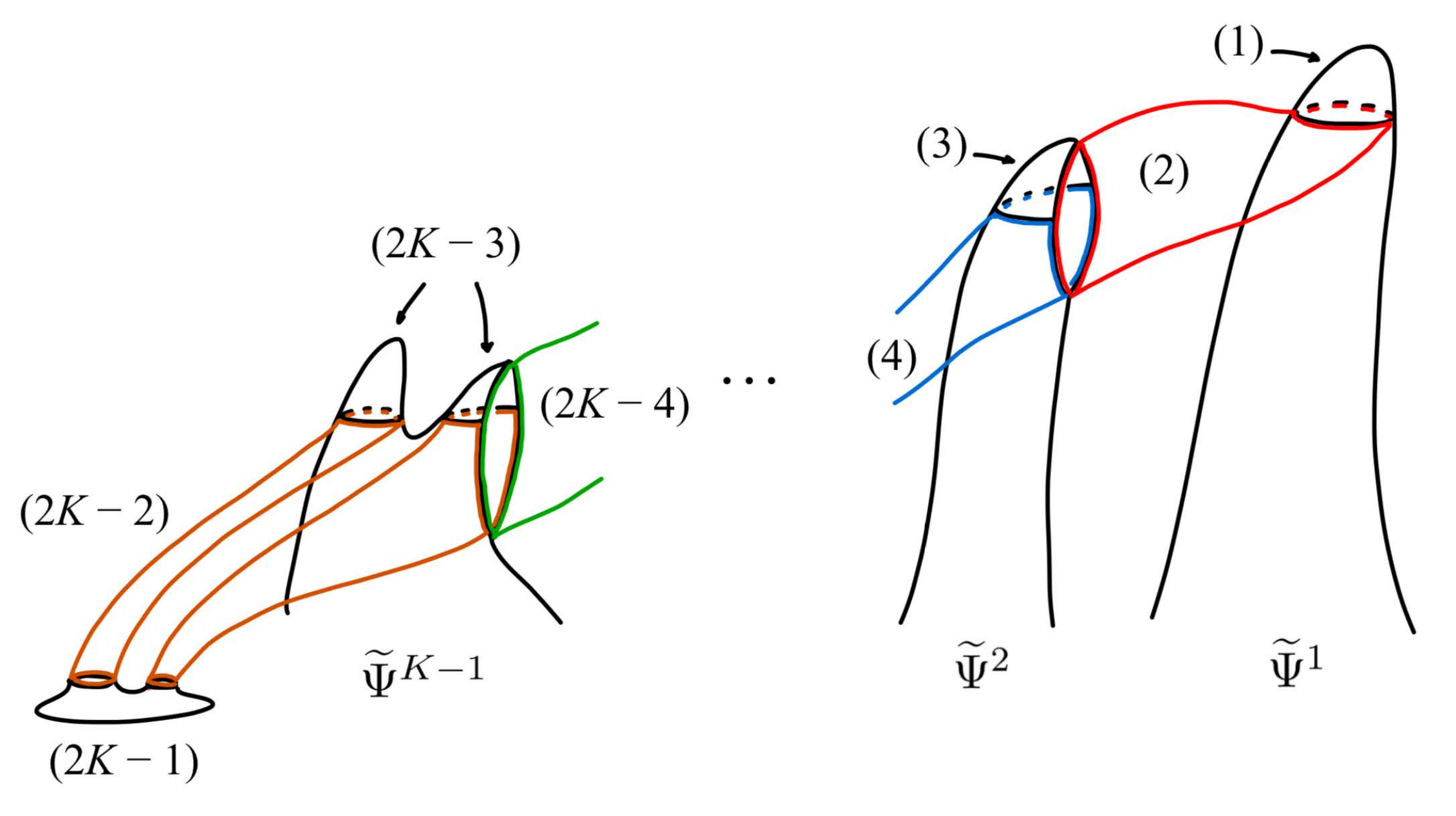}}
        \caption{Constructing $\Xi$.}
        \label{fig:gluingScheme}
    \end{figure}

    By the definition of these $2K-1$ families, there is a natural way to glue them together. Namely, for each $k=1,\cdots,K-2$, the following pairs of families share a common subfamily:
    \begin{itemize}
        \item ($2k$) and ($2k-1$)
        \item ($2k$) and ($2k+1$)
        \item ($2k$) and ($2k+2$)
    \end{itemize}
    Additionally, ($2K-2$) and ($2K-1$) share a common subfamily as well. The overall gluing scheme is described in Figure~\ref{fig:gluingScheme}. 

    Let the newly obtained family be called $\Xi$. It is also a Simon-Smith family of genus $\leq 1$.
    \begin{prop}\label{prop:XiPsiHomotopic}
        There exist:
        \begin{itemize}
            \item a simplicial complex $\tilde W$, containing $Y$ and $\dmn(\Xi)$ as subcomplexes,
            \item a map $\tilde F: [0,1]\x \tilde W\to \tilde W$ such that the map $(t,w)\mapsto\tilde F(1-t,w)$, with $(t,w)\in [0,1]\x\tilde W$, is a strong deformation retraction of $\tilde W$ onto $Y$, and $\tilde F(0,\cdot)|_{\dmn(\Xi)}$ is a surjective homotopy equivalence from $\dmn(\Xi)$ onto $Y$,
            \item a Simon-Smith family of genus $\leq 1$,
                \[
                    \tilde \Xi:\tilde W\to \GS^*(S^3),
                \]
        \end{itemize}
        such that:  
        \begin{enumerate}[label=\normalfont(\arabic*)]
            \item $\tilde\Xi|_{Y}=\Psi,\;\;\tilde\Xi|_{\dmn(\Xi)}=\Xi.$
            \item For each $x\in \dmn(\Xi)$, the family $t\mapsto \tilde\Xi(\tilde F(t,x))$, with $t\in [0,1]$, is a pinch-off process.
        \end{enumerate}
    \end{prop}
    Thus, heuristically, the map $(t,w)\mapsto \tilde\Xi(\tilde F(1-t,w))$, where $(t,w)\in[0,1]\x \dmn(\Xi)$, can be viewed as a \emph{generalized homotopy} that deforms $\Xi$ back to $\Psi$.

    \begin{proof} By examining the construction of $\Xi$, there is a natural, canonical way to define the objects mentioned in the proposition. Specifically, we introduce a sequence of ``$\bF$-homotopies" that, in a certain sense, homotope $\Xi$ back to $\Psi$. Before proceeding, we introduce the following notations:
        \begin{itemize}
            \item The $\bF$-metric in $\cZ_2(S^3;\Z_2)$ does not induce a metric on the space $\GS^*(S^3)$. However, the $\bF$-metric would induce a pseudometric, and therefore a topology, on $\GS^*(S^3)$. With this in mind, the list in the next paragraph would describe a sequence of homotopies, called {\em $\bF$-homotopies} under this topology.
            \item We use ``$+$'' to define the addition of chains with $\Z_2$-coefficients in $\GS^*(S^3)$, where overlapping portions cancel out each other (see, for example,  Figure~\ref{fig:gluingScheme}). 
            \item We use the numbering $(1), (2),\cdots,(2K-1)$ for the subfamilies of $\Xi$ as stated at the beginning of \S \ref{subsect:new_family_Xi}.
        \end{itemize}
        
        Let us now consider the following sequence of $\bF$-homotopies. For visualization purposes, readers may refer to Figure ~\ref{fig:gluingScheme}. 
        \begin{itemize}
            \item The family $\Xi$ can be written as the sum of three subfamilies:
            \[
                \Xi=(2K-1)+(2K-2)+\textrm{ remaining part}.
            \]
            We homotope the subfamily $(2K-1)+(2K-2)$ back to $\Psi^{K-1}$  (using the family $H^K$ and the strong deformation retraction of $W^K$ onto $Y^{K-1}$ induced by $F^K$), while fixing the remaining part. Call this new family $\Xi^{K-1}$.
            \item The family $\Xi^{K-1}$ can be written as the sum of three subfamilies:
            \[
                \Xi^{K-1}=\tilde \Psi^{K-1}+(2K-4) + \textrm{ remaining part.}
            \]
            We homotope the subfamily $\tilde \Psi^{K-1}+(2K-4)$ to $\Psi^{K-2}$ (using the family $H^{K-1}$ and  the strong deformation retraction of $W^{K-1}$ onto $Y^{K-2}$ induced by $F^{K-1}$),  while fixing the remaining part. Call this new family $\Xi^{K-2}$.
            \item ...
            \item The family $\Xi^{2}$ can be written as the sum of three subfamilies:
            \[
                \Xi^2=\tilde \Psi^{2}+(2) + \textrm{ remaining part}.
            \]
            We homotope the subfamily $\tilde \Psi^{2}+(2)$ to $\Psi^{1}$ (using the family $H^{2}$ and  the strong deformation retraction of $W^{2}$ onto $Y^{1}$ induced by $F^{2}$),  while fixing the remaining part. Call this new family $\Xi^{1}$, which is actually $\tilde \Psi^1$.
            \item Finally, we homotope  $\Xi^1=\tilde \Psi^1$ to $\Psi$ using the family $H^{1}$ and  the strong deformation retraction of $W^{1}$ onto $Y^{0}=Y$ induced by $F^1$.
        \end{itemize}
        
        From the above chain of $\bF$-homotopies, one can easily construct $\tilde W$, $\tilde F$, and $\tilde\Xi$ satisfying the desired conditions.
    \end{proof}

\subsection{Decomposing $\Xi$} 
    Recall that for each $k=1, \cdots,K-1$, during the $k$-th min-max process, one of two cases occurred: (1) Some multiplicity one tori $T^k_1, \cdots,T^k_{n_k}$ were detected, and (2) a multiplicity one sphere $V^k$ was detected. For each $k$,  if Case 1 occurred, then $C^k\subset\dmn(\tilde\Psi^k)$ can be decomposed into a union of $C^k_1, \cdots,C^k_{n_k}$, and we will call each $C^k_j$ a {\em genus $1$ cap}. Recall that $[\tilde \Psi^k|_{C^k_j}]$ is $\bF$-close to $[T^k_j]$. If instead Case 2 occurred, we call $C^k$ a {\em genus $0$ cap}. Recall that $\tilde \Psi^k|_{C^k}$ is a Simon-Smith family of genus 0.

    If $C\subset\dmn(\tilde\Psi^k)$ is any such cap obtained at the $k$-th stage of min-max, we define the {\em trace of $C$}, which is a subset of $\dmn(\Xi)$, as follows. We first define the following sets (see Figure~\ref{fig:trace2}):
    \begin{itemize}
        \item {\color{red}$B_k$}$\;:=C$, a subcomplex of $\dmn(\tilde\Psi^k)=\tilde Y^k$.
        \item ${\color{RoyalBlue}B_{k+1}}:=(F^{k+1}(1,\cdot))^{-1}(B_k\cap Y^k)\subset W^{k+1}$. Here $B_k\cap Y^k$ can also be viewed  as a subset of $\dmn(\Psi^k)=\partial_0 W^{k+1}$.
        \item ${\color{ForestGreen}B_{k+2}}\;:=(F^{k+2}(1,\cdot) )^{-1}(B_{k+1}\cap Y^{k+1} )
        \subset W^{k+2}$.
        \item ...
    \end{itemize}
    Note that each  of $B_k, B_{k+1},\cdots,B_K$ is regarded as a subset of $\dmn(\Xi)$. 
    Then the trace of $C$ is defined as their union
    \[
        T(C):=B_k\cup B_{k+1}\cup\cdots\cup B_K\subset\dmn(\Xi)\,.
    \]
    
    \begin{figure}[h]
        \centering
        \makebox[\textwidth][c]{\includegraphics[width=2.7in]{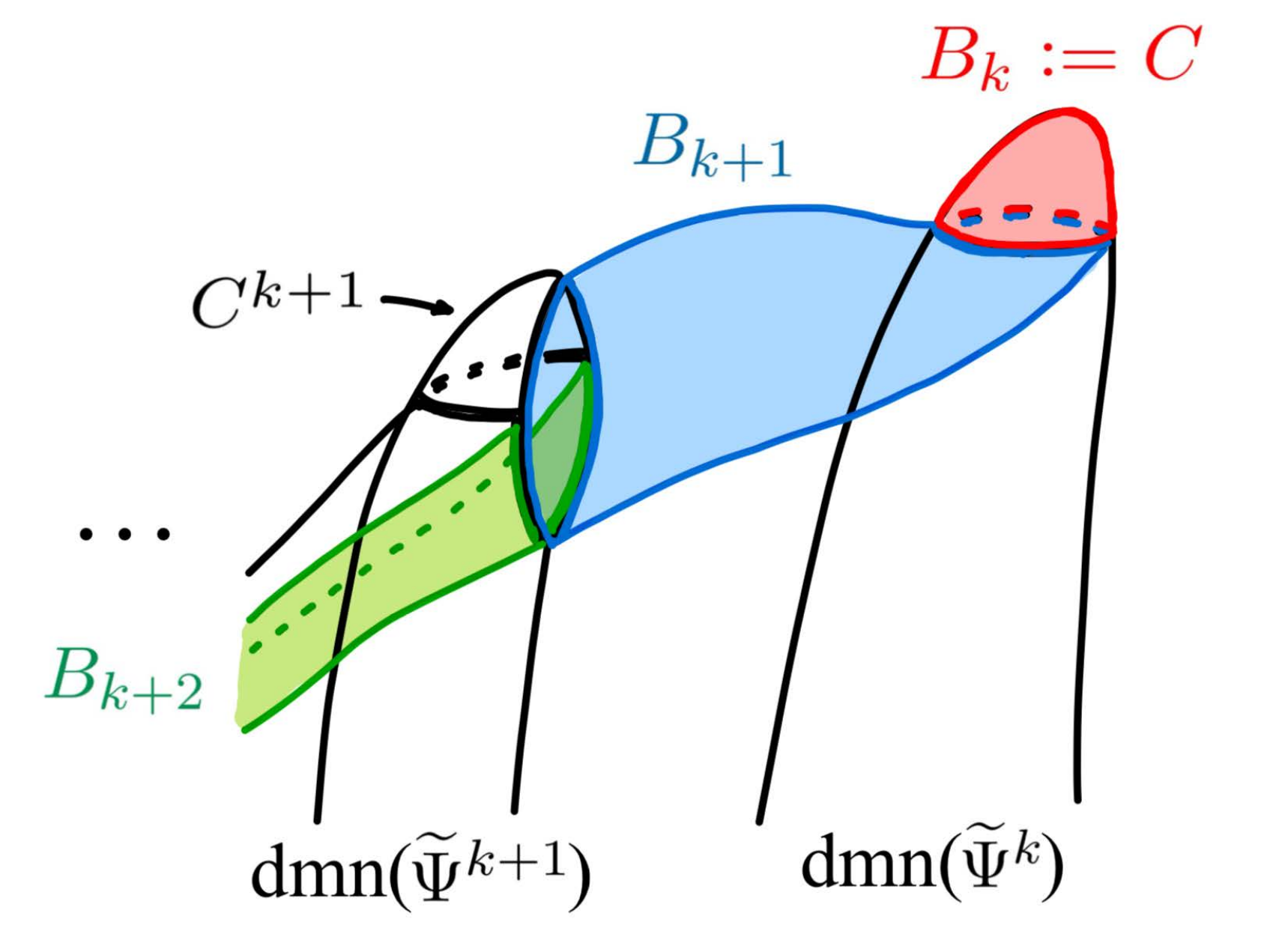}}
        \caption{The trace $T(C)$ in $\dmn(\Xi)$.}
        \label{fig:trace2}
    \end{figure}

    Now, let $N$ be the number of all the genus 1 caps we obtained throughout the first $K-1$ stages of min-max. To prove Theorem~\ref{thm:main}, we just need to show that $N\geq 5$.

    We decompose $\dmn(\Xi)$ as the union of the following $N+1$ subcomplexes: 
    \begin{itemize}
        \item Let $D_0$ be the union of $\dmn(\Psi^K)$ and $\bigcup_C T(C)$, where $C$ ranges over the genus $0$ caps. 
        \item For each genus $1$ cap $C$, we consider its trace $T(C)$. There are in total $N$ such traces, which are denoted by $D_1, \cdots, D_N$.
    \end{itemize}
    It follows directly from the definition of $\Xi$ that
    \[
        \dmn(\Xi)=D_0\cup D_1\cup\cdots\cup D_N\,.
    \]
    
    \begin{prop}\label{prop:XiProperty}
        The decompositions above satisfy the following two properties:
        \begin{enumerate}[label=\normalfont(\arabic*)]
            \item\label{item:XiD0}   $\Xi|_{\bigcup_C T(C)}$,  where $C$ ranges over the genus $0$ caps,  is a Simon-Smith family of genus $0$. 
            \item\label{item:traceDeformRetract} For each genus $0$ or genus $1$ cap $C$, there exists a strong deformation retraction of the trace $T(C)$  back to $C$ in $\dmn(\Xi)$. 
        \end{enumerate}    
    \end{prop}
    \begin{proof}
        For \ref{item:XiD0}, if $C$ is a genus 0 cap obtained at the $k$-stage of min-max, then $\tilde \Psi^k|_{C^k}$ is a Simon-Smith family of genus $0$.
        Then it follows from the definition of the homotopies $H^k$, namely Theorem~\ref{thm:mapping_cylinder}~\ref{item:mapping_cylinder}~\ref{item_item:pinch-off}, that $\Xi|_{T(C)}$ is also a Simon-Smith family of genus $0$. This proves \ref{item:XiD0}. Item \ref{item:traceDeformRetract} follows directly from the definition of $T(C)$, using the maps $F^{k+1},\cdots,F^K$.
    \end{proof}

\subsection{Topological arguments}\label{subsect:LSArgument} 

    In this subsection, we will show that $N\geq 5$. 
    
    First, we define three elements $\lambda,\alpha,\beta$ in $H^1( Y;\Z_2)$:
    \begin{itemize}
        \item Let $\bar\lambda$ be the non-trivial element of $H^1(\cZ_2(S^3;\Z_2);\Z_2)$. Define $\lambda$ as the pullback of $\bar \lambda$ under $[\Psi]$.
        \item Let $a= \RP^1\x\RP^2\subset\RP^2\x \RP^2$, and  $A\subset Y$ be the $\RP^5$-subbundle over $a$. Define $\alpha$ to be the Poincar\'e dual $PD(A)$ of $A$.
        \item Let $b= \RP^2\x\RP^1\subset\RP^2\x \RP^2$, and  $B\subset Y$ be the $\RP^5$-subbundle over $b$. Define $\beta$ to be the Poincar\'e dual $PD(B)$ of $B$.
    \end{itemize}
    Here we assumed $\RP^2$ is built from attaching a $2$-disk to the non-trivial loop $\RP^1\subset\RP^2$.

    The following topological fact about $Y$ is crucial and is the main reason we used the family $\Psi$ in the first place.
    \begin{thm}\label{thm:cupLength9}
        In the cohomology ring $H^*(Y;\Z_2)$,
        \[
            \lambda^5\cup\alpha^2\cup\beta^2\ne 0\,.
        \]
    \end{thm}
    We will prove this theorem in \S\ref{sect:family_Psi}.

    Now, to prove $N\geq 5$, let us suppose by contradiction that $N\leq 4$.  

    By Proposition~\ref{prop:XiPsiHomotopic}, $\dmn(\Xi)$ and $Y$ are homotopy equivalent, so, by abuse of notation, we will also treat view $\lambda,\alpha,\beta$ as elements of $H^1(\dmn(\Xi);\Z_2)$ as well. 
    For each $j=0,1,\cdots,N$, let $i_j$ be the inclusion map $D_j\hookrightarrow \dmn(\Xi)$. Moreover, given a cohomology class $\gamma$ of $\dmn(\Xi)$, we denote its pullback under  
    \[
        (i_j)^*: H^*(\dmn(\Xi);\Z_2)\to H^*(D_j;\Z_2)
    \]
    by $\gamma|_{D_j}$. Then by Theorem~\ref{thm:cupLength9} and a Lyusternik-Schnirelmann argument, we obtain:

    \begin{lem}\label{lem:LS}
        \begin{itemize}
            \item []
            \item If $N=4$, then one of $\lambda^5|_{D_0}, \alpha|_{D_1}, \alpha|_{D_2}, \beta|_{D_3}, \beta|_{D_4}$ is non-zero.
            \item If $N=3$, then one of $\lambda^5|_{D_0}, \alpha|_{D_1}, \alpha|_{D_2}, \beta^2|_{D_3}$ is non-zero.
            \item If $N=2$, then one of $\lambda^5|_{D_0}, \alpha|_{D_1}, (\alpha \cup \beta^2)|_{D_2}$ is non-zero.
            \item If $N=1$, then one of $\lambda^5|_{D_0}, (\alpha^2\cup \beta^2)|_{D_1}$ is non-zero.
            \item If $N=0$, i.e. $\dmn(\Xi)=D_0$, then $\lambda^5|_{D_0}$ is non-zero.
        \end{itemize}
    \end{lem}

    By Proposition~\ref{prop:XiPsiHomotopic}, regarding the induced maps $[\Psi]$ and $[\Xi]$ into $\cZ_2(S^3;\Z_2)$, we have
    \[
        [\Xi]^*(\bar\lambda)=[\Psi]^*(\bar\lambda)=\lambda\,.
    \]
    Hence,
    \[
        \lambda^5|_{D_0}=i_0^*([\Xi]^*(\bar \lambda)^5)=([\Xi\circ i_0]^*(\bar \lambda))^5=([\Xi|_{D_0}]^*(\bar \lambda ))^5\,,
    \]
    which is non-zero if and only if $[\Xi|_{D_0}]$ is a $5$-sweepout.
    Consequently, it suffices to prove the following two facts:
    \begin{prop}\label{prop:not5sweepout}
        $\Xi|_{D_0}$ is not a $5$-sweepout.
    \end{prop}
    \begin{thm}\label{thm:trivialInFirsthomo} Suppose $1\leq N\leq 4$.
        Then for each $j=1,\cdots,N$, the map 
        \[
            (i_j)_*:H_1(D_j;\Z_2)\to H_1(\dmn(\Xi);\Z_2)
        \]
        is the zero map. 
    \end{thm}

    Indeed, Proposition~\ref{prop:not5sweepout} implies $\lambda^5|_{D_0}= 0$, and Theorem~\ref{thm:trivialInFirsthomo}, via the universal coefficients theorem, implies the map 
    \[
        (i_j)^*: H^1(\dmn(\Xi);\Z_2)\to H^1(D_j;\Z_2)
    \] is trivial for each $j=1,\cdots,N$. This implies that all the cohomology classes in Lemma~\ref{lem:LS} must be zero, contradicting the assumption that $N \leq 4$. Thus, it remains to prove Proposition~\ref{prop:not5sweepout} and Theorem~\ref{thm:trivialInFirsthomo}.

    \begin{proof}[Proof of Proposition~\ref{prop:not5sweepout}] 
        In this proof, we will treat $\Xi|_{D_0}$ as a Simon-Smith family of genus $\leq 1$ in the unit 3-sphere $(S^3,\bar g)$ instead of $(S^3,g')$. Recall that $D_0$ is the union of $E:= \bigcup_CT(C)$, where $C$ runs over the genus $0$ caps, and $\dmn(\Psi^K)$.
        
        By Proposition~\ref{prop:XiProperty}, $\Xi|_E$ is a Simon-Smith family of genus 0, while 
        \[
            \sup  \cH^2_{\bar g}\circ\Xi|_{\dmn(\Psi^K)}= \sup \cH^2_{\bar g}\circ\Psi^K<1\;\;(<4\pi)
        \] by definition. Thus, if we apply the {\it relative} Simon-Smith min-max, Theorem~\ref{thm:relative_minMax}, on the unit round 3-sphere to $\Xi|_{D_0}$, relative to $\dmn(\Psi^K)\subset D_0$, we will detect the multiplicity one equatorial 2-sphere. This immediately implies that the width  $\bL(\Lambda(\Xi|_{D_0}))$, in the {\it non-relative} Simon-Smith min-max setting, is at most $4\pi$. 
        
        Note the following fact:
        \begin{claim}
            In a Riemannian $3$-manifold $M$, if $\Phi$ is a Simon-Smith family  such that $[\Phi]$ is a $p$-sweepout, then the Simon-Smith min-max width $\bL(\Lambda(\Phi))$ satisfies
            \[
                \bL(\Lambda(\Phi))\geq \omega_p(M)\,.
            \]
        \end{claim}
        \begin{proof}
            This is true because it follows  from definition that a homotopy in the  sense of Simon-Smith min-max is continuous in the flat topology.
        \end{proof}
        
        Thus, whenever $[\Xi |_{D_0}]$ is a $p$-sweepout,
        \[
            4\pi=\bL(\Lambda(\Xi|_{D_0}))\geq \omega_p(S^3,\bar g)\,.
        \]
        But by C. Nurser~\cite{Nur16}, the first four widths of the unit $3$-sphere are $4\pi$, while the fifth is $2\pi^2$. Thus, $[\Xi|_{D_0}]$ is not a $5$-sweepout. This finishes the proof of Proposition~\ref{prop:not5sweepout}
    \end{proof}

    We postpone the detailed proof of Theorem~\ref{thm:trivialInFirsthomo} to \S \ref{sect:genus1caps}. For now, let us outlined the idea here. Let $1\leq j\leq N$. By definition, $D_j$ is the trace $T(C)$ of some genus 1 cap $C\subset\dmn(\Xi)$ obtained at the $k$-th stage of min-max, for some $1\leq k\leq K-1$. By Proposition~\ref{prop:XiProperty}, there exists a strong deformation retraction of $D_j$ onto $C$, so it suffices to show that for the inclusion map $i:C\hookrightarrow\dmn(\Xi)$, the induced map $i_*$ between the first homology groups of $C$ and $\dmn(\Xi)$ is zero. By definition, the entire image of $[\Xi|_{C}]$ is $d_0$-close to a single minimal torus  in the $\bF$-metric {\em for currents}. We will show in \S \ref{sect:genus1caps} that this means $i_*$ indeed has to be the zero map, thereby proving Theorem~\ref{thm:trivialInFirsthomo}.

    Finally, as previously explained, Proposition~\ref{prop:not5sweepout} and Theorem~\ref{thm:trivialInFirsthomo} together contradict Lemma~\ref{lem:LS}, thereby proving that the assumption $N\leq 4$ is false. Consequently, at least $5$ multiplicity one, embedded, minimal tori are detected through the repetitive min-max process. This completes the proof of Theorem~\ref{thm:main}.

\part{Technical ingredients}

\section{Min-max results I: The Simon-Smith min-max theorem}\label{sect:min-max_i}
    The goal of this section is to prove Theorem~\ref{thm:minMax} and subsequently Theorem~\ref{thm:relative_minMax}.

\subsection{Deformations in annuli}

    One of the innovative ingredients in the Almgren-Pitts min-max theory is the notion of almost-minimizing varifolds in annuli, which was established using combinatorial arguments for deformations in annuli. In this subsection, we adapt these tools to our setting.

    For an open subset $U \subset M$, we use $\Is(U)$ to denote the set of isotopies $\{\varphi(t)\}_{t \in [0, 1]}$ in $U$.

    \begin{defn}[$(\varepsilon, \delta)$-deformation]\label{def:epsilon-delta-def}
        Given $\varepsilon, \delta > 0$, an open set $U \subset M$, and a punctate surface $\Sigma \in \GS(M)$, we call an isotopy $\psi \in \Is(U)$ {\it an $(\varepsilon, \delta)$-deformation of $\Sigma$ in $U$} provided that: 
        \begin{enumerate}[label=\normalfont(\arabic*)]
            \item $\cH^2(\psi(t, \Sigma)) \leq \cH^2(\Sigma) + \delta$ for all $t \in [0, 1]$.
            \item $\cH^2(\psi(1, \Sigma)) \leq \cH^2(\Sigma) - \varepsilon$.
        \end{enumerate}
    
        We define $\mathfrak{a}(U; \varepsilon, \delta)$ to be the set of all punctate surfaces that do not admit $(\varepsilon, \delta)$-deformations in $U$.
    \end{defn}

    \begin{defn}[Admissible annuli]
        Given $K \in \mathbb{N}$ and $p \in M$, a collection of annuli centered in $p$
        \[
            A(p; s_1, r_1), \cdots, A(p; s_K, r_K),
        \]
        is called {\it $K$-admissible} if $2r_{i + 1} < s_i$ for all $i = 1, \cdots, K - 1$.

        For $R \in (0, \infty]$, if $\sup_i r_i \leq R$, we will say that these annuli are {\em of outer radius at most $R$}.
    \end{defn}

    The following proposition is an adaptation of Pitts' combinatorial result~\cite[Theorem~4.10]{Pit81}.

    \begin{prop}\label{prop:an_deform}
        Given $X$ a finite cubical subcomplex of some cube $I(m, k)$, $D \subset X$ a compact subset (not necessarily a subcomplex) of $X$, $\Phi: X \to \GS^*(M)$ a Simon-Smith family, $R \in (0, \infty]$ and a pair positive numbers $(\varepsilon, \delta)$, we set $K = K(m) := 3^{m3^m}$. Suppose that for every $x \in D$, there exists a $K$-admissible collection of annuli of outer radius at most $R$, denoted by $\{A_{x, i}\}^K_{i = 1}$, so that
        \[
            \Phi(x) \notin \bigcup^K_{i = 1}\mathfrak{a}(A_{x, i}; 2\varepsilon, \delta / 2)\,.
        \]
        Then there exists $\Phi^* \in \Lambda(\Phi)$ in the homotopy class of $\Phi$ such that:
        \begin{enumerate}[label=\normalfont(\arabic*)]
            \item For every $x \in D$, 
                \[
                    \cH^2(\Phi^*(x)) < \cH^2(\Phi(x)) - \varepsilon + (3^m - 1)\delta\,.
                \]
            \item For every $x \in X \setminus D$, 
                \[
                    \cH^2(\Phi^*(x)) < \cH^2(\Phi(x)) + 3^m\delta\,.
                \]
            \item For every $x \in X$, there exist $t$ many points $p_1, \cdots, p_t \in M$ such that $t \leq K(m)$ and
            \[
                \Phi^*(x) \setminus(\overline{B}_{2R}(p_1) \cup \cdots \cup \overline{B}_{2R}(p_t)) = \Phi(x) \setminus(\overline{B}_{2R}(p_1) \cup \cdots \cup \overline{B}_{2R}(p_t))\,.
            \]
            \item Moreover, for another compact set $D' \subset X$ with $D' \cap D = \emptyset$, we can choose $\Phi^*$ such that $\Phi^*(x) = \Phi(x)$ holds for all $x \in D'$.
        \end{enumerate}
    \end{prop}
    \begin{proof}
        First, by our assumption, for each $x \in D$ and $i \in \{1, 2, \cdots, K\}$, there exists an isotopy $\psi_{x, i} \in \Is(A_{x, i})$ such that:
        \begin{itemize}
            \item $\cH^2(\psi_{x, i}(t, \Phi(x))) \leq \cH^2(\Phi(x)) + \delta / 2$ for all $t \in [0, 1]$.
            \item $\cH^2(\psi_{x, i}(1, \Phi(x))) \leq \cH^2(\Phi(x)) - 2\varepsilon$.
        \end{itemize}
        By Proposition~\ref{prop:SS_AP}, $[\Phi]$ is continuous in the $\bF$-metric, and thus, for each $x \in D$, there exists an open neighborhood $O_x \subset X$ of $x$ such that for every $y \in O_x$ and every $i \in \{1, 2, \cdots, K\}$, we have:
        \begin{itemize}
            \item $\cH^2(\psi_{x, i}(t, \Phi(y))) \leq \cH^2(\Phi(y)) + \delta$ for all $t \in [0, 1]$.
            \item $\cH^2(\psi_{x, i}(1, \Phi(y))) \leq \cH^2(\Phi(y)) - \varepsilon$.
        \end{itemize}

        Next, since the set $D$ is compact, we can refine $X$ so that it is a cubical subcomplex of $I(m, k')$ and every cell $\sigma$ of $X$ with $\sigma \cap D \neq \emptyset$ is contained within some $O_{x_\sigma}$ where $x_\sigma \in D$. We denote by $X_b$ the smallest cubical subcomplex of $X$ containing
        \[
            \{\sigma \in X : \sigma \cap D \neq \emptyset\}\,.
        \]
        Moreover, to ensure that conclusion (4) holds, we also require that $X'_b$ the smallest cubical subcomplex of $X$ containing
        \[
            \{\sigma \in X : \sigma \cap D' \neq \emptyset\}\,,
        \]
        satisfies the condition that $X_b \cap X'_b = \emptyset$.
        
        It follows from Proposition 4.9 of~\cite{Pit81} that for every cell $\sigma$ of $X_b$ we can choose an annulus $A_\sigma \in \{A_{x_\sigma, 1}, \cdots, A_{x_\sigma, K}\}$ and the corresponding isotopy $\psi_\sigma \in \Is(A_\sigma)$ with the following property. For every pair of distinct cells $\sigma, \tau$ of $X_b$, if they are faces of a common cell $\gamma$ of $X_b$, then 
        \[
            \dist(A_\sigma, A_\tau) > 0\,.
        \]

        Additionally, for every cell $\sigma$ of $X_b$, we can choose a cut-off function 
        \[
            \eta_\sigma: X \to [0, 1]\,,
        \] such that:
        \begin{itemize}
            \item $\eta_\sigma(x) \equiv 0$ or every $x \in X$ with $\|x - c_\sigma\|_{\ell_\infty} \geq 3^{-k'} / 2$,
            \item $\eta_\sigma(x) \equiv 1$ for every $x \in X$ with $\|x - c_\sigma\|_{\ell_\infty} \leq 3^{-k'} / 3$,
        \end{itemize}
        where $c_\sigma$ is the center of $\sigma$.
        Consequently, we obtain a continuous map
        \[
            \hat\psi_\sigma: [0, 1] \times X \to \operatorname{Diff}^\infty(M), \quad (t, x) \mapsto \psi_\sigma(t \eta_\sigma(x))\,.
        \]
        Note that $\hat\psi_\sigma(t, x) = \operatorname{Id}$ outside $A_\sigma$.

        Finally, we define a continuous map
        \[
            \varphi: [0, 1] \times X \to \operatorname{Diff}^\infty(M)
        \]
        by composing all the $\hat \psi_\sigma$
        \[
            \varphi(t, x) := \circ_{\sigma \text{ is a cell of } X_b} \left(\hat\psi_\sigma(t, x)\right)\,.
        \]
        Note that no two annuli $A_\sigma$ and $A_\tau$ intersect unless either $\hat \psi_\sigma(t, x)$ or $\hat \psi_\tau(t, x)$ is an identity map. Consequently, the order of $\sigma$'s does not affect the definition of $\varphi$. And we can define 
        \[
            \Phi^* := \varphi(1, \cdot)(\Phi_i(\cdot)) \in \Lambda(\Phi)\,.
        \]

        Consequently, by the definition of $\eta_\sigma$, for every $x \in X$, we have:
        \begin{itemize}
            \item There are at most $3^m$ cells $\sigma$ of $X_b$ such that $\eta_\sigma(x) > 0$.
            \item If two cells $\sigma, \tau$ of $X_b$ has $\eta_\sigma(x), \eta_\tau(x) > 0$, then 
            \[
                \operatorname{dist}(A_\sigma, A_\tau) > 0\,.
            \]
            \item If $x \in X_b$, then there exists a cell $\sigma_0$ of $X_b$ such that $\eta_{\sigma_0}(x) = 1$.
        \end{itemize}
        Therefore, if $x \in X \setminus D$, we have
        \begin{align*}
            \cH^2(\Phi^*(x)) &= \sum_{\sigma:\eta_\sigma(x) > 0} \cH^2(\Phi^*(x) \cap A_\sigma) + \cH^2\left(\Phi(x) \setminus \bigcup_{\sigma: \eta_\sigma(x) > 0} A_\sigma\right)\\
                &\leq \sum_{\sigma: \eta_\sigma(x) > 0} (\cH^2(\Phi(x) \cap A_\sigma) + \delta) + \cH^2\left(\Phi(x) \setminus \bigcup_{\sigma: \eta_\sigma(x) > 0} A_\sigma\right)\\
                &\leq \cH^2(\Phi(x)) + 3^m \delta\,.
        \end{align*}
        On the other hand, if $x \in D \subset X_b$, let $\sigma_0$ of $X_b$ be a cell such that $\eta_{\sigma_0}(x) = 1$, and then we have
        \begin{align*}
            \cH^2(\Phi^*(x)) &= \sum_{\sigma:\eta_\sigma(x) > 0} \cH^2(\Phi^*(x) \cap A_\sigma) + \cH^2\left(\Phi(x) \setminus \bigcup_{\sigma:\eta_\sigma(x) > 0} A_\sigma\right)\\
                &\leq \sum_{\sigma : \sigma \neq \sigma_0,\ \eta_\sigma(x) > 0} (\cH^2(\Phi(x) \cap A_\sigma) + \delta) + (\cH^2(\Phi(x) \cap A_{\sigma_0}) - \varepsilon) \\
                & \quad + \cH^2\left(\Phi(x) \setminus \bigcup_{\sigma:\eta_\sigma(x) > 0} A_\sigma\right)\\
                &\leq \cH^2(\Phi(x)) + (3^m - 1) \delta - \varepsilon\,.        \end{align*}

        In addition, for $x \in D' \subset X'_b$, we have that $\eta_{\sigma}(x) = 0$ holds for any cell $\sigma$ of $X_b$, and thus,
        \[
            \Phi^*(x) = \Phi(x)\,,
        \]
        which confirms the statement (4).

        The statement (3) follows from the construction of the map $\varphi$.
    \end{proof}

\subsection{Existence and regularity of a smooth minimal surface}

    From the definition of width, it is clear that one can choose a minimizing sequence $\{\Phi_i\}$ in $\Lambda(\Phi)$.

    Recall that any minimizing sequence can be pulled-tight via a family of diffeomorphisms.

    \begin{prop}[Existence of pulled-tight sequence, {\cite[Proposition~3.1]{CD03}\cite[Theorem~4.3]{Pit81}}]\label{prop:Pulltight}
        Given any minimizing sequence $\{\Phi_i\}$ in a homotopy class $\Lambda$, there exists a pulled-tight minimizing sequence $\{\Phi^*_i\}$ in $\Lambda$ such that 
        \[
            \bC(\{\Phi^*_i\}) \subset \bC(\{\Phi_i\})\,.
        \]
    \end{prop}
    \begin{proof}
        Following Pitts~\cite[p. 153]{Pit81} (See also~\cite{CD03, MN14}), there exists a continuous map
        \[
            f: [0, 1] \x \{V \in \cV_2(M): \|V\|(M) \leq 2L\} \to \operatorname{Diff}^\infty(M)
        \]
        such that 
        \begin{enumerate}
            \item $f(0, \cdot) = \id$;
            \item $f(t, V) = \id$ for all $t \in [0, 1]$ if $V$ is stationary;
            \item $\|f(1, V)_\# V\|(M) < \|V\|(M)$ if $V$ is not stationary.
        \end{enumerate}

        By Proposition~\ref{prop:SS_AP} (\ref{item:prop:SS_AP_2}), each $[\Phi_i]: X \to \cZ_2(M; \bF; \Z_2)$ is continuous, and thus, $|\Phi_i|: X \to \cV_2(M)$ is continuous. Hence, for each $i$, we have a continuous map $\varphi_i: [0, 1] \x X \to \operatorname{Diff}^\infty(M)$ defined by 
        \[
            \varphi_i(t, x) := f(t, |\Phi_i(x)|)\,.
        \]
        Let us define $\Phi^*_i$ by
        \[
            \Phi^*_i(x) := \varphi_i(1, \|\Phi_i\|(x))(\Phi_i(x))\,.
        \]

        It follows immediately from the construction that each $\Phi^*_i$ is homotopic to $\Phi_i$, $\bC(\{\Phi^*_i\}) \subset \bC(\{\Phi_i\})$ and every varifold in $\bC(\{\Phi^*_i\})$ is stationary.
    \end{proof}

    For simplicity, let us assume that the minimizing sequence $\{\Phi_i\}$ is already pulled-tight. Now we aim to show that 
    \[
        \bC(\{\Phi_i\}) \cap \cW_L \neq \emptyset.
    \]
    Indeed, by Proposition~\ref{prop:reg_am_aa}, it suffices to show that at least one varifold in $\bC(\{\Phi_i\})$ is {\em almost minimizing in admissible annuli} as defined in Definition~\ref{def:am_aa}.
    
    \begin{defn}[Almost minimizing]
        For an open subset $U \subset M$ and a sequence of punctate surfaces $\{\Sigma_j\} \subset \GS^*(M)$, a varifold $V \in \mathcal{V}_2(M)$ is called \textit{almost minimizing with respect to $\{\Sigma_j\}$} if there exist two sequences of positive real numbers $\varepsilon_j \to 0$, $\delta_j \to 0$ such that:
        \begin{enumerate}[label=\normalfont(\arabic*)]
            \item $\bF(|\Sigma_j|, V) < \varepsilon_j$.
            \item $\Sigma_j \in \mathfrak{a}(U; \varepsilon_i, \delta_i)$.
        \end{enumerate}
    \end{defn}
    
    \begin{defn}[Almost minimizing in admissible annuli]\label{def:am_aa}
        Given $R \in (0, \infty]$, a sequence of punctate surfaces $\{\Sigma_j\} \subset \GS^*(M)$ and a varifold $V \in \mathcal{V}_2(M)$, we say that a varifold $V$ is \textit{almost minimizing in every $K$-admissible collection of annuli of outer radius at most $R$ with respect to $\{\Sigma_j\}$}, if there exists $\varepsilon_j \to 0$, $\delta_j \to 0$ such that:
        \begin{enumerate}[label=\normalfont(\arabic*)]
            \item $\bF(|\Sigma_j|, V) < \varepsilon_j$.
            \item for any $K$-admissible collection of annuli $\{A_i\}^K_{i = 1}$ each of outer radius at most $R$, and any $\Sigma_j$, 
            \[
                \Sigma_j \in \bigcup^K_{i = 1}\mathfrak{a}(A_i; \varepsilon_j, \delta_j)\,.
            \]
        \end{enumerate}
    \end{defn}
    \begin{rmk}\label{rmk:stable_in_annuli}
        Suppose that a varifold $V$ is almost minimizing in every $K$-admissible collection of annuli of outer radius at most $R$. Then from the definition, it is clear that for any $K$-admissible collection of annuli $\{A_i\}^K_{i = 1}$ each of radius at most $R$, $V$ is stable in at least one $A_i$. 
        
        In particular, if $\Sigma = \spt \|V\|$ is a smooth surface, we can say that $\Sigma$ is {\em stable in every $K$-admissible collection of annuli of outer radius at most $R$}.
    \end{rmk}

    \begin{prop}[Regularity of varifolds almost minimizing in admissible annuli]\label{prop:reg_am_aa}
        Given a closed Riemannian manifold $(M, g)$, $K \in \N^+$, $R \in (0, +\infty]$, a sequence $\{\Sigma_j\} \subset \GS(M)$, and a stationary varifold $0 \neq V \in \cV_2(M)$, suppose that:
        \begin{enumerate}[label=\normalfont(\roman*)]
            \item We can choose for each $j \in \N^+$ a finite set $P_j$ such that $\Sigma_j \setminus P_j$ is a smooth surface and 
            \[
                P = \lim_{j \to \infty} P_j
            \]
            in the Hausdorff sense where $P$ is a finite set in $M$.
            \item $V$ is almost minimizing in every $K$-admissible collection of annuli of radius at most $R$ with respect to $\{\Sigma_j\}$.
        \end{enumerate}
        Then $V \in \cW_L$, where $L = \|V\|(M)$.
    \end{prop}
    \begin{proof}
        Let $\{\varepsilon_j\}$ and $\{\delta_j\}$ be the sequences of positive real numbers for $V$ and $\{\Sigma_j\}$ as in Definition~\ref{def:am_aa}.

        For every fixed $p \in M$, we can choose $R(p) \in (0, R)$ such that 
        \[
            B(p, R(p)) \cap (P \setminus \{p\}) = \emptyset\,.
        \] 
        Then by the definition of almost minimizing varifold in admissible annuli, for every $p \in M$ and every subsequence $\{\Sigma_{j_k}\}^\infty_{k = 1}$, there exists a radius $r(p) \in (0, R(p))$ such that for every pair $(s, r) \in \R^2$ with $0 < s < r < r(p)$ and a further subsequence $\{\Sigma_{j_{k_l}}\}^\infty_{l = 1}$, such that: 
        \begin{enumerate}
            \item $V$ is almost minimizing in $A(p; s, r)$ with respect to $\{\Sigma_{j_{k_l}}\}$.
            \item For every $l \in \N^+$, $\Sigma_{j_{k_l}} \cap A(p; s, r)$ is a smooth surface.
        \end{enumerate}

        Then, following the proof of~\cite[Theorem~7.1]{CD03}, one can show that $V$ is associated with a disjoint union of minimal surfaces with multiplicities, i.e., $V \in \cW_L$. Note that $1/j$ should be replaced by $\varepsilon_j$ and $1/(8j)$ by $\delta_j$ therein. It is straightforward to verify that all the replacement arguments hold immediately.
    \end{proof}

    The following proposition generalizes~\cite[Proposition 5.3]{CD03} to a multi-parameter Simon-Smith family, using Pitts's combinatorial argument~\cite[Proposition~4.9]{Pit81}. It is worth noting that De Lellis-Ramic has proved a similar multi-parameter generalization in~\cite{DR18}. However, their homotopy class includes families not derived from isotopies, leading to a slight different definition of almost-minimizing compared to~\cite{CD03} and ours. We provide a proof here for the sake of completeness.
    
    \begin{prop}[Existence of varifolds almost minimizing in admissible annuli]\label{prop:exist_am_aa}
        Suppose that $X$ is a finite cubical subcomplex of some cube $I(m, k)$, $\Phi: X \to \GS^*(M)$ is a Simon-Smith family and $\{\Phi_i\}$ is a pulled-tight sequence in $\Lambda(\Phi)$. Then for $K = K(m) := 3^{m3^m}$ and any $R \in (0, \infty]$, there exists an integer and a min-max subsequence $\Sigma_j := \Phi_{i_j}(x_j)$ such that $\Sigma_j$ converges to some $W \in \bC(\{\Phi_i\})$ in the varifold sense. Moreover, $W$ is almost minimizing in every $K$-admissible collection of annuli of outer radius at most $R$ with respect to $\{\Sigma_j\}$.
    \end{prop}
    \begin{proof}[Proof]
        We set
        \[
            \varepsilon_i = \frac{1}{i}\,, \quad \delta_i = \frac{\varepsilon_i}{2 K(m)}\,.
        \]
        Then up to a subsequence, we may assume that
        \[
            \sup_{x \in X} \mathcal{H}^2(\Phi_i(x)) < \bL(\Lambda(\Phi)) + \delta_i\,.
        \]

        Suppose for the sake of contradiction that for some large $i$ and for every $\Phi_{i}(x)$ with $\cH^2(\Phi_i(x)) \geq \bL(\Lambda(\Phi)) - 2\varepsilon_i$ (and thus, $\Phi_i(x) \in \GS(M)$), there exists a $K(m)$-admissible collection of annuli $\{A_{x, j}\}^K_{j = 1}$ of outer radius at most $R$ such that
        \[
            \Phi_i(x) \notin \bigcup^K_{i = 1} \mathfrak{a}(A_{x, j}; 2\varepsilon_i, \delta_i / 2)\,.
        \]

        It follows from Proposition~\ref{prop:an_deform} that there exists $\Phi^*_i \in \Lambda(\Phi)$ such that if $x \in X \setminus D$, 
        \begin{align*}
            \cH^2(\Phi^*_i(x)) &\leq \cH^2(\Phi_i(x)) + 3^m \delta_i\\
                &\leq (\bL(\Lambda(\Phi)) - 2\varepsilon_i) + \varepsilon_i / 2\\
                &< \bL(\Lambda(\Phi)) - \delta_i\,.
        \end{align*}
        On the other hand, if $x \in D$, 
        \begin{align*}
            \cH^2(\Phi^*_i(x)) &\leq \cH^2(\Phi_i(x)) + (3^m - 1) \delta_i - \varepsilon_i\\
                &< \bL(\Lambda(\Phi)) - \delta_i\,.
        \end{align*}
        
        In conclusion, we have $\sup_{x \in X} \cH^2(\Phi^*_i(x)) < \bL(\Lambda(\Phi)) - \delta_i$, a contradiction. This completes the proof.
    \end{proof}

    \begin{proof}[Proof of Theorem~\ref{thm:minMax} \ref{item:minMaxPulltight} without genus bound]
        By Proposition~\ref{prop:Pulltight}, there exists a pulled-tight sequence $\{\Phi_i\}$ in $\Lambda(\Phi)$.

        Then it follows from Proposition~\ref{prop:exist_am_aa}, there exists an integer $K \in \N^+$ and a varifold $W \in \bC(\{\Phi_i\})$ which is almost minimizing in every $K$-admissible collection of annuli with respect to some $\{\Sigma_j = \Phi_{i_j}(x_j)\}$.

        Finally, as discussed in the second paragraph of Remark~\ref{rmk:homotopy}, we can choose $P_j$ for each $\Sigma_j$ such that 
        \[
            \sup_j \# P_j < \infty\,,
        \] and thus, up to a subsequence, $P_j$ converges to a finite set $P$ in the Hausdorff sense. By Proposition~\ref{prop:reg_am_aa}, we can conlude that
        \[
            W \in \bC(\{\Phi_i\}) \cap \cW_L\,.
        \]
    \end{proof}
    
\subsection{Genus bound}
    In the previous subsection, we have proved that the existence of a varifold $W \in \bC(\{\Phi_i\}) \cap \cW_L$ almost minimizing in every $K$-admissible collection of annuli of outer radius at most $R$ with respect to $\{\Sigma_j\}$. In this and the next subsection, we take $R = +\infty$.
    
    To justify Theorem~\ref{thm:minMax}~\ref{item:minMaxPulltight}, it suffices to show that this $W$ satisfies the genus estimates as follows.

    \begin{prop}[Genus bound of varifolds almost minimizing in admissible annuli]\label{prop:am_genus_bound}
        Given a closed Riemannian manifold $(M, g)$, $K \in \N^+$, $\mathfrak{g}_0 \in \N$, $R \in (0, \infty]$, a sequence $\{\Sigma_j\} \subset \GS(M)$ and a stationary varifold $V \in \cV_2(M)$, suppose that:
        \begin{enumerate}[label=\normalfont(\roman*)]
            \item We can choose for each $j \in \N^+$ a finite set $P_j$ such that $\Sigma_j \setminus P_j$ is a smooth surface and 
            \[
                P = \lim_{j \to \infty} P_j
            \]
            in the Hausdorff sense where $P$ is a finite set in $M$;
            \item $\sup_j \mathfrak{g}(\Sigma_j) \leq \mathfrak{g}_0$;
            \item $V$ is almost minimizing in every $K$-admissible collection of annuli of outer radius at most $R$ with respect to $\{\Sigma_j\}$.
        \end{enumerate}
        Then $V \in \cW_{L, \leq \mathfrak{g}_0}$ for $L = \|V\|(M)$.
    \end{prop}

    The genus bound result essentially relies on the improved lifting lemma with multiplicity proved by Ketover~\cite[Proposition~2.2]{Ket19}, which we have tailored to our setting in the following proposition. More precisely, we do not assume that the limiting sequence consists of smooth surfaces.

    \begin{prop}[Improved lifting lemma with multiplicity]\label{prop:improv_lifting}
        Given a closed Riemannian manifold $(M, g)$, $K \in \N^+$, a sequence $\{\Sigma_j\} \subset \GS(M)$ and a stationary varifold $V \in \cV_2(M)$, suppose that:
        \begin{enumerate}[label=\normalfont(\roman*)]
            \item We can choose for each $j \in \N^+$ a finite set $P_j$ such that $\Sigma_j \setminus P_j$ is a smooth surface and 
            \[
                P = \lim_{j \to \infty} P_j
            \]
            in the Hausdorff sense where $P$ is a finite set in $M$;
            \item $V$ is almost minimizing in every $K$-admissible collection of annuli with respect to $\{\Sigma_j\}$.
        \end{enumerate}
        Thus, by Proposition~\ref{prop:reg_am_aa}, there exists a disjoint set of smooth, connected, embedded minimal surface $\{\Gamma_i\}^l_{i = 1}$ and a set of positive integers $\{m_i\}^l_{i = 1}$ such that
        \[
            V = m_1|\Gamma_1| + \cdots m_l |\Gamma_l|\,.
        \]
        We denote $\Gamma := \bigcup^l_{i = 1} \Gamma_i$.

        Assume that there exists a collection of points $\{q_i \in \Gamma_i\}^l_{i = 1}$ and a collection of simple closed curves $\{\gamma_k\}^t_{k = 1}$ contained in $\Gamma \setminus P$ so that for all $k_1 \neq k_2 \in \{1, 2, \cdots, t\}$, if $\gamma_{k_2}, \gamma_{k_2} \subset \Gamma_i$ then $\gamma_{k_1} \cap \gamma_{k_2} = \{q_i\}$. Then there exists $\varepsilon_0 > 0$, so that for any $\varepsilon < \varepsilon_0$, there exists curves $\{\tilde\gamma_k\}^t_{k = 1}$, a subsequence of $\{\Sigma_j\}$ (still labeled as $\{\Sigma_j\}$), and punctate surfaces $\{\tilde \Sigma_j\}$ each of which is obtained from $\Sigma_j$ by finitely many {\em neck-pinch surgeries} outside $P_j$ with the following properties.
        \begin{enumerate}
            \item Each $\tilde \gamma_k$ is homotopic to $\gamma_k$ in $\Gamma$ and $\tilde \gamma_k \subset T_\varepsilon(\gamma_k)$;
            \item $|\tilde \Sigma_j| \to V$ as varifolds;
            \item For each $k \in \{1, 2, \cdots, t\}$, if $\tilde \gamma_k \subset \Gamma_i$, then exactly one of the following holds:
            \begin{enumerate}[label=\normalfont(\arabic*)]
                \item $\pi^{-1}(\tilde \gamma_k) \cap T_\varepsilon(\Gamma_i) \cap \tilde \Sigma_j$ is a union of $m_i$ closed curves, each of which projects via $p$ onto $\tilde \gamma_i$ with degree one.
                \item $\pi^{-1}(\tilde \gamma_k) \cap T_\varepsilon(\Gamma_i) \cap \tilde \Sigma_j$ is a union of $m_i / 2$ closed curves, each of which projects via the closest-point projection $p$ onto $\tilde \gamma_k$ with degree two. In this case, $\Gamma_i$ is non-orientable and $\pi^{-1}(\tilde \gamma_k)$ is a M\:obius band.
            \end{enumerate}
            Here, $\pi: T_\varepsilon(\Gamma_i) \to \Gamma_i$ is the closest-point projection onto $\Gamma_i$.
        \end{enumerate}
    \end{prop}
    
    Indeed, in~\cite[Section 4]{Ket19}, the proof of lifting lemma, the only point where Ketover essentially relies on the smoothness of the sequence $\{\Sigma_i\}$ is in the application of the following Genus Collapse Lemma by Colding-Minicozzi.
    
    \begin{lem}[Genus Collapse, {\cite[Lemma I.0.14]{CM04}}]\label{lem:genus_collapse}
        Given a $3$-manifold $(M, g)$, a sequence of smooth surfaces $\{\Sigma_j\}$ of genus at most $\mathfrak{g}_0$ in $(M, g)$, there exists finitely many points in the manifold $\{x_i\}^{m}_{i = 1}$ with $m \leq \mathfrak{g}_0$ and a subsequence of the surfaces, still denoted $\{\Sigma_j\}$, such that for all $x \notin \{x_i\}^{m}_{i=1}$, there is a radius $r_x > 0$ such that $\Sigma_j \cap B_{r_x}(x)$ is a union of planar domains, i.e.,
        \[
            \mathfrak{g}(\Sigma_j \cap B_{r_x}(x)) = 0\,.
        \]
    \end{lem}
    \begin{rmk}
        In the original statement in Colding-Minicozzi~\cite{CM04}, it appears they require each $\Sigma_j$ is minimal; However, following their proof, as in Ketover’s cited statement~\cite[Lemma~3.4]{Ket19}, this minimality condition is unnecessary.

        In addition, the requirement by Colding-Minicozzi that the ambient manifold $(M, g)$ be closed is also unnecessary.
    \end{rmk}

    \begin{proof}[Proof of Proposition~\ref{prop:improv_lifting}]
        We choose $\varepsilon_0$ as in~\cite[Set up 4.2]{Ket19} and fix $\varepsilon < \varepsilon_0$.

        We choose $s \in (0, \dist(P, \bigcup_k \gamma_k))$ so that for every $p \in P$, $B_s(p) \cap \Gamma$ is either a disk or an empty set. Then we set
        \[
            \tilde M := M \setminus B_s(P)\,.
        \]

        By possibly choosing a subsequence, without loss of generality, we may assume that every $\Sigma_j \cap \tilde M$ is a smooth surface. 
        
        Note that all the arguments in~\cite[Section 4]{Ket19} are confined to a neighborhood of $\bigcup_k \gamma_k$. With Lemma~\ref{lem:genus_collapse}, we can follow these arguments verbatim to prove our proposition.
    \end{proof}

    \begin{proof}[Proof of Proposition~\ref{prop:am_genus_bound}]
        In the previous subsection, we prove $V \in \cW_L$, i.e.,
        \[
            V = m_1|\Gamma_1| + \cdots + m_l |\Gamma_l|
        \]
        for a disjoint set of smooth, connected, embedded minimal surface $\{\Gamma_i\}^l_{i = 1}$ and a set of positive integers $\{m_i\}^l_{i = 1}$.

        We follow the proof of~\cite[Theorem~1.2]{Ket19} to select a set of curves $\{\alpha_k\}^t_{k = 1} \subset \Gamma$. Since $\# P < \infty$, we may choose $\{\alpha_k\}^t_{k = 1} \subset \Gamma \setminus P$. This allows us to apply our Proposition~\ref{prop:improv_lifting} to obtain a sequence of punctate surfaces $\{\tilde \Sigma_j\}$. However, the only issue to continue the arguments from~\cite{Ket19} is that our $\tilde \Sigma_j$ is not necessarily smooth everywhere.

        To address this issue, for each $j$, we choose $s_j \to 0$ such that we can replace $\tilde{\Sigma}_j \cap B_{s_j}(P_j) = \Sigma_j \cap B_{s_j}(P_j)$ in $\tilde{\Sigma}_j$ with finitely many disks through surgeries and the removal of connected components, thereby obtaining a sequence of smooth surfaces $\{\tilde\Sigma'_j\}$. For sufficiently large $j$, we have: 
        \begin{enumerate}
            \item $B_{s_j}(P_j) \cap \pi^{-1}(\bigcup^t_{k = 1} \alpha_k) = \emptyset$;
            \item $\tilde\Sigma'_j \setminus B_{s_j}(P_j) = \tilde\Sigma_j \setminus B_{s_j}(P_j)$;
            \item $|\tilde\Sigma'_j| \to W$ as varifolds;
            \item $\mathfrak{g}(\tilde\Sigma'_j) \leq \mathfrak{g}(\tilde\Sigma_j) \leq \mathfrak{g}(\Sigma_j) \leq \mathfrak{g}_0$.
        \end{enumerate}

        Consequently, with the sequence of smooth surfaces $\{\tilde\Sigma'_j\}$, we can follow the proof of~\cite[Theorem~1.2]{Ket19} verbatim to obtain \begin{equation}
            \sum^l_{i = 1} m_i \mathfrak{g}(\Gamma_i) \leq \mathfrak{g}_0\,.
        \end{equation}
    \end{proof}

\subsection{Multiplicity one theorem}

    The Multiplicity one result is an adaption of Theorem B in~\cite{WZ23}, which essentially follows from the PMC min-max theorem, Theorem 2.4 therein.

\subsubsection{PMC Simon-Smith min-max theory}

    Let us first recall the concepts of $\cA^h$-functional, $\VC$-space, $C^{1,1}$ boundary and strong $\cA^h$-stationarity introduced in~\cite{WZ23}.

    \begin{defn}[Prescribed mean curvature functional]
        For a function $h \in C^\infty(M)$, we define the {\em prescribed mean curvature functional associated with $h$} to be
        \[
            \cA^h: \cV_2(M) \times \cC(M) \to \R\,, \quad (V, \Omega) \mapsto \|V\|(M) - \int_\Omega h d\cH^3\,.
        \]
    \end{defn}

    \begin{defn}[$\cA^h$-stationary pairs]
        Given an open subset $U \subset M$, a pair $(V, \Omega) \in \cV_2(M) \times \cC(M)$ is {\em $\cA^h$-stationary in $U$}, if for any vector field $X \in \cX(U)$ with associated flow $\varphi^t$,
        \[
            \delta \cA^h(V, \Omega)(X) := \frac{d}{dt}\vert_{t = 0} \cA^h(\varphi^t_\#(V, \Omega)) = 0\,.
        \]

        An $\cA^h$-stationary pair $(V, \Omega)$ is {\em $\cA^h$-stable in $U$} if if for any vector field $X \in \cX(U)$ with associated flow $\varphi^t$,
        \[
            \delta^2 \cA^h(V, \Omega)(X, X) := \frac{d^2}{dt^2}\vert_{t = 0} \cA^h(\varphi^t_\#(V, \Omega)) \geq 0\,.
        \]
    \end{defn}

    \begin{defn}[$\VC$-space]
        The {\em $\VC$-space on $M$}, denoted by $\VC(M)$, is the space of all pairs $(V, \Omega) \in \cV_2(M) \times \cC(M)$ such that their is a sequence $\{\Omega_k\} \subset \cC(M)$ satisfying
        \[
            |\partial^* \Omega_k| \to V\,, \quad \Omega_k \to \Omega\,.
        \]

        Given two pairs $(V, \Omega)$ and $(V', \Omega')$ in $\VC(M)$, the $\mathscr{F}$-distance between them is
        \[
            \mathscr{F}((V, \Omega), (V', \Omega')) := \bF(V, V') + \cF(\Omega, \Omega')\,.
        \]
    \end{defn}

    \begin{defn}[$C^{1,1}$ almost embedding]
        Given an open subset $U \subset M$, a $C^{1,1}$ immersed surface $\phi: \Sigma \to U$ with $\phi(\partial \Sigma) \cap U = \emptyset$ is called a {\em $C^{1,1}$ almost embedded surface} in $U$ provided that at every point $p \in \phi(\Sigma)$ where $\phi$ is not an embedding, there exists a neighborhood $W \subset U$ of $p$, such that:
        \begin{enumerate}
            \item $\Sigma \cap \phi^{-1}(W)$ is a disjoint union of connected components $\bigsqcup^\ell_{i = 1} \Gamma^i$.
            \item $\phi: \Gamma^i \to W$ is a $C^{1,1}$ embedding for each $i$.
            \item For each pair $i \neq j$ , $\phi(\Gamma^j)$ lies on one-side of $\phi(\Gamma_i)$ in $W$.
        \end{enumerate}
        We call the set of all such $p$ as the {\em touching set} of $\Sigma$. For simplicity, we will denote $\phi(\Sigma)$ by $\Sigma$ and $\phi(\Gamma^i)$ by $\Gamma^i$.
    \end{defn}

    \begin{defn}[$C^{1,1}$ boundary]
        Given an open subset $U$ of $M$, a $C^{1,1}$ almost embedded surface $\phi:\Sigma \to U$ and $\Omega \in \cC(U)$, $(\Sigma, \Omega)$ is called a {\em $C^{1,1}$ boundary} in $U$, provided that $\Sigma$ is orientable and
        \[
            \phi_\#([\Sigma]) = \partial^*\Omega\,.
        \]

        For a function $h \in C^\infty(M)$, a $C^{1,1}$ boundary $(\Sigma, \Omega)$ in $U$ is called a {\em $C^{1,1}$ $h$-boundary} if $(|\Sigma|, \Omega)$ is $\cA^h$-stationary in $U$.
    \end{defn}

    \begin{defn}[Strong $\cA^h$-stationarity]
        Given an open subset $U$ of $M$, a $C^{1,1}$ $h$-boundary $(\Sigma, \Omega)$ is said to be {\em strongly $\cA^h$-stationary in $U$} provided the following holds.

        For every $p \in U$ in the touching set of $\Sigma$, there exists a neighborhood $W \subset U$ of $p$, and a decomposition $\Sigma \cap W = \bigcup^\ell_{i = 1} \Gamma^i$ into $\ell := \Theta^2(\Sigma, p) \geq 2$ connected disks with a natural ordering
        \[
            \Gamma^1 \leq \Gamma^2 \leq \cdots \leq \Gamma^\ell\,.
        \]
        Let $W^1$ and $W^\ell$ be the bottom and the top components of $W \setminus \Sigma$. For $i = 1$ or $\ell$ and all $X \in \cX(W)$ pointing into $W^i$ along $\Gamma^i$, if $W^i \subset \Omega$, then
        \[
            \delta \cA^h(\Gamma^i, W^i)(X) \geq 0\,;
        \]
        otherwise, $W^i \cap \Omega = \emptyset$ and
        \[
            \delta \cA^h(\Gamma^i, W \setminus W^i)(X) \geq 0\,.
        \]
    \end{defn}
    
    To adapt their PMC Simon-Smith min-max theory to our setting, note that for every  $\Sigma \in \GS^*(M)$, by definitions, $[\Sigma]$ can be viewed as the reduced boundary of a Caccioppoli set $\Omega \in \cC(M)$ and its complement $M \setminus \Omega \in\cC(M)$. Therefore, we can extend the definition of $\mathscr{E}$ in~\cite[Section~2]{WZ23} to
    \[
        \tilde{\mathscr{E}} := \{(\Sigma, \Omega) : \Sigma \in \GS^*(M),\ [\Sigma] = \partial^* \Omega\}\,.
    \]

    \begin{defn}[PMC Simon-Smith family]
        Let $X$ be a cubical subcomplex of some $I(m, j)$.
        A map $(\Phi, \Omega): X \to \tilde{\mathscr{E}}$ is called a {\em PMC Simon-Smith family}, provided that:
        \begin{enumerate}
            \item Its first component $\Phi: X \to \GS^*(M)$ is a Simon-Smith family.
            \item Its second component $\Omega: X \to \cC(M)$ is continuous.
        \end{enumerate}
    \end{defn}

    \begin{defn}[Relative PMC homotopy class]
        Let $Z \subset X$ be cubical subcomplexes of some $I(m, j)$. Two PMC Simon-Smith families $(\Phi_0, \Omega_0)$ and $(\Phi_1, \Omega_1)$ parametrized by the same parameter space $X$ with $(\Phi_0, \Omega_0)\vert_Z = (\Phi_1, \Omega_1)\vert_Z$ are said to be {\em homotopic relative to $(\Phi_0, \Omega_0)\vert_Z$} to each other if there exists a continuous map
        \[
           \varphi: [0, 1] \times X \to \operatorname{Diff}^\infty(M)
        \] such that 
        \begin{enumerate}[label=\normalfont(\arabic*)]
            \item $\varphi(0, x) = \operatorname{Id}$ for all $x \in X$.
            \item $\varphi(t, z) = \operatorname{Id}$ for all $t \in [0, 1]$ and $z \in Z$.
            \item $(\varphi(1, x)(\Phi_0(x)), \varphi(1, x)(\Omega_0(x))) = (\Phi_1(x), \Omega_1(x))$ for all $x \in X$.
        \end{enumerate} 
        
        The set of all families homotopic relative to $(\Phi, \Omega)\vert_Z$ to a PMC Simon-Smith family $(\Phi, \Omega)$ is called the {\em relative $(X, Z)$-homotopy class of $(\Phi, \Omega)$}, and is denoted by $\Lambda_Z(\Phi, \Omega)$.
    \end{defn}

    \begin{defn}
        Given a $(X, Z)$-relative homotopy class $\Lambda_Z$ of a PMC Simon-Smith family and a function $h \in C^\infty(M)$, its {\em $h$-width} is defined by
        \[
            \mathbf{L}^h(\Lambda_Z):=\inf_{(\Phi, \Omega)\in \Lambda_Z}\sup_{x\in X}\cA^h(|\Phi|(x), \Omega(x))\,.
        \]
    \end{defn}
    \begin{defn}
        A sequence $\{(\Phi_i, \Omega_i)\}$ in $\Lambda_Z$ is said to be {\em minimizing associated with $h \in C^\infty(M)$} if 
        \[
            \lim_{i\to\infty} \sup_{x\in X}\cA^h(|\Phi|(x), \Omega(x)) =\mathbf{L}^h(\Lambda_Z)\,.
        \]
        If $\{(\Phi_i \Omega_i)\}$ is a minimizing sequence associated with $h$ and $\{x_i\} \subset X$ satisfies 
        \[
            \lim_{i\to\infty} \cA^h(|\Phi_i(x_i)|, \Omega_i(x_i))=\mathbf{L}^h(\Lambda_Z)\,,
        \]
        then $\{\Phi_i(x_i)\}$ is called a {\em min-max sequence associated with $h$}. 
        
        For a minimizing sequence $\{(\Phi_i, \Omega_i)\}$, we define its \textit{critical set associated with $h$} to be the set of all subsequential varifold-limit of its min-max sequences:
        \begin{align*}
            \bC^h(\{(\Phi_i, \Omega_i)\}):=\{(V, \Omega) = \lim_j (|\Phi_{i_j}(x_j)|, \Omega_{i_j}(x_j)) \in &\VC(M) : x_j\in X, \, \\ &\cA^h(V, \Omega)=\bL^h(\Lambda_Z)\}\,.
        \end{align*}
    \end{defn}

    \begin{defn}[$(\cA^h, \varepsilon, \delta)$-deformation]
        Given $\varepsilon, \delta > 0$, an open set $U \subset M$, a function $h \in C^\infty(M)$ and a pair $(\Sigma, \Omega) \in \tilde{\mathscr{E}}$, we call an isotopy $\psi \in \Is(U)$ {\it an $(\cA^h, \varepsilon, \delta)$-deformation of $\Sigma$ in $U$ associated with $h$} provided that:
        \begin{enumerate}
            \item $\cA^h(\psi(t)_\# (|\Sigma|, \Omega)) \leq \cA^h(|\Sigma|, \Omega) + \delta$ for all $t \in [0, 1]$.
            \item $\cA^h(\psi(t)_\# (|\Sigma|, \Omega)) \leq \cA^h(|\Sigma|, \Omega) - \varepsilon$.
        \end{enumerate}
    
        We define $\mathfrak{a}^h(U; \varepsilon, \delta)$ to be the set of all pairs in $\tilde{\mathscr{E}}$ that do not admit $(\cA^h, \varepsilon, \delta)$-deformations in $U$.
    \end{defn}
    
    \begin{defn}[Almost minimizing in admissible annuli]
        Given a sequence of pairs $\{(\Sigma_j, \Omega_j)\} \subset \tilde{\mathscr{E}}$ and a pair $(V, \Omega) \in \VC(M)$, we say that $(V, \Omega)$ is \textit{$\cA^h$-almost minimizing in every $K$-admissible collection of annuli of outer radius at most $R$ with respect to $\{(\Sigma_j, \Omega_j)\}$}, if there exists $\varepsilon_j \to 0$, $\delta_j \to 0$ such that: 
        \begin{enumerate}[label=\normalfont(\arabic*)]
            \item $\mathscr{F}((|\Sigma_j|, \Omega_j), (V, \Omega)) < \varepsilon_j$.
            \item for any $K$-admissible collection of annuli $\{A_i\}^K_{i = 1}$, and any pair $(\Sigma_j, \Omega_j)$, 
            \[
                (\Sigma_j, \Omega_j) \in \bigcup^K_{i = 1}\mathfrak{a}^h(A_i; \varepsilon_j, \delta_j)\,.
            \]
        \end{enumerate}
    \end{defn}
    \begin{rmk}\label{rmk:Ah_stable_in_annuli}
        Similar to Remark~\ref{rmk:stable_in_annuli}, if $\Sigma = \spt \|V\|$ is a minimal surface, then for every $K$-admissible collection of annuli, $(|\Sigma|, \Omega)$ is $\cA^h$-stable in at least one annulus.
    \end{rmk}
    
    \begin{thm}[Relative PMC min-max theorem, {\cite[Theorem~2.4]{WZ23}}]\label{thm:rel_PMC_min_max}
        Given $Z \subset X$ cubical subcomplexes of some $I(m, j)$, a function $h \in C^\infty(M)$ and a PMC Simon-Smith family $(\Phi_0, \Omega_0)$, suppose that 
        \[
            \bL^h(\Lambda_Z(\Phi_0, \Omega_0)) > \max \{\sup_{z \in Z} \cA^h(|\Phi_0(z)|, \Omega_0(z)), 0\}\,.
        \]
        Then there exists a minimizing sequence $\{(\Phi_i, \Omega_i)\} \subset \Lambda_Z(\Phi_0, \Omega_0)$, and a strongly $\cA^h$-stationary, $C^{1,1}$ $h$-boundary $(\Sigma, \Omega)$ with $(|\Sigma|, \Omega) \in \bC^h(\{(\Phi_i, \Omega_i)\})$, such that
        \[
            \cA^h(\Sigma, \Omega) = \bL^h(\Lambda_Z(\Phi_0, \Omega_0))\,.
        \]
        Moreover, there exists an integer $K = K(m)$ and a min-max subsequence $(\Sigma_j, \Omega_j) := (\Phi_{i_j}(x_j), \Omega_{i_j}(x_j)$ such that $(|\Sigma|, \Omega)$ is $\cA^h$-almost minimizing in every $K$-admissible collection of annuli with respect to $\{(\Sigma_j, \Omega_j)\}$.
    \end{thm}
    \begin{proof}
        Firstly, as in the proof of our Proposition~\ref{prop:exist_am_aa}, one can adapt the proof of~\cite[Theorem~3.8]{WZ23} to obtain a pair $(V_0, \Omega_0) \in \VC(M)$ which is $\cA^h$-almost minimizing in small annuli with respect to $\{(\Sigma_j, \Omega_j) = (\Phi_{i_j}(x_j), (\Omega_{i_j}(x_j))\}$.

        Then, as discussed in the second paragraph of Remark~\ref{rmk:homotopy}, we can choose $P_j$ for each $\Sigma_j$ such that 
        \[
            \sup_j \# P_j < \infty
        \] and thus, up to a subsequence, $P_j$ converges to a finite set $P$ in the Hausdorff sense. Since $\Sigma_j \setminus P_j$ is a smooth surface, one can follow verbatim the proof of~\cite[Theorem~4.7]{WZ23} and conclude that $(V_0, \Omega_0)$ is $C^{1,1}$ and strongly $\mathcal{A}^h$-stationary boundary in $M \setminus P'$ for some finite set $P' \supset P$.

        Finally, one can follow their $C^{1,1}$-version of removal singularity as in the proof of~\cite[Theorem~2.4]{WZ23} to conclude the theorem.
    \end{proof}
    
\subsubsection{Proof of Theorem~\ref{thm:minMax}~\ref{item:minMaxMultiplicityone}}
    Let $\mathcal{W}^{m}_{L, \leq \mathfrak{g}_0}$ be the set of all varifolds $W \in \mathcal{W}_{L, \leq \mathfrak{g}_0}$, whose support is a smooth embedded minimal surface $\Sigma$, such that for every $K(m) := 3^{m3^m}$-admissible collection of annuli, $\Sigma$ is stable in at least one annulus. Note that this is just a reformulation of {\em Property (R')} in~\cite{WZ23}. Proposition~\ref{prop:reg_am_aa} with $R = +\infty$ implies that 
    \[
        W \in \mathbf{C}(\{\Phi_i\}) \cap \cW^m_{L, \leq \mathfrak{g}_0}\,.
    \]
    
    By Sharp's compactness theorem~\cite{Sha17} (See also Theorem~\cite[Theorem~7.1]{WZ23}), the set $\cW^m_{L, \leq\mathfrak{g}_0}$ is compact. In particular, if $g$ is a bumpy metric on $M$, then $\cW^m_{L, \leq\mathfrak{g}_0}$ is a finite set.

    As in the proof of~\cite[Theorem~7.3]{WZ23}, it suffices to work on bumpy metrics $g$ and the general results follow from a approximation by bumpy metrics. 

    By Proposition~\ref{prop:SS_AP}, for any Simon-Smith family $\Phi: X \to \GS^*(M)$, the map $[\Phi]$ is continuous in the $\bF$-metric, and thus, it can be lifted to a PMC Simon-Smith family
    \[
        (\tilde \Phi, \tilde \Omega): \tilde X \to \tilde{\mathscr{E}}\,,
    \]
    where $\tilde X$ is a double-cover of $X$.

    Hence, one can follow verbatim the proofs of Theorem 7.2 and Theorem 7.3 of~\cite{WZ23} through replacing Theorem 2.4 therein by our Theorem~\ref{thm:rel_PMC_min_max}. Note that by Remark~\ref{rmk:Ah_stable_in_annuli}, Theorem~\ref{thm:rel_PMC_min_max} also implies the required {\em Property (R')} as defined in~\cite[Corollary~3.11]{WZ23}. This completes the proof of Theorem~\ref{thm:minMax}~\ref{item:minMaxMultiplicityone}.

\subsection{Refinement of the critical set by Marques-Neves}\label{subsect:Refinement_MN}
	
    Conclusion~\ref{item:minMaxW} is adapted from~\cite[Theorem~4.7]{MN21}, which essentially follows from the combinatorial argument of~\cite[Theorem~4.10]{Pit81}. Note that our Definition~\ref{def:epsilon-delta-def} of $(\varepsilon, \delta)$ deformation is continuous, so we do not need to appeal to the intricate interpolation lemmas as in~\cite{MN21}.
    
    Let us first recall a useful lemma proved in Marques-Neves~\cite[Lemma~4.5]{MN21}.
    
    \begin{lem}\label{lem:close_replace_balls}
        For every $r, L > 0$ and $m \in \N$, there exists $\tilde \eta = \tilde \eta(r, L, m) > 0$ so that for every $V, T \in \cV_2(M)$ with $\bM(V) \leq 2L, \bM(T) \leq 2L$, $V$ stationary,
        \[
            V \llcorner (M \setminus (\overline{B}_{2\tilde \eta}(p_1) \cup \cdots \cup \overline{B}_{2\tilde \eta}(p_t))) = T\llcorner (M \setminus (\overline{B}_{2 \tilde \eta}(p_1) \cup \cdots \cup \overline{B}_{2\tilde \eta}(p_t)))
        \]
        for some collection $\{p_1, \cdots, p_t\} \subset M$, $t \leq K(m) := 3^{m3^m}$, and
        \[
            \|V\|(M) - \tilde \eta \leq \|T\|(M) \leq \|V\|(M) + \tilde \eta\,,
        \]
        then $\mathbf{F}(V, T) < r/4$.
    \end{lem}
    
    With this lemma, we can fix $\tilde \eta$ from Lemma~\ref{lem:close_replace_balls} using the parameters $r, L, m$ from Theorem~\ref{thm:minMax}.
    
    As discussed in Remark~\ref{rmk:stable_in_annuli}, the support of the varifold generated from the min-max theory is also stable in admissible annuli. Let $\mathcal{W}^{m, \tilde \eta}_{L, \leq \mathfrak{g}_0}$ be the set of all varifolds $W \in \mathcal{W}_{L, \leq \mathfrak{g}_0}$, whose support is a smooth embedded minimal surface $\Sigma$, such that for every $K(m) := 3^{m3^m}$-admissible collection of annuli of outer radius at most $\tilde \eta$, $\Sigma$ is stable in at least one annulus.

    Let $\{\Phi_i\}$ be a pulled-tight sequence from Theorem~\ref{thm:minMax}~\ref{item:minMaxPulltight}. In the following, we will deform $\{\Phi_i\}$ and obtain a new pulled-tight sequence $\{\tilde \Phi_i\}$ such that there exists $\eta > 0$ such that for all sufficiently large $i$,
    \begin{equation}\label{eqn:minMaxhatW}
        \cH^2(\Phi^*_i(x)) \geq L - \eta \implies |\Phi^*_i(x)| \in \bB^\bF_r(\cW^{m, \tilde \eta}_{L, \leq \mathfrak{g}_0})\,.
    \end{equation}
    Note that 
    \[
        \cW^{m}_{L, \leq \mathfrak{g}_0} \subset \cW^{m, \tilde \eta}_{L, \leq \mathfrak{g}_0} \subset \cW_{L, \leq \mathfrak{g}_0}\,.
    \]
    
    It follows from Remark~\ref{rmk:stable_in_annuli} and Proposition~\ref{prop:am_genus_bound} (with $R = \tilde \eta$) that for each $V \in \bC(\{\Phi_i\}) \setminus \cW^{m, \tilde \eta}_{L, \leq \mathfrak{g}_0}$, there exists a $K(m) := 3^{m3^m}$-admissible collection of annuli of outer radius at most $\tilde \eta$, denoted by $\{A_{V, j}\}^{K(m)}_{j = 1}$, such that $V$ is not almost-minimizing in any $A_{V, j}$ with respect to any min-max subsequence. In particular, for every $V \in \bC(\{\Phi_i\}) \setminus \cW^{m, \tilde \eta}_{L, \leq \mathfrak{g}_0}$, there exists $\varepsilon_V > 0$ and $N_V \in \N^+$ with the following property: For every $\delta > 0$ and every $\Phi_i(x)$ with $i \geq N_V$, if $\mathbf{F}(V, |\Phi_i(x)|) < \varepsilon_V$, then
    \[
        \Phi_i(x) \notin \bigcup^K_{j = 1} \mathfrak{a}(A_{V, j}; \varepsilon_V, \delta)\,.
    \]
    
    Let $\cK := \bC(\{\Phi_i\}) \setminus \bB^\bF_{r/2}(\cW^{m, \tilde \eta}_{L, \leq \mathfrak{g}_0})$. Since $\mathcal{K}$ is compact, we can find a finite set $\{V_k\}^\nu_{k = 1} \subset \mathcal{K}$ such that
    \[
        \mathcal{K} \subset \bigcup^\nu_{k = 1} \bB^\bF_{\frac{\varepsilon_k}{4}} (V_k)\,,
    \]
    where $\varepsilon_k := \varepsilon_{V_k}$. We set 
    \[
        \varepsilon^{(1)} := \frac{\min_k \varepsilon_k}{10}\,, \quad N^{(1)} := \max_k N_{V_k}\,.
    \]

    Since $\{\Phi_i\}$ is pulled-tight, we can choose an integer $N^{(2)} \geq N^{(1)}$ and $\varepsilon^{(2)} \in (0, \varepsilon^{(1)})$ such that for each $i \geq N^{(2)}$, and any $x \in X$, if
    \[
      \cH^2(\Phi_i(x)) \geq \sup_{y \in X} \cH^2(\Phi_i(y)) - 2\varepsilon^{(2)} \,, \quad \mathbf{F}(|\Phi_i(x)|, \cW^{m, \tilde \eta}_{L, \leq \mathfrak{g}_0}) \geq \frac{r}{2}\,,
    \]
    then $\bF(|\Phi_i(x)|, V_{k_{i, x}}) < \frac{\varepsilon_{k_{i, x}}}{2}$ for some $k_{i, x} \in \{1, 2, \cdots, \nu\}$, and thus, for every $\delta > 0$, we have
    \begin{equation}\label{eqn:Phi_i_ann_def}
        \Phi_i(x) \notin \bigcup^K_{j = 1} \mathfrak{a}(A_{V_{k_{i,x}}, j}; \varepsilon_{V_{k_{i,x}}}, \delta)\,.
    \end{equation}
    We denote by $D_i$ the set of all such $x$ (``bad points'') for each $i \geq N^{(2)}$. 
    
    We set 
    \[
        \varepsilon := \frac{\min\left\{\varepsilon^{(1)}, \varepsilon^{(2)}, r, \tilde \eta\right\}}{10}\,.
    \]
    For each $i \geq N^{(2)}$, we also choose $\delta_i > 0$ such that 
    \[
        \lim_i \delta_i = 0\,.
    \]

    By \eqref{eqn:Phi_i_ann_def}, for every $i \geq N^{(2)}$, if $x \in D_i$, then
    \begin{equation}
        \Phi_i(x) \notin \bigcup^K_{j = 1} \mathfrak{a}(A_{V_{k_{i,x}}, j}; 2 \varepsilon^{(1)}, \delta_i / 2)\,.
    \end{equation}
    Hence, by Proposition~\ref{prop:an_deform} with the parameters $R = \tilde \eta$, $\varepsilon = \varepsilon^{(1)}$, $\delta = \delta_i$ and $D = D_i$, we obtain $\Phi^*_i \in \Lambda(\Phi)$. Moreover, for every $x \in X$,
    \[
        \cH^2(\Phi^*_i(x)) \leq \cH^2(\Phi^*_i(x)) + 3^m \cdot \delta_i\,.
    \]
    Since $\delta_i \to 0$, $\{\Phi^*_i\}$ is also a minimizing sequence.
    
    To show that $\bC(\{\Phi^*_i\}) \subset \bB^\bF_{\frac{3}{4}r}(\cW^{m, \tilde \eta}_{L, \leq \mathfrak{g}_0})$, we need to prove the following claim.
    
    \begin{claim}
        For sufficiently large $i$, if $\cH^2(\Phi^*_i(x)) \geq \sup_{y \in X} \cH^2(\Phi_i(y)) - \varepsilon / 4$, then 
        \[
            \bF(|\Phi_i(x)|, \cW^{m, \tilde \eta}_{L, \leq \mathfrak{g}_0}) < r / 2\,.
        \]
    \end{claim}
    \begin{proof}
        Since $\lim_i \delta_i = 0$, for sufficiently large $i$, we have
        \[
            \delta_i < \frac{\varepsilon}{1000 K(m)}\,. 
        \]
        
        By the definition of $D_i$, for $x \in X$, if $\bF(|\Phi_i(x)|, \cW^{m, \tilde \eta}_{L, \leq \mathfrak{g}_0}) \geq r / 2$, we have either $\Phi_i(x) \in D_i$ or $\cH^2(\Phi_i(x)) < \sup_{y \in X} \cH^2(\Phi_i(y)) - 2 \varepsilon^{(2)}$.
        
        If $\Phi_i(x) \in D_i$, i.e.,
        \[
            \cH^2(\Phi_i(x)) \geq \sup_{y \in X} \cH^2(\Phi_i(y)) - 2\varepsilon^{(2)}\,, \quad \mathbf{F}(|\Phi_i(x)|, \cW^{m, \tilde \eta}_{L, \leq \mathfrak{g}_0}) \geq r/2\,,
        \]
        by Proposition~\ref{prop:an_deform}, we have
        \[
            \cH^2(\Phi^*_i(x)) \leq \cH^2(\Phi_i(x)) - \varepsilon^{(2)} + (3^m - 1) \delta_i < \sup_y \cH^2(\Phi_i(y)) - \varepsilon / 4\,.
        \]

        If $\cH^2(\Phi_i(x)) < \sup_{y \in X} \cH^2(\Phi_i(y)) - 2 \varepsilon^{(2)}$, then by Proposition~\ref{prop:an_deform} again, we have
        \[
            \cH^2(\Phi^*_i(x)) \leq \cH^2(\Phi_i(x)) + 3^m \cdot \delta_i < \sup_{y \in X} \cH^2(\Phi_i(y)) - 2 \varepsilon^{(2)} + 3^m \cdot \delta_i < \cH^2(\Phi_i(x)) - \varepsilon / 4\,.
        \]
        
        This finishes the proof of the claim.
    \end{proof}
    
    Now, every $W \in \bC(\{\Phi^*_i\})$ is a varifold limit of some subsequence $\Phi^*_{i_j}(x_{i_j})$ for $j \to \infty$. After passing to a subsequence, $\Phi_{i_j}(x_{i_j})$ converges to a varifold $V$. By Proposition~\ref{prop:an_deform}, we know that there exist some points $p_1, \cdots, p_t \in M$ ($t \leq K(m)$) such that
    \[
        V\llcorner(M \setminus (\overline{B}_{2\tilde \eta}(p_1) \cup \cdots \cup \overline{B}_{2\tilde \eta}(p_t))) = W\llcorner(M \setminus (\overline{B}_{2\tilde \eta}(p_1) \cup \cdots \cup \overline{B}_{2\tilde \eta}(p_t)))\,.
    \]
    
    Furthermore, since
    \begin{align*}
        L \geq \|V\|(M) = \lim_{j \to \infty} \cH^2(\Phi_{i_j}(x_{i_j})) &\geq \lim_{j \to \infty} \cH^2(\Phi^*_{i_j}(x_{i_j})) - K(m) \cdot \delta_{i_j}\\
            &= \|W\|(M) = L\,,
    \end{align*}
    by Lemma~\ref{lem:close_replace_balls}, we can conlude that
    \[
        \mathbf{F}(V, W) < r / 4\,.
    \]
    The previous claim implies that $\mathbf{F}(V, \cW^{m, \tilde \eta}_{L, \leq \mathfrak{g}_0}) < r / 2$, so we have 
    \[
        \mathbf{F}(W, \cW^{m, \tilde \eta}_{L, \leq \mathfrak{g}_0}) < \frac{3}{4} r\,.
    \]
    
    Note that the new minimizing sequence $\{\Phi^*_i\}$ is not necessarily pulled-tight. However, by Proposition~\ref{prop:Pulltight}, we can obtain a pulled-tight minimizing sequence $\{\tilde \Phi_i\}$ with
    \[
        \bC(\{\tilde \Phi_i\}) \subset \bC(\{\Phi^*_i\}) \subset \bB^\bF_{\frac{3}{4}r}(\cW^{m, \tilde \eta}_{L, \leq \mathfrak{g}_0})\,.
    \]
    The conclusion follows immediately from a suitable choice of $\eta$.

\subsection{Proof of Theorem~\ref{thm:relative_minMax}}
    Since $\bL(\Lambda_Z(\Phi)) > \sup_{z \in Z} \cH^2(\Phi(z))$, for each $\Phi' \in \Lambda_Z(\Phi)$, one can restrict all the deformations in the previous subsections to occur away from the compact set
    \[
        X' := \left\{x \in X : \cH^2(\Phi(x)) \leq \frac{\bL(\Lambda_Z(\Phi)) + \sup_{z \in Z} \cH^2(\Phi(z))}{2}\right\}\,.
    \]
    For example, all deformations can be composed with a cut-off function on $X$ that vanishes on $X'$.

    Consequently, all the results follow immediately.
    
\section{The family $\Psi$}\label{sect:family_Psi}
 
\subsection{The space of Clifford Tori}\label{subsect:spaceClifford}
    Let us recall the parametrization of the space $\cC$ of unoriented Clifford tori. 
    A Clifford torus is uniquely determined by an unordered pair of orthogonal unoriented $2$-subspaces in $\mathbb{R}^4$. Namely, given two orthogonal $2$-subspaces $P$ and $Q$, the Clifford torus they define consists of the points on the unit sphere $\mathbb S^3$ that are equidistant  from the equators $P\cap \mathbb S^3$ and $Q\cap \mathbb S^3$. Note that for every $2$-subspace of $\mathbb{R}^4$, there exists exactly one orthogonal $2$-subspace $Q$. Hence, to choose an inside direction for a Clifford torus, i.e. to assign an orientation, one simply needs to specify one $2$-plane $P$, so the space of oriented Clifford tori is $G_2(\R^4)$, the space of unoriented 2-planes in $\R^4$.
    
    By~\cite[\S 1]{HO80}, the set $G_2^+(\mathbb{R}^4)$ of {\it oriented} 2-planes in $\mathbb{R}^4$ can be identified with the set 
    \[
        Q_2 := \{[z_1:z_2:z_3:z_4]: z_1^2 + z_2^2 + z_3^2 + z_4^2 = 0\}\subset \mathbb{CP}^3\,.
    \]
    Namely, for an oriented orthonormal basis $(u,v)$ of a 2-plane in $\mathbb{R}^4$, we can define 
    \[
        (z_1,z_2,z_3,z_4) := u + iv\,.
    \] 
    Moreover, there is a biholomorphism 
    \[
        F : \mathbb{CP}^1 \times \mathbb{CP}^1 \rightarrow Q_2\,,
    \]
    defined in~\cite[\S 2]{HO80}:
    \begin{align*}
        F([w_1 : 1], [w_2 : 1]) &= (1 + w_1 w_2, i(1 - w_1 w_2), w_1 - w_2, -i(w_1 + w_2))\,,  \\
        F([w_1 : 0], [w_2 : 1]) &= (w_2, -iw_2, 1, -i)\,, \\
        F([w_1 : 1], [w_2 : 0]) &= (w_1, -iw_1, -1, -i)\,,  \\
        F([w_1 : 0], [w_2 : 0]) &= (1, -i, 0, 0)\,. 
    \end{align*}
    This gives a homeomorphism between $S^2\x S^2$ and $G_2^+(\R^4)$. It follows from a straightforward calculation (see~\cite[\S 3.4.3]{Nur16}) that $G_2(\R^4)$, or the space of oriented Clifford tori, is homeomorphic to $S^2\x S^2/\sim$ with $(x,y)\sim (-x,-y)$. Moreover, one can verify that $(x,y)$ and $(x,-y)$ correspond to two {\it orthogonal} $2$-subspaces (meaning every vector in one plane is perpendicular to every vector in the other plane), so the set $\cC$ of unoriented Clifford tori is given by $S^2\x S^2/\sim$ with 
    \[
        (x,y)\sim (-x,-y)\sim (x,-y)\sim (-x,y)\,,
    \]
    which is an $\RP^2\x\RP^2$ (see also~\cite[\S 4]{Whi91}). 

\subsection{Definition of $\Psi$} 
    In this subsection, we define the 9-parameter family $\Psi$. We will reuse many of the notations and presentations in~\cite{MN14} by Marques-Neves,~\cite{Nur16} by C. Nurser, and an unpublished manuscript~\cite{Mar23} by F. C. Marques, in which he constructed a 9-parameter family (in a slightly different way than us) that is an 8-sweepout to prove that the $8$-width of the unit $3$-sphere is $2\pi^2$.

    Let $\mathbb S^3$ be the unit 3-sphere. Let $Z=\RP^2\x\RP^2$ parametrize the space of unoriented Clifford tori, and $\tilde  Z$ denote the space of oriented Clifford tori. Then $\tilde{ Z}$ is connected, and there is a natural double-cover map $q : \tilde{ Z} \to  Z$ given by forgetting the orientation. Let $\sigma : \tilde{ Z} \to \tilde{ Z}$ be the nontrivial deck transformation. Note that $\sigma(\sigma(z)) = z$. 
    
    For each $z \in \tilde{ Z}$, let $\Sigma(z)$ be the corresponding oriented Clifford torus, and let $A(z)$ and $A^*(z)$ denote the two connected components of  $\mathbb S^3\backslash \Sigma(z)$ such that $A(z)$ and $A^*(z)$ vary continuously in $z$. Let $N(z)$ denote the unit normal vector field on $\Sigma(z)$ that points into $A^*(z)$. Note that for every $z \in \tilde Z$, and $p \in \Sigma(z)$,
    \[
        A(z)=A^*(\sigma(z))\,, \quad N(z)_{-p}=-N(z)_p\,, \quad N(z)=-N(\sigma(z))\,.
    \]

    In ~\cite[\S 3.4.2]{Nur16}, C. Nurser defined for each oriented Clifford torus $\Sigma(z)$ a 5-sweepout $\Phi^{\Sigma(z)}_5:\RP^5\to\cZ_2(\mathbb S^3;\Z_2)$. Note that as we will see below, each element of $\Phi^{\Sigma(z)}_5$ can actually be represented as a closed set in $\mathbb S^3$. Here, we view $\RP^5$ as the space $\overline{{\mathbb B^4}}\x[-\pi,\pi]$ with boundary points $(v,t)$ identified with $(-v,-t)$. Then, we can construct the corresponding family of closed surfaces. 

    \begin{prop}\label{prop:TildePsiSym}
        For each oriented Clifford torus $\Sigma(z)$, with $z\in \tilde Z$, there exists a family $\Psi^{\Sigma(z)}_5$ of closed subsets of $\mathbb S^3$, parametrized by $\RP^5$, such that:
        \begin{enumerate}[label=\normalfont(\arabic*)]
            \item The maps  $[\Psi^{\Sigma(z)}_5]$ and $\Phi^{\Sigma(z)}_5$ (defined by C. Nurser) are the same.
            \item For every $[(v,t)]\in\RP^5$ and $z\in\tilde Z$,
            \[
                \Psi^{\Sigma(z)}_5([(v,t)])=\Psi^{\Sigma(\sigma(z))}_5([(v,-t)])\,.
            \]
        \end{enumerate}
    \end{prop}

    We will prove this proposition in the following subsections.
    
    Let us assume this proposition for now, and consider the $\RP^5\x\tilde Z$-family $\tilde \Psi$ of closed subsets of $\mathbb S^3$ defined by 
    \[
        \tilde\Psi([(v,t)],z):=\Psi^{\Sigma(z)}([(v,t)])\,.
    \]
    Moreover, consider the $\Z_2$-action on $\RP^5\x\tilde Z$ where the non-trivial element of $\Z_2$ acts by 
    \begin{equation}\label{eq:Z2action}
        ([(v,t)],z)\mapsto ([(v,-t)],\sigma(z))\,,
    \end{equation}
    and let $Y$ be the quotient space. Using Proposition~\ref{prop:TildePsiSym}, the family $\tilde \Psi$  immediately induces  a family $ \Psi$ of closed subsets of $\mathbb S^3$, with parameter space $Y$. Furthermore, $Y$ is an $\RP^5$-bundle over $Z=\RP^2\x\RP^2$.

\subsubsection{Conformal maps and level surfaces}

    For each $v \in {\mathbb B^4}$, we consider the conformal map
    \[
        F_v : \mathbb S^3 \to \mathbb S^3\,, \quad x \mapsto \frac{(1 - |v|^2)}{|x-v|^2} (x - v) - v\,.
    \]
    For each $(v,z)\in {\mathbb B^4}\x\tilde Z$,  we define $d(v,z):\mathbb S^3\to\R$ to be the signed distance to $F_v(\Sigma(z))$, which is positive in $F_v(A^*(z))$ and negative in $F_v(A(z))$. 
    Then for each $v \in {\mathbb B^4}$ and $t \in \R$, we define the open set
    \[
        A(v,t,z)=\{x\in \mathbb S^3:d(v,z)(x)<t\}\,,
    \]
    and the closed set
    \[
        \Sigma(v, t, z) = \partial A(v, t, z)\,.
    \]
    Note that  $\Sigma(v,t,z)$ has a natural  orientation and induces an element in $\cZ(\mathbb S^3;\Z)$, and
    \[
        A(v, t, \sigma(z)) = \mathbb S^3 \setminus(\Sigma(z)\cup  A(v, -t, z))\,.
    \]

    In order to define an $\RP^5$-family, we need to reparametrize $A$ by performing blow-up suitably, because $A$ does not give a  continuous family near $v\in \Sigma(z)$. 

\subsubsection{Reparametrizing $A$ via blow-up}

    Define the right-half disk on $\mathbb{R}^2$ as
    \[
        D^2_+(r) := \{s = (s_1, s_2) \in \mathbb{R}^2 : |s| < r, s_1 \geq 0\}\,.
    \]
    For $\varepsilon > 0$ sufficiently small, we define the map $\Lambda(z) : \Sigma(z) \times D^2_+(3\varepsilon) \to \overline{{\mathbb B^4}}$  by
    \[
        \Lambda(z)(p, s) = (1 - s_1)(\cos(s_2)p + \sin(s_2)N(z)_p)\,,
    \]
    which is a diffeomorphism onto a neighborhood of $\Sigma(z)$ in $\overline{{\mathbb B^4}}$. Notice $\Lambda(z)$ maps into $\mathbb S^3$ if $s_1=0$. 
    
    Let $\Omega_r(z)$ be  the image $\Lambda(z)(\Sigma(z) \times D^2_+(r))$ for all $r \leq 3\varepsilon$. Then as $s$ approaches $(0,3\varepsilon)$ (resp. $(0,-3\varepsilon)$), $\Lambda(z)(p,s)$ approaches $A^*(z)\backslash \Omega_{3\varepsilon}(z)$ (resp. $A(z)\backslash \Omega_{3\varepsilon}(z)$).
    
    Now, we define a continuous map $T(z) : \overline{{\mathbb B^4}} \to \overline{{\mathbb B^4}}$ which collapses the tubular neighborhood $\Omega_\varepsilon(z)$ of $\Sigma(z)$ onto $\Sigma(z)$:
    \begin{itemize}
        \item $T(z)$ is the identity on $\overline{{\mathbb B^4}} \setminus \Omega_{3\varepsilon}(z)$.
        \item On $\Omega_{3\varepsilon}(z)$, 
        \[
            T(z)(\Lambda(z)(p,s))=\Lambda(z)(p,\psi(|s|)s)\,,
        \]
        where $\psi$ is smooth, 0 on $[0,\varepsilon]$, strictly increasing on $[\varepsilon,2\varepsilon]$, and 1 on $[2\varepsilon,3\varepsilon]$. 
    \end{itemize}
    Note that $T(z)=T(\sigma(z))$.

    For every $p\in\Sigma(z)$ and $k\in [-\infty,+\infty]$, define
    \begin{equation}
        \overline{Q}_{p,k,z} = -\frac{k}{\sqrt{1+k^2}}p - \frac{1}{\sqrt{1+k^2}}N(z)_p \in \mathbb S^3\,.
    \end{equation}
    (Note we use the convention that $\frac{k}{\sqrt{1+k^2}}=\pm 1$ and $\frac{1}{\sqrt{1+k^2}}=0$ for $k=\pm \infty$.)
    We define the generalized Gauss map $\overline{Q}(z) : \mathbb S^3 \cup \overline{\Omega}_\varepsilon(z) \to \mathbb S^3$ by:
    \begin{equation}
        \overline{Q}(z)(v) = 
        \begin{cases} 
            -T(v) & v \in A^*(z) \setminus \overline{\Omega}_\varepsilon(z) \\ 
            T(v) & v \in A(z) \setminus \overline{\Omega}_\varepsilon(z) \\ 
            \overline{Q}_{p,k(s),z} & v = \Lambda(z)(p, s) \in \overline{\Omega}_\varepsilon(z)\,,
        \end{cases} 
    \end{equation}
    where $s=(s_1,s_2)\in D^2_+(\varepsilon)$ and
    \begin{equation}
        k(s) = \frac{s_2}{\sqrt{\varepsilon^2 - s_2^2}}\in[-\infty,+\infty]\,.
    \end{equation}
    The key property of $\overline Q$ is that, as $v\in \mathbb S^3$ moves from $A(z)\backslash {\Omega}_{3\varepsilon}(z)$, crosses $\Sigma(z)$, and arrives in $A^*(z)\backslash {\Omega}_{3\varepsilon}(z)$ (so that $v\in \mathbb S^3$ crosses from one side of $\Sigma(z)$ into another), $\overline Q (z)$ flips from being the identity map to the antipodal map, continuously. And the continuous function $k(s)$ is such that $k(s)$ goes to $+\infty$ (resp. $-\infty$) if $\Lambda(z)(p,s)$ goes towards  $A^*(z)\backslash \Omega_\varepsilon(z)$ (resp. $A(z)\backslash \Omega_\varepsilon(z)$).

    We define $\overline{r}(z) : \mathbb S^3 \cup \overline{\Omega}_\varepsilon(z) \to [0, \pi]$ by:
    \begin{equation}
        \overline{r}(z)(v) = 
        \begin{cases} 
            0 & v \in A^* (z)\setminus \overline{\Omega}_\varepsilon(z) \\ 
            \pi & v \in A (z)\setminus \overline{\Omega}_\varepsilon(z) \\ 
            \overline{r}_{k(s)} & v = \Lambda(z)(p, s) \in \overline{\Omega}_\varepsilon(z)\,,
        \end{cases}
    \end{equation}
    where $\bar r_k=\pi/2-\arctan k$. The key point is that $\overline r(z)$ changes from 0 to $\pi$ continuously across the intermediate region $\overline\Omega_\varepsilon (z)$.

    With the above preparation, we can start defining  $\Psi^{\Sigma(z)}_5$. 

    We know the map $\overline Q(z)$ satisfies $\overline Q(\sigma(z)) = -\overline Q(z)$ from $N(z)=-N(\sigma(z))$, and the function $\overline r(z)$ satisfies $r(\sigma(z)) = \pi - \overline r(z)$. Notice that the subset $\overline\Omega_\varepsilon(z)$ is invariant under the antipodal map, and $\overline\Omega_\varepsilon(\sigma(z)) = \overline\Omega_\varepsilon(z)$. The function $\overline r$ can initially be considered as a continuous function defined on
    \[
        \{(z, v) : z \in \tilde{ Z}, v \in \mathbb S^3 \cup \overline\Omega_\varepsilon(z)\}\,,
    \]
    which is a closed subset of $\tilde{ Z} \times {\mathbb B^4}$. We extend it continuously to a function
    \[
        \overline r : \tilde{ Z} \times \overline {\mathbb B^4} \to [0, \pi]
    \]
    that still satisfies $\overline r(\sigma(z)) = \pi - \overline r(z)$. For each $(v,t,z)\in \overline {\mathbb B^4} \times \R \times \tilde{ Z}$, we define the set
    \[
        U(v, t, z) =
        \begin{cases}
            A(T(z)(v), t, z) & \text{if } v \in {\mathbb B^4} \setminus \overline\Omega_\varepsilon(z)\,, \\
            B_{\overline r(z)(v) + t}(\overline Q(z)(v)) & \text{if } v \in \mathbb S^3 \cup \overline\Omega_\varepsilon(z)\,.
        \end{cases}
    \] 
    Notice that $U(v, \pi, z) = \mathbb S^3$ and $U(v, -\pi, z)$ is empty for every $v \in {\mathbb B^4}$. And for $v\in \mathbb S^3\cup \overline\Omega_\varepsilon(z)$, $U$ gives geodesic spheres of varying radius, and the radius always covers the full range $[0,\pi]$ as $t$ ranges over $[-\pi,\pi]$. 
    
    Marques-Neves~\cite[Theorem 5.1]{MN14} showed that the family $\partial U$ is continuous in the flat topology. 

\subsubsection{Closing up the parameter space}
    Let $\gamma : [0,1] \to [0,1]$ be the continuous function satisfying $\gamma(t) = 0$ if $t \leq \frac{1}{2}$ and $\gamma(t) = 2t - 1$ if $t \geq \frac{1}{2}$. For  each 
    \[
        (v,t,z)\in \overline {\mathbb B^4} \times [-1, 1] \times \tilde{ Z} 
    \]
    we define the open set 
    \[
        U'(v, t, z) = U\left(v, 2\pi t + \gamma(|v|)\left(\frac{\pi}{2} - \overline r(z)(v)\right), z\right)\,.
    \]
    Reparametrizing $t$ as above allows us to close up the boundary of $\overline {\mathbb B^4}\x [-1,1]\x\{z\}$ via the antipodal map, as follows. 

    For 
    \[
        (v,t)\in\partial(\overline {\mathbb B^4}\x [-1,1])=(\mathbb S^3\x(-1,1))\cup (\overline {\mathbb B^4}\x\{-1,1\})\,,
    \]
    it is straightforward to check that
    \[
        U'(v, t, z) =
        \begin{cases}
            B_{\frac{\pi}{2} + 2\pi t}(\overline Q(z)(v)) & \text{if } (v,t) \in \mathbb S^3\x(-1,1)\,, \\
            0& \text{if } (v,t)\in \overline {\mathbb B^4}\x\{-1\}\,,\\
            \mathbb S^3& \text{if } (v,t)\in \overline {\mathbb B^4}\x\{1\}\,.
        \end{cases}
    \]
    Note that the smallness of $\varepsilon$ is used here for the case where $v\in \overline \Omega_\varepsilon(z)$.
    Now, using the fact that $B_{\pi-r}(-p)= \mathbb S^3\backslash \overline{B_r(p)}$, $T(z)(-v) = -T(z)(v)$,  and $\overline Q(z)(-v) = -\overline Q(z)(v)$ by~\cite[\S 3.4.3]{Nur16}, it follows immediately that 
    for every $(v, t) \in \partial(\overline {\mathbb B^4} \times [-\pi, \pi])$,
    \begin{equation}\label{eq:U'boundary}
        U'(-v, -t, z) = \mathbb S^3 \setminus \overline{U'(v, t, z)}\,,
    \end{equation}
    This allows us to define an $\RP^5$-family  $\Psi^{\Sigma(z)}_5$ of closed subsets  by 
    \[
        \Psi^{\Sigma(z)}_5 ([(v,t)])=\partial U'(v,t,z)\,.
    \]
    Note that we have disregarded orientation, and a geodesic sphere of radius $0$ or $\pi$ should be both treated as a single point, ensuring $\Psi_5^{\Sigma(z)}$ satisfies Definition~\ref{def:Simon_Smith_family}~\ref{item:closedFamily}. 

\subsubsection{Symmetry of $\Psi^{\Sigma(z)}_5$ in $z$}
    Furthermore, it follows from a straightforward computation that the map $U'$ satisfies the identity
    \begin{equation*}
        U'(v, t, z) = \mathbb S^3 \setminus \overline{U'(v, -t, \sigma(z))}\,,
    \end{equation*}
    for every $(v, t, z) \in {\mathbb B^4} \times [-1, 1] \times \tilde{ Z}$.
    Thus, we have
    \[
        \Psi^{\Sigma(z)}_5([(v,t)])=\Psi^{\Sigma(\sigma(z))}_5([(v,-t)])\,.
    \]
    This finishes the proof of Proposition~\ref{prop:TildePsiSym}.

\subsection{Proof of Theorem~\ref{thm:PsiDef}}
    From the above construction of $\Psi$, it is clear that for each oriented Clifford torus $\Sigma$,  there exists a homeomorphism $f:Y_{[\Sigma]}\to\RP^5$ such that $\Phi^\Sigma_5\circ f=[\Psi|_{Y_{[\Sigma]}}]$. So to prove Theorem~\ref{thm:PsiDef},  it suffices to show that $\Psi$ is a Simon-Smith family of genus $\leq 1$.
    We start with studying the topology of the surfaces in $\Psi$.

\subsubsection{Topology of $\Sigma(v,t,z)$}
    Fix a $z_0\in\tilde Z$ which corresponds to the Clifford torus $\Sigma_0$ given by
    \[
        X(\alpha,\beta)=\frac1{\sqrt 2}(\cos\alpha,\sin\alpha,\cos\beta,\sin\beta)\,,
    \]
    with $\alpha,\beta\in S^1$ (where $S^1$ is viewed as $[0,2\pi]$ with endpoints identified), oriented such that the distance function $d(0,z_0)$ is negative on the ``core circle"
    \[
        w(\beta)=\frac1{\sqrt 2} (0,0,\cos\beta,\sin\beta),\quad \beta\in S^1\,.
    \]

    \begin{prop}\label{lem:behaviourSigmavtz}
        Let $(v,t)\in {\mathbb B^4}\x[-\pi,\pi]$. We have the following descriptions regarding the closed set
        \[
            \Sigma(v,t):=\partial\{x\in \mathbb S^3:d(v,z_0)(x)<t\}\,.
        \]
        
        If $v=0$, then $\Sigma(v,t)$ is a smooth torus for $t\in (-\pi/4,\pi/4)$,  a great circle for $t=\pm\pi/4$, and empty for other $t$.
             
        If $v\ne 0$ and $v/|v|$ lies on the circle 
        \begin{equation}\label{eq:circle}
            \frac 1{\sqrt 2}(\cos\alpha,\sin\alpha,0,0)\,,\quad \alpha\in S^1\,,
        \end{equation}
        then there exist
        $-\pi<t_1<0<t_2<\pi,$ depending on $v$,
        such that:
        \begin{enumerate}[label=\normalfont(\arabic*)]
            \item For $t_1<t<t_2$, $\Sigma(v,t)$ is a smooth torus.
            \item For $t=t_1$ and $t=t_2$, $\Sigma(v,t)$ is a great circle.
            \item For $t< t_1$ and $t> t_2$, $\Sigma(v,t)$ is empty.
        \end{enumerate}
    
        If $v$ does not fall into the above two cases, then there exist 
        \[
            -\pi<t_1<t_2<0<t_3<t_4<\pi\,,
        \] depending on $v$, such that:
        \begin{enumerate}[label=\normalfont(\arabic*)]
            \item For $t_2<t<t_3$, $\Sigma(v,t)$ is a smooth torus.
            \item For $t=t_2$ and $t=t_3$, $\Sigma(v,t)$ is topologically a sphere with two points identified such  that it is smooth except at a point.  
            \item For $t_1<t<t_2$ and $t_3<t<t_4$, $\Sigma(v,t)$ is topologically a sphere, smooth except at two points.  
            \item For $t= t_1$ and $t= t_4$, $\Sigma(v,t)$ is a point.
            \item For $t< t_1$ and $t> t_4$, $\Sigma(v,t)$ is empty.
        \end{enumerate}
    \end{prop}
    \begin{proof}
        First, the case $v=0$ is clear. 

        Second, in the case where $v\ne 0$ and $v/|v|$ lies on the circle \eqref{eq:circle}, note that both the Clifford torus $\Sigma_0$ and the map $F_v$, are symmetric about the antipodal points $\pm v/|v|$. Therefore, the image $F_v(\Sigma_0)$, as well as the level surfaces is parallel to $F_v(\Sigma_0)$, share this symmetry, and the desired result follows immediately.

        Finally, consider the case where $v\ne 0$ with $v/|v|$ not lying on the circle \eqref{eq:circle}. 
        The Clifford torus $\Sigma_0$ is the envelope of the 2-spheres $S_\beta\subset \mathbb S^3$ with center $w(\beta)$ and radius $\pi/4$ under the standard metric on $\mathbb S^3$, where $\beta\in S^1$. Thus, $\Sigma_0$ is a {\em channel surface}. The sphere $S_\beta$ is given by the following equations on $x\in\R^4$:
        \[
            |x|=1\,, \quad x\cdot w(\beta)=\cos(\pi/4)\,.
        \]
    
        Now, when we apply the conformal diffeomorphism $F_v$ to $S_\beta$, we obtain another $2$-sphere $F_v(S_\beta)$ in $\mathbb S^3$, with a new center $\tilde w(\beta)$ and a new radius $\tilde r(\beta)\in (0,\pi)$ under the standard metric on $\mathbb S^3$. 
        
        We see that the function $\tilde r: S^1 \to \mathbb{R}$ has the following property.
    
        \begin{claim}\label{claim:tildeROneMinOneMax}
            If $v\in {\mathbb B^4}$ is non-zero and $v/|v|$ does not lie on the circle (\ref{eq:circle}), then the function $\tilde r:S^1\to \R$ has exactly one local maximum and one local minimum.
        \end{claim}
    
        We will postpone the proof to the end, but first, let us note that this lemma is sufficient to justify the final case of our proposition.
     
        Without loss of generality, we can focus on the case $t < 0$, as the case $t>0$ follows from an identical analysis on the Clifford torus $-\Sigma_0$, i.e. the Clifford torus with the opposite orientation to $\Sigma_0$.
        
        Let us consider the closed envelop given by the collection of spheres 
        \[
            \{\partial B_{\tilde r(\beta)+t}(\tilde w(\beta))\}_{\beta\in S^1}\,,
        \]
        where we use the convention that $\partial B_0(p) = \{p\}$ and $\partial B_r = \emptyset$ for $r < 0$. Given that conformal diffeomorphisms preserve channel surfaces, it is straightforward to verify that this envelop is exactly the set $\Sigma(v,t)$.
        
        Now, suppose the function $\tilde r:S^1\to \R$ has exactly one local maximum and one local minimum. Let us decrease $t$ starting from $0$. When $|t|$ is small, the envelop described above forms a torus. And there would be a $t_2 < 0$ such that, when $t=t_2$, the torus neck pinches at exactly one point, corresponding to the $\beta$ where $\tilde r$ achieves a minimum. As we continue decreasing $t$, the envelop becomes a topological sphere with two singular points, until it eventually vanishes at a point, corresponding to the $\beta$ where $\tilde r$ achieves a maximum. It happens at some time $t_1<0$. This finishes the proof of the third case of Proposition~\ref{lem:behaviourSigmavtz}.
    
        \begin{proof}[Proof of Claim~\ref{claim:tildeROneMinOneMax}]
            Recalling that $F_v$ has a symmetry about the points $\pm v/|v|$, and $F_v$ pushes everything away from $\frac{v}{|v|}$ and towards $-\frac{v}{|v|}$, we know the  center
            $\tilde w(\beta)$ of the sphere 
            $F_v(S_\beta)$ takes the form
            \[
                \tilde w(\beta)=\frac{-\cos(\theta(\beta))\frac v{|v|}+\sin(\theta(\beta))w(\beta)}{\left|-\cos(\theta(\beta))\frac v{|v|}+\sin(\theta(\beta))w(\beta)\right|}
            \]
            for some function $\theta(\beta)\in (0,\pi/2)$. For convenience, let us denote the above denominator by $l(\beta)$. Additionally, the new sphere $F_v(S_\beta)$ is characterized by the following equations for $y\in\R^4$:
            \[
                |y|=1\quad \textrm{ and }\quad y\cdot \tilde w(\beta)=\textrm{constant}\,,
            \]
            where the constant equals to $\cos$ of the radius of the new sphere.
    
            Thus, let us compute $F_v(x) \cdot \tilde w(\beta)$ for $x\in S_\beta$. For simplicity, we omit writing the dependence of $l, \theta$ and $w$ on $\beta$.
            \begin{align}
                &\quad \nonumber F_v(x)\cdot \tilde w(\beta)\\\nonumber
                &= \left(\frac{1-|v|^2}{|x-v|^2}(x-v)-v\right)\cdot \frac{-\cos\theta\frac v{|v|}+\sin\theta \;w}{l}\\
                \label{eq:tildeFv}&=\frac 1{l}\frac{1-|v|^2}{1+|v|^2-2x\cdot v}\left(-\cos\theta\frac{x\cdot v}{|v|}+\cos\theta \;|v|+\sin\theta\frac 1{\sqrt 2}-\sin\theta \;w\cdot v\right)\\
                \nonumber&\;\;+\frac 1{l}(\cos\theta \;|v|-\sin\theta\;w\cdot v)\,.
            \end{align}
            Note that in the last equality, we use $w\cdot x=1/\sqrt 2$, which holds because $x\in S_\beta$. Hence, for some constant $C=C(v,\beta)$, independent of $x\in S_\beta$, we have 
            \[
                -\cos\theta\frac{x\cdot v}{|v|}+\cos\theta \;|v|+\sin\theta\frac 1{\sqrt 2}-\sin\theta \;w\cdot v=C\left({1+|v|^2-2x\cdot v}\right)\,.
            \]
            Now, by setting both the zeroth-order term and the first-order term in $x$ to zero, and canceling $C$ from these two equations, we obtain
            \begin{equation}\label{eq:cosThetaSinTheta}
                \frac{1-|v|^2}{2|v|}\cos\theta=\left(\frac 1{\sqrt 2}-w\cdot v\right)\sin\theta\,,
            \end{equation}
            which leads to 
            \begin{equation}\label{eq:cotTheta}
                \cot\theta=\frac{2|v|}{1-|v|^2}\left(\frac 1{\sqrt 2}-w\cdot v\right)\,.
            \end{equation}
    
            Hence, substituting (\ref{eq:cosThetaSinTheta}) into (\ref{eq:tildeFv}), the radius $\tilde r(\beta)\in (0,\pi)$ satisfies:
            \begin{align*}
                \cos\tilde r(\beta) &= F_v(x)\cdot \tilde w(\beta)\\
                &= \frac 1{l}\frac{1-|v|^2}{|x-v|^2}(x-v)\cos\theta\left(-\frac{x\cdot v}{|v|}+|v|+\frac{1-|v|^2}{2|v|}\right)\\
                &\;\;+\frac 1{l}(\cos\theta \;|v|-\sin\theta\;w\cdot v)\\
                &=\frac 1{l}\left(\frac{1-|v|^2}{2|v|}\cos\theta+(\cos\theta \;|v|-\sin\theta\;w\cdot v)\right)\\
                &=\frac 1{l}\left(\frac 1{|v|}\cos\theta-\frac 1{\sqrt 2}\sin\theta\right)\,.
            \end{align*}
            Now, by a straightforward computation of $l$ using (\ref{eq:cosThetaSinTheta}), we find that
            \[
                l=\sqrt{\frac 1 2\sin^2\theta+\left(\frac1{|v|}\cos\theta-\frac 1{\sqrt 2}\sin\theta\right)^2}\,.
            \]
            It follows straightforward that
            \[
                \cos^2\tilde r= \left(1+\frac 1{(\frac {\sqrt{2}}{|v|} \cot\theta-1)^2}\right)\,.
            \]
            Since $\tilde r$ depends smoothly on $\theta$, the equation can be extended to the case where the denominator is zero. Together with (\ref{eq:cotTheta}) one can derive that 
            \begin{equation}\label{eq:cot2tildeR}
                \cot\tilde r(\beta)= \frac{\sqrt{2}}{|v|}\cot\theta - 1 = \frac{2(1-\sqrt 2 w\cdot v)}{1-|v|^2}-1\,.
            \end{equation}
            Note that the sign of the right-hand side is positive sign due to the monotonicity of $\cot$ and the geometric behavior of the conformal map $F_v$:  when $w \cdot v$ is larger, $\tilde r$ is larger.
        
            Since $\cot $ is a strictly decreasing function with non-zero slope on $(0, \pi)$, by (\ref{eq:cot2tildeR}) and the fact that $\tilde r$ depends smoothly on $w\cdot v$, it suffices to show that $w\cdot v(\beta)$, which is a function of $\beta \in S^1$, has exactly one local minimum and one local maximum. This holds true because $w(\beta)$ describe a circles in $\mathbb S^3$ that is not perpendicular to $v$. This completes the proof of the claim.
        \end{proof}

        This finishes the proof of Proposition~\ref{lem:behaviourSigmavtz}.
    \end{proof}

\subsubsection{$\Psi$ is a Simon-Smith family of genus $\leq 1$}\label{subsubsect:PsiIsSimonSmith}

    Let us finish the proof of Theorem~\ref{thm:PsiDef}, by showing that $\Psi$ is a Simon-Smith family of genus $\leq 1$.

    By the definition of $\Psi$ in the previous section and Proposition~\ref{lem:behaviourSigmavtz}, we can conclude that $\Psi$ consists of the following types of closed subsets:
    \begin{enumerate}
        \item A smooth torus: This arises as the nearby level surface of the signed distance function to $F_v(\Sigma(z))$, where $v \notin \Sigma(z)$.
        \item A smooth sphere:  They are geodesic spheres on $\mathbb S^3$ with radius strictly between $0$ and $\pi$.
        \item A great circle: This occurs in the second case of $v$ in Proposition~\ref{lem:behaviourSigmavtz}.
        \item A topological sphere with two points identified, so that it is smooth except at a point: This occurs in the third case of $v$ in Proposition~\ref{lem:behaviourSigmavtz}.
        \item A topological sphere, smooth except at two points: This occurs in the third case of $v$ in Proposition~\ref{lem:behaviourSigmavtz}.
        \item A single point: It has  several forms. (1) A geodesic sphere on $\mathbb S^3$ with radius 0. (2) A geodesic sphere on $\mathbb S^3$ with radius $\pi$. (3) It also occurs in the third case of $v$ in Proposition~\ref{lem:behaviourSigmavtz}. 
        \item The empty set.
    \end{enumerate}
    Then, it follows easily that each member of $\Psi$ is an element of $\cS(S^3)$,  and $\Psi$ satisfies Definition~\ref{def:Simon_Smith_family}, \ref{item:Hausdorff_cts}, \ref{item:closedFamily}, \ref{item:SingPointsUpperBound}, and also the genus bound~\ref{item:SimonSmithFamilyGenus}.

    As for Definition~\ref{def:Simon_Smith_family}~\ref{item:SimonSmithFamilyLocalSmooth}, regarding the local $C^\infty$-convergence, we know from the definition of $\Psi$ that it suffices to fix a $z = z_0$ and show that $\Psi^{\Sigma(z_0)}_5$ satisfies Definition~\ref{def:Simon_Smith_family}~\ref{item:SimonSmithFamilyLocalSmooth}. In fact, if we define the family
    \[
        \Psi'_5 (v,t):=\partial U(v,t,z_0),\;\; (v,t)\in \overline {\mathbb B^4}\x [-2\pi,2\pi]
    \]
    (with $\Psi'_5(v,t)$ treated as a single point if it is a geodesic sphere of radius 0 or $\pi$),
    and notice that $\Psi^{\Sigma(z_0)}_5$ was defined as a subfamily of $\Psi'_5$ with suitable closing-up of the boundary of the parameter space, it becomes clear that it suffices to show that $\Psi'_5$ satisfies Definition~\ref{def:Simon_Smith_family}~\ref{item:SimonSmithFamilyLocalSmooth}. 

    Let us first show that $\Psi'_5(\cdot,0)$ satisfies Definition~\ref{def:Simon_Smith_family}~\ref{item:SimonSmithFamilyLocalSmooth}. We now explicitly describe, for each $\Psi'_5(v,0)$, what the associated set $P(v,0)$ of points (according to Definition~\ref{def:Simon_Smith_family}) should be. By definition,  $\Psi'_5(v,0)$ takes on one of the following forms, for $v\in \overline{B}^4$.
    \begin{enumerate}
        \item If $v\in {\mathbb B^4}\backslash\overline{\Omega_\varepsilon(z_0)}$, $\Psi'_5(v,0)$ is a smooth torus.
        \item If $v\in \mathbb S^3\backslash \overline{\Omega_\varepsilon(z_0)}$, $\Psi'_5(v,0)$ is empty.
        \item If $v\in \mathbb S^3\cap \overline{\Omega_\varepsilon(z_0)}$, $\Psi'_5(v,0)$ is a single point.
        \item If $v\in\overline{ \Omega_\varepsilon(z_0)}\backslash \mathbb S^3$, $\Psi'_5(v,0)$ is a smooth geodesic sphere.
    \end{enumerate}
    For the first two case, we set $P(v,0)$ to be empty. For the third case, we take $P(v,0)= \Psi'_5(v,0)$. For the last case, we need to be a bit more careful. If $v$ lies in the interior of $\Omega_\varepsilon(z_0)$, we can set $P(v,0)$ to be empty. However, if $v\in {\mathbb B^4}\cap \partial \Omega_\varepsilon(z_0)$, we should let $P(v,0)$ consist of the point $-v/|v|$. Indeed, from~\cite[Theorem 5.1]{MN14}, we already know $\Psi'_5(\cdot,0)$ is continuous in the flat topology near each $v$. However, the convergence is not smooth when approaching such $v$ from outside $\overline{\Omega_\varepsilon(z_0)}$, as we would see a handle of shrinking size near the point $-v/|v|$. Wth such definition of $P(v,0)$ for each $v$, it is straightforward to verify from the definition that $\Psi'_5(\cdot,0)$ satisfies  Definition~\ref{def:Simon_Smith_family}~\ref{item:SimonSmithFamilyLocalSmooth}.

    Now, consider the case $t\ne 0$. By the definition of $U$, each $\Psi'_5(v,t)$ is a level surface of the signed distance function to $\Psi'_5(v,0)$. Thus, using Proposition~\ref{lem:behaviourSigmavtz} and the fact that  $\Psi'_5(\cdot,0)$ satisfies  Definition~\ref{def:Simon_Smith_family}~\ref{item:SimonSmithFamilyLocalSmooth}, it is easy to check that $\Psi'_5$ also satisfies Definition~\ref{def:Simon_Smith_family}~\ref{item:SimonSmithFamilyLocalSmooth}, with $P(v, t)$ be the union of projection of $P(v, 0)$ to $\Psi'_5(v, t)$ and finite singular points as described in Proposition~\ref{lem:behaviourSigmavtz}. This finishes the proof that  $\Psi'_5$ satisfies Definition~\ref{def:Simon_Smith_family}~\ref{item:SimonSmithFamilyLocalSmooth}.

    Finally, to show that the closing-up $\Psi^{\Sigma(z_0)}_5$ satisfies Definition~\ref{def:Simon_Smith_family}~\ref{item:SimonSmithFamilyLocalSmooth}, one simply need to merge the associated set $P(v, t)$ and $P(v', t')$ if $(v, t) \sim (v', t')$. Since $\Psi$ is induced from $\Psi^{\Sigma(z_0)}_5$ via $SO(4)$-action and another closing-up, it follows that $\Psi$ also satisfies Definition~\ref{def:Simon_Smith_family}~\ref{item:SimonSmithFamilyLocalSmooth}.
    
    As a result, $\Psi$ is a Simon-Smith family of genus $\leq 1$.

\subsection{A 7-sweepout in $\Psi$}\label{subsect:7sweepout}
    As part of the setup for the proof of Theorem~\ref{thm:cupLength9}, we consider a subset $C\subset Y=\dmn(\Psi)$, where $C\cong \RP^4\x\RP^2\x \RP^2$, such that $[\Psi|_C]$ is a 7-sweepout. 

    First, we define the submanifold
    \[
        \tilde C:=\{([(v,t)],z)\in \RP^5\x\tilde Z:t=0\}\,.
    \]
    Evidently, $\tilde C\cong\RP^4\x\tilde Z$.  Note that $\tilde C$ is invariant under the $\Z_2$-action on $\RP^5\x\tilde Z$  described in (\ref{eq:Z2action}).  Let $C$ be the quotient of   $\tilde C$ by this $\Z$-action. Then $ C\cong\RP^4\x Z$, and $C$ can be viewed as a subset of $Y$. Geometrically, $\Xi|_C$ consists of images of Clifford tori under conformal maps, and geodesic spheres. The fact that $C$ is a product, instead of a non-trivial $\RP^4$-bundle over $Z$, reflects the irrelevance of Clifford tori's orientations in defining $\Xi|_C$, as opposed to the situation over $Y\backslash C$.

    Now, 
    \[
        H_1(C;\Z_2)\cong H_1(\RP^4;\Z_2)\oplus H_1(\RP^2;\Z_2)\oplus H_1(\RP^2;\Z_2)\cong \Z_2\oplus\Z_2\oplus\Z_2\,.
    \]
    Let $\omega_C,\alpha_C,\beta_C\in H^1(C;\Z_2)$ be respectively the Hom-duals of 
    \begin{equation}\label{eq:3homologyClasses}
        (1,0,0),(0,1,0),(0,0,1)\in H_1(C;\Z_2).
    \end{equation}
    By the K\"unneth formula and the fact that the cohomology ring $H^*(\RP^k;\Z_2)$ is $\Z_2[x]/(x^{k+1})$, we have
    \[
        H^*(C;\Z_2)\cong \Z_2[\omega_C,\alpha_C,\beta_C]/(\omega_C^5,\alpha_C^3,\beta_C^3)\,.
    \]
    As verified in~\cite[\S 3.4.3]{Nur16}, the three homology classes in \eqref{eq:3homologyClasses} all correspond to $1$-sweepouts under $[\Psi|_C]$, and thus the element 
    \[
        \lambda_C:=([\Psi|_C])^*(\bar\lambda)\in H^1(C;\Z_2)
    \]
    is equal to $\omega_C+\alpha_C+\beta_C$. Hence, by the K\"unneth formula,
    \[
        (\lambda_C)^7=\omega_C^4\cup\alpha_C^2\cup\beta_C+\omega_C^4\cup\alpha_C\cup\beta_C^2\ne 0
    \]
    in $\Z_2$-coefficients. So $[\Psi|_C]$ is a 7-sweepout, as $[\Psi]$ is continuous in the $\cF$-norm by Proposition~\ref{prop:SS_AP}.

\subsection{Proof of Theorem~\ref{thm:cupLength9}} 
    Recall the three elements $\lambda,\alpha,\beta$ in $H^1( Y;\Z_2)$ defined by:
    \begin{itemize}
        \item Let $\bar\lambda$ be the non-trivial element of $H^1(\cZ_2(S^3;\Z_2);\Z_2)$; define $\lambda=[\Psi]^*(\bar \lambda)$.
        \item Let $a= \RP^1\x\RP^2\subset\RP^2\x \RP^2$, and  $A\subset Y$ be the $\RP^5$-subbundle over $a$; define $\alpha$ to be the Poincar\'e dual $PD(A)$ of $A$.
        \item Let $b= \RP^2\x\RP^1\subset\RP^2\x \RP^2$, and  $B\subset Y$ be the $\RP^5$-subbundle over $b$; define $\beta$ to be the Poincar\'e dual $PD(B)$ of $B$.
    \end{itemize}
    Our goal is to prove $\lambda^5\cup\alpha^2\cup\beta^2\ne 0$.

    Let us first write down the generators of $H^*(Y;\Z_2)$. Let $C\subset Y$ be the 8-dimensional subset obtained in the previous subsection such that $[\Psi|_C]$ is a $7$-sweepout, and $i_C:C\to Y$ be the inclusion map. Recall that the element
    \[
        \lambda_C:=i_C^*(\lambda)=([\Psi]\circ i_C)^*(\bar \lambda)\in H^1(C;\Z_2)
    \]
    satisfies 
    \begin{equation}\label{eq:7powerNonZero}
        (\lambda_C)^7\neq 0\;\;{\textrm{in} }\; H^*(C; \Z_2)\,.
    \end{equation} Now, set $\gamma:=PD(C)\in H^1(Y;\Z_2)$.

    Recall that $Y$ is an $\RP^5$-bundle over $Z=\RP^2\times \RP^2$:
    \[
        \RP^5\to Y=\dmn(\Psi) \to Z\,.
    \] 
    Let us denote the projection map by $\pi: Y \to Z$.

    \begin{claim}\label{claim:cohoRingY}
        The cohomology ring $H^*(Y;\Z_2)$ is generated by $\alpha, \beta, \gamma.$
    \end{claim}
    \begin{proof}
        First, by definition, $\alpha,\beta$ are the pullbacks of the two generators of 
        $H^*(Z;\Z_2)$ under the projection  map $\pi:Y\to Z$. Moreover, note that for each $z\in Z$, the intersection of $C$ with every fiber $Y_z\cong\RP^5$ of $Y$ is an $\RP^4$, whose Poincar\'e dual in $H^1(Y_z;\Z_2)$ generates 
        \[
            H^*(Y_z;\Z_2)\cong H^*(\RP^5;\Z_2)\,.
        \] 
        Thus, if we pullback $\gamma$ to $H^1(Y_z;\Z_2)$ via the inclusion map $Y_z\hookrightarrow Y$, we obtain the generator of $H^*(Y_z;\Z_2)$.
        Then the claim follows immediately from the Leray–Hirsch Theorem (see~\cite[Theorem 4D.1]{Hat02}). 
    \end{proof}

    Let $\pi_1, \pi_2: Z\to \RP^2$ be the projection onto the first and second factor respectively. Let $k_1, k_2\subset \RP^2$ be two transverse isotopic loops in $\RP^2$ with a single intersection point $p$. Let $Y_{(p,p)}\subset Y$ be the fiber over $(p,p)$. 

    \begin{claim}\label{claim:alphaBetaCupProduct}
        $\alpha^2\cup\beta^2=PD(Y_{(p,p)})$, and $\alpha^r\cup \beta^s=0$ whenever $r\geq 3$ or $s\geq 3$. 
    \end{claim}
    \begin{proof}
        By definition, $\alpha$ is the Poincar\'e dual of $A_i:= (\pi_1\circ\pi)^{-1}(k_i)$ while $\beta$ is the Poincar\'e dual of $B_i:= (\pi_2\circ\pi)^{-1}(k_i)$, for $ i=1, 2$. Therefore, $\alpha^2\cup \beta^2$ is the Poincar\'e dual of $A_1\cap A_2\cap B_1\cap B_2$, which is the fiber  $Y_{(p, p)}$. 
        
        If $r\geq 3$ or $s\geq 3$, we just consider three perturbed copies of $k_1\subset\RP^2$ that do not share a common point, and their preimages under $\pi_1\circ\pi$ or $\pi_2\circ\pi$, similarly as above. Then second claim would follow.
    \end{proof}

    \begin{claim}\label{claim:notJustAB}
        In $H^1(Y;\Z_2)$, $\lambda$ is not a linear combination of just $\alpha$ and $\beta$.
    \end{claim}
    \begin{proof}
        Suppose by contradiction that $\lambda=c_1\alpha + c_2 \beta$ for some $c_1, c_2\in \Z_2$.
        Then 
        \[
            \lambda_C=i_C^*(\lambda)=i_C^*(c_1\alpha + c_2 \beta)\,.
        \]
        Recall that $(\lambda_C)^7\ne 0$ by (\ref{eq:7powerNonZero}). But by Claim~\ref{claim:alphaBetaCupProduct}, $(c_1\alpha + c_2 \beta)^7=0$, which leads to a contradiction.
    \end{proof}

    From Claim~\ref{claim:cohoRingY} and~\ref{claim:notJustAB}, we have $\lambda = \gamma + c_1\alpha + c_2\beta$ for some $c_1,c_2\in\Z_2$. Theorem~\ref{thm:cupLength9} will follow from the following claim.
    \begin{claim}
        $\gamma^5\cup \alpha^2\cup\beta^2 \neq 0$.
    \end{claim}
    \begin{proof}
        Since $C \to Z$ is a trivial $\mathbb{RP}^4$-bundle, we can perturb it to obtain $ C_1, C_2, C_3, C_4,C_5$ such that
        \begin{itemize}
            \item Within a trivialization of $Y$ over some open neighborhood $U \subset Z$ of $(p, p)$, for each $i$, the restriction of $C_i$ over $U$ is equal to $\hat{C}_i \times U$ for some $\hat{C}_i \subset \mathbb{RP}^5$ homeomorphic to an $\RP^4$;
            \item All $\hat{C}_i$'s transversely intersect at a single point.
        \end{itemize}
        Thus, $Y_{(p,p)}\cap \bigcap_{i=1}^5 C_i$ is just a single point in $Y$. Together with Claim~\ref{claim:alphaBetaCupProduct} and the definition $\gamma=PD(C)$, we know $\gamma^5\cup \alpha^2\cup\beta^2 \neq 0 $. 
    \end{proof}

    Thus, by the above claim and   Claim~\ref{claim:alphaBetaCupProduct}, we have
    \begin{align*}
        \lambda^5\cup\alpha^2\cup\beta^2 = (\gamma+(c_1\alpha+c_2\beta))^5\cup \alpha^2\cup\beta^2 = \gamma^5\cup \alpha^2\cup\beta^2 \neq 0
    \end{align*}
    in $H^*(Y; \mathbb{Z}_2)$. This completes the proof of Theorem~\ref{thm:cupLength9}.

\section{Perturbing the metric}\label{sect:perturbMetric}
    Note that in a $3$-dimensional Riemannian manifold $(S^3, g)$ with positive Ricci curvature, there is no stable minimal surface, and moreover, by~\cite{CS85}, there exists $L > 0$, such that the area of every minimal surface with genus at most $1$ is no greater than $L / 2$. Therefore, Proposition~\ref{prop:perturbMetric} follows from the following general result. 

    \begin{prop}[Generalization of Proposition~\ref{prop:perturbMetric}] \label{Prop_MetricPerturb_Main}
        Let $(S^3, g_0)$ be a Riemannian $3$-sphere with only finitely many embedded minimal tori and no degenerate stable embedded minimal sphere. For any $\varepsilon>0$ and any $L$ greater than the largest area among all minimal tori in $(S^3, g_0)$, there exists a metric $g'$ that is $\varepsilon$-close to $g_0$ in $C^\infty$ such that among all embedded $g'$-minimal surfaces with area less than $L$:
        \begin{enumerate}[label=\normalfont(\arabic*)]
            \item \label{item:sameNumber_inproof} The number of embedded $g'$-minimal tori is the same as that of $g_0$-minimal tori.
            \item There are only finitely many embedded minimal spheres $\{S_1,\cdots,S_q\}$, and each of them is non-degenerate.
            \item \label{item:linearIndep_inproof} The $g'$-areas
            \[
                \area_{g'}(S_1), \dots, \area_{g'}(S_q)
            \]
            are $\mathbb{Z}$-linearly independent. Moreover, their $\mathbb{Z}$-linear combination can never achieve the area of any $g'$-minimal tori with $g'$-area less than $L$.
        \end{enumerate}
    \end{prop}

    Let us recall a compactness result for minimal surfaces essentially following from~\cite[Page~390, Claim(*)]{CS85} (cf.~\cite{And85,Whi87},~\cite[Proposition~7.14]{CW11}).

    \begin{lem} \label{lem:min_surf_cpt}
        In $S^3$, given a non-negative integer $\mathfrak{g}_0$, a sequence of Riemannian metrics $\{g_i\}^\infty_{i = 1}$ and a sequence of connected embedded closed surfaces $\{\Sigma_i\}^\infty_{i = 1}$, suppose that there exists another Riemannian metrics $g_\infty$ such that $g_i \to g_\infty$ in $C^3$ and each $\Sigma_i$ is a $g_i$-minimal surface satisfying
        \[
            \sup_{i} \operatorname{area}_{g_i}(\Sigma_i) < + \infty\,, \quad \sup_i \mathfrak{g}(\Sigma_i) \leq \mathfrak{g}_0\,.
        \]
        Then there exists a subsequence of $\{\Sigma_i\}$, still denoted by $\{\Sigma_i\}$, a connected $g_\infty$-minimal surface $\Sigma_\infty$ and a finite number of points $Z = \{x_1, \cdots, x_l\}$ such that in the varifold topology,
        \[
            |\Sigma_i| \to k |\Sigma_\infty|, \quad k \mathfrak{g}(\Sigma_\infty) \leq \mathfrak{g}_0\,,
        \]
        with some $k \in \N^+$, and for any $r > 0$, $\Sigma_i \to \Sigma_\infty$ in $S^3 \setminus B_r(Z)$ (multi-)graphically in the $C^3$ topology.

        Furthermore, if $k = 1$, then $Z = \emptyset$, and thus, for sufficiently large $i$, $\mathfrak{g}(\Sigma_i) = \mathfrak{g}(\Sigma_\infty)$; if $k > 1$, then $\Sigma_\infty$ is degenerate stable.
    \end{lem}
    \begin{rmk}
        Note that $S^3$ does not admit non-orientable closed surfaces, so all $\Sigma_i$ and $\Sigma_\infty$ are oriented.
    \end{rmk}
    \begin{proof}
        The first part, except for the genus bound of $\Sigma_\infty$, follows from~\cite[Theorem~3]{Whi87}, since a $g_i$-minimal surface can be viewed as an embedded $\Phi_i$-stationary surface in $(S^3, g_\infty)$ for some even elliptic integrand $\Phi_i$, and $\Phi_i$ converges to the constant function $1$. By the area monotonicity formula for $g_i$-minimal surfaces, $\Sigma_i$ also converges to $\Sigma_\infty$ in the Hausdorff distance sense.

        Then there exists $r_0 > 0$ such that $\mathfrak{g}(\Sigma_\infty \setminus B_{r_0}(Z)) = \mathfrak{g}(\Sigma_\infty)$. Since $\Sigma_i \to \Sigma_\infty$ in $S^3 \setminus B_{r_0}(Z)$ (multi-)graphically in the $C^3$ topology, for sufficiently large $i$, $\Sigma_i \setminus B_{r_0}(Z)$ has exactly $k$ connected components $\{\Gamma^j_i\}^k_{j = 1}$, and
        \[
            \sum^k_{j = 1} \mathfrak{g}(\Gamma^j_i) \leq \mathfrak{g}(\Sigma_j) \leq \mathfrak{g}_0\,.
        \]
        Hence, for sufficiently large $i$ we have
        \[
            \mathfrak{g}(\Sigma_\infty) = \mathfrak{g}(\Sigma_\infty \setminus B_{r_0}(Z)) = \min_j \mathfrak{g}(\Gamma^j_i) \leq \frac{\mathfrak{g}_0}{k}\,.
        \]
        
        When $k = 1$, $Z = \emptyset$ follows immediately from Allard's regularity~\cite{All72}. When $k > 1$, it follows from~\cite[Claim~5~and~Claim~6]{Sha17} that there exists a positive Jacobi fields on $\Sigma_\infty$, and thus, $\Sigma_\infty$ is degenerate stable.
    \end{proof}

    We prove the following useful lemma.

    \begin{lem} \label{Lem_MetricPerturb_Neighb of Torus}
        Suppose $(S^3, g_0)$ contains only finitely many embedded minimal tori, given by $\cT:= \{T_1, T_2, \dots, T_l\}$. Then for any $L> 0$, there exists an open subset $U\supset \bigcup_{1\leq j\leq l}T_j$ in $S^3$ and an open neighborhood $\cO$ of $g_0$ in $C^\infty$ such that for every $g\in \cO$,
        \begin{enumerate}[label=\normalfont(\arabic*)]
            \item Every embedded $g$-minimal $2$-sphere $S$ with $g$-area less than $L$ in $S^3$ is not entirely contained in $U$;
            \item If we further assume that $(S^3, g_0)$ does not admit any degenerate stable embedded minimal sphere with area less than $L$, and $g|_U = g_0|_U$. Then every embedded $g$-minimal torus $T$ with $g$-area less than $L$ in $(S^3, g)$ is inside $\cT$.
        \end{enumerate}
    \end{lem}
    \begin{proof}
        We first prove that (1) holds for some $U$ and $\cO$. Suppose for contradiction that there exists an $L>0$ and a sequence of metrics $\{g_j\}$, which is $C^3$-converging to $g_0$, as well as a sequence of $g_j$-minimal spheres $\{S_j\}$ in $S^3$, each of which has $g_j$-area less than $L$ and is contained in $1/j$-tubular neighborhood of $\bigcup_{1\leq j\leq l}T_j$.

        By Lemma~\ref{lem:min_surf_cpt}, up to a subsequence, $S_j$ converges to some minimal sphere $S_\infty$ with multiplicity in $\bigcup_{1\leq j\leq l}T_j$ as varifolds. However, there is no minimal sphere lying in $\bigcup_{1\leq j\leq l}T_j$, a contradiction. 
        
        Hence, we finish the proof of (1).

        Now that we fix $L, U$ as above and argue that for a smaller neighborhood $\cO' \subset \cO$ of $g_0$, (b) is true. To see that, suppose again for contradiction that there exists a sequence of metrics $\{g_j\}_J$ $C^3$-converging to $g_0$ with $g_j|_U = g_0|_U$, as well as a sequence of $g_j$-minimal torus $\{T_j \notin \cT\}$ in $S^3$ with $g_j$-area less than $L$. By Lemma~\ref{lem:min_surf_cpt}, $|T_j|$ varifold subconverges to $k|\Sigma|$ for some smooth minimal surface $\Sigma\subset (S^3, g_0)$.

        Moreover, if $\Sigma$ is also a torus, then $\Sigma\in \cT$ and $T_j\subset U$ for $j>>1$, but since $g_j|_U = g_0|_U$, such $T_j$ is also $g_0|_U$ minimal, which contradicts to the definition of $\cT$. Therefore, $\Sigma$ is a sphere and the multiplicity $k\geq 2$. Then again by Lemma~\ref{lem:min_surf_cpt}, $\Sigma$ is a degenerate stable minimal sphere in $(S^3, g_0)$, which contradicts to the assumption in (2).

        This completes the proof of (2).
    \end{proof}

    Now we adapt White's bumpy metric argument to our setting. In the following, let $3\leq q\leq \infty$ be an integer (or infinity), $L>\sup_j\{\area_{g_0}(T_j)\}$ be fixed, $U, \cO$ be determined from previous Lemma. Let 
    \begin{align*}
        \Gamma^q & := \{C^q\text{-metric }g\in \cO\text{ conformal to }g_0: g|_U = g_0|_U\} \\ 
        \cM^q & := \{(g, S): g\in \Gamma^q\,;\; S \text{ is a minimal sphere in }(S^3, g), \area_g(S)<L\}\,;
    \end{align*}
    Here we use $C^q$-topology on Riemannian metrics and $C^{2,\alpha}$-topology on submanifolds, $\alpha\in (0,1)$.

    \begin{lem} \label{Lem_MetricPerturb_WhiteBumpyMetric}
        For every integer $3\leq q< \infty$, $\cM^q$ is a separable $C^{q-2}$ Banach manifold modeled on $C^q(S^3)$ and 
        \[ 
            \Pi: \cM^q \to \Gamma^q \,, \;\;\; (g, S)\mapsto g 
        \] 
        is a $C^{q-2}$ Fredholm map with index $0$.
        
        Moreover, $g\in \Gamma^q$ is a regular value of $\Pi$ if and only if for every $(g, S)\in \cM^q$, $S$ is non-degenerate.
    \end{lem}
    \begin{proof}
        For every $(\hat{g}, \hat{S})\in \cM^q$ with unit normal field $\hat{\nu}$, we apply~\cite[Theorem~1.2]{Whi91} with
        \begin{align*}
            & \Gamma := \Gamma^q\,, \;\;\;\;\;  X := C^{2,\alpha}(\hat{S})\,,  \;\;\;\;\;  Y:= C^\alpha(\hat{S})\,, \;\;\;\;\;  \cH:= L^2(\hat{S})\,, \\
            & A: \Gamma\times X \to \R \,, \;\; (g, w) \mapsto \area_g(\graph_{\hat{S}, \hat{g}}(w))\,, \\
            & H: \Gamma\times X\to Y\,, \; \text{ given by } \langle H(g, w), v\rangle_{L^2} = \left.\frac{d}{dt}\right|_{t=0} A(g, w+tv) \,.
        \end{align*}
        where every norm is measured under metric $\hat{g}$. Note that $\Gamma^q$ can be identified as an open subset of 
        \[
            C^q_U(S^3):= \{f\in C^q(S^3): f|_U=0\}\,.   
        \]
        endowed with $C^q$-topology. This is also a separable Banach space.

        Notice that near $(\hat{g}, \hat{S})$, $\cM^q = \{(g, S)\in \Gamma\times X: H(g, S) = 0\}$. And as is verified in~\cite[Theorem~1.1]{Whi91}, 
        \[
            D_2H(\hat{g}, 0) = -(\Delta_{\hat{S}} + |A_{\hat{S}}|^2 + \Ric_{\hat{g}}(\hat{\nu}, \hat{\nu})) =: L_{\hat{S}, \hat{g}} : X\to Y
        \] 
        is a Fredholm map with index $0$. Hence to prove the first part of the Lemma, it suffices to verify the condition (C) in~\cite[Theorem~1.2]{Whi91}, namely, for every $0\neq \kappa\in \Ker D_2H(\hat{g}, 0)$, there exists a one parameter family $\{\gamma(s)\in \Gamma^q\}_{s\in (-1, 1)}$ such that $\gamma(0) = \hat{g}$ and 
        \begin{align}\label{Equ_MetricPerturb_Condition(C)}
            \left(\frac{\partial^2}{\partial s\partial t}\right)_{s=t=0} A(\gamma(s), t\kappa) \neq 0 \,.    
        \end{align}
        To see this, let $\gamma(s) = (1+sf)\hat{g}$ where $f\in C^q_U(S^3)$. According to the calculation in~\cite[Theorem~2.1]{Whi91}, 
        \[
            \left(\frac{\partial^2}{\partial s\partial t}\right)_{s=t=0} A(\gamma(s), t\kappa) = \int_{\hat{S}} \hat{\nu}(f)\kappa \ dx\,.
        \]
        Recall that by Lemma~\ref{Lem_MetricPerturb_Neighb of Torus}, $\hat{S}\setminus \bar{U}\neq \emptyset$; And by the unique continuation property, $\kappa|_{\hat S\setminus U}$ is not identically $0$. Therefore, (\ref{Equ_MetricPerturb_Condition(C)}) is true by taking $f\in C_c^q(S^3\setminus \bar{U})$ such that $\hat{\nu}(f)$ is an $L^2$-approximation of $\kappa\cdot \chi_{S^3\setminus U}$.

        To study the regular value of $\Pi$, first notice that for every $(g', S')\in \cM^q$, 
        \begin{align*}
            \Tan_{(g', S')} \cM^q & = \{(f, w): f\in C^q_U(S^3),\; L_{S', g'} = \hat{\nu}(f) \text{ on }S'\} \,; \\
            d\Pi|_{(g', S')}& : \Tan_{(g', S')} \cM^q \to C^q_U(S^3)\,, \;\; (f, w)\mapsto f \,.
        \end{align*}
        Thus, if $g'\in \Gamma^q$ is a regular value of $\Pi$, then for every $(g', S')\in \cM^q$ with unit normal field $\nu'$, every $\kappa\in \Ker L_{S', g'}$ and every $\phi\in C^q_c(S'\setminus \bar{U})$, we take $f\in C^q_U(S^3)$ such that $\nu'(f) = \phi$, then there exists $w_\phi\in C^{2, \alpha}(S')$ such that $L_{S',g'}w_\phi = \nu'(f) = \phi$. Then 
        \[
            0 = \int_{S'} w_\phi\cdot L_{S', g'}\kappa = \int_{S'} \kappa \cdot L_{S',g'} w_\phi = \int_{S'} \kappa \phi \,. 
        \]
        In other words, $\kappa\perp C^q_c(S'\setminus \bar{U})$. Therefore, $\kappa|_{S'\setminus \bar{U}} = 0$. Again by Lemma~\ref{Lem_MetricPerturb_Neighb of Torus} and the unique continuation property, $\kappa \equiv 0$ on $S'$. Hence, $\Ker L_{S', g'} = 0$, that is, $S'$ is non-degenerate.
    \end{proof}

    To obtain the third assertion of Proposition~\ref{Prop_MetricPerturb_Main}, we need the following lemma in linear algebra.

    \begin{lem}\label{Lem_MetricPerturb_Abstract}
        Let $k, l\geq 1$. Given $\va\in \R^k, \vb\in \R^l$, consider the linear function 
        \begin{align*}
            f_{\va}: \Z^k\to \R\,, \quad \vu\mapsto \vu\cdot \va \,; \qquad
            h_{\vb}: \Z^l\to \R\,, \quad \vv\mapsto \vv\cdot \vb \,.
        \end{align*}
        Then for every $\vb\in \R^l$, 
        \[
            \cR:= \{\va\in \R^k: \Ker(f_\va) = \{0\}\,, \; f_\va(\Z^k)\cap h_{\vb}(\Z^l) = \{0\}\}
        \]
        is a countable intersection of open dense subset of $\R^k$ (and thus dense in $\R^k$).
    \end{lem}
    \begin{proof}
       Let $L_k:\Z^k\to \N$ and $L_l: \Z^l\to \N$ be bijections, $\Lambda_{k,N}\subset \Z^k$, $\Lambda_{l,N}\subset \Z^l$ be the inverse image of $[0, N]$ under $L_k$, $L_l$ correspondingly.  Clearly, $\cR$ is the intersection of \[
         \cR_N:= \left\{\va\in \R^k: \Ker(f_\va)\cap \Lambda_{k,N} \subset\{0\}\,, \; f_\va(\Z^k \cap \Lambda_{k,N})\cap h_{\vb}(\Z^l\cap \Lambda_{l,N}) \subset \{0\} \right\}   
       \]
       over $N\in \N$, where each $\cR_N$ is open and dense in $\R^k$.
    \end{proof}

    \begin{proof}[Proof of Proposition~\ref{Prop_MetricPerturb_Main}]
        For every fixed integer $q\geq 3$, $\varepsilon>0$ and $L$ greater than the largest area among all minimal tori in $(S^3, g_0)$, let $U\subset S^3$ and $\cO\ni g_0$ be determined by Lemma~\ref{Lem_MetricPerturb_Neighb of Torus}. By the Sard-Smale Theorem, the regular values of $\Pi: \cM^q\to \Gamma^q$ forms a residual (and hence dense) subset $\Gamma^q_{bumpy}\subset \Gamma^q$. By Lemma~\ref{lem:min_surf_cpt},~\ref{Lem_MetricPerturb_Neighb of Torus},~\ref{Lem_MetricPerturb_WhiteBumpyMetric} and Sharp's compactness theorem~\cite{Sha17}, the first two assertion of Proposition~\ref{Prop_MetricPerturb_Main} hold for every $g\in \Gamma^q_{bumpy}$. Then by the argument in~\cite{Whi17}(See also the proof of~\cite[Theorem~9]{ACS18}), the first two assertion of Proposition~\ref{Prop_MetricPerturb_Main} also hold for every $g\in \Gamma^q_{bumpy}\cap \Gamma^\infty$, which is dense in $\Gamma^\infty$.

        Moreover, for each $g'_0 \in \Gamma^q_{bumpy}\cap \Gamma^\infty$, there are only finitely many minimal spheres with area at most $L$. Therefore, by Lemma~\ref{lem:min_surf_cpt}, and Sharp's compactness theorem again, there exists a neighborhood $\mathcal{U}$ of $g'_0$ in $\Gamma^\infty$, such that $\mathcal{U} \subset \Gamma^q_{bumpy}\cap \Gamma^\infty$. In other words, $\Gamma^q_{bumpy}\cap \Gamma^\infty$ is both open and dense in $\Gamma^\infty$.
        
        To obtain the third assertion, let $g_0'\in \Gamma^q_{bumpy}\cap \Gamma^\infty$ be $\varepsilon/2$-$C^\infty$ close to $g_0$ and let 
        \[
            \{S_1, S_2, \dots, S_k\}
        \]
        be all the embedded minimal 2-spheres in $(S^3, g_0')$ with area $<L$. Let $r>0$ and $\{p_j\}_{j=1}^k$ be points in $S^3$ such that for every $j$, $p_j\in S_j$ and 
        \[
            B_r(p_j)\cap U\cup \bigcup_{i\neq j} S_i = \emptyset \,.  
        \]
        by deforming in $B_r(p_j)$, one can construct an $\R^k$-parametrized family of metrics $\{g(\va)'\}_{\va\in \R^k}\subset \Gamma^q_{bumpy}\cap \Gamma^\infty$ such that $g(0)' = g_0'$, all the embedded $g(\va)'$-minimal spheres $\{S_j(\va)\}_{j=1}^k$ are continuous deformation of $\{S_j\}_{j=1}^k$ and that 
        \[
            \va \mapsto \left(\area_{g(\va)'}(S_1(\va)), \area_{g(\va)'}(S_2(\va)), \dots, \area_{g(\va)'}(S_k(\va)) \right)
        \]
        is a local diffeomorphism near $0$ onto its image. By Lemma~\ref{Lem_MetricPerturb_Abstract} with 
        \[
            \mathbf{b} = (\area_{g_0}(T_1), \area_{g_0}(T_2), \dots, \area_{g_0}(T_l))\,,
        \] in a small neighborhood of $0$, those $\va\in \R^k$ such that $g(\va)'$ satisfies the third assertion of Proposition~\ref{Prop_MetricPerturb_Main} is residual and hence dense. We choose such a $g' = g(\va)'$ within $\varepsilon/2$-neighborhood of $g_0'$ in $\Gamma^\infty$, which satisfies all the assertion in Proposition~\ref{Prop_MetricPerturb_Main}. 
    \end{proof}

\section{Min-max results II: Convergence as currents}\label{sect:min-max_ii}

    Theorem~\ref{thm:currentsCloseInBoldF} is analogous to Proposition 4.10 and Theorem 4.11 in~\cite{MN21}. However, in our definition of homotopy class (Definition~\ref{def:homotopyClass}, two Simon-Smith families must be connected through an isotopy, which precludes the use of the well-known Almgren interpolation scheme. Therefore, the proofs in~\cite{MN21} cannot be followed verbatim to justify Theorem~\ref{thm:currentsCloseInBoldF}. Instead, we will use results from~\cite{Ket19} and~\cite{DP10} to construct appropriate isotopies.

    Let $\Phi': X \to \GS^*(M)$ be a Simon-Smith family of genus $\leq \fg_0$ in the homotopy class $\Lambda$ from Theorem~\ref{thm:currentsCloseInBoldF}. Suppose that $X$ is a cubical subcomplex of some $I(m, k)$ and $K(m) = 3^{m3^m}$. In addition, by Remark~\ref{rmk:homotopy}, we may assume that $N_P \equiv N \in \N$ within the homotopy class $\Lambda$. For $L = \bL(\Lambda)$, since $\cW_{L, \leq \mathfrak{g}_0}$ consists of finitely many varifolds associated with connected multiplicity-one minimal surfaces, we can enumerate these surfaces as 
    \[
        \Gamma_1, \Gamma_2, \cdots, \Gamma_q\,.
    \]
    
    We fix $r_0 > 0$ such that for each $\Gamma_k$, each point $p \in \Gamma_k$, and any $r \in (0, 2 r_0)$,
    \[
        B_{r}(p) \cap \Gamma_k
    \]
    is a topological disk, and moreover, $\partial B_{r}(p)$ is a mean-convex sphere with transversal intersection with $\Gamma_k$:
    \[
        \partial B_{r}(p) \pitchfork \Gamma_k\,.
    \]
    
    \begin{lem}[{\cite[2.1 (f)]{Pit81}}]\label{lem:Pitts_currents_varifolds}
        Suppose that $T, T_1, T_2, \cdots \in \cZ_2(M; \Z_2)$, $\lim_i T_i = T$ in the flat topology and $\lim_i |T_i| = V \in \cV_2(M)$ in the weak topology. Then 
        \[
            \|T\| \leq \|V\|\,.
        \]
    \end{lem}
    
    \begin{cor}\label{cor:BoldFclose_twocomps}
        For any $r > 0$, there exists $r_1 > 0$ and $\delta_1 > 0$ such that for every $V \in \cW_{L, \leq \mathfrak{g}_0}$ associated with some $\Gamma_k$, the set
        \[
            \cD_k := \{T \in \cZ_2(M; \bF; \Z_2) : \bF(|T|, V) < r_1\}
        \]
        has exactly two connected components $\cD^0_k \cup \cD^1_k$ characterized by
        \begin{align*}
            \cD^0_k &= \{T \in \cD_k : \cF(T, 0) < \delta_1\}\,,\\
            \cD^1_k &= \{T \in \cD_k : \cF(T, 0) > \delta_1\} = \{T \in \cD_k : \bF(T, [\Gamma_k]) < r\}\,.
        \end{align*}
    \end{cor}
    \begin{proof}
        By Lemma~\ref{lem:Pitts_currents_varifolds}, every sequence $\{T_i\}_i \subset \cZ_2(M; \Z_2)$ satisfying $\lim_i T_i = T$ and $\lim_i |T_i| = V$ has 
        \[
            \|T\| \leq \|V\|\,.
        \]
    
        Since $V$ is associated with a connected multiplicity-one minimal surface $\Gamma_k$ and $\partial T = 0$, by the constancy theorem, we have either $T = [\Gamma_k]$ or $T = 0$. Thus, by choosing suitable $r_1$ and $\delta_1$, the conclusion follows immediately.
    \end{proof}
    
    Fix $r$ from Theorem~\ref{thm:currentsCloseInBoldF} and then choose $r_1$ and $\delta_1$ from the previous Corollary. The crucial proposition we need is the following:
    
    \begin{prop}\label{prop:AM_current_close}
        For every $\alpha \in (0, 2r_0)$, there exists $\varepsilon \in (0, r_1)$ such that for any punctate surface $\Sigma \in \GS(M)$ and any $\delta > 0$, if for some $\Gamma_k$ and some $p \in \Gamma_k$,
        \begin{enumerate}[label=\normalfont(\roman*)]
            \item $\bF(|\Sigma|, |\Gamma_k|) < \varepsilon$,
            \item $\Sigma \cap B_\alpha(p)$ is smooth,
            \item $\mathfrak{g}(\Sigma \cap B_\alpha(p)) = 0$,
            \item $[\Sigma] \in \mathfrak{a}(B_\alpha(p); \varepsilon, \delta)$,
        \end{enumerate}
        then $\cF([\Sigma], 0) > \delta_1$.
    \end{prop}
    \begin{proof}
        Suppose for the sake of contradiction that for $\alpha > 0$, $\Gamma_k$ and $p \in \Gamma_k$, there exists a sequence of punctate surfaces $\{\Sigma_i\} \subset \GS(M)$ and a sequence of positive number $\{\delta^{(i)}\}$ such that for all large $i$,
        \begin{enumerate}[label=\normalfont(\roman*)]
            \item $\bF(|\Sigma_i|, |\Gamma_k|) < \frac{1}{i}$,
            \item $\Sigma_i \cap B_\alpha(p)$ is smooth,
            \item $\mathfrak{g}(\Sigma_i \cap B_\alpha(p)) = 0$,
            \item $[\Sigma_i] \in \mathfrak{a}(B_\alpha(p); \frac{1}{i}, \delta^{(i)})$,
        \end{enumerate}
        but $\cF([\Sigma_i], 0) \leq \delta_1$. Hence, by Lemma~\ref{lem:Pitts_currents_varifolds} and  Corollary~\ref{cor:BoldFclose_twocomps}, we have $\cF([\Sigma_i], 0) < \delta_1$ and 
        \[
            \lim_{i \to \infty} \cF([\Sigma_i], 0) = 0\,.
        \]
    
        We can choose $\bar r \in (\alpha / 2, \alpha) \subset (0, 2r_0)$ such that for each $i$, $\partial B_{\bar r}(p) \pitchfork \Sigma_i$. Note that by definition, we also have $|\Sigma_i| \in \mathfrak{a}(B_{\bar r}(p); \frac{1}{i}, \delta^{(i)})$.
    
        Following~\cite[Section~3.2]{Ket19}, let $\mathfrak{Is}_i(\Sigma_i, B_{\bar r}(p))$ denote the set of all isotopies $\phi$ of $M$ supported in $B_{\bar r}(p)$ such that for all $t \in [0, 1]$,
        \[
            \cH^2(\phi(t, \Sigma_i)) \leq \cH^2(\Sigma_i) + \delta^{(i)}\,.
        \]
        For each $i = 1, 2, 3, \dots$, we can take a {\it minimizing sequence $\phi^l(1, \Sigma_i)$} for the Problem$(\Sigma_i, \mathfrak{Is}_i(\Sigma_i, B_{\bar r}(p)))$, i.e., 
        \[
            \lim_{l \to \infty} \cH^2(\phi^l(1, \Sigma_i)) = \inf_{\phi \in \mathfrak{Is}_i(\Sigma_i, B_{\bar r}(p))} \mathbf{M}(\phi(1, \Sigma_i))\,,
        \]
        and let $V_i$ be the varifold limit of $\phi^l(1, \Sigma_i)$. 
        
        It follows from~\cite[Lemma~3.12]{Ket19} that as $i \to \infty$,
        \[
            V_i \to |\Gamma_k|
        \]
        in the varifold topology, and the convergence is smooth in compact subsets of $B_{\bar r}(p)$. By~\cite[Proposition 3.2]{DP10}, we also know that $V_i \cap B_{\bar r}(p)$ is associated with a smooth minimal surface $\Delta_i$ such that $\partial \Delta_i = \partial (\Sigma_i \cap B_{\bar r}(p))$.
    
        By~\cite[Proposition~4.7]{Ket19}, we have
        \[
            \mathfrak{g}(\Delta_i) \leq \mathfrak{g}(\Sigma_i \cap B_{\bar r}(p)) = 0\,,
        \]
        so $\Delta_i$ is a minimal disk. In particular, $\Delta_i$ separates $B_{\bar r}(p)$, and $V_i$ is associated with a cycle $T_i \in \mathcal{Z}_2(M; \Z_2)$.
        
        On one hand, note that
        \[
            T_i \llcorner (M \setminus B_{\bar r}(p)) = [\Sigma_i] \llcorner (M \setminus B_{\bar r}(p))\,.
        \]
        In particular,
        \[
            T_i \to 0
        \]
        outside $B_{2r_0}(p)$.
        
        On the other hand, in $B_{\alpha/2}(p)$, for sufficiently large $i$, $\Delta_i$ is a normal graph over $\Sigma_i$, $B_{\alpha/2}(p) \setminus \Delta_i$ has two connected components $E^1_i$ and $E^2_i$ and $B_{\alpha/2}(p) \setminus \Gamma_k$ has two connected components $E^1$ and $E^2$ satisfying
        \[
            \min\{\operatorname{vol}(E^1_i), \operatorname{vol}(E^2_i)\} \geq \min\{\operatorname{vol}(E^1), \operatorname{vol}(E^2)\} / 2 > 0\,.
        \]
        Therefore, we have 
        \[
            \liminf_i \mathcal{F}(T_i, 0) \geq \min\{\operatorname{vol}(E^1), \operatorname{vol}(E^2)\} / 2 > 0\,.
        \] Since $|T_i|=V_i \to |\Gamma_k|$, by Lemma~\ref{lem:Pitts_currents_varifolds}, we have $\lim T_i = [\Gamma_k]$. This contradicts that $T_i \to 0$ outside $B_{2r_0}(p)$.
    \end{proof}

    \begin{proof}[Proof of Theorem~\ref{thm:currentsCloseInBoldF}]
        As above, for $r > 0$, we choose $r_1$ and $\delta_1$ from Corollary~\ref{cor:BoldFclose_twocomps} and also choose $\varepsilon$ from Proposition~\ref{prop:AM_current_close} with the parameters $r_1$ and $\delta_1$.

        Let $\alpha = r_0 / (16 \cdot 8^{\bar N})$ and choose $\varepsilon$ from Proposition~\ref{prop:AM_current_close}. We set $r_3 := \min(r, r_1, \varepsilon)$.

        For each $\Gamma_k$, choose $p_k \in \Gamma_k$ and for each $i = 1, \cdots, \bar N$, we set $r_i = \frac{7r_0}{8^i}$ and $s_i = \frac{5r_0}{8^i}$, and we obtain a $\bar N$-admissible collection of annuli, $\{A^k_i := A(p_k; s_i, r_i)\}$. We also choose $\{p^1_k, \cdots, p^{\bar N}_k\} \subset \Gamma_k$ such that for each $i = 1, \cdots, \bar N$,
        \[
            \dist(p_k, p^i_k) = \frac{3r_0}{4 \cdot 8^{i - 1}}\,.
        \]
        and thus, $B(p^i_k, \alpha) \subset A^k_i$.
        
        By Theorem~\ref{thm:minMax}, there exists a pulled-tight sequence $\{\Phi_i\}$ in $\Lambda$ and some $\eta_3 \in (0, r_3)$ such that for each $i$ and each $x \in X$,
        \[
            \cH^2(\Phi_i(x)) \geq L - \eta_3 \implies |\Phi_i(x)| \in \bB^\bF_{r_3}(\cW_{L, \leq \mathfrak{g}_0})\,.
        \]
        
        It follows from Corollary~\ref{cor:BoldFclose_twocomps} that the compact set
        \[
            Y_i := \{x \in X : \cH^2(\Phi_i(x)) \geq L - \eta_3\}
        \]
        can be decomposed as two disjoint compact sets $Y^0_i$ and $Y^1_i$:
        \begin{align*}
            Y^0_i &= \{x \in Y_i : \cF([\Phi_i(x)], 0) < \delta_1\}\\
            Y^1_i &= \{x \in Y_i : \cF([\Phi_i(x)], 0) > \delta_1\} = \{x \in Y_i : \bF([\Phi_i(x)], [\cW_{L, \leq \mathfrak{g}_0}]) < r\}\,.
        \end{align*}
        
        Now for each $x \in Y^0_i$, we can choose $k$ such that $\bF(|\Phi_i(x)|, |\Gamma_k|) < r_3 \leq \varepsilon$. Since $\mathfrak{g}(\Phi_i(x)) \leq \mathfrak{g_0}$ and we can choose a set of at most $N$ points $P_{\Phi_i}(x)$ such that $\Phi_i(x) \setminus P_{\Phi_i}(x)$ is a smooth surface, there are at least $K(m)$ points $\{p^{j_l}_{k}\}^{K(m)}_{j_l}$ from $\{p^j_k\}^{\bar N}_{j = 1}$ where for each $l$,
        \begin{enumerate}
            \item $\Phi_i(x) \cap B_{\alpha}(p^{j_l}_k)$ is smooth,
            \item $\mathfrak{g}(\Phi_i(x) \cap B_{\alpha}(p^{j_l}_k)) = 0$,
        \end{enumerate}
        hence, $\Sigma$ admits an $(\varepsilon, \frac{\eta_3}{1000 K(m)})$ deformation in $B_{\alpha}(p^{j_l}_k)$.

        By Proposition~\ref{prop:an_deform}, we can obtain $\Phi^*_i$ for each $\Phi_i$. Moreover, for sufficiently large $i$, since $\eta_3 < \varepsilon$, we have the following properties.

        For every $x \in Y^0_i$, 
        \begin{align*}
            \cH^2(\Phi^*_i(x)) &< \cH^2(\Phi_i(x)) - \varepsilon + (3^m - 1) \frac{\eta_3}{1000 K(m)}\\
                &< \cH^2(\Phi_i(x)) - \eta_3 / 2 \\
                &< L - \eta_3 / 2\,.
        \end{align*}

        For every $x \in X \setminus Y_i$,
        \begin{align*}
            \cH^2(\Phi^*_i(x)) &< \cH^2(\Phi_i(x)) + 3^m \frac{\varepsilon}{1000 K(m)}\\
                &< L - \eta_3 + 3^m \frac{\eta_3}{1000 K(m)}\\
                &< L - \eta_3 / 2\,.
        \end{align*}

        Since $Y^0_i$ and $Y^1_i$ are disjoint compact set, we also have $\Phi^*_i(x) = \Phi_i(x)$ holds for every $x \in Y^1_i$.
        
        Therefore, for every $x \in X$ with $\cH^2(\Phi^*_i(x)) \geq L - \eta_3 / 2$, then $x \in Y^1_i$ and thus,
        \[
            [\Phi_i(x)] \in \bB^\bF_{r_3}([\cW_{L, \leq \mathfrak{g}_0}]) \subset \bB^\bF_r([\cW_{L, \leq \mathfrak{g}_0}])\,,
        \]
        which completes the proof with $\eta = \eta_3 / 2$ and $\Phi = \Phi^*_i$.
    \end{proof}

\section{Min-max results III: Interpolation}\label{sect:min-max_iii}
    The goal of this section is to prove Theorem~\ref{thm:mapping_cylinder}. Let $(M, g)$ be a $3$-dimensional closed Riemannian manifold.

    The key is the following deformation proposition. Intuitively, for a Simon-Smith family of genus $\leq 1$ very close to spheres, we show that we can pinch all the small necks and shrink all small connected components  in a continuous way to obtain a Simon-Smith family of genus $0$. 
    
    \begin{prop}\label{Prop_Technical Deformation}
        Let $X$ be a cubical subcomplex of $I(m, k)$, $\mathscr{S}$ be a compact subset of embeddings of $\mathbb{S}^2$ into $(M, g)$ (endowed with smooth topology), $q \in \N$ and $\varepsilon \in (0,1)$. Then there exists $\delta = \delta(\varepsilon, q, m, \mathscr{S})\in (0, 1)$ with the following property.

        If $\Phi: X\to \GS(M)$ is a Simon-Smith family of genus $\leq 1$, $N_P(\Phi) < q$ and $S: X\to \mathscr{S}$ is a map such that for every $x\in X$, 
        \[
            \bF([\Phi(x)], [S(x)]) \leq \delta
        \]
        then there exists a Simon-Smith family $H: [0, 1] \times X \to \GS(M)$ of genus $\leq 1$ such that:
        \begin{enumerate}[label=\normalfont(\arabic*)]
            \item For every $x\in X$, $H(0, x) = \Phi(x)$, $\mathfrak{g}(H(1, x)) = 0$.
            \item\label{item:area_control} For every $x\in X$ and $0\leq t \leq t'\leq 1$,
                \begin{align*}
                    \cH^2(H(t, x)) \leq \cH^2(\Phi(x)) + \varepsilon\,,\quad \mathfrak{g}(H(t', x)) \leq \mathfrak{g}(H(t, x)) \,.
                \end{align*}
            \item For each $x \in X$, the Simon-Smith family $\{H(x, t)\}_{t \in [0, 1]}$ is a pinch-off process.
        \end{enumerate}
    \end{prop}

    We will prove the proposition using the following technical lemmas.
    
\subsection{Nontrivial loop in a small ball}

    In this subsection, we show that if a (singular) torus is sufficiently close to a sphere, we can find a loop contained within a small ball. In the next subsection, we will explain how to pinch the torus along the boundary of the small ball and then shrink all connected components inside to a point. The small size of the ball helps control the area increase, as required in Proposition~\ref{Prop_Technical Deformation}~\ref{item:area_control}. This existence result is closely related to the well-known systolic inequality.

    \begin{thm}[Loewner torus inequality~{\cite[Section~3]{Pu52}}, Gromov's homological systolic inequality~{\cite[Section~2.C]{Gro92}}]\label{thm:systolic_ineq} 
        Given a smooth orientable closed surface $\Gamma$ of positive genus, there exists a universal constant $C > 0$ such that
        \[
            \operatorname{sys} H_1(\Gamma; \Z_2) := \inf \left\{\mathcal{H}^1(c) : [c] \neq 0 \in H_1(\Gamma; \Z_2)\right\} \leq C \sqrt{\mathcal{H}^2(\Gamma)}\,.
        \]
        In particular, there exists a simple closed geodesic $c \in \Gamma$ such that $[c] \neq 0 \in H_1(\Gamma; \Z_2)$ and $\mathcal{H}^1(c) \leq C \sqrt{\mathcal{H}^2(\Gamma)}$.
    \end{thm}

    Let us first prove the existence of a nontrivial small loop for embedded tori sufficiently close to an embedded sphere.

    \begin{lem}[Existence of small nontrivial loop: Smooth surfaces] \label{lem:small_loop} 
        In a closed Riemannian $3$-manifold $(M, g)$, let $S$ be an embedded sphere. For any $\varepsilon\in (0,1)$, there exists $\delta = \delta_0(S, \varepsilon) \in (0, \varepsilon)$ with the following property. 

        In $(M, g)$, for any embedded torus $T$ with $\mathbf{F}([S], [T]) \leq \delta$, there exists $p\in M$ and a simple loop $\gamma\subset \Sigma\cap B(p, \varepsilon)$ such that the closed surface obtained by neck-pinch along $\gamma$ of $\Sigma$ has genus $0$. 
    \end{lem}
    \begin{proof} 
        Without loss of generality, we may assume that $S$ is a unit round $2$-sphere, and it has a tubular neighborhood $N_{\zeta} (S) \cong (S \times [-\zeta, \zeta], g_S + dt^2)$ with $\zeta < \varepsilon / 10$. We denote by 
        \[
            \pi: N_{\zeta}(S) \to S
        \] the corresponding closest-point projection.
        
        The general case follows by changing the metric $g$ and taking a smaller $\delta$.

        \medskip
        \paragraph*{\textbf{Step 0. Set up}} 
        
        Fix a triangulation $\mathcal{T}$ of $S$ such that:
        \begin{enumerate}[label = \normalfont(0.\roman*)]
            \item\label{item:Delta_small_area} For any $\Delta \in \mathcal{T}^{(2)}$, $\mathcal{H}^2(\Delta) \in \left(0, \frac{\varepsilon^2}{100}\right)$.
            \item\label{item:Delta_small_length} For any $e \in \mathcal{T}^{(1)}$, $\mathcal{H}^1(e) \in \left(0, \frac{\varepsilon^2}{100}\right)$.
        \end{enumerate} 
        Let $m$ denote the total number of cells in $\mathcal{T}$. Set $\tau := \frac{1}{1000m^4}$. 

        For convenience, in the following, for any $v \in \mathcal{T}^{(0)}$, $e \in \mathcal{T}^{(1)}$, $\Delta \in \mathcal{T}^{(2)}$ and $\varphi \in \operatorname{Isom}(S) = SO(3)$, we define:
        \begin{itemize}
            \item the vertical line $\mathrm{I}_{v, \varphi} := \pi^{-1}(\varphi(v))$;
            \item the rectangle $\mathrm{R}_{e,\varphi} := \pi^{-1}(\varphi(e))$;
            \item the cylinder $\mathrm{Cyl}_{\Delta, \varphi} := \pi^{-1}(\partial \varphi(\Delta))$;
            \item the prism $\mathrm{Pri}_{\Delta, \varphi} := \pi^{-1}(\varphi(\Delta))$.
        \end{itemize} 

        \medskip
        \paragraph*{\bf Step 1. Existence of a good triangulation} 

        We show that there exists $\delta' \in (0, 1)$ with the following property: 
        
        For any closed embedded  surface $T \subset N_{\zeta}(S)$ and $\mathbf{F}([S], [T]) < \delta'$, there exist $\varphi \in \operatorname{Isom}(S) = SO(3)$ such that: 
        \begin{enumerate}[label=\normalfont(\arabic*)]
            \item\label{item:v_transverse} For any $v \in \mathcal{T}^{(0)}$, $\mathrm{I}_{v, \varphi} \pitchfork T$ is a single point in $T$. 
            \item\label{item:Delta_transverse} For any $\Delta \in \mathcal{T}^{(2)}$, $\mathrm{Cyl}_{\Delta, \varphi} \pitchfork T = \bigcup^{n_\Delta}_{k = 1} c_{\Delta, k}$ is a finite union of disjoint simple loops in $T$ satisfying
            \[
                \frac{\mathcal{H}^1(\mathrm{Cyl}_{\Delta, \varphi}\cap T)}{\mathcal{H}^1(\partial \Delta)} \in (1 - \tau, 1 + \tau)\,,
            \]
            and
            \[
                \frac{\mathcal{H}^2(\mathrm{Pri}_{\Delta, \varphi})\cap T)}{\mathcal{H}^2(\Delta)} \in (1 - \tau, 1 + \tau)\,,
            \]
        \end{enumerate} 
        Moreover, for any $\Delta \in \mathcal{T}^{(2)}$, we can arrange $\{c_{\Delta, k}\}^{n_\Delta}_{k = 1} $ in such an order that: 
        \begin{enumerate}[label=\normalfont(\arabic*)]
            \setcounter{enumi}{2}
            \item\label{item:long_loop} $c_{\Delta, 1}$ satisfies $[c_{\Delta, 1}] \neq 0 \in H_1(\mathrm{Cyl}_{\Delta, \varphi}; \Z_2)$ and
            \[
                \frac{\mathcal{H}^1(c_{\Delta, 1})}{\mathcal{H}^1(\partial \Delta)} \in (1 - \tau, 1 + \tau)\,;
            \] we call it the \emph{long loop} for $\Delta$. 
            \item\label{item:short_loop} Every $c_{\Delta, k}$ with $k \in \{2, \dots, n_\Delta\}$ satisfies $[c_{\Delta, k}] = 0 \in H_1(\mathrm{Cyl}_{\Delta, \varphi}; \Z_2)$, $c_{\Delta, k} \cap \bigcup_{v \in \mathcal{T}^{(0)}} \mathrm{I}_{v, \varphi}) = \emptyset$, and
            \[
                \frac{\sum^{n_\Delta}_{k = 2} \mathcal{H}^1(c_{\Delta, k})}{\mathcal{H}^1(\partial \Delta)} \in (0, \tau)\,;
            \] we call them the \emph{short loops} for $\Delta$. 
        \end{enumerate} 

        Let $\mu$ be the Haar measure on $\operatorname{Isom}(S) = SO(3)$.

        By the isoperimetric inequality, we choose $\delta$ small enough, depending on $S$ and $\zeta$, such that $\mathbf{F}([S], [T]) \leq \delta$ and $T \subset N_\zeta(S)$ imply $T$ and $S$ are homologous to each other in $N_{\zeta}(S)$ with $\Z_2$ coefficients.
        
        This implies that for any $v \in \mathcal{T}^{(0)}$, $e \in \mathcal{T}^{(1)}$, $\Delta \in \mathcal{T}^{(2)}$, and $\varphi \in \operatorname{Isom}(S)$, if $\mathrm{I}_{v, \varphi}$ intersects $T$ transversely, then
        \begin{equation*}
            [\mathrm{I}_{v, \varphi} \cap T] = [\varphi(v)] \neq 0 \in H_0(\mathrm{I}_{v, \varphi}; \Z_2)\,;
        \end{equation*}
        if $\mathrm{R}_{e, \varphi}$ intersects $T$ transversely, then
        \begin{equation*}
            [\mathrm{R}_{e, \varphi} \cap T] = [\varphi(e)] \neq 0 \in H_1(\mathrm{R}_{e, \varphi}; \Z_2)\,;
        \end{equation*}
        if $\mathrm{Cyl}_{\Delta, \varphi}$ intersects $T$ transversely, then
        \begin{equation}\label{eqn:transverse_Delta}
            [\mathrm{Cyl}_{\Delta, \varphi} \cap T] = [\varphi(\partial \Delta)] \neq 0 \in H_1(\mathrm{Cyl}_{\Delta, \varphi}; \Z_2)\,,
        \end{equation}
        and
        \begin{equation*}
            [\mathrm{Pri}_{\Delta, \varphi} \cap T] = [\varphi(\Delta)] \neq 0 \in H_2(\mathrm{Pri}_{\Delta, \varphi}; \Z_2)\,,
        \end{equation*}
        By Sard's theorem, the transversality holds for a full measure subset $E$ of $SO(3)$. Note that these give lower bounds on $\cH^0(\mathrm{I}_{v, \varphi} \cap T)$, $\cH^1(\mathrm{R}_{e, \varphi} \cap T)$ and $\cH^2(\mathrm{Cyl}_{\Delta, \varphi} \cap T)$ for all $\varphi \in E$.

        Therefore, by the coarea formula and possibly taking $\delta$ smaller, there exists a subset $W_T \subset E$ with $\mu(W_T) \geq 1/2$ such that (1) and (2) hold. Meanwhile, (3) and (4) follow from \eqref{eqn:transverse_Delta}.
        
        \medskip
        \paragraph*{\textbf{Step 2. Special case}} 

        By Step 1, for every torus $T \subset N_\zeta(S)$ with $\bF([T], [S]) < \delta'$, we can choose a $\varphi \in W_T$. Then for every $\Delta \in \mathcal{T}^{(2)}$, $v \in \Delta^{(0)}$,
        \[
            \mathrm{I}_{v, \varphi} \cap T = \mathrm{I}_{v, \varphi} \cap c_{\Delta, 1}\,.
        \]
        Therefore, every short loop $c_{\Delta, k}$ with $k \geq 2$ lies in the interior of $\mathrm{R}_{e, \varphi}$ for some $e \in \Delta^{(1)}$. In particular, every short loop for $\Delta \in \mathcal{T}^{(2)}$ coincides with exactly one short loop in another $\Delta' \in \mathcal{T}^{(2)}$. 
        
        First, we cut $T$ along every loop $\bigcup_{\Delta \in \mathcal{T}^{(2)}}\{c_{\Delta, k}\}^{n_\Delta}_{k = 2}$. Then for any $\Delta \in \mathcal{T}^{(2)}$, 
        \[
            T_{\Delta} := T \cap \mathrm{Pri}_{\Delta, \varphi}
        \] 
        is a surface with boundary; the boundary consists of smooth short loops $\{c_{\Delta, k}\}^{n_\Delta}_{k = 1}$ and a piecewise smooth long loop $c_{\Delta, 1}$. We glue disks $\{D_{\Delta, k}\}^{n_\Delta}_{k = 2}$ of arbitrary small areas to its boundary and obtain a orientable surface $T'_\Delta$ with boundary. If $\fg(T'_\Delta) = 1$, then we find a nontrivial simple loop $\gamma$ of $T$ contained in $\mathrm{Pri}_{\Delta, \varphi} \subset B(p, \varepsilon / 2)$ for some $p \in M$, so we are done. Otherwise, each $T'_\Delta$ is a union of multiple spheres and a disk whose boundary is $c_{\Delta, 1}$. So by Step 1~\ref{item:v_transverse} and~\ref{item:long_loop}, the whole surface $T'$ is a genus $0$ surface. This implies that at least one short loop $\gamma$ is a non-trivial simple loop in $T$, and by Step 2~\ref{item:short_loop}, it is contained in a small ball $B(p, \varepsilon / 2)$.  

        \medskip
        \paragraph*{\bf Step 3. General case} 

        By choosing $\delta$ sufficiently small, for a general torus $T$ with $\bF([T], [S]) < \delta$, by the coarea formula, we can find $d \in (\zeta / 2, \zeta)$ such that we can cut $T \pitchfork \partial N_d(S) = \bigcup^{n_T}_{j = 1} c_j$ and attach disks $\{D_j\}^{n_T}_{j = 1}$ to obtain two disjoint orientable closed surfaces $T^1$ and $T^2$ satisfying:
        \begin{enumerate}[label = (3.\roman*)]
            \item $T^1 \subset N_{\zeta}(S)$.
            \item $\mathbf{F}([T^1], [S]) < \delta'$.
            \item $\mathcal{H}^2(T^2) < \varepsilon^2 / 100C$, where $C$ is the universal constant from Theorem~\ref{thm:systolic_ineq}.
            \item $\sum^{n_T}_{j = 1} \mathcal{H}^1(c_j) < \varepsilon / 100$.
        \end{enumerate} 

        If $\fg(T^1) = 1$, we can apply Step 2 to $T^1$ to obtain a nontrivial simple loop $\gamma^1 \in T^1$ and a small ball $B(p, \varepsilon / 2) \supset \gamma^1$.

        If $\fg(T^2) = 1$, then by Theorem~\ref{thm:systolic_ineq}, we can find a nontrivial simple loop $\gamma^2 \in T^2$ of length $\leq \varepsilon / 10$. Therefore, there exists a geodesic ball $B(p, \varepsilon / 2) \supset \gamma^2$.

        In either case, $\gamma^1$ or $\gamma^2$ might intersect with some $D_j$'s. Since $D_j$'s are contractible, we can push $\gamma^1$ or $\gamma^2$ to the boundary, and obtain a nontrivial simple loop $\gamma \subset T$. In addition, $\gamma \subset B(p, \varepsilon)$ by the length estimates of $\cH^1(c_j)$.

        Finally, if $\fg(T^1) = \fg(T^2) = 0$, then one of $c_j$ is a nontrivial simple loop in $T$ and the conclusion follows.
    \end{proof}    

    To prove the existence of a nontrivial loop for punctate surfaces, we first need to approximate them by smooth surfaces as follows.

    \begin{lem}\label{lem:approx_PS}
        In an orientable closed Riemannian $3$-manifold $(M, g)$, suppose that $S$ is an embedded smooth orientable closed surface. For any $q\in \mathbb{N}^+$ and $\varepsilon \in \left(0, \operatorname{injrad}(M, \mathbf{g})/2\right)$, there exists a constant $\delta = \delta_1(S, q, \varepsilon) \in (0, \varepsilon)$ with the following property. 

        For any punctate surface $\Sigma \in \mathcal{S}(M)$ with a punctate set $P$ such that
        \[
            \# P \leq q\,, \quad \mathbf{F}([S], [\Sigma]) < \delta\,,
        \]
        there exist an integer $\tilde q \leq q$, a real number $r \in (0, \varepsilon)$ and geodesic balls $\{B(p_i, r)\}^{\tilde q}_{i = 1}$ in $(M, g)$ such that:
        \begin{enumerate}[label=\normalfont{(\arabic*)}]
            \item\label{item:approx_PS_disjoint} The closed balls $\{\overline{B(p_i, r)}\}^{\tilde q}_{i = 1}$ are pairwise disjoint geodesic balls.
            \item\label{item:approx_PS_genus} $\Sigma \setminus \bigcup^{\tilde q}_{i = 1} B(p_i, r)$ is an embedded smooth surface with boundary and genus at most $\mathfrak{g}(\Sigma)$.
            \item\label{item:approx_PS_approx} By replacing $\Sigma \cap \bigcup^{\tilde q}_{i = 1} B(p_i, r)$ by finitely many disks $\{D_j\}$, we obtain an embedded smooth orientable closed surface $T$ satisfying $\bF([T], [S]) < \varepsilon$ and $\sum^n_{j = 1} \mathcal{H}^1(\partial D_j) < \varepsilon$.
        \end{enumerate}
    \end{lem} 
    \begin{proof}
        Since $S$ is a smooth closed surface, there exists a positive constant $C_S > 1$, such that for any $p \in M$ and $r \in (0, \operatorname{injrad}(M, g) / 2)$,
        \[
            \mathcal{H}^2(\overline{B(p, r)} \cap S) \leq C_S r^2\,, \quad \mathcal{H}^3(\overline{B(p, r)}) \leq C_S r^3\,.
        \]

        Let $\eta$ be a sufficiently small positive constant determined later. 
        
        By the definition of $\mathbf{F}$-metric, there exists $\delta \in (0, \varepsilon / 2)$ such that for any $\Sigma \in \GS(M)$ with $\mathbf{F}([S], [\Sigma]) \leq \delta$ and any $p \in M$,
        \[
            \mathcal{H}^2(\overline{B(p, 2\eta)} \cap \Sigma) \leq 8 C_S \eta^2\,.
        \] 
        Suppose that the punctate set $P$ has $\# P \leq q$.

        By Sard's theorem and the coarea formula, we can choose a constant $C_q > 0$, only depending on $q$, and a radius $\tilde r \in \left(\eta_3 / 3^q, 2\eta_3 / 3^q\right)$ such that, for any $p \in P$ and any $l \in \{1, 3, 3^2, \dots, 3^q\}$, $\partial B(p, l\tilde r) \pitchfork S$ is a finite union of disjoint simple loops and
        \[
            \mathcal{H}^1(\partial B(p, l\tilde r) \cap \Sigma) \leq C_q \mathcal{H}^2(\overline{B(p, 2\eta)} \cap \Sigma) / \eta < 8C_qC_S \eta\,.
        \] 
        Then one can find a subset $P' \subset P$ and a radius $r \in \{\tilde r, 3\tilde r, 3^2 \tilde{r}, \dots, 3^q \tilde{r}\} \subset (0, \eta)$ such that $\Sigma \setminus \bigcup_{p' \in P'} B(p', r)$ is a smooth surface with boundary satisfying
        \[
            \sum_{p' \in P'} \mathcal{H}^1(\partial B(p', r) \cap \Sigma) < 8qC_qC_S \eta\,,
        \]
        $P \subset \bigcup_{p' \in P'} B(p', r)$ and $\{\overline{B(p', r)}\}_{p' \in P'}$ are disjoint from each other. Clearly, we also have $\mathfrak{g}(\Gamma) \leq \mathfrak{g}(S)$. These confirms~\ref{item:approx_PS_disjoint} and~\ref{item:approx_PS_genus}. %

        By choosing $\eta$ sufficiently small, ~\ref{item:approx_PS_approx} follows from the isoperimetric inequality.
    \end{proof} 

    \begin{lem}[Existence of small nontrivial loops: Punctate surfaces] \label{Lem_Existence of small  nontrivial loop}
        Let $\mathscr{S}$ be a compact subset of embeddings of $\mathbb{S}^2$ into $(M, g)$ (endowed with smooth topology), $q \in \mathbb{N}^+$, $\varepsilon\in (0,1)$. Then there exists $\delta = \delta_2(\mathscr{S}, q, \varepsilon)\in (0, 1)$ with the following property.
        
        Suppose that $\Sigma\in \GS(M)$ has $\mathfrak{g}(\Sigma) = 1$, its punctate set $P$ has $\# P \leq q$, and there exists a $S\in \mathscr{S}$ such that $\bF([\Sigma], [S])\leq \delta$. Then there exists $p\in S^3$ and an embedded loop $\gamma\subset \Sigma\cap B(p, \varepsilon)$ such that the surface obtained by neckpinch along $\gamma$ of $\Sigma$ has genus $0$.
    \end{lem}
    \begin{proof}
        For every $S \in \mathscr{S}$, let $\delta_0(S) := \delta_0(S, \varepsilon / 2)$ be as defined in Lemma~\ref{lem:small_loop} and $\delta_1(S) := \delta_1\left(S_i, q, \delta_0(S)/2\right)$ as defined in Lemma~\ref{lem:approx_PS}. The compactness of $\mathscr{S}$ implies that there exists finite subset $\{S_i\}$ such that for every $S' \in \mathscr{S}$, there exists $S_{i_0}$ with 
        \[ 
            \mathbf{F}([S'], [S_{i_0}]) \leq \delta_1(S_{i_0}) / 2\,.
        \]
        We define
        \[
            \delta := \min_{i} \delta_1\left(S_i\right).
        \]
        
        For every $\Sigma\in \GS(M)$ with $\mathfrak{g}(\Sigma) = 1$, if its punctate set $P$ has $\# P \leq q$, and there exists a $S\in \mathscr{S}$ such that $\bF([\Sigma], [S])\leq \delta$, then there exists $S_{i_0}$ such that
        \[
            \mathbf{F}([\Sigma], [S_{i_0}]) \leq \delta_1(S_{i_0})\,.
        \]
        We can apply Lemma~\ref{lem:approx_PS} to approximate $\Sigma$ by an embedded closed surface $T$ where 
        \[
            T \setminus \bigcup B(p_i, r) = \Sigma \setminus \bigcup B(p_i, r)
        \]
        for a disjoint union of balls and $r \in (0, \varepsilon / 2)$ and $\bF([T], [S_{i_0}]) < \delta_0(S_{i_0})$.
        
        Since $\fg(\Sigma) = 1$, $\fg(T) = 0$ or $1$: 
        \begin{itemize}
            \item If $\fg(T) = 0$, then there exists a nontrivial simple loop of $\Sigma$ lying inside some $B(p_i, r)$.
            \item Otherwise, by Lemma~\ref{lem:small_loop}, there exists a nontrivial simple loop $\gamma \subset T$ contained in a ball $B(p, \varepsilon / 2)$. As before, $\gamma$ might not lie inside $\Sigma$ but since each disk in $T \cap \bigcup B(p_i, r)$, we can push $\gamma$ to their boundary. The resulting simple loop $\gamma' \subset \Sigma$ and is inside the ball $B(p, \varepsilon)$. 
        \end{itemize} 
        
        These complete the proof.
    \end{proof}

\subsection{Local deformations}
    In the previous subsection, we located nontrivial loops for singular tori within small balls, provided that they are sufficiently close to a sphere. One key reason we do not simply pinch along the loops, which resemble mean curvature flow, is the need to construct the pinching in a continuous manner along a Simon-Smith family. The singular set presents a technical obstruction to this straightforward approach. Therefore, we employ a two-step modification process, Pinching and Shrinking, as detailed in the following lemmas.
    
    \begin{lem}[Local Deformation A: Pinching] \label{Lem_Loc Deform A}
        There exists a constant $C_g$ depending only on $(M, g)$ with the following property.
        
        Let $\Lambda>0$, $0<r<\injrad(M, g)/4$, $p\in M$. Suppose that $\Phi: X \to \GS^*(M)$ is a Simon-Smith family and $x_0 \in X$ so that $\Phi(x_0) \in \GS(M)$ satisfies
        \[
            \cH^2(\Phi(x_0) \cap A(p; r, 2r)) \leq \Lambda r^2\,.
        \] Then there exist $s\in (r, 2r)$, $\zeta\in (0, \min\{\Lambda r, s-r, 2r-s\}/5)$, an integer $K\geq 0$, a neighborhood $O_{x_0}$ of $x_0$ in $X$, and a Simon-Smith family 
        \[
            H: [0, 1] \times O_{x_0}\to \GS(M) \,,   
        \]
        with the following properties.
        \begin{enumerate}[label=\normalfont(\arabic*)]
            \item For every $y \in O_{x_0}$, $\Sigma_y := \Phi(y) \cap A(p; s-5\zeta, s+5\zeta)$ is a surface, varying smoothly in $y$.
            \item For every $s'\in [s-4\zeta, s+4\zeta]$ and every $y \in O_{x_0}$, we have $\Phi(y)$ intersects $\partial B_{s'}(p)$ transversally, and 
            \[
                \Phi(y)\cap \partial B_{s'}(p) =: \bigsqcup_{i=1}^K \gamma_i(s', y)\,, \quad \sum_{i=1}^K \cH^1(\gamma_i(s', y)) \leq 2\Lambda r\,.
            \]
        \end{enumerate}
        
        If $K=0$, then $H(t, y) \equiv \Phi(y)$ for every $y \in O_{x_0}$ and $t\in [0, 1]$; and if $K\geq 1$, then
        \begin{enumerate}[label=\normalfont(\arabic*)]
            \setcounter{enumi}{2}
            \item For every $y \in O_{x_0}$, $H(0, y) = \Phi(y)$.
            \item There exists $t_0 = 0<t_1<t_2<\dots<t_K<1 = t_{K+1}$ such that for every $0\leq i\leq K$, $t\in (t_i, t_{i+1})$ and $y \in O_{x_0}$, 
            \[
                \Sigma_{t, y} := H(t, y) \cap A(p, s-5\zeta, s+5\zeta) 
            \] 
            is a surface, which is diffeomorphic to $\Sigma_y \setminus \bigsqcup_{1\leq j\leq i} \gamma_j(s, y)$ attached with $2i$ disks. Moreover, $H(1, y)\cap A(p, s-5\zeta, s+5\zeta)$ is also a surface and
            \[
                H(1, y)\cap A(p, s-\zeta, s+\zeta) = \emptyset \,.
            \]
            \item\label{item:surgery_process} For every $i \in \{1, 2, \cdots, K\}$ and $y \in O_{x_0}$, near $t_i$, $H(t, y)$ is a surgery process via pinching the cylinder $\bigcup_{s' \in (s - 2\zeta, s + 2\zeta)} \gamma_i(s', y)$.
            \item For every $y \in O_{x_0}$ and $t \in [0, 1]$, we have
            \begin{align*}
                H(t, y)\setminus A(p; s-3\zeta, s+3\zeta) & = \Phi(y)\setminus A(p; s-3\zeta, s+3\zeta) \,,\\
                \cH^2(H(t, y)) & \leq \cH^2(\Phi(y)) + C_g\Lambda^2 r^2 \,.
            \end{align*}
        \end{enumerate}
    \end{lem}
    \begin{rmk}
        For every $y \in O_{x_0}$, by~\ref{item:surgery_process}, we see that $\{H(t, y)\}_{t \in [0, 1]}$ is a pinch-off process.
    \end{rmk}
    \begin{proof}
        Since $\cH^2(\Phi(x_0)) \leq \Lambda r^2$ and $\Phi(x_0) \setminus P$ is a smooth surface for some finite set $P$, by the slicing theorem and Sard's lemma, there exists $s \in (r, 2r)$ such that $\Phi(x_0) \pitchfork \partial B_r(p)$ is a finite union of loops $\{\gamma_i(s, x_0)\}^K_{i = 1}$ with 
        \[
            \sum_{i=1}^K \cH^1(\gamma_i(s, x_0)) \leq \Lambda r\,.
        \]
        
        By Definition~\ref{def:Simon_Smith_family}~\ref{item:SimonSmithFamilyLocalSmooth}, we can choose $\zeta$ and $O_{x_0}$ such that (1) and (2) hold. Moreover, one can find a continuous family of  $\{\varphi_y\}_{y \in O_{x_0}} \subset \operatorname{Diff}^\infty(M)$ such that 
        \[
            \varphi_y(\Phi(x_0)) \cap A(p; s - 5\zeta, s + 5\zeta) = \Sigma_y \,.
        \]
        Therefore, by possibly choosing a smaller $O_{x_0}$, it suffices to find a continuous map $\{H(t, x_0)\}_{t \in [0, 1]}$ for $\Sigma$ supported in $A(p; s - 4\zeta, s + 4\zeta)$ satisfying (3) and
        \begin{enumerate}
            \item[(4)'] There exists $t_0 = 0<t_1<t_2<\dots<t_K<1 = t_{K+1}$ such that for every $0\leq i\leq K$, $t\in (t_i, t_{i+1})$ and $x_0 \in O_{x_0}$, 
            \[
                \Sigma_{x_0, t} := H(t, x_0) \cap A(p, s-5\zeta, s+5\zeta) 
            \] 
            is a surface, which is diffeomorphic to $\Sigma_{x_0} \setminus \bigsqcup_{1\leq j\leq i} \gamma_j(s, {x_0})$ attached with $2i$ disks. Moreover, $H(1, x_0)\cap A(p, s-5\zeta, s+5\zeta)$ is also a surface and
            \[
                H(1, x_0)\cap A(p, s- 1.1\zeta, s+ 1.1\zeta) = \emptyset \,.
            \]
            \item[(5)'] For every $i \in \{1, 2, \cdots, K\}$, near $t_i$, $H(t, x_0)$ is a surgery process via pinching the cylinder $\bigcup_{s' \in (s - 1.9\zeta, s + 1.9\zeta)} \gamma_i(s', x_0)$;
            \item[(6)'] For every $t \in [0, 1]$, we have
            \begin{align*}
                H(t, x_0)\setminus A(p; s-3.1\zeta, s+3.1\zeta) & = \Phi({x_0})\setminus A(p; s-3.1\zeta, s+3.1\zeta) \,; \\
                \cH^2(H(t, x_0)) & \leq \cH^2(\Phi({x_0})) + C_g \Lambda^2 r^2 / 1.1 \,.
            \end{align*}
        \end{enumerate}
        Indeed, the general $H: [0, 1] \times O_x\to \GS(M)$ can be defined by 
        \[
            H(t, y) := \left(\varphi_y(H(t, x_0)) \cap A(p; s - 5\zeta, s + 5\zeta)\right) \cup \left(\Phi(y) \setminus A(p; s - 5\zeta, s + 5\zeta)\right)\,,
        \]
        and one can verify that $H$ satisfies (3) - (6) following the properties of $H(\cdot, x_0)$.

        Note that in $(M, g)$, for every $r \in (0, \injrad(M, g) / 4)$ and $p \in S$, the annulus $A(p; r, 2r)$ is diffeomorphic to $\mathbb{A}(\mathbf{0}; r, 2r) \subset \R^3$ with a uniform bi-Lipschitz constant $C^{(1)}_g$ depending only on $g$. Without loss of generality, we may assume that $A(p; r, 2r)$ is $\mathbb{A}(\mathbf{0}; r, 2r)$.

        By the isoperimetric inequality in standard spheres, there exists a dimensional constant $C^{(2)} > 0$ such that for any loop $\gamma$ in $\gamma$ in a sphere $\partial \mathbb{B}_s$ for any $s > 0$, $\gamma$ bounds a disk $D$ in $\partial \mathbb{B}_s$ with
        \[
            \cH^2(D) \leq C^{(2)} \cH^1(\gamma)^2\,.
        \]
        Therefore, together with Jordan curve theorem, given a finite set of disjoint loops $\{\gamma_i\}$ in $\partial \mathbb{B}_s$, each of them bounds a disk $D_i$ in $\partial \mathbb{B}_s$ such that 
        \[
            \sum_i \cH^2(D_i) \leq C^{(2)} \sum_i \cH^1(\gamma_i)^2 \leq C^{(2)} (\sum_i \cH^1(\gamma_i))^2\,.
        \]
        In addition, for every pair $i$ and $j$, we have a trichotomy: $D_i \cap D_j = \emptyset$, $D_i \subset D_j$ or $D_j \subset D_i$.

        Hence, for $\Phi(x_0)$ with $\Phi(x_0) \cap \partial \mathbb{B}_s(p) = \bigsqcup^K_{i = 1}\gamma_i$, we can perform the neckpinch surgery starting from the innermost $D_i$ to the outermost $D_i$ in a the neighborhood $\mathbb{A}(\mathbf{0}; s - 1.9\zeta, s + 1.9\zeta)$ through removing cylinders in $\mathbb{A}(\mathbf{0}; s - 1.1\zeta, s + 1.1\zeta)$ and gluing disks. Since the total area of all the gluing disks is bounded by
        \[
            3\sum_i \cH^2(D_i) + 6(\sum_i \cH^1(\gamma_i)) \zeta \leq C^{(3)} \Lambda^2 r^2 \,,
        \]
        for some dimensional constant $C^{(3)}$. 
        
        In a general annulus $A(p; r, 2r)$, we only need to replace $C^{(3)}$ by a constant $C(C^{(1)}_g, C^{(2)}, C^{(3)})$, and thus, the statements (3), (4)', (5)' and (6)' above follow immediately.
    \end{proof}

    \begin{lem}[Local Deformation B: Shrinking] \label{Lem_Loc Deform B}
        In a $3$-dimensional Riemannian manifold $(M, g)$, let $p\in M$ and $0< r^- < r^+ < \injrad(M, g)/4$. Then there exists a smooth one parameter family of maps $\{\cR_t: M\to M\}_{t\in [0,1]}$ such that 
        \begin{enumerate}[label=\normalfont(\arabic*)]
            \item $\cR_t = \id$ when $t\in [0, 1/4]$; $\cR_t(B(p, r^-)) = \{p\}$ when $t\in [3/4, 1]$.
            \item $\cR_t$ is a diffeomorphism when $t\in [0, 3/4)$.
            \item $\cR_t|_{M \setminus B(p, r^+)} = \id$ and $\cR_t(B(p, r^\pm))\subset B(p, r^\pm)$ for every $t\in [0, 1]$.
            \item $\cR_t|_{B(p, r^-)}$ is $1$-Lipschitz for every $t\in [0, 1]$.
        \end{enumerate}
    \end{lem}
    \begin{proof}
        Let $\psi\in C^\infty(\R)$ be a non-decreasing function such that $\psi=0$ on $\R_{\leq 1/4}$, $\psi=1$ on $\R_{\geq 3/4}$, $\psi\in (0, 1)$ on $(1/4, 3/4)$.
        Let $\eta\in C^\infty_c(\R_{< r_+})$ be a decreasing function such that $\eta=1$ on $\R_{\leq r_-}$. 
        
        Working under normal coordinates of $(M, g)$ at $p$, it is easy to check that the following $\cR_t$ satisfies (1)-(4) in the Lemma: 
        \begin{align*}
          \cR_t(q) := \begin{cases}
            q\,, & \text{ if } q\notin B(p, r^+)\,, \\
            \left(1 - \psi(t)\eta(|q|)\right)q\,, & \text{ if } q\in B(p, r^+)\,.
          \end{cases}
        \end{align*}
    \end{proof}

    \begin{lem} \label{Lem_Loc Deform Composition}
        In $(M, g)$, for $N \in \N^+$, let $\{A(p_j; r_j^-, r_j^+)\}_{j=1}^N$ be a collection of pairwise disjoint annuli in $M$, where $r_j^+<\injrad(M, g)/4$. For each $j \in \{1, 2, \cdots, N\}$, let $\cR^j_t: M\to M$ be a map from Lemma~\ref{Lem_Loc Deform B} with $p_j, r_j^\pm$ in place of $p, r^\pm$.
        For simplicity, we denote
        \begin{align*}
            A_j := A(p_j; r^-_j, r^+_j)\,,\quad B^\pm_j:= B(p_j; r^\pm_j)\,.
        \end{align*}
        Then the following hold.
        \begin{enumerate}[label=\normalfont(\arabic*)]
            \item There exists a permutation $\sigma\in \mathfrak{S}_N$ such that for every $1\leq i<j\leq N$, either $B_{\sigma(i)}^+\cap B_{\sigma(j)}^+ = \emptyset$, or $B_{\sigma(i)}^+\subset B_{\sigma(j)}^-$. In this case, we call such a $\sigma$ {\em admissible}.
            \item For every admissible $\sigma$, the map 
            \[
               \cR^{\sigma(N)}_t\circ\cR^{\sigma(N-1)}_t\circ \cdots \circ \cR^{\sigma(2)}_t\circ\cR^{\sigma(1)}_t: M \to M
            \]
            is $1$-Lipschitz on every connect component of $M\setminus \bigcup_{j=1}^N A_j$; and it is independent of the choice of admissible $\sigma$.
        \end{enumerate}
    \end{lem}
    \begin{proof}
        Since $\{A_j\}$ is pairwise disjoint, for every $i \neq j$, we have exactly one of the three possibilities:
        \begin{itemize}
            \item $B_{\sigma(i)}^+\cap B_{\sigma(j)}^+ = \emptyset$;
            \item $B_{\sigma(i)}^+\subset B_{\sigma(j)}^-$;
            \item $B_{\sigma(j)}^+\subset B_{\sigma(i)}^-$.
        \end{itemize}
        Hence, to achieve (1), we can simply choose a permutation $\sigma$ such that for every $1 \leq i < j \leq N$, $r^+_{\sigma(i)} \leq r^+_{\sigma(j)}$.

        For (2), the $1$-Lipschitz property follows from Lemma~\ref{Lem_Loc Deform B} (3) and (4). The independence on the choice of admissible permutations follows from the fact that if $B^+_i \cap B^+_j = \emptyset$, then $\spt \cR^i_t \cap \spt \cR^j_t = \emptyset$.
    \end{proof}

\subsection{Combinatorial arguments}
    
    To extend the local deformations from the previous subsection to a global deformation along a Simon-Smith family, we rely on the following two useful combinatorial arguments, inspired by~\cite[Lemma~4.8,~Proposition~4.9]{Pit81} and~\cite[Lemma~5.7]{DLJ18}. 
    
    \begin{lem}[Combinatorial Argument I]\label{lem:combinatorial_I}
        For any positive integers $m \in \N^+$ and $q \in \N^+$, there exists $N = N(m, q) \in \N^+$ with the following property.

        Given any $k \in \N^+$ and a cubical complex $I(m, k)$, suppose that $\{\cF_\sigma\}_{\sigma\in I(m,k)}$ is a family of collections of open sets in $M$ assigned to each cell $\sigma$ of $I(m, k)$, where each collection $\cF_\sigma$ has the form
        \[
            \cF_\sigma = (\cO_{\sigma, 1}, \cO_{\sigma,2}, \cdots, \cO_{\sigma, q})
        \]
        and each $\cO_{\sigma, i} = \{U^1_{\sigma, i}, U^2_{\sigma, i}, \cdots, U^N_{\sigma, i}\}$ satisfies
        \begin{equation}\label{eqn:dist_diam}
            \dist(U^r_{\sigma, i}, U^s_{\sigma, i}) \geq 2 \min\{\diam(U^r_{\sigma, i}), \diam(U^s_{\sigma, i})\}\,,
        \end{equation}
        for all $\sigma \in I(m, k)$, $i\in \{1, 2, \cdots, q\}$ and $r, s \in \{1, 2, \cdots, N\}$ with $r \neq s$.
        Then we can extract a family of open sets 
        \[
            \{(U_{\sigma, 1}, U_{\sigma,2}, \cdots, U_{\sigma, q})\}_{\sigma \in I(m, k)}
        \]
        where $U_{\sigma, i} \in \cO_{\sigma, i}$ such that
        \[
            \dist(U_{\sigma, i} \cap U_{\tau, j}) > 0\,,
        \] whenever $\sigma, \tau \in I(m, k)$, $(\sigma, i) \neq (\tau, j)$ and $\sigma, \tau$ are faces of a common cell $\gamma$ of $I(m, k)$.

        Moreover, this is also true if the number of collections in each $\mathcal{F}_\sigma$ is no greater than $q$.
    \end{lem}
    \begin{proof}
        The proof is similar to that of~\cite[Lemma~5.7]{DLJ18}.
        
        Let $\{\cF_\sigma\}_{\sigma\in I(m,k)}$ be a family as in the lemma with some $N \in \N^+$ to be determined later.
        
        Note that by our assumption \eqref{eqn:dist_diam}, for each $U \in \cO_{\sigma, i}$ and $V^1, \cdots, V^l \in \cO_{\tau, j}$ with $\sigma \neq \tau$, if $\diam(U) \leq \diam(V^j)$ for all $j \in \{1, 2, \cdots, l\}$, then there exists at most one $V^j$ with 
        \[
            \dist(U, V^j) = 0\,.
        \]

        Here is the algorithm to generate the family of open sets:
        \begin{enumerate}
            \item We take an open set $U^r_{\sigma, i} \in \cO_{\sigma, i}$ with the smallest diameter among all the remaining open sets $U^\cdot_{\cdot, \cdot}$ and define $U_{\sigma, i} := U^r_{\sigma, i}$, and then remove all the open sets in $\cO_{\sigma, i}$.
            \item For every $\tau \in I(m, k)$ and every $j \in \{1, 2, \cdots, q\}$, where $(\sigma, i) \neq (\tau, j)$ and $\sigma, \tau$ are faces of some cell $\gamma \in I(m, k)$, and we remove all $V \in \cO_{\tau, j}$ with $\dist(U_{\sigma, i}, V) = 0$.
            \item Repeat Step (1) and Step (2), until there is no open set left. 
        \end{enumerate}

        In Step (1), we remove all the open sets in $\cO_{\sigma, i}$ after defining $U_{\sigma, i}$, so $U_{\sigma, i}$ can be only defined at most once. 
        
        In Step (2), by the argument above, before $U_{\tau, j}$ is defined, the size of $\cO_{\tau, j}$ will decrease at most by one if some $U_{\sigma, i}$ is defined where $(\sigma, i) \neq (\tau, j)$ and $\sigma, \tau$ are faces of some cell $\gamma \in I(m, k)$. Since there are at most $5^m$ many such $\sigma$, the size of $\cO_{\tau, j}$ will decrease at most by $5^m  \cdot q$ in total before $U_{\sigma, i}$ is defined. 
        
        Therefore, it suffices to choose
        \[
            N := 5^m \cdot q + 1\,,
        \]
        and then for every $\sigma \in I(m, k)$ and $i \in \{1, \cdots, q\}$, the open set $U_{\sigma, i}$ with required properties is defined before our algorithm halts.

        For the general case where $|\cF_\sigma| = q' < q$, we can replace each $\cF_\sigma$ by
        \[
            (\cO_{\sigma, 1}, \cO_{\sigma, 2}, \cdots, \cO_{\sigma, q'}, \cO_{\sigma, q'}, \cdots, \cO_{\sigma, q'})\,.
        \]
        The same arguments apply immediately.
    \end{proof}

    \begin{lem}[Combinatorial argument II]\label{lem:combinatorial_II}
        Given any integers $q, N \in \N^+$ and any $R_0 \in (0, \injrad(M, g) / 4)$, if $P$ is a finite set of at most $q$ points in $M$, then there exists a radius $R \in (5^{-2Nq^2}R_0, 5^{-2N} R_0)$ such that for every $p \in P$,
        \[
            A(p; R, 5^{2N} R) \cap P = \emptyset\,.
        \]
    \end{lem}
    \begin{proof}
        For $1 \leq l \leq q^2$, we define 
        \[
            s_l = 5^{-2N l} R_0\,, \quad r_l = 5^{-2N(l - 1)} R_0\,.
        \]

        For every $p \in P$, since $\# P \leq q$, there are at most $q - 1$ many $l$ such that
        \[
            A(p; s_l, r_l) \cap P \neq \emptyset\,.
        \]
        Since $q(q - 1) < q^2$, by the pigeon hole principle, there is a $l_0$ such that $R := 5^{-2N l_0} R_0$ satisfies the requirement of the lemma.
    \end{proof}

\subsection{Proof of Proposition~\ref{Prop_Technical Deformation}}
        The proof will proceed in three steps.

        \medskip
        
        \paragraph*{\bf Step 1. Set up}
        Let $C_g$ be as in Lemma~\ref{Lem_Loc Deform A} and $N := N(m, q)$ as in Lemma~\ref{lem:combinatorial_I}. Without loss of generality, we assume that 
        \[
            \varepsilon < 5^{-N-1} \min\left\{\injrad(M, g), \frac{1}{10000 \cdot 3^m q C_g }\right\}\,,
        \]
        and we set $\varepsilon_1:= 5^{-2N-10}\varepsilon$, and $\varepsilon_2:= 5^{-2Nq^2}\varepsilon_1$.
        
        Since $\mathscr{S}$ is a compact set, we can choose $\delta_0 = \delta_0(\mathscr{S}) > 0$ such that for every $\Sigma \in \GS^*(M)$, if $\bF([\Sigma], [\mathscr{S}]) < \delta_0$, then for every $r \in (\varepsilon_2, \varepsilon_1)$, 
        \begin{equation}\label{eqn:area_ratio}
            \sup_{p \in M} \left\{\cH^2\left(\Sigma\cap A(p, r, 2r)\right)\right\} \leq 100 r^2\,.
        \end{equation}

        We select $\delta_1 := \delta_1(\mathscr{S}, q, \varepsilon_2)$ from Lemma~\ref{Lem_Existence of small  nontrivial loop}, and set 
        \[
            \delta(\varepsilon, q, m, \mathscr{S}) := \min\{\delta_0, \delta_1\}\,.
        \]
        
        For a Simon-Smith family $\Phi$ from the proposition, by Lemma~\ref{Lem_Existence of small  nontrivial loop}, we see that for every $x \in X$ with $\mathfrak{g}(\Phi(x)) = 1$, there exists $p_x \in S^3$ such that $B_{\varepsilon_2}(p_x)$ contains a non-trivial loop in $\Phi(x)$. 
        
        For each $x \in X$, we set 
        \[
            \tilde{P}(x) := \begin{cases}
                P_\Phi(x) \cup \{p_x\} &\text{if } \mathfrak{g}(\Phi(x)) = 1\\
                P_\Phi(x) &\text{if } \mathfrak{g}(\Phi(x)) = 0\,.
            \end{cases}
        \]
        By Lemma~\ref{lem:combinatorial_II}, since $\# \tilde P(x) \leq q$, for each $x \in X$, take $r_x \in (\varepsilon_2, 5^{-2N}\varepsilon_1)$ such that for every pair $p, p' \in \tilde{P}(x)$, we have 
        \[
            p' \notin A(p; r_x, 5^{2N} r_x)\,.
        \]
        For each $l \in \{1, 2, \cdots, N\}$, we set 
        \[
            r_{x, l}:= 5^{2(l - 1)} r_x\,.
        \]

        By \eqref{eqn:area_ratio}, for every $1 \leq l \leq N$, $x \in X$ and $p \in \tilde P(x)$ we have
        \begin{align}
            \cH^2 \left(\Phi(x)\cap A(p; r_{x,l}, 2r_{x,l})\right) \leq 100 r_{x,l}^2 \,. \label{Equ_Uniform area bound in small annuli}
        \end{align}
        and thus, applying Lemma~\ref{Lem_Loc Deform A} to $\Phi(x)$ in $A(p; r_{x,l}, 2r_{x,l})$ with $\Lambda = 100$, we obtain the following data:
        \begin{enumerate}
            \item $s_{x,l,p}\in (r_{x,l}, 2r_{x,l})$,
            \item $\zeta_{x,l,p}\in (0, \min\{100 r_{x,l}, s_{x,l}-r_{x,l}, 2r_{x,l}-s_{x,l}\}/5)$,
            \item a neighborhood $O_{x,l,p}\subset X$ of $x$,
            \item and a Simon-Smith family
            \[
                H_{x,l,p}: [0, 1] \times O_{x,l,p}\to \GS(M)\,.
            \]
        \end{enumerate}
        For each $x \in X$, we set a neighborhood of $x$
        \[
            O_x \subset \bigcap_{l \in \{1, 2, \cdots, N\}, p \in \tilde P(x)} O_{x, l, p}\,,
        \]
        such that $y \mapsto \Phi(y) \setminus \bigcup_{p \in \tilde P(x)} B(p; r_{x, 1, p}, 2r_{x, 1, p})$ is continuous in the smooth topology for $y \in O_x$.

        Note that by the definition of $\tilde P(x)$, $\fg(\Phi(x) \setminus \bigcup_{p \in \tilde P(x)} B(p; r_{x, 1, p})) = 0$ and for every $y \in O_x$, we also have $\fg(\Phi(y) \setminus \bigcup_{p \in \tilde P(x)} B(p; r_{x, 1, p})) = 0$. Therefore, if $\fg(\Phi(y)) = 1$, then there exists at least one nontrivial loop contained in $\Phi(y) \cap \bigcup_{p \in \tilde P(x)} B(p; r_{x, 1, p})$.

        \medskip

        \paragraph*{\bf Step 2. Refinement}

        Since $X$ is compact, we can take a finite cover from $\{O_x\}_{x \in X}$. Then 
        we refine $X$ so that for every cell $\sigma$ of $X$, we can find an $x_\sigma \in X$ satisfying that every cell $\tau$ of $X$ is a subset of $O_{x_\sigma}$ provided that $\sigma$ and $\tau$ are faces of some cell $\gamma$ of $X$. Note that for each $\sigma$, the number of such $\tau$ is no more than $5^m$. 
        
        By applying Lemma~\ref{lem:combinatorial_I} with 
        \[
            \cF_\sigma = (\{A(p; r_{x_\sigma, i}, 2r_{x_\sigma, i})\}^N_{i = 1})_{p \in \tilde P(x_\sigma)}\,,
        \] 
        we obtain, associated with each $\sigma$, a collection of annuli denoted by
        \[
            (A_{\sigma, p})_{p \in \tilde P(x_\sigma)}
        \] such that whenever $\sigma$ and $\tau$ are faces of some cell $\gamma$ of $X$ and for any $p \in \tilde P(x_\sigma)$ and $p' \in \tilde P(x_\tau)$, we have 
        \[
            A_{\sigma, p} \cap A_{\tau, p'} = \emptyset\,,
        \]
        unless $(\sigma, p) = (\tau, p')$.
    
        For each cell $\sigma$ and $p \in \tilde P(x_\sigma)$, for convenience, we denote the data associated with each $A_{\sigma, p}$, constructed from Lemma~\ref{Lem_Loc Deform A} at the end of Step 1, as follows:
        \begin{enumerate}
            \item $s_{\sigma, p} \in (r_{\sigma, p}, 2r_{\sigma, p})$, 
            \item $\zeta_{\sigma, p} \in (0, \min\{s_{\sigma, p}-r_{\sigma, p}, 2r_{\sigma, p}-s_{\sigma, p}\}/5)$, 
            \item a neighborhood $O_{\sigma} := O_{x_\sigma}$,
            \item and a Simon-Smith family 
            \[
                H_{\sigma, p}: [0, 1] \times O_{\sigma} \to \GS(M)\,.
            \]
        \end{enumerate}
        We also denote
        \begin{align*}
            r_{\sigma, p}^\pm:= s_{\sigma, p}\pm \zeta_{\sigma, p}\,, \quad
            B_{\sigma, p}^\pm := B_{r_{\sigma, p}^\pm}(p_{\sigma, p})\,, \quad
            \hat{A}_{\sigma, p} := \overline{B_{\sigma, p}^+\setminus B_{\sigma, p}^-}\subset A_{\sigma, p} \,.
        \end{align*}

        \medskip

        \paragraph*{\bf Step 3. Construction of $H$}
        
        After the refinement in the previous step, suppose that $X$ is a cubical subcomplex of $I(m, k')$. For each cell $\sigma$ of $X$ and for each $x \in X$, we define
        \[
            \bd_\sigma(x) := \min\left\{\frac{2\|x- c_\sigma\|_{\ell_\infty}}{3^{-k'}} , 1\right\}\,.
        \]
        to be the normalized $l^\infty$ distant function to $\sigma$. Here, $c_\sigma$ is the center of $\sigma$.
        
        For each $t \in [0, 1/2]$ and each $x \in X$, we define $H(x, t)$ using pinching surgeries as described in Lemma~\ref{Lem_Loc Deform A}: 
        
        For every cell $\sigma$ of $X$ and every $p \in \tilde P(x_\sigma)$,
        \begin{align*}
            H(t, x)\cap A_{\sigma, p} = H_{\sigma, p}\left(\min\{8t(1-\bd_\sigma(x)), 1\}, x\right) \cap A_{\sigma, p}
        \end{align*}
        and that $H(t, x)=\Phi(x)$ outside $\bigcup_{\sigma \in X, p \in \tilde P(x_\sigma)} A_{\sigma, p}$. The map has the following properties for every $x \in X$.
        \begin{itemize}
            \item The set
            \[
                Z_x:= \{\sigma: \bd_\sigma(x)<1\}
            \] 
            has the property that every cell $\sigma$ in $Z_x$ is a face of the cell $\gamma$, which is the smallest cell containing $x$. By Step 2, the corresponding annuli $\{A_{\sigma, p}\}_{\sigma \in Z_x, p \in \tilde P(x_\sigma)}$ are pairwise disjoint, and the number is no greater than $3^m \cdot q$. In particular, the map $H(x, t)$ is well-defined.
            \item By Lemma~\ref{Lem_Loc Deform A} (4) and (5), $t \mapsto \mathfrak{g}(H(x, t))$ is non-increasing and $t \mapsto H(x, t)$ is a pinch-off process.
            \item By (\ref{Equ_Uniform area bound in small annuli}) and Lemma~\ref{Lem_Loc Deform A} (6), we have for every $t\in [0, 1/2]$, 
            \begin{align*}
                \cH^2(H(t, x)) &\leq \cH^2(\Phi(x)) + (3^m \cdot q) \cdot C_g 100^2 \varepsilon^2_1\\ 
                &\leq \cH^2(\Phi(x)) + \varepsilon \,.
            \end{align*}
            \item If $\bd_{\sigma}(x)\leq 3/4$, then for each $p \in \tilde P(x_\sigma)$, in $A_{\sigma, p}$, $H(1/2, x) = H_{\sigma, p}(1, x)$, thus by Lemma~\ref{Lem_Loc Deform A} (4), we know that 
            \[
                H(1/2, x)\cap \hat{A}_{\sigma, p} = \emptyset \,.
            \]
            \item If $\mathfrak{g}(\Phi(x)) = 1$ and $x \in \sigma$, then by the discussion at the end of Step 1, we know that for some $p \in \tilde P(x_\sigma)$, $B_{s_{\sigma, p}}(p)$ contains a nontrivial loop of $\Phi(x)$. Thus, the components of $H(1/2, x)$ outside $B_{\sigma}^+$ always has genus $0$.
        \end{itemize}

        Finally, for each $t \in [1/2, 1]$ and each $x \in X$, we define $H(x, t)$ using shrinking process as described in Lemma~\ref{Lem_Loc Deform B}:

        For each cell $\sigma$ of $X$ and $p \in \tilde P(x_\sigma)$, we let $\cR_{(\sigma, p), t}: S^3\to S^3$ be the one-parameter family of shrinking deformation from Lemma~\ref{Lem_Loc Deform B} with $r^- = r^-_{\sigma, p}$ and $r^+ = r^+_{\sigma, p}$.
        
        Then for each $x \in X$, we label $\{(\sigma, p) \mid \sigma \in Z_x,\ p \in P'(x_\sigma)\}$ as
        \[
            \{(\sigma_i, p_i)\}_{1 \leq i \leq N'}
        \]
        such that for every $1 \leq i < j \leq N'$, either $B_{\sigma_i, p_i}^+\cap B_{\sigma_j, p_j}^+ = \emptyset$, or $B_{\sigma_i, p_i}^+\subset B_{\sigma_j, p_j}^-$. The existence of such labeling follows from Lemma~\ref{Lem_Loc Deform Composition} (1). Then, we denote for simplicity, 
        \[
            \hat{\cR}^{(i)}_{t, x} := \cR_{(\sigma_i, p_i), (2t-1)(1-\bd_{\sigma}(x))} \,.
        \]
        Note that 
        \begin{itemize}
            \item By Lemma~\ref{Lem_Loc Deform B} (1), if $\bd_{\sigma_i}(x)\geq 3/4$ or $t\in [1/2, 5/8]$, then $\hat{\cR}^{i}_{t, x}=\id$; when $\bd_{\sigma_i}(x)< 3/4$, by the construction above, $H(1/2, x)\cap \hat{A}_{\sigma_i, p_i} = \emptyset$.
            \item if $\bd_{\sigma_i}(x)\leq 1/4$, then $\hat{\cR}^{(i)}_{1, x}(B_{\sigma_i, p_i}^-) \subset \{p_i\}$.
        \end{itemize}
    
        Now for $t\in [1/2, 1]$ and $x \in X$, we define
        \begin{align*}
            H(t, x):= \hat{\cR}^{(N')}_{t, x}\circ\cdots\circ \hat{\cR}^{(2)}_{t, x}\circ \hat{\cR}^{(1)}_{t, x}\left(H(x, 1/2) \right) \,.
        \end{align*}
        Intuitively, $\{H(t, x)\}_{t \in [1/2, 1]}$ is obtained by shrinking some connected components of $H(1/2, x)$ to points. Therefore, for every $x \in X$, $t\mapsto \mathfrak{g}(H(t, x))$ is non-increasing, and $\{H(t, x\}\}_{t \in [1/2, 1]}$ is a pinch-off process. This confirms statement (3) of the proposition.
    
        By Lemma~\ref{Lem_Loc Deform B} (3), (4) and Lemma~\ref{Lem_Loc Deform Composition} (2), we see that $\cH^2(H(t, x))$ is non-increasing for $t \in [1/2, 1]$. In particular, for every $t\in [1/2, 1]$, 
        \[
            \cH^2(H(t, x)) \leq \cH^2(\Phi(x)) + \varepsilon \,.
        \]
        This confirms statement (2) of the proposition.
    
        Furthermore, if $\fg(H(1/2, x)) = 1$ and $x \in \sigma$, then by the construction above, $H(1, x)$ consists of some connected components of $H(1/2, x)$ outside $\bigcup_{p \in \tilde P(x_\sigma)} B_{s_{\sigma, p}}(p)$, along with a finite set of points. From the final property of $H(1/2, x)$, it follows that $\fg(H(1, x)) = 0$. This verifies statement (1) of the proposition and completes the proof.

\subsection{Proof of Theorem~\ref{thm:mapping_cylinder}}
    Now we use Proposition~\ref{Prop_Technical Deformation} and the following local min-max theorem to prove Theorem~\ref{thm:mapping_cylinder}.

    \begin{thm}[Local min-max theorem {\cite[Theorem~6.1]{MN21}}]\label{thm:local-min-max}
        Let $\Sigma$ be a closed, smooth, embedded non-degenerate minimal surface with Morse index $k$ and multiplicity one, in a closed $3$-dimensional manifold $(M, g)$. For every $\beta > 0$, there exists $\varepsilon_0 > 0$ and a smooth family $\{F_v\}_{v \in \overline{\mathbb{B}}^k} \subset \operatorname{Diff}^\infty(M)$ such that
        \begin{enumerate}[label=\normalfont(\arabic*)]
            \item $F_0 = \operatorname{Id}$, $F_v = F^{-1}_v$ for all $v \in \overline{\mathbb{B}}^k$.
            \item $\|F_v - \operatorname{Id}\|_{C^1} < \beta$ for all $v \in \overline{\mathbb{B}}^k$.
            \item The function 
            \[
                A^\Sigma: \overline{\mathbb{B}}^k \to [0, \infty], v \mapsto \cH^2((F_v)_\# \Sigma)\,,
            \]
            is strictly concave.
            \item For every $T \in \cZ_2(M; \Z_2)$ with $\mathcal{F}(T, [\Sigma]) < \varepsilon_0$, we have
            \[
                \max_{v \in \overline{\mathbb{B}}^k} \bM((F_v)_\# T) \geq \cH^2(\Sigma)
            \]
            with equality only if $[\Sigma] = (F_v)_\# T$ for some $v \in \overline{\mathbb{B}}^k$.
        \end{enumerate}
    \end{thm}
    
    We proceed the proof in three steps.

    \medskip
    
    \paragraph*{\bf Step 1. Set up}
    
    Since $\cW_{L, \leq \mathfrak{g}_0}(M, g)$ consists of a varifold associated with a non-degenerate multiplicity-one minimal sphere $S$, let $\{F_v\}_{v \in \overline{\mathbb{B}}^k}$ be the smooth family associated with $S$ given by Theorem~\ref{thm:local-min-max} such that
    \begin{itemize}
        \item[(i)] For every $T_1, T_2 \in \cZ_n(M; \Z_2)$ and $v \in \overline{\mathbb{B}}^k$, 
        \[ 
            \bF((F_v)_\#(T_1), (F_v)_\#(T_2)) \leq 2\bF((T_1), (T_2))\,. 
        \]
    \end{itemize}
    We can choose smaller $\varepsilon_0 > 0$ such that 
    \begin{itemize}
        \item[(ii)] For every $T \in \cZ_2(M; \Z_2)$ with $\mathbf{F}(T, [S]) < \varepsilon_0$, the function
        \[
            A^T: \overline{\mathbb{B}}^k \to [0, \infty], v \mapsto \bM((F_v)_\# T)\,,
        \]
        is strictly concave with a unique maximum in $\mathbb{B}^k_{1/2}$.
        \item[(iii)] There exists $\varepsilon_1 > 0$ such that for every $T \in \cZ_2(M; \Z_2)$ with $\mathbf{F}(T, [S]) < \varepsilon_0$,
        \[
            \varepsilon_1 < \min_{v \in \partial \mathbb{B}^k, T } (\bM(S) - \bM((F_v)_\# S))\,.
        \]
    \end{itemize}
    Let $\mathscr{S} := \{(F_v)_\#(S)\}_{v \in \overline{\mathbb{B}}^k}$, which is a compact subset of  embedding of $\mathbb{S}^2$ into $(M, g)$. Let
    \[
        \varepsilon_1 := \min_{v \in \partial \mathbb{B}^k} (\bM(S) - \bM((F_v)_\# S)) > 0\,.
    \]

    Let $\delta(\varepsilon_1 / 10, N_P(\Phi) + 1, m + 1, \mathscr{S})$ be chosen as in Proposition~\ref{Prop_Technical Deformation}.

    \medskip

    \paragraph*{\bf Step 2. Initial Simon-Smith family}
    
    It follows from Theorem~\ref{thm:currentsCloseInBoldF} with $r = \min(\varepsilon_0, \delta(\varepsilon_1 / 10, \mathcal{S}, m) / 10)$ that there exists $\eta > 0$ and $\Phi_1 \in \Lambda(\Phi)$ such that 
    \[
        \bM(\Phi_1(x)) \geq L - 4\eta \implies [\Phi_1(x)] \in \bB^\bF_r([S])\,.
    \]
    
    In particular, there exists a Simon-Smith family 
    \[
        H_1: [0, 1] \times X \to \GS^*(M)
    \] with $H_1(0, \cdot) = \Phi$ and $H_1(1, \cdot) = \Phi_1$, which is a pinch-off process, because for every $x \in X$, $H_1(\cdot, x)$ is induced by a one-parameter group of diffeomorphisms.
    
    We can refine the cubical subcomplex $X$ so that, each cell $\sigma$ of $X$, exactly one of the following conditions holds:
    \begin{itemize}
        \item There exists a point $x_0 \in \sigma$ such that $\bM(\Phi_1(x_0)) \geq L - 2\eta$. In this case, for every $x \in \sigma$, $\bM(\Phi_1(x)) \geq L - 4\eta$.
        \item For every $x \in \sigma$, $\bM(\Phi_1(x)) < L - 2\eta$.
    \end{itemize}
    Let $X_0$ be the smallest cubical subcomplex of $X$ containing all cells that satisfy the first condition, and let $X_1$ be the smallest cubical subcomplex containing all cells that satisfy the second condition. Let
    \[
        Z := X_0 \cap X_1\,,
    \]
    which is also compact.
    
    Clearly for every $z \in Z$, 
    \[
        L - 4\eta \leq \cH^2(\Phi_1(z)) \leq L - 2\eta\,.
    \]
    and thus, $\bF([\Phi_1(z)], [S]) < r \leq \varepsilon_0$. For every $z \in Z$, let $A^z: \overline{\mathbb{B}}^k \to [0, \infty)$ be the function
    \[
        A^z(v) := \bM((F_v)_\#([\Phi_1(z)]))\,,
    \]
    By (ii) of Step 1, $A^z$ is strictly concave and has a unique maximum $m(z) \in \mathbb{B}^k_{1/2}$. By Proposition~\ref{prop:SS_AP}, $[\Phi_1]$ is continuous in the $\bF$-metric, so the function
    \[
        m: Z \to \mathbb{B}^k_{1/2}
    \]
    is continuous.

    It follows from Theorem~\ref{thm:local-min-max} (4) and $\bM([\Phi_1(z)]) = \cH^2(\Phi_1(z)) < L$, that $m(z) \neq 0$ for every $z \in Z$. Hence, there exists $\alpha > 0$ such that 
    \[
        \alpha \leq |m(z)| < 1/2\,.
    \]
    Consider the one-parameter flow $\{\phi^z(\cdot, t)\}_{t \geq 0} \subset \operatorname{Diff}(\overline{\mathbb{B}}^k)$ generated by
    \[
        v \mapsto -(1 - |v|^2) \nabla A^z(v)\,.
    \]
    For every $v \in \overline{\mathbb{B}}^k \setminus m(y)$, $t \mapsto A^z(\phi^z(v, t))$ is decreasing, and the limit $\lim_{t \to \infty} \varphi^z(v, t) \in \partial \mathbb{B}^k$. In particular, by (iii) of Step 1, we have
    \[
        A^z(0) - \lim_{t \to \infty} A^z(\phi^z(v, t)) \geq \varepsilon_1\,.
    \]

    By the compactness of $Z$, we can choose $T_0 > 0$ such that for every $z \in Z$, $t \mapsto A^z(\phi^z(v, t))$ is decreasing along $[0, T_0]$ and
    \[
        A^z(0) - A^z(\phi^z(v, T_0)) \geq \varepsilon_1 / 2\,.
    \]

    Now, let us consider a cubical complex $X'$ in $I^{m+1}$, whose underlying space is
    \[
        X_0 \times \{0\} \cup Z \times [0, 1] \cup X_1 \times \{1\}\,.
    \]
    From the construction, we also know that the map 
    \[
        f: X' \to X,\quad (x, y) \mapsto x
    \] is a homotopy equivalence. Note that $f$ is a surjective cubical map, after refining both $X$ and $X'$.
    
    We can define a Simon-Smith family $\Phi'_2: X' \to \GS^*(M)$ by
    \begin{align*}
        \Phi'_2(x, 0) &= \Phi_1(x)\,, \quad \forall x \in X_0\,,\\
        \Phi'_2(x, 1) &= \Phi_1(x)\,, \quad \forall x \in X_1\,,\\
        \Phi'_2(x, y) &= F_{\phi^z(0, 3 y T_0)}(\Phi'(x))\,, \quad \forall x \in Z, y \in [0, 1/3]\,,\\
        \Phi'_2(x, y) &= F_{\phi^z(0, T_0)}(\Phi'(x))\,, \quad \forall x \in Z, y \in [1/3, 2/3]\,,\\
        \Phi'_2(x, y) &= F_{\phi^z(0, 3 (1 - y) T_0)}(\Phi'(x))\,, \quad \forall x \in Z, y \in [2/3, 1]\,.
    \end{align*}

    Note that we can define two Simon-Smith family $\Phi': X' \to \GS^*(M)$ and $\Phi'_1: X' \to \GS^*(M)$ by extending $\Phi$ and $\Phi_1$ to the domain $X'$:
    \begin{align*}
        \Phi'(x, y) = \Phi(x)\,, \quad \forall (x, y) \in X'\,,
    \end{align*}
    and
    \begin{align*}
        \Phi'_1(x, y) = \Phi_1(x)\,, \quad \forall (x, y) \in X'\,,
    \end{align*}
    Similarly, we also have two deformations $H'_1: [0, 1] \times X' \to \GS^*(M)$ from $\Phi'$ to $\Phi'_1$ and $H'_2: [0, 1] \times X' \to \GS^*(M)$ from $\Phi'_1$ to $\Phi'_2$:
    \begin{align*}
        H'_1(t, x, y) = H_1(t, x)\,, \quad \forall t \in [0, 1], (x, y) \in X'\,,
    \end{align*}
    and
    \begin{align*}
        H'_2(t, x, 0) &= \Phi_1(x)\,, \quad \forall t \in [0, 1], x \in X_0\,,\\
        H'_2(t, x, 1) &= \Phi_1(x)\,, \quad \forall t \in [0, 1], x \in X_1\,,\\
        H'_2(t, x, y) &= F_{\phi^z(0, 3 t y T_0)}(\Phi'(x))\,, \quad \forall t \in [0, 1], x \in Z, y \in [0, 1/3]\\
        H'_2(t, x, y) &= F_{\phi^z(0, t T_0)}(\Phi'(x))\,, \quad \forall t \in [0, 1], x \in Z, y \in [1/3, 2/3]\,,\\
        H'_2(t, x, y) &= F_{\phi^z(0, 3 t (1 - y) T_0)}(\Phi'(x))\,, \quad \forall t \in [0, 1], x \in Z, y \in [2/3, 1]\,.
    \end{align*}
    
    For any $i \in {1, 2}$ and any $(x, y) \in X'$, $\{H'_i(t, x, y)\}_{t \in [0, 1]}$ is induced by an isotopy, and thus, a pinch-off process. Furthermore, we have
    \begin{equation}\label{eqn:NP_estimates}
        N_P(\Phi) = N_P(\Phi_1) = N_P(\Phi') = N_P(\Phi'_1) = N_P(\Phi'_2)\,.
    \end{equation}

    Furthermore, $\Phi'_2$ has the following properties:
    For every $(x, y) \in X'$ with $y \geq 1/3$,
    \begin{equation}\label{eqn:comp_upper_bound}
        \cH^2(\Phi'_2(x, y)) \leq L - 2\eta\,.
    \end{equation}
    For every $(x, y) \in X'$ with $y \leq 2/3$, by (i) of Step 1, 
    \begin{equation}\label{eqn:close_to_cS}
        \bF([\Phi'_2(x, y)], [\mathscr{S}]) \leq 2 r < \delta(\varepsilon_1 / 10, m + 1, N_P(\Phi'_2) + 1, \mathscr{S})\,.
    \end{equation}
    For every $(x, y) \in X'$ with $y \in [1/3, 2/3]$,
    \begin{equation}\label{eqn:bdry_upper_bound}
        \cH^2(\Phi'_2(x, y)) \leq L - 2\eta - \varepsilon_1 / 2\,.
    \end{equation}

    \medskip

    \paragraph*{\bf Step 3. Interpolation}

    Let $Y' := \{(x, y) \in X : y \leq 1/3\}$. By \eqref{eqn:close_to_cS}, for every $x' \in Y'$, $\Phi'_2(x') \in \GS(M)$, and we can apply Proposition~\ref{Prop_Technical Deformation} to $\Phi'_2|_{Y'}$ and obtain a deformation 
    \[
        H'_{3, Y'}: [0, 1] \times Y' \to \GS(M)\,.
    \]

    Consequently, we can define a Simon-Smith family $H'_3: [0, 1] \times X' \to \GS^*(M)$ as
    \begin{align*}
        H'_3(t, x, y) &= H'_3(t, x, y)\,, \quad \forall t \in [0, 1], (x, y) \in X', y \in [0, 1/3]\,,\\
        H'_3(t, x, y) &= H'_3((2 - 3y)t, x, 1/3)\,, \quad \forall t \in [0, 1], x \in Z, y \in [1/3, 2/3]\,,\\
        H'_3(t, x, y) &= \Phi'_2(x, y)\,, \quad \forall t \in [0, 1], (x, y) \in X', y \in [2/3, 1]\,.
    \end{align*}
    We denote $H'_3(1, \cdot, \cdot): X' \to \GS^*(M)$ by $\Phi'_3$.
    
    By Proposition~\ref{Prop_Technical Deformation} (1), for every $(x, y) \in X'$ with $y \in [0, 1/3]$, 
    \[
        \fg(\Phi'_3(x, y)) = \fg(H'_3(1, x, y)) = 0\,.
    \] 
    By Proposition~\ref{Prop_Technical Deformation} (2) and \eqref{eqn:bdry_upper_bound}, for every $(x, y) \in X'$ with $y \in [1/3, 2/3]$, 
    \[
        \cH^2(\Phi'_3(x, y)) = \cH^2(H'_3((2 - 3y)t, x, 1/3)) \leq \cH^2(\Phi'_2(x, 1/3)) + \varepsilon_1 / 10 \leq L - 2\eta\,.
    \]
    By \eqref{eqn:comp_upper_bound}, for every $(x, y) \in X'$ with $y \in [2/3, 1]$,
    \[
        \cH^2(\Phi'_3(x, y)) = \cH^2(\Phi'_2(x, y)) \leq L - 2\eta\,.
    \]
    Therefore, for every $(x, y) \in X'$,
    \begin{equation}\label{eqn:genus_bound}
        \cH^2(\Phi'_3(x, y)) \geq L - \eta \implies y \in [0, 1/3] \implies \mathfrak{g}(\Phi'_3(x, y)) = 0\,.
    \end{equation}
    
    Moreover, by Proposition~\ref{Prop_Technical Deformation} (3), for every $(x, y) \in X'$, $\{H'_3(t, x, y)\}_{t \in [0, 1]}$ is a pinch-off process.

    Finally, we can set $\Phi' = \Phi'_3$, which satisfies (1) of the theorem.
    Let $W := M_f$, and consider the deformation map
    \[
        H': [0, 1] \times X' \to \GS^*(M)\,,
    \]
    by 
    \[
        H'(t, x) = \begin{cases}
            H_1(3t, x) & t \in [0, 1/3]\,,\\
            H_2(3t - 1, x) & t \in [1/3, 2/3]\,,\\
            H_3(3t - 2, x) & t \in [2/3, 1]\,.
        \end{cases}
    \]
    Since $H'(0, x) = \Phi'(x) = \Phi(f(x))$, $H'$ induces a Simon-Smith family $H: W \to \GS^*(M)$ of genus $\leq 1$ satisfying (2)(a) and (2)(b) of the theorem.

    This completes the proof.

\section{Pinch-off process and mean curvature flow}\label{min-max_iiiOffAndMCF}
    In this section, we show that a pinch-off process (Definition~\ref{defn:pinch_off}) would possess certain topological properties similar to those of mean curvature flow. These properties will be used in \S \ref{sect:genus1caps}.

    Given a level set flow $\{M(t)\}_{t\in[0,\infty)}$ in some $(n+1)$-Riemannian manifold $N$, consider the complements of the time-slices,   
    \[
        W(t):=\{t\}\x (N\backslash M(t))\subset [0,\infty)\x N\,,
    \]
    and 
    \[
        W[t_1,t_2]:=\bigcup_{t\in[t_1,t_2]}W(t)\subset [0,\infty)\x N\,.
    \]
    In his work~\cite{Whi95}, B. White ingeniously applied the avoidance principle to show:
    \begin{itemize}
        \item For any $T>0$, any loop in $W[0,T]$ is homotopic to some loop in $W[0]$.
        \item If $N$ is a complete Riemannian manifold with Ricci curvature bounded from below, then for any $T>0$, the homomorphism
        \[
            H_{n-1}(W(T);\Z)\to H_{n-1}(W[0,T];\Z)
        \]
        induced by the inclusion map $W(T)\hookrightarrow W[0,T]$ is an monomorphism.
        \item The rank of $H_1(W(t);\Z)$ is non-increasing in $t$.
    \end{itemize}

    In the following, we prove a similar theorem regarding pinch-off processes. Given a compact Riemannian $3$-manifold $M$, and a pinch-off process $\Phi:[0,T]\to \GS(M)$, consider the complements of the time-slices, 
    \[
        W(t):=\{t\}\x (M\backslash S(t))\subset [0,T]\x M\,,
    \]
    and 
    \[
        W[t_1,t_2]:=\bigcup_{t\in[t_1,t_2]}W(t)\subset [0,T]\x M\,.
    \]
 
    \begin{prop}\label{prop:analogousToMCF} 
        If $\Phi:[0,T]\to \GS(M)$ is a pinch-off process, we have:
        \begin{enumerate}[label=\normalfont(\arabic*)]
            \item Any loop in $W(T)$ is homotopic in $W[0,T]$ to a loop in $W(0)$.
            \item If a loop is homologically non-trivial in $W(T)$, then it is homotopic in $W[0,T]$ to a homologically non-trivial loop in $W(0)$.
            \item By Definition~\ref{def:punctate_surf}, for each $t\in[0,T]$, $W(t)$ can be written as the disjoint union of two open sets $\ins(\Phi(t))$ and $\out(\Phi(t))$, such that both of their reduced boundaries are  $\Phi(t) \setminus \Phi(t)_\text{iso}$, and they both vary continuously in $t$ (as Caccioppoli sets). Then, 
            \[ \rank(H_1(\ins(\Phi(t));\Z)) \text{ and } \rank(H_1(\out(\Phi(t));\Z))\]
            are both non-increasing in $t$.
        \end{enumerate}
    \end{prop}

    In the remainder of this section, we will prove the above proposition. We will follow the notation in Definition~\ref{defn:pinch_off}. For simplicity, we assume that  the singular times $t_1,\cdots,t_n$ are distinct. It would be clear that the following strategy can still be applied when some of the $t_i$ coincide, by treating the neighborhood of each $(t_i,p_i)$ individually.
 
    \subsection*{Item (1)}
        Let $\gamma$ be a loop in $W(T)$. 
    In  order to homotope $\gamma$ back to some loop in $W(0)$, let us  construct a homotopy $\{\gamma_t\}_{t\in[0,T]}$ with $\gamma_T=\gamma$, and $\gamma_t\subset W(t)$ for each $t$, backwardly from time $t = T$ to $t = 0$. When we decrease $t$ from $T$ to $0$, if $t$ is not equal to any of the $t_i$, then near $t$,  $\Phi(\cdot)$ is moving by isotopy, so clearly the homotopy $\{\gamma_t\}$ can be constructed. So, let us assume $t$ is approaching  some  $t_i$ from above. Suppose  $(t_i,p_i)$ corresponds to ``shrinking some components into $p_i$". In this case, $\Phi(t)$ varies smoothly on some $[t_i,t_i+\epsilon]$, so we can define $\{\gamma_t\}$ up to $t=t_i$. As $t$ further decreases from $t_i$, the reversed shrinking process just produces some connected components, so clearly $\{\gamma_t\}$ can exist past $t_i$. Finally, the case where  $(t_i,p_i)$ corresponds to a neck-pinch surgery can be argued similarly (we can by perturbation assume $\gamma_{t_i}$ avoids $p_i$, and continue to decrease $t$).  Thus, the desired homotopy $\{\gamma_t\}_{t\in[0,T]}$ can be constructed, and so (1) is true.
    
\subsection*{Item (2)} To prove (2),  considering the homotopy  $\{\gamma_t\subset W(t)\}_{t\in[0,T]}$ of loops obtained previously, it suffices to show that if $\gamma_0$ is homologically trivial in $W(0)$ then $\gamma_T$ is homologically trivial in $W(T)$. Suppose $\gamma_0$ bounds a 2-chain in $W(0)$. We are going to show that $\gamma_t$ bounds some 2-chain $ C_t$ for each $t$. Note $C_t$ need not vary continuously in $t$.
        
        When we increase $t$ from $0$ to $T$, if $t$ is not equal to any of the  $t_i$, then as before, near $t$, $\Phi(\cdot)$ is moving by isotopy, so it is possible to construct the desired homotopy of 2-chains. So let us assume $(t_i,p_i)$ is a singularity. Now, we fix a continuous choice of inside region and outside region for $\Phi(t)$, for all $t$. Then there are four cases regarding $(t_i,p_i)$:
        
\begin{enumerate}[label=(\alph*)]
    \item $(t_i,p_i)$ is a neck-pinch point, such the solid cylinder region lies in the inside region.
    \item $(t_i,p_i)$ corresponds to shrinking some component of the inside region into the point $p_i$, at time $t_i$.
    \item $(t_i,p_i)$ is a neck-pinch point, such the solid cylinder region lies in the outside region.
    \item $(t_i,p_i)$ corresponds to shrinking some component of the outside region into the point $p_i$, at time $t_i$.
\end{enumerate}
Without loss of generality let us assume every $\gamma_t$ lies in the inside region. For the general case, we just consider the inside and outside components of $\gamma_t$ separately, and argue similarly.

In each of these four cases, let $[s_1,t_i]\x U$ be a spacetime neighborhood  within which the pinch-off process takes the form as described above. We can assume that we already have the family $\{ C_t\}$ defined  for $t$ up to $s_1$. Our goal is to extend it pass the time $t_i$.

\subsection*{Case (b)} 
By choosing the spacetime neighborhood small enough, we can assume that $[s_1,t_i]\x  U$ does not intersect $\gamma_t$  for every $t\in [s_1,T]$. Hence,    by the definition of case (b), {\it $C_{s_1}\cap U$ is a union of  $2$-cycles}. Removing these 2-cycles from $C_{s_1}$, and we can extend the family  onto the  time interval $[s_1,t_i]$, and after that.
        
\subsection*{Case (a)} From our  construction of the  family $\{\gamma_t\}_{t\in[0,T]}$, we may, by taking the surgery region $U$ to be sufficiently small, assume that $\gamma_{s_1}$ avoids $U$, even though $ C_{s_1}$ could intersect $U$. Now, the boundary of the cylinder $\Phi(s_1)\cap U$ consists of two loops. Pick one such loop, and   denote by $D$ the  disc that it  bounds in the sphere $\partial U$. When we intersect the  two chain $ C_{s_1}$ with $D$, we can assume that the cross section consists of finitely many loops $c_1,\cdots,c_m$, with each $c_j$  bounding a disc $d_j$ within $D$. By gluing two copies of each $d_j$ to $ C_{s_1}$, and slightly opening up a gap between the two $d_j$, one can obtain a 2-chain bounded by $\gamma_{s_1}$ that avoids $D$. By deforming this 2-chain, we can further assume that it avoids the whole ball $U$. Now,  let $C_{s_1}$  denote  this new  2-chain instead, replacing the old one. Then, we can extend the family $\{C_t\}$ pass the surgery time $t_i$. 

\subsection*{Case (c)} It is clear from the definition of Case (c) that we can construct the desired family $\{C_t\}$ pass the time $t_i$. Intuitively, this is because the neck-pinch is ``outward", so that the inside region is  gaining ground. 

\subsection*{Case (d)} Note  $C_{s_1}$ may intersect $U$. We perform surgery for $C_{s_1}$ along every loop component of $C_{s_1}\cap \partial U$, so that   it avoids the sphere $\partial U$. After discarding everything inside $U$,  we still call the new 2-chain $C_{s_1}$, replacing the old one. Now we can construct the desired family $\{C_t\}$ pass the time $t_i$.
    
        \subsection*{Item (3)} Let us just do the case for $\ins(\Phi(t))$, as the case for $\out(\Phi(t))$ is similar. Suppose we have loops $\gamma^1_T,\cdots,\gamma^k_T$ in $\ins(\Phi(T))$ such that the elements they induce in $H_1(\ins(\Phi(T));\Z)$ are linearly independent. From the proof of (1), we can construct for each $j=1,\cdots,k$ a homotopy $\{\gamma^j_t\subset \ins(\Phi(t))\}_{t\in[0,T]}$ of loops. To prove (3), it suffices to show that $\gamma^1_0,\cdots,\gamma^k_0$ are also linear independent in $H_1(\ins(\Phi(0));\Z)$. Suppose by contradiction that $\gamma^1_0+\cdots+\gamma^k_0$  bounds some 2-chain $C_0\subset \ins(\Phi(0))$. Following the proof of (2), we can similarly construct  for each $t$ a 2-chain $C_t\subset\ins(\Phi(t))$ bounded by $\gamma^1_t+\cdots+\gamma^k_t$. Putting $t=T$, contradiction arises. Thus (3) is true too.

\section{Topology of genus one cap}\label{sect:genus1caps} 

    The goal of this section is to prove Theorem~\ref{thm:trivialInFirsthomo}. 
    
    As explained at the end of \S \ref{sect:mainProof}, it suffices to show that for each genus $1$ cap $C\subset\dmn(\Xi)$, the map 
    \[
        i_*:H_1(C;\Z_2)\to H_1(\dmn(\Xi);\Z_2)
    \]
    induced by the inclusion $i:C\hookrightarrow\dmn(\Xi)$ is trivial. 
    
    In particular, for any fixed loop $c \subset C$, we aim to show that 
    \[
        [i(c)] = 0 \in H_1(\dmn(\Xi);\Z_2)
    \] For simplicity, we may just view $c$ as a subset of $\dmn(\Xi)$, we will show 
    \begin{equation}\label{eq:c=0}
         [c] = 0 \in H_1(\dmn(\Xi);\Z_2)\,.
    \end{equation}
    Note that, by the definition of a genus $1$ cap, the image
    \begin{equation}\label{eq:nearATorus}
        [\Xi](c)\subset  \bB^\bF_{d_0}([T])
    \end{equation}
    for some embedded minimal torus $T$. 

    The proof of (\ref{eq:c=0}) will consist of three main steps. First, we need to understand better the topology of the members of $\Xi|_{c}$.  In the second step, using family $\tilde\Xi$ obtained in Proposition~\ref{prop:XiPsiHomotopic}, we, in a certain sense, homotope the family $\Xi|_c$ back to some subfamily $\Phi|_{c_0}$ of $\Phi$, for some loop $c_0\subset Y$, while keeping track of the topology of the members of this homotopy. We will show that one can assume  $\Phi|_{c_0}$ to consist entirely Clifford tori (if $\Psi$ is viewed in $\mathbb S^3$), and to prove (\ref{eq:c=0}) it suffices to show that $c_0$ is homologically trivial in $Y$. In the third step, we prove that $c_0$ is homologically trivial  by studying the family $\Psi$.

\subsection{The family $\Xi|_c$}
    In Proposition~\ref{prop:XiPsiHomotopic}, we obtained a mapping cone $\tilde W$ containing $Y$ and $\dmn(\Xi)$, a map $\tilde F:[0,1]\x\tilde W\to\tilde W$, and a Simon-Smith family $\tilde\Xi:\tilde W\to \cS^*(S^3)$ of genus $\leq 1$. Then we can define a map 
    \[
        G_1:[0,1]\x S^1\to \tilde W
    \] 
    such that $G_1(1,\cdot)$ parametrizes a loop $c\subset\tilde W$, and for every $t \in [0,1]$ and every $\theta \in S^1$, 
    \[
        G_1(t,\theta) := \tilde F(t,G_1(1,\theta))\,.
    \] 
    By Proposition~\ref{prop:XiPsiHomotopic}, for each $\theta\in S^1$, the family $\{\tilde\Xi\circ G_1(t,\theta)\}_{t \in [0, 1]}$ is a pinch-off process.

    To prove Theorem~\ref{thm:trivialInFirsthomo}, we need more information about $\Xi(x)$ for $x \in C$ beyond the fact that $C$ is a genus $1$ cap. More precisely, we need that the two regions, inside and outside, enclosed by each surface resemble solid tori.

    \begin{prop}\label{prop:tauSigma}
        In a Riemannian $3$-sphere $(S^3,g)$, let $\Sigma$ be a smooth, embedded torus. There exists a constant $\tau(\Sigma)>0$ with the following property.
    
        Let $S\in \cS(S^3)$. By definition~\ref{def:punctate_surf}, let $\Omega,\Omega'\subset S^3$ be the two open subsets of $S^3 \setminus S$, whose reduced boundaries are both $S \setminus S_\text{iso}$. Assume that:
        \begin{enumerate}[label=\normalfont(\roman*)]
            \item $H_1(\Omega;\Z)$ and $H_1(\Omega';\Z)$ are both either $0$ or $\Z$.
            \item $\bF([S],[\Sigma])<\tau(\Sigma)$.
        \end{enumerate}
        Then, 
        \begin{enumerate}[label=\normalfont(\arabic*)]
            \item\label{item:homologyIsZ} $H_1(\Omega;\Z)\cong H_1(\Omega';\Z)\cong \Z.$
            \item\label{item:volClose} Fixing an inside direction for $\Sigma$ pointing towards $\ins(\Sigma)$, there is a unique choice of an inside direction for $S$ such that
            \[
                \vol(\ins(S)\triangle\ins(\Sigma)),\vol(\out(S)\triangle\out(\Sigma))<\tau(\Sigma)\,.
            \]
            Here, $\{\ins(S), \out(S)\} = \{\Omega, \Omega'\}$ and $\{\ins(\Sigma), \out(\Sigma)\}$ is the connected components of $S^3 \setminus \Sigma$.
            \item\label{item:uniqueGen} Fixing a generator $a_0$ for $H_1(\ins(\Sigma);\Z)$, there is a unique generator $a_1$ for $H_1(\ins(S);\Z)$ such that there exists a loop $\gamma_\ins\subset \ins(S)\cap \ins(\Sigma)$ that induces both $a_0$ and $a_1$.  And an analogous statement holds with $\out(\cdot)$ in place of $\ins(\cdot)$.
        \end{enumerate}
    \end{prop}
    \begin{rmk}
        The condition ``$H_1(\Omega;\Z)$ and $H_1(\Omega';\Z)$ are both either $0$ or $\Z$" is different from requiring ``$\fg(S)=0$ or 1". Indeed, consider this counterexample: Take three 2-spheres, identify all of their north poles as a point, and similarly, identify all of their three south poles as another point.
    \end{rmk}
        
    \begin{proof}
        First, choose a constant $d_1>0$ small enough such that $\Sigma$ has a tubular neighborhood of width $2d_1$. Fix an inside direction for $\Sigma$. For each $t\in (0,d_1)$, the boundary of the  $t$-neighborhood of $\Sigma$ consists of two smooth surfaces, $S_t$ and $S_{-t}$: We assume that $S_t$ lies outside $\Sigma$, while $S_{-t}$ inside. We denote $S_0:=\Sigma$. Now, for sufficiently small $\tau(\Sigma)$, from the fact that 
        \[
            \cF([S],[\Sigma])\leq  \bF([S],[\Sigma])<\tau(\Sigma)\,,
        \]
        item~\ref{item:volClose} follows. 

        Fix a small $\delta = \delta(\Sigma) >0$ to be determined by $\Sigma$ only. By $ \bF(|S|,|\Sigma|)<\tau(\Sigma)$ and by assuming $\tau(\Sigma)$ sufficiently small (depending on $\delta$ and $\Sigma$), we can assert that:
        By the coarea formula and Sard's theorem, there exists a  $t_0\in (-d_1,-d_1/2)$  such that:
        \begin{itemize}
            \item The intersection $S\cap S_{t_0}$ is transverse and is a finite union of  smooth loops.
            \item These loops are contained in finitely many balls $\{B_{r_i}(q_i)\}_i$ with
            \[
                \sum_i r_i<\delta\,.
            \]
        \end{itemize}
        
        Since $\Sigma$ is a torus, we can fix generating elements 
        \[
            a_0\in H_1(\ins(\Sigma);\Z)\cong \Z\,, \quad b_0\in H_1(\out(\Sigma);\Z)\cong \Z\,,
        \]
        such that they have linking number 1 in $S^3$. Then, by the bullet points above and $\cF([S],[\Sigma])<\tau(\Sigma)$, we can choose a sufficiently small $\delta = \delta(\Sigma)$, and thus $\tau(\Sigma)$, such that there exists a loop $\gamma_\ins\subset S_{t_0}\cap \ins(S)$ which avoids all loops in $S_{t_0} \cap S$ and satisfies $[\gamma_\ins]=a_0 \in H_1(\ins(\Sigma);\Z)$. Arguing similarly, we can choose a loop $\gamma_\out\subset \out(\Sigma)\cap \out(S)$ such that $[\gamma_\out]=b_0 \in H_1(\out(\Sigma);\Z)$. Noting the linking number $\link(\gamma_\ins,\gamma_\out)$ is $1$, we obtain item~\ref{item:homologyIsZ} of the proposition. Now, set 
        \[
            a_1:=[\gamma_\ins]\in H_1(\ins(S);\Z)\,, \quad b_1:=[\gamma_\out]\in H_1(\out(S);\Z)\,.
        \] 
        From $\link(\gamma_\ins,\gamma_\out)= 1$, item \ref{item:uniqueGen} also follows easily. 
    \end{proof}

    Recall that by assumption, $(S^3, g')$ has only finitely many embedded minimal tori. For each embedded minimal torus $\Sigma$, the above proposition gives a constant $\tau(\Sigma)>0$.
    
    Without loss of generality, we now assume that the constant $d_0>0$, chosen in the proof of Theorem~\ref{thm:main} (in the paragraph right before \S \ref{subsubsect:firstStage}) is smaller than $\tau(\Sigma)$ for every embedded minimal torus $\Sigma$. Then for each $x\in c\;(\subset \dmn(\Xi))$, we have the following:
    \begin{itemize}
        \item $\bF([\Xi(x)],[T])<\tau(T)$. In particular, $\Xi(x) \in \GS(M)$.
        \item Since the family $ \tilde\Xi\circ \tilde F(\cdot,x )$ is a pinch-off process (by Proposition~\ref{prop:XiPsiHomotopic}), by considering the shape of the members in original family $\Psi$ (see \S \ref{subsubsect:PsiIsSimonSmith}), we know by Proposition~\ref{prop:analogousToMCF}, that 
        $\Xi(x)$ satisfies the assumptions on $S$ in the first bullet point of Proposition~\ref{prop:tauSigma}.
    \end{itemize}
    As a result, we can apply Proposition~\ref{prop:tauSigma} to  $\tilde\Xi\circ G_1(1,\theta)$ for each $\theta\in S^1$ (as $G_1(1,\cdot)$ is a parametrization for the loop $c$). Thus, fixing an inside direction for the torus $T$, we have:
    \begin{enumerate}
        \item For each $\theta\in S^1$,
        \begin{equation}\label{eq:firstHomoZ}
            H_1(\ins(\tilde\Xi\circ G_1(1,\theta));\Z)\cong H_1(\out(\tilde\Xi\circ G_1(1,\theta));\Z)\cong \Z\,.
        \end{equation}
        \item For each $\theta$, there is a unique way to choose an inside direction for $\tilde\Xi\circ G_1(1,\theta)$ such that 
        \[
            \vol(\ins(T)\triangle\ins(\tilde\Xi\circ G_1(1,\theta))),\vol(\out(T)\triangle\out(\tilde\Xi\circ G_1(1,\theta)))<\tau(T)\,,
        \]
        and this choice is continuous in $\theta$.
        \item\label{item:uniqueGenXiG1} Let us fix a generator $a_0$ for $H_1(\ins(T);\Z)$. For each $\theta$, there is a unique generator $a(\theta)$ of $H_1(\ins(\tilde\Xi\circ G_1(1,\theta));\Z)$ such that there exists a loop $\gamma_a(\theta) \subset \ins(\tilde\Xi\circ G_1(1,\theta))\cap\ins(T)$ that induces both $a_0$ and $a(\theta)$. \item\label{item:uniqueGenXiG1outside} Let us fix a generator $b_0$ for $H_1(\out(T);\Z)$. For each $\theta$, there is a unique generator $b(\theta)$ of $H_1(\out(\tilde\Xi\circ G_1(1,\theta));\Z)$ such that there exists a loop $\gamma_b(\theta) \subset \out(\tilde\Xi\circ G_1(1,\theta))\cap\out(T)$ that induces both $b_0$ and $b(\theta)$.
    \end{enumerate}

    \subsection{Relating $\Xi|_c$ to $\Psi$}
    Then, using the item (2) above above, we can continuously and uniquely choose an inside direction for $\tilde\Xi\circ G_1(t,\theta)$ for each $(t,\theta)\in [0,1]\x S^1$. Moreover, from the fact that $\tilde\Xi\circ G_1(\cdot,\theta)$ is a pinch-off process for each $\theta$, we know by Proposition~\ref{prop:analogousToMCF}   that 
    \[
        \rank(H_1(\ins(\tilde\Xi\circ G_1(t,\theta));\Z))\,,\quad \rank(H_1(\out(\tilde\Xi\circ G_1(t,\theta));\Z))\,,
    \]
    are both non-increasing in $t$. Thus, together with (\ref{eq:firstHomoZ}) and the description of the family $\Psi$ given in \S \ref{subsubsect:PsiIsSimonSmith}, we obtain:
    \begin{itemize}
        \item For each $(t,\theta) \in [0, 1] \times S^1$, 
        \[
            H_1(\ins(\tilde\Xi\circ G_1(t,\theta));\Z)\cong H_1(\out(\tilde\Xi\circ G_1(t,\theta));\Z)\cong \Z\,.
        \]
        (Recall that the first homology, in $\Z$-coefficients, of an open subset of $S^3$ has no torsion.)
        \item The family $\tilde\Xi\circ G_1(0,\cdot)$, which can be viewed as the same as $\Psi|_{\tilde F(0,c)}$, consists entirely of smooth tori, where $\tilde F(0,c)$ is a loop in $Y$. 
    \end{itemize}

    Now, we prove a lemma stating that all smooth tori in $\Psi$ can be deformation retracted to Clifford tori.
    
    \begin{lem}
        If we view $\Psi$ as a Simon-Smith family in the unit $3$-sphere $\mathbb S^3$, and let $Z_0, Z_1\subset Y$ denote the sets of parameters corresponding to Clifford tori and smooth tori, respectively, then $Z_1$ can be deformation retracted onto $Z_0$.
    \end{lem}
    \begin{proof}
        Proposition~\ref{lem:behaviourSigmavtz}  tells us that for each $v\in {\mathbb B^4}$ and $z\in\tilde \cC$, we have an interval $I_{v,z}\subset (-\pi,\pi)$ such that for each $t\in I_{v,z}$,  
        \[
            \Sigma(v,t,z):=\partial\{x\in S^3:d(v,z)(x)<t\}
        \]
        is a smooth torus.  Thus, by shrinking the interval $(-\pi,\pi)$ into the point $\{0\}$, we can deform $I_{v,z}$ into the point $\{0\}$ for each $v $ and $z$.  This  gives us a deformation retraction of $Z_1$ onto a subset $Z_2$ of $Y$. Note that 
        \[
            Z_2\subset \{[(v,t,z)]\in Y:v\in {\mathbb B^4},t=0\}.
        \]
        Thus, $Z_2$ can be deformation retracted onto the set 
        \[
            Z_0=\{[(v,t,z)]\in Y:v=0,t=0\}\,,
        \]
        which parametrizes the set of Clifford tori under $\Psi$ in $\mathbb S^3$.
    \end{proof}
    Thus, by the above lemma, there exists a homotopy 
    \[
        G_2:[0,1]\x S^1 \to Y \subset \tilde W
    \]
    such that $G_2(0,\cdot)$ lies in $Z_0$ and $G_2(1,\cdot )=G_1(0,\cdot)$. Let $c_0$ denote the loop in $Z_0$ parametrized by $G_2(0,\cdot)$. By concatenating the two homotopies $G_1$ and $G_2$, we  obtain a homotopy 
    \[
        G_3:[0,1]\x S^1\to \tilde W
    \]
    such that:
    \begin{itemize}
        \item $G_3(0,\cdot)$ parametrizes $c_0\subset Z_0\subset \tilde  W$.
        \item $G_3(1,\cdot)$ parametrizes $c\subset \dmn(\Xi)\subset \tilde  W$. 
        \item For each $\theta \in S^1$, the family $\tilde\Xi\circ G_3(\cdot,\theta)$ is a pinch-off process.
        \item For each $(t,\theta) \in [0, 1] \times S^1$, 
        \[
            H_1(\ins(\tilde\Xi\circ G_3(t,\theta));\Z)\cong H_1(\out(\tilde\Xi\circ G_3(t,\theta));\Z)\cong \Z\,.
        \]
    \end{itemize}
    In particular, by the first two bullet points and the fact that $\tilde F$ induces a strong deformation retraction of $\tilde W$ onto $Y$,  to prove that $[c]=0$ in $H_1(\dmn(\Xi);\Z_2)$, it suffices to show that $[c_0]=0$ in $H_1(Z_0;\Z_2)$.

    Recall that $Z_0\cong \RP^2\x\RP^2$, so $H_1(Z_0;\Z_2) = \Z^2 \times \Z^2$ has 3 non-trivial elements: $(1,0), (0,1), (1,1)$. 
    Note, on one hand,  $(1,0)$ and $(0,1)$ (but not $(1,1)$) both correspond to 1-sweepouts under $\Psi$ as we discussed in \S \ref{subsect:7sweepout}. On the other hand, $\Xi|_c$ is not a 1-sweepout as it is close to a single embedded minimal torus in the flat topology by (\ref{eq:nearATorus}), and this implies $\Psi|_{c_0}$, which is homotopic to $\Xi|_c$ in the flat topology via $\tilde\Xi\circ G_3$,  is not a 1-sweepout either. Thus, $[c_0]$ cannot be $(1,0)$ or $(0,1)$.

    It suffices rule out the possibility that $[c_0]=(1,1)$. Let us extend the inside directions we chose for the family $\tilde\Xi\circ G_1$ to the family $\tilde\Xi\circ G_3$. Note that this extension is continuous and unique. As a result, we can ``naturally identify" the groups 
    \[
        H_1(\ins(\tilde\Xi\circ G_3(t,\theta));\Z)\,,\quad(t,\theta)\in [0,1]\x S^1\,,
    \] in the following sense. Let $k\in \N$. For each $i,j=1,\cdots,k$, define the closed rectangle 
    \[
        R_{i,j}=\left[\frac {i-1}k,\frac{i}k\right]\x \left[2\pi\frac {j-1}k ,2\pi\frac{j}k \right]\subset[0,1]\x S^1\,.
    \]

    \begin{prop}\label{prop:identifyHomoGroups}
        For sufficiently large $k$, there exist for each $(i,j)$ loops $\gamma^\ins_{i,j}$ and $\gamma^\out_{i,j}$ with linking number $1$ in $S^3$
        such that the following hold: If $R_{m,n}$ is one of the rectangles that intersect $R_{i,j}$ then for every $p\in R_{m,n}$ we have:
        \begin{enumerate}[label=\normalfont(\arabic*)]
            \item $\gamma^\ins_{i,j}$ lies in $\ins(\tilde\Xi\circ G_3(p))$ and generates $H_1(\ins(\tilde\Xi\circ G_3(p));\Z)$.
            \item $\gamma^\out_{i,j}$ lies in $\out(\tilde\Xi\circ G_3(p))$ and generates $H_1(\out(\tilde\Xi\circ G_3(p));\Z)$.
            \item  $\gamma^\ins_{i,j}$ and $\gamma^\ins_{m,n}$ are homologous in $\ins(\tilde\Xi\circ G_3(p))$. 
            \item $\gamma^\out_{i,j}$ and $\gamma^\out_{m,n}$ are homologous in $\out(\tilde\Xi\circ G_3(p))$.
        \end{enumerate}
    \end{prop}
    \begin{proof}
        For each $(t,\theta) \in [0, 1] \times S^1$, since
        \[
            H_1(\ins(\tilde\Xi\circ G_3(t,\theta));\Z)\cong H_1(\out(\tilde\Xi\circ G_3(t,\theta));\Z)\cong \Z\,,
        \]
        there exist $\gamma^\ins_{(t, \theta)} \subset \ins(\tilde\Xi\circ G_3(t,\theta))$ and $\gamma^\out_{(t, \theta)} \subset \out(\tilde\Xi\circ G_3(t,\theta))$ generating these two homology groups, respectively. Moreover, from the fact that $\tilde\Xi\circ G_3$ is a Simon-Smith family (Definition \ref{def:Simon_Smith_family} \ref{item:closedFamily}), there exists a neighborhood $O_{(t, \theta)} \subset [0, 1] \times S^1$ of $(t, \theta)$, such that $\gamma^\ins_{(t, \theta)} \subset \ins(\tilde\Xi\circ G_3(p))$ and generates 
        \[
            H_1(\ins(\tilde\Xi\circ G_3(p));\Z)\,,
        \]
        and $\gamma^\out_{(t, \theta)} \subset \out(\tilde\Xi\circ G_3(p))$ and generates 
        \[
            H_1(\out(\tilde\Xi\circ G_3(p));\Z)\,. 
        \]

        Clearly, $\{O_{(t, \theta)}\}_{(t, \theta) \in [0, 1] \times S^1}$ is an open covering of $[0, 1] \x S^1$. Since $[0, 1] \times S^1$ is compact, we can find a finite covering, denoted by
        \[
            \{O_l\}^t_{l = 1}\,,
        \]
        together with the corresponding loops $\{\gamma^\ins_l\}$ and $\{\gamma^\out_l\}$.
        Moreover, we can find a sufficiently large $k$ such that for each $i, j = 1, \cdots, k$, there exists some $l$, such that
        \[
            N(R_{i, j}) \subset O_l
        \]
        where $N(R_{i, j})$ is the union of $R_{i,j}$ and all the rectangles adjacent to $R_{i, j}$. 

        Let us first work on the rectangles $R_{k,1},\cdots,R_{k,k}$. 
        
        For each $R_{k, j}$, it contains some point $(1, \theta)$. Using the property (\ref{item:uniqueGenXiG1}) and (\ref{item:uniqueGenXiG1outside}) for the family $\tilde\Xi\circ G_3(1,\cdot)$, there exists a loop $\gamma_a(\theta)$ that lies in 
        \[
            \ins(\tilde\Xi\circ G_3(1,\theta))\cap\ins(T)\,.
        \]
        induces both $a_0$ and $a(\theta)$. Suppose that $R_{k, j} \subset O_l$, and then we can choose $\gamma^\ins_{k, j}$ to be either $\gamma^\ins_l$ or $-\gamma^\ins_l$ such that
        \[
            [\gamma^\ins_{k, j}] = a_0 \in H_1(\ins(T); \Z)\,,
        \]
        and thus,
        \[
            [\gamma^\ins_{k, j}] = a(\theta) \in H_1(\ins(\tilde\Xi\circ G_3(1, \theta));\Z)\,.
        \]
        Similarly, we can choose a $\gamma^\out_{k, j}$. This confirms statements (1) and (2) of the proposition for $(k, j)$. Note that the loops $\gamma^\ins_{k, j}$ and $\gamma^\out_{k, j}$ have linking number $1$, and statements (3) and (4) of the proposition follows from the fact that all $\gamma^\ins_{k, j}$ generates $a_0 \in H_1(\ins(T); \Z)$ and all $\gamma^\out_{k, j}$ generates $b_0 \in H_1(\out(T); \Z)$.

        Using backward induction, suppose that for some $i$, $\gamma^\ins_{i, j}$ and $\gamma^\out_{i, j}$ have been chosen for all $j = 1, 2, \cdots, k$. For each $(i - 1, j)$, as before, suppose that $R_{k, i} \subset O_l$, and then we can choose $\gamma^\ins_{k, j}$ to be either $\gamma^\ins_l$ or $-\gamma^\ins_l$ such that
        \[
            [\gamma^\ins_{i - 1, j}] = [\gamma^\ins_{i, j}] \in H_1(\ins(\tilde\Xi\circ G_3(i/k, 2\pi j/k));\Z)\,.
        \]
        We can similarly choose a $\gamma^\out_{k, j}$. This confirms statements (1) and (2) of the proposition for $(i - 1, j)$. 
        From the construction and the choice of $k$, that for each rectangle $R_{m,n}$ with $m \geq i - 1$, which intersects $R_{i-1,j}$,
        \[
            \link(\gamma^\ins_{m,n},\gamma^\out_{k,j})=1,\quad \link(\gamma^\out_{m,n},\gamma^\ins_{k,j})=1\,.
        \]
        This confirms statements (3) and (4) of the proposition in this case.

        In conclusion, one can find such a large $k$ and the corresponding loops.
    \end{proof}

\subsection{The family $\Phi|_{c_0}$}
    Focusing on the part in $\{0\}\x S^1\subset [0,1]\x S^1$ in Proposition~\ref{prop:identifyHomoGroups}, one obtains a map $L_0:S^1\x S^1\to S^3$ such that for each $\theta\in S^1$, the loop $L_0(\theta,\cdot) = \gamma^\ins_{0, j}$ for some $j$, which lies in $\ins(\tilde\Xi\circ G_3(0,\theta))$ and generates 
    \[
        H_1(\ins(\tilde\Xi\circ G_3(0,\theta));\Z)\cong \Z\,,
    \]
    satisfying conclusions (3) and (4) of the previous proposition.
    
    However, by the topological lemma below, this is impossible if $[c_0]=(1,1)$.

    \begin{lem}
        Let $d:[0,2\pi]\to Z_0$ with $d(0) = d(2\pi)$ parametrize a loop representing $(1,1)\in H_1(Z_0;\Z_2)$. We continuously choose an inward direction for $\Psi\circ d(\theta)$, for each $\theta\in [0,2\pi]$. Let 
        \[
            L:[0,2\pi]\x S^1\to S^3
        \] be a continuous map such that each loop $L(\theta,\cdot)$ lies in $\ins(\Psi\circ d(\theta))$ and generates 
        \[
            H_1(\ins(\Psi\circ d(\theta));\Z)\cong\Z\,.
        \]
        Then 
        $[L(0,\cdot)]$ and $[L(2\pi,\cdot)]$, which both are elements in $H_1(\ins(\Psi\circ d(0));\Z)$, differ exactly by a sign.
    \end{lem}
    \begin{proof}
        Let us view $\Psi$ as a Simon-Smith family in the unit $3$-sphere $\mathbb S^3$, so that each $\Psi\circ d(\theta)$ is an unoriented Clifford torus. Since $Z_0\cong \RP^2\x\RP^2$, $\pi_1(Z_0)=\Z_2\x\Z_2$, so we can also have $[d]=(1,1)$ in $\pi_1(Z_0)$. 

        Fix a pair $(x,y)\in S^2\x S^2$, such that the point $[(x,y)]$ in $S^2\x S^2/\sim$ with $(x,y)\sim (-x,-y)$  corresponds to the Clifford torus $\Psi\circ d(0)$ equipped with the chosen inside direction (recall the parametrization of the set of oriented Clifford tori described in \S ~\ref{subsect:spaceClifford}). We consider a path $\tilde d$ in $S^2\x S^2$ joining $(x,y)$ to $(-x,-y)$, which gives us the loop $d$ in $Z_0\cong \RP^2\x\RP^2$ under the projection map $S^2\x S^2\to\RP^2\x\RP^2$.

        Consider the set 
        \begin{align}\label{eq:cliffordTori}
            \{(\Omega,a_0):\;&\Omega\subset \mathbb S^3 \textrm{ is a solid torus bounded by some Clifford torus},\\
            \nonumber & a_0 \textrm{ is one of the two generators of }H_1(\Omega;\Z)\cong \Z\}.
        \end{align}
        Evidently, there is a natural topology that can equipped on this set, such that it is a double cover of the space 
        \[
            \{\Omega\subset \mathbb S^3:\Omega \textrm{ is a solid torus bounded by some Clifford torus}\}\,.
        \] 
        Note that $L$ induces a map $L'$ from $[0,2\pi]$ into   the space (\ref{eq:cliffordTori}). To prove the lemma, it suffices to show that, while the solid tori corresponding to $L'(0)$ and $L'(2\pi)$ are the same, the first homology classes of the solid torus given by $L'(0)$ and $L'(2\pi)$ differ by a sign.

        Following the discussion in \S \ref{subsect:spaceClifford}, we know that the space (\ref{eq:cliffordTori}) is homeomorphic to $ G_2^+(\R^4)$, the space of oriented 2-planes in $\R^4$. But $G_2^+(\R^4)\cong S^2\x S^2$. Thus, to prove the lemma, it suffices to show that the two oriented 2-planes corresponding to $(x,y),(-x,-y)\in S^2\x S^2$ are the same except differing by an orientation. This was already proved in~\cite[\S 3.4.3]{Nur16}.
    \end{proof}

    In conclusion, $[c_0]\ne (1,1)$ in $H_1(Y;\Z_2)$, so $[c_0]=(0, 0)$, as desired. This finishes the proof of Theorem~\ref{thm:trivialInFirsthomo}.

\printbibliography
\end{document}